\documentclass{amsart}

\usepackage[all]{xy}
\newdir{ >}{{}*!/-5pt/@{>}} % cf xyguide exercise 14
\SelectTips{cm}{}

\usepackage{amssymb,mathrsfs}
\usepackage[alphabetic]{amsrefs}

\newcommand{\anonym}{01}

\usepackage{tikz}
\usetikzlibrary{calc}

\newcommand{\ds}{\displaystyle}
\newcommand{\ev}{\mathrm{ev}}
\newcommand{\xibar}{\bar\xi}
\newcommand{\sh}{\mathrm{sh}}
\newcommand{\Sq}{Sq}

\let\P\relax
\newcommand{\can}{\mathrm{can}}
\newcommand{\P}{P}
\newcommand{\wEG}{\widetilde{EG}}

\newcommand{\sA}{\mathscr{A}}
\newcommand{\sB}{\mathscr{B}}
\newcommand{\sC}{\mathscr{C}}
\newcommand{\sD}{\mathscr{D}}

\newcommand{\bR}{\mathbb{R}}
\newcommand{\bC}{\mathbb{C}}
\newcommand{\bE}{\mathbb{E}}
\newcommand{\bF}{\mathbb{F}}
\newcommand{\bT}{\mathbb{T}}
\newcommand{\bZ}{\mathbb{Z}}
\newcommand{\longto}{\longrightarrow}
\renewcommand{\:}{\colon}
\newcommand{\<}{\langle}
\renewcommand{\>}{\rangle}
\newcommand{\ctensor}{\mathbin{\widehat{\otimes}}}
\newcommand{\cotensor}{\mathbin{\square}}
\newcommand{\THH}{T\hspace{-0.25mm}H\hspace{-0.25mm}H}
\newcommand{\wC}{\widehat{C}}

\newcommand{\alg}{\mathrm{Alg}}
\newcommand{\dgalg}{\mathrm{dgAlg}}

\newcommand{\leftsuspension}{S}

\newcommand{\bbamod}{\mathrm{Mod_\sA^{bb}}}
\newcommand{\bbacomod}{\mathrm{{}_{\sA_*}\!CoMod^{bb}}}

\newcommand{\Rmod}{{}_R\mathrm{Mod}}
\newcommand{\grRmod}{\mathrm{gr{}_{R}Mod}}
\newcommand{\grfpmod}{\mathrm{gr{}_{\bF_p}Mod}}
\newcommand{\bbgrRmod}{\mathrm{gr}_{R}\mathrm{Mod^{bb}}}
\newcommand{\bbgrfpmod}{\mathrm{gr}_{\bF_p}\mathrm{Mod^{bb}}}
\newcommand{\filgrfpmod}{\mathrm{fil gr}_{\bF_p}\mathrm{Mod}}
\newcommand{\filgrRmod}{\mathrm{fil gr}_{R}\mathrm{Mod}}

\newcommand{\rbbfilgrRmod}{\mathrm{fil}^{\mathrm{rbb}}\mathrm{gr}{}_R \mathrm{Mod}}
\newcommand{\rbbfilgrfpmod}{\mathrm{fil^{rbb}gr}_{\bF_p}\mathrm{Mod}}
\newcommand{\rbbfilamod}{\mathrm{fil^{rbb}Mod_\sA}}
\newcommand{\rbbfilamodc}{\mathrm{fil^{rbb}Mod^\wedge_\sA}}
\newcommand{\rbbfilacomodc}{\mathrm{fil}^\mathrm{rbb}\llap{${}_{\sA_*}\!$} \mathrm{CoMod}^\wedge}

\DeclareMathOperator{\im}{im}

\DeclareMathOperator{\aw}{aw}
\DeclareMathOperator{\sd}{sd}
\DeclareMathOperator{\Ext}{Ext}
\DeclareMathOperator{\Tor}{Tor}
\DeclareMathOperator{\cExt}{{}^{\mathit{c}} Ext}

\DeclareMathOperator{\diag}{diag}
\DeclareMathOperator{\AW}{AW}
\DeclareMathOperator{\gr}{gr}
\DeclareMathOperator{\fil}{fil}
\DeclareMathOperator{\inj}{inj}
\DeclareMathOperator{\Fil}{Fil}
\DeclareMathOperator{\Fun}{Fun}
\DeclareMathOperator{\map}{map}

\DeclareMathOperator{\sk}{sk}
\DeclareMathOperator{\Sp}{Sp}
\DeclareMathOperator{\Tot}{Tot}
\DeclareMathOperator{\TOT}{TOT}

\DeclareMathOperator*{\colim}{colim}
\DeclareMathOperator*{\hocolim}{hocolim}
\DeclareMathOperator*{\holim}{holim}
\DeclareMathOperator*{\rlim}{Rlim}

\DeclareMathOperator{\Hom}{Hom}

\newtheorem{theorem}{Theorem}[section]
\newtheorem{lemma}[theorem]{Lemma}
\newtheorem{proposition}[theorem]{Proposition}
\newtheorem{corollary}[theorem]{Corollary}
\theoremstyle{definition}
\newtheorem{definition}[theorem]{Definition}
\theoremstyle{remark}
\newtheorem{remark}[theorem]{Remark}
\theoremstyle{remark}
\newtheorem{example}[theorem]{Example}

\numberwithin{equation}{section}
\numberwithin{figure}{section}
\numberwithin{table}{section}

\usepackage[mathlines, pagewise]{lineno}
\if 10
  \linenumbers
\fi

 % Adams spectral sequences;


\title[Topological periodic homology of complex cobordism]
	{Continuous homology of topological periodic
	homology of complex cobordism}

  \if \anonym
  \else
    \author{Sverre Lun{\o}e--Nielsen}
    \address{Department of Engineering Sciences, University of Agder, Norway}
    \email{sverreal@uia.no}
    
    \author{John Rognes}
    \address{Department of Mathematics, University of Oslo, Norway}
    \email{rognes@math.uio.no}
  \fi

\begin{document}

\maketitle

\begin{abstract}
  We determine the continuous mod~$p$ homology of the topological
  periodic homology $TP(MU)$ of the complex cobordism spectrum, as a
  graded algebra with Steenrod operations.  The answer is given in
  terms of an explicit and purely algebraic construction~$C_+$,
  analogous to Singer's construction~$R_+$.  Its $\Ext$-algebra
  provides the $E_2$-term for a multiplicative Adams-type spectral
  sequence converging strongly to the homotopy of $p$-completed
  $TP(MU)$.
\end{abstract}

\setcounter{tocdepth}{3}
\tableofcontents{}

\section{Introduction}

Let $p$ be a prime.  We determine the continuous mod~$p$ homology
$H^c_*(TP(MU))$ of the topological periodic homology of the complex
cobordism spectrum, as a graded algebra with Steenrod operations.  Its
$\Ext$-algebra provides the starting page for a multiplicative Adams-type
spectral sequence converging strongly to $\pi_* TP(MU)^\wedge_p$.

\subsection{Context}
For $E_1$ ring spectra~$B$, the $p$-complete algebraic $K$-theory
$K(B)^\wedge_p$ \cite{EKMM97} and its close approximation, the topological
cyclic homology $TC(B)^\wedge_p$ \cite{BHM93}, \cite{DGM13}, capture
$p$-primary arithmetic information about~$B$.  In the case~$B = S$ these
invariants are related to the geometric topology of highly-connected
compact manifolds \cite{WJR13}, while for $B = \bZ$ they are related to
global and local number theory \cite{Tat63}, \cite{Kur92}, \cite{RW00}.
Interpolating between these, we have the complex cobordism spectrum~$MU$
\cite{Mil60} and the truncated Brown--Peterson spectra $BP\<n\>$
of heights $0 < n < \infty$ \cite{JW73}.  We can view $S \to MU$
as a Hopf--Galois extension \cite{Rog08}, with $MU$ being close
enough to~$S$ that one can approach $K(S)$ or $TC(S)$ by descent from
$K(MU)$ or~$TC(MU)$ \cite{DR18}, but far enough from~$S$ that $\pi_*(MU)$
remains manageable.  We therefore view the $E_\infty$ ring spectrum~$MU$
as an important intermediary, bridging between $S$ and $\bZ$, and we
seek to understand the homotopy-theoretic information captured by $K(MU)$
and $TC(MU)$, with the expectation that it will shed light
on $S$, $K(S)$ and~$TC(S)$.

\subsection{A homological approach}
Waldhausen's perspective~\cite{Wal84} on the Quillen--Lichtenbaum
conjectures, and the calculations for $BP\<1\>$ in~\cite{AR02},
motivated the so-called redshift conjectures predicting that
$K(BP\<n\>)$ and $TC(BP\<n\>)$ have telescopic height~$n+1$.  These
were confirmed by Hahn--Wilson in~\cite{HW22}, playing a role in the
disproof by Burklund--Hahn--Levy--Schlank~\cite{BHLS} of Ravenel's
telescope conjecture.  However, this bound on telescopic height also
means that one cannot directly detect height $\ge n+2$ phenomena in
$TC(BP\<n\>)$, and must therefore expect to go to $TC(MU)$ to gain
information about these for all~$n$.  The $TC$-calculations for
$BP\<n\>$ are manageable~\cite{AKACHR25} because they can be carried
out with coefficients in finite spectra of type~$n+1$, but for $MU$ no
such finite type coefficients are available, and a direct calculation
of $TC(MU)$ in terms of homotopy groups appears to be out of
reach~\cite{Rog20}.  A homological approach to topological cyclic
homology \cite{AR05}, \cite{BR05} instead endeavors to first determine
the mod~$p$ homology $H_* TC(MU)$ as a right $\sA$-module algebra, or
equivalently as a left $\sA_*$-comodule algebra, where $\sA$ and
$\sA_*$ denote the mod~$p$ Steenrod algebra and its dual,
respectively, and then to approach the complexity of $\pi_* TC(MU)$
through the multiplicative Adams spectral sequence
\[
E_2^{s,t} = \Ext_{\sA}^{s,t}(\bF_p, H_* TC(MU)) \Longrightarrow
\pi_{t-s} TC(MU)^\wedge_p \,.
\]
Having access to the algebra structure in this spectral sequence is
likely to be essential for any detailed calculations.

Using the Nikolaus--Scholze formalism~\cite{NS18}, we are considering
the cyclotomic spectrum $X := H \wedge THH(MU)$, where $H := H\bF_p$
denotes the mod~$p$ Eilenberg--MacLane spectrum with the trivial
cyclotomic structure.  Its topological cyclic homology $TC(X) \simeq
H \wedge TC(MU)$ \cite{CMM21}*{Def.~2.2, Thm.~2.7}
then has homotopy $\pi_* TC(X) \cong H_* TC(MU)$,
and comes equipped with a natural action by the Steenrod algebra,
arising from graded self-maps of~$H$.  In this paper, we shall determine
the topological periodic homology $TP(X) = (H \wedge THH(MU))^{t\bT}$,
where $\bT$ is the circle group acting in the standard way on~$THH(MU)$.
We refer to its homotopy $\pi_* TP(X) = H^c_* TP(MU)$ as the continuous
homology of $TP(MU)$, which also arises naturally \cite{LNR12}*{Def.~4.7}
as the limit of the homologies of a tower of bounded-below spectra with
homotopy limit the $\bT$-Tate construction on $THH(MU)$, in a manner
that is compatible with the right $\sA$-actions.  The corresponding
left $\sA_*$-coaction takes values in a completed tensor product $\sA_*
\ctensor H^c_* TP(MU)$, formed with respect to a specified filtration
on the continuous homology, yielding what we call a complete coaction.

\subsection{Continuous homology of $\bT$-Tate fixed points}
We recall in Subsections~\ref{sec:greenlees-may-filtration}
and~\ref{sec:hesselholt-madsen} the Greenlees--May~\cite{GM95} and
Hesselholt--Madsen~\cite{HM03} spectrum level filtrations on~$(H \wedge
THH(MU))^{t\bT}$, and deduce from~\cite{HR24} that they induce the same
ascending and multiplicative right $\sA$-module filtration $\{F_n H^c_*
TP(MU)\}_{n \in \bZ}$ on~$H^c_* TP(MU)$, which we call the ($\bT$-)Tate
filtration.  It agrees with the abutment filtration in the multiplicative
and strongly convergent homologically indexed $\bT$-Tate spectral sequence
\[
\hat E^2_{n,*} = \hat H^{-n}(\bT; H_* THH(MU))
	\Longrightarrow H^c_{n+*} TP(MU) \,,
\]
where $\hat H^{-*}(\bT; \bF_p) = P(t^{\pm1}) = \bF_p[t^{\pm1}]$ is
the Laurent polynomial ring generated by $t \in \hat H^2$.  The fact
that this is an upper half-plane spectral sequence translates to the
complete, Hausdorff and exhaustive Tate filtration being what we call
`relatively bounded below', later abbreviated to `rbb'.

We show in Section~\ref{sec:limit-of-adams-ss} that there is a
multiplicative and strongly convergent limit of Adams spectral
sequences
\begin{equation} \label{eq:limAdamsSS}
E_2^{s,t} = \Ext_{\sA}^{s,t}(\bF_p, H^c_* TP(MU))
	\Longrightarrow \pi_{t-s} TP(MU)^\wedge_p \,,
\end{equation}
which we (a little awkwardly) call the limit Adams spectral sequence.
Under an equivalence between right $\sA$-actions and complete
left $\sA_*$-coactions upon relatively bounded below filtered
graded $\bF_p$-vector spaces, the right $\sA$-module $\Ext$-groups
$\Ext_{\sA}^{*,*}(\bF_p, -)$ displayed above correspond to left
$\sA_*$-comodule continuous $\Ext$-groups $\cExt_{\sA_*}^{*,*}(\bF_p, -)$.
It is the latter groups that most naturally arise in the course of our
construction and analysis of this spectral sequence.

\subsection{The $\bT$-Singer construction}
In Subsection~\ref{sec:t-singer} we introduce an algebraic
construction $C_+(M_*; \sigma)$ that is to the homological $C_p$-Singer
construction $R_+(M_*)$ for a right $\sA$-module~$M_*$ as the Tate
cohomology of the circle group $\bT$ is to the Tate cohomology of its
order~$p$ subgroup~$C_p$, and which we call the (homological)
$\bT$-Singer construction.  It is defined for right $\sA$-modules~$M_*$
such that $\beta_*(x) = 0$ for all $x \in M_*$, where $\beta_*$ denotes
the right action by the mod~$p$ Bockstein element $\beta \in \sA$.
Moreover, $M_*$ needs to come equipped with a right $\sA$-linear
differential $\sigma \: S M_* \to M_*$, where $S M_*$ denotes the left
suspension of~$M_*$.  Our main example will be
\begin{equation} \label{eq:HTHHMU}
M_* = H_* THH(MU) \cong P(m_\ell \mid \ell\ge1) \otimes E(\sigma m_\ell
\mid \ell \ge 1) \,,
\end{equation}
with the differential~$\sigma$ induced by the left $\bT$-action.  We write
$E(\sigma m_\ell) = \Lambda(\sigma m_\ell)$ for exterior algebras.
Here $H_*(MU) = P(m_\ell \mid \ell\ge1)$ with $|m_\ell| = 2\ell$,
chosen so that $m_\ell \mapsto \xibar_k \in \sA_*$ for $\ell = p^k-1$,
and $m_\ell$ is left $\sA_*$-comodule primitive for $\ell \ne p^k-1$.
When $p=2$ one should read $\xibar_k$ as $\bar\zeta_k^2$.  In particular,
$\beta_*(x) = 0$ for all $x \in H_* THH(MU)$.

In general, we define
\begin{equation}
  c_+(M_*; \sigma) := \hat H^{-*}(\bT; \bF_p) \otimes M_*
\end{equation}
with a right $\sA$-action specified by the formulas
\begin{align}
\P^s_*(t^r \otimes x) &= \sum_k \binom{-1-r-s(p-1)}{s-pk}
	t^{r+(s-k)(p-1)} \otimes \P^k_*(x) \\
\beta_*(t^r \otimes x) &= - t^{r+1} \otimes \sigma(x)
\end{align}
for $p$ odd, and by
\begin{align}
\Sq^{2s}_*(t^r \otimes x) &= \sum_k \binom{-1-r-s}{s-2k}
	t^{r+s-k} \otimes \Sq^{2k}_*(x) \\
\Sq^1_*(t^r \otimes x) &=  t^{r+1} \otimes \sigma(x)
\end{align}
for $p=2$.  Here $\P^s_*$ and $\Sq^{2s}_*$ denote the right actions by
the Steenrod operations $\P^s$ and $\Sq^{2s}$, respectively, with $\Sq^1_*
= \beta_*$.  We give $c_+(M_*; \sigma)$ the ascending filtration $\{F_n
c_+(M_*; \sigma)\}_{n \in \bZ}$ where $t^r \otimes x$ first appears in
filtration
\[
\Fil(t^r \otimes x) = -2r - |x|(p-1) \,.
\]
Here $|x|$ denotes the degree of~$x \in M_*$.  Then
\begin{equation}
  C_+(M_*; \sigma) := c_+(M_*; \sigma)^\wedge
  = \lim_n \frac{c_+(M_*; \sigma)}{F_n c_+(M_*; \sigma)}
  \cong \hat H^{-*}(\bT; \bF_p) \ctensor M_*
\end{equation}
is defined as the completion of $c_+(M_*; \sigma)$ with respect to this
filtration.  An element of $C_+(M_*; \sigma)$ is a possibly infinite sum
$\sum_r t^r \otimes x_r$ of elements all in the same degree, with $x_r
= 0$ for $r \ll 0$.  The latter condition is automatically satisfied
if $M_*$ is bounded below, which we hereafter assume.  The completion
inherits a filtration from that on~$c_+(M_*; \sigma)$, which is complete,
Hausdorff, exhaustive and relatively bounded below.  Moreover, the right
$\sA$-action extends over the completion and respects the filtration.

When $M_*$ is a right $\sA$-module algebra, $\sigma$ is a derivation and
$p$ is an odd prime, we give $c_+(M_*; \sigma)$ the algebra structure
induced by the cup product in Tate cohomology, namely $(t^r \otimes x)
\cdot (t^s \otimes y) = t^{r+s} \otimes xy$.  For $p=2$ we must add a
correction term $t^{r+s+1} \otimes \sigma(x) \sigma(y)$ to this product.
In each case the algebra structure then extends to $C_+(M_*; \sigma)$,
making the latter a relatively bounded below complete, Hausdorff and
exhaustively filtered right $\sA$-module algebra.

\subsection{Main theorem}
Here is our main result, which we prove as Theorem~\ref{thm:finale-MU}
in the body of the paper.

\begin{theorem}
There is a homomorphism
\[
\Theta_{MU}^{\bT} \: H^c_* TP(MU)
	\longto C_+(H_* THH(MU); \sigma)
\]
of relatively bounded below complete filtered right $\sA$-module algebras.
The underlying homomorphism of graded $\bF_p$-vector spaces is an
isomorphism, and induces an isomorphism
\[
\Theta^{\bT}_{MU*} \: \Ext_{\sA}^{*,*}(\bF_p, H^c_* TP(MU))
	\overset{\cong}\longto
	\Ext_{\sA}^{*,*}(\bF_p, C_+(H_* THH(MU); \sigma))
\]
of $\Ext_{\sA}^{*,*}(\bF_p, \bF_p)$-algebras.
\end{theorem}

\begin{remark}
  The homomorphism $\Theta_{MU}^{\bT}$ is filtration-preserving in the
  weak sense that it takes $F_n$ of the domain to $F_n$ of the codomain,
  but it will strictly decrease the filtration index of some classes.
  Hence it is not an isomorphism of filtered objects, and this is the
  reason for the two-stage formulation of the theorem.

  As we have reviewed, the input $(H_* THH(MU), \sigma)$ for the
  $\bT$-Singer construction~$C_+$ is given by~\eqref{eq:HTHHMU}, and
  by the theorem, its output determines the starting page of the limit
  Adams spectral sequence~\eqref{eq:limAdamsSS} converging to
  $\pi_* TP(MU)^\wedge_p$.

  A completely analogous theorem holds for the Brown--Peterson
  spectrum~$BP$, with $H_* BP = P(\xibar_k \mid k\ge1)$ and
  $H_* THH(BP) \cong P(\xibar_k \mid k\ge1) \otimes E(\sigma\xibar_k
  \mid k\ge1)$, with essentially the same proof.
\end{remark}

To proceed to~$TC(MU)$, we should determine the topological negative
homology $TC^-(X) = (H \wedge THH(MU))^{h\bT}$, with $\pi_* TC^-(X)
= H^c_* TC^-(MU)$ the continuous homology of $TC^-(MU)$, and then
identify the equalizer of the two $E_\infty$ ring spectrum maps $\can$
and $\varphi \: TC^-(X) \to TP(X)$.  The case of $B = MU$ is special
in that
\if \anonym
  there are 
\else
  we have already established 
\fi
equivalences
\begin{align*}
\Gamma \: TF(MU)^\wedge_p &\overset{\simeq}\longto TC^-(MU)^\wedge_p \\
  \hat\Gamma \: TF(MU)^\wedge_p &\overset{\simeq}\longto TP(MU)^\wedge_p
 \if \anonym
 \,.
 \fi
\end{align*}
\if \anonym
  See ~\cite{LNR11}*{Cor.~1.2}.  Here,
\else
  in previous work~\cite{LNR11}*{Cor.~1.2}, where
\fi
$TF(MU) = \lim_{n,F} THH(MU)^{C_{p^n}}$ is the topological
Frobenius homology.  The main
remaining task is therefore to control the homomorphisms~$\can_*$
and~$\varphi_* \: H^c_* TC^-(MU) \to H^c_* TP(MU)$ in terms of the
description above.  We plan to complete this task in a
forthcoming paper.

\subsection{Continuous homology of $C_p$-Tate fixed points}
Momentarily ignoring filtrations, the isomorphism~$\Theta_{MU}^{\bT}$ is
realized as the map of cycles $\ker(\bar\sigma) \to \ker(\bar\sigma)$
for an isomorphism
\begin{equation} \label{eq:ThetaMU}
\Theta_{MU} \: (H^c_* THH(MU)^{tC_p}, \bar\sigma)
	\overset{\cong}\longto (R_+(H_* THH(MU)), \bar\sigma)
\end{equation}
of differential graded right $\sA$-module algebras.  Its domain is the
continuous homology $H^c_* THH(MU)^{tC_p} = \pi_* X^{tC_p}$ of
$THH(MU)^{tC_p}$, which we equip with a complete, Hausdorff and
exhaustive ($C_p$-)Tate filtration
$\{F_n H^c_* THH(MU)^{tC_p}\}_{n \in \bZ}$, much as
for~$TP(X) = X^{t\bT}$, making $H^c_* THH(MU)^{tC_p}$ a relatively
bounded below filtered right $\sA$-module algebra.  Moreover, there is
a residual $\bar\bT = \bT/C_p$-action on $X^{tC_p}$, with
$\bar\bT$-fixed points $X^{t\bT}$, which we discuss in
Section~\ref{sec:residual-circle-action}.  The fundamental class
$\bar e_1 \in H_1(\bar\bT)$ induces a degree~$+1$ filtration-shifting
differential and derivation~$\bar\sigma$ on $H^c_* THH(MU)^{tC_p}$,
making it a filtered differential graded algebra, compatibly with the
structures already mentioned.

We know from~\cite{BBLNR14} that the comparison map $G \: X^{t\bT} \to
(X^{tC_p})^{h\bar\bT}$, from $\bar\bT$-fixed points to $\bar\bT$-homotopy
fixed points of $X^{tC_p}$, is an equivalence.  Hence there is a
$\bar\bT$-homotopy fixed point spectral sequence
\[
E^2_{n,*} = H^{-n}(\bar\bT; H^c_* THH(MU)^{tC_p})
	\Longrightarrow H^c_{n+*} TP(MU) \,,
\]
where $H^{-*}(\bar\bT; \bF_p) \cong P(t) = \bF_p[t]$ and $d^2(t^r
\otimes x) = t^{r+1} \otimes \bar\sigma(x)$ for $r\ge0$ and $x \in H^c_*
THH(MU)^{tC_p}$.  We shall prove in Lemma~\ref{lemma:sigmaone-exact}
and Proposition~\ref{prop:theta-is-dga-morphism} that $\im(\bar\sigma) =
\ker(\bar\sigma)$ inside $H^c_* THH(MU)^{tC_p}$, so that this spectral
sequences collapses to filtration $n=0$ at the $E^3$-term.  Hence the
edge homomorphism
\[
F_* \: H^c_* TP(MU) \longto H^c_* THH(MU)^{tC_p}
\]
identifies the domain of $\Theta_{MU}^{\bT}$ with $\ker(\bar\sigma)$ in the
codomain of~$F_*$, where we write $F \: X^{t\bT} \to X^{tC_p}$ for the
map that forgets some invariance.

\subsection{The topological $C_p$-Singer construction}
The codomain of $\Theta_{MU}$ will be the homological Singer construction
$R_+(H_* THH(MU))$.  To properly introduce it, we first extend
\if \anonym
  results from~\cite{LNR12}
\else
  our previous work~\cite{LNR12}
\fi
on the topological $C_p$-Singer construction
$R_+(B) := (B^{\wedge p})^{tC_p}$, where $B$ now can be any spectrum
and $C_p$ acts on $B^{\wedge p}$ by cyclically permuting the factors.
Lifting an observation of Miller~\cite{BMMS86}*{\S II.3}
to the spectrum level, there is an equivalence
\[
R_+(B) \simeq \holim_q \Sigma^{1+q} D_{C_p}(\Sigma^{-q} B) \,,
\]
where $D_{C_p}$ denotes the $p$-th cyclic extended power.
Suppose hereafter that $B/p$ is bounded below.  Then the multiplicative
homological $C_p$-Tate spectral sequence
\[
\hat E^2_{n,*} = \hat H^{-n}(C_p; H_*(B)^{\otimes p})
	\Longrightarrow H^c_* R_+(B)
\]
is strongly convergent and collapses at the $E^2$-term.  The associated
filtration $\{F_n H^c_* R_+(B)\}_{n \in \bZ}$ of $H^c_* R_+(B) = \pi_*
(H \wedge B^{\wedge p})^{tC_p}$ is complete, Hausdorff and exhaustive,
with the classes detected by $u^i t^r \otimes x^{\otimes p} \in \hat
E^\infty_{-n,*}$ first appearing in Tate filtration $n = -i-2r$.
Here $i \in \{0,1\}$, $r \in \bZ$ and $\hat H^{-*}(C_p; \bF_p) = E(u)
\otimes P(t^{\pm1})$ with $u \in \hat H^1$ and $t \in \hat H^2$ for $p$
odd, while for $p=2$ we must replace $u^2 = 0$ with $u^2 = t$.

In Section~\ref{sec:topsinger} we construct a natural `$H$-based Tate
diagonal' map $\epsilon^H_B \: H \wedge B \to (H \wedge B^{\wedge
p})^{tC_p}$ of $H$-modules, and use it to define a natural
homomorphism
\[
\omega_B \: \hat H^{-*}(C_p; \bF_p) \otimes H_*(B)
	\longto H^c_* R_+(B) \,.
\]
Moreover, we show that $\omega_B$ is a monoidal natural transformation
between lax symmetric monoidal functors.  In particular, when $B$ is
a ring spectrum, $\omega_B$ is an $\bF_p$-algebra homomorphism.  This
strengthens 
\if \anonym
  results from~\cite{LNR12},
\else
  our previous work,
\fi
which did not account for multiplicative structure.

\subsection{The homological $C_p$-Singer construction}
In Section~\ref{sec:cp-singer} we let
\begin{equation}
r_+(M_*) := \hat H^{-*}(C_p; \bF_p) \otimes M_*
\end{equation}
for any graded $\bF_p$-vector space~$M_*$, and give it the filtration
$\{F_n r_+(M_*)\}_{n \in \bZ}$ where $u^i t^r \otimes x$ first appears
in filtration
\[
\Fil(u^i t^r \otimes x) = -i -2r - |x| (p-1) \,.
\]
We define the homological Singer construction
\begin{equation}
  R_+(M_*) := r_+(M_*)^\wedge = \lim_n \frac{r_+(M_*)}{F_n r_+(M_*)}
  \cong \hat H^{-*}(C_p; \bF_p) \ctensor M_*
\end{equation}
as the completion of $r_+(M_*)$ with respect to this filtration.
An element of $R_+(M_*)$ is a possibly infinite
sum $\sum_{i,r} u^i t^r \otimes x_{i,r}$ of elements all in the same
degree, with $x_{i,r} = 0$ for $r \ll 0$.  The latter condition is
automatically satisfied if $M_*$ is bounded below.
The cup product in Tate cohomology turns $r_+$ and $R_+$ into
lax symmetric monoidal functors to filtered graded $\bF_p$-vector
spaces.

We prove in Proposition~\ref{prop:omega-B} that $\omega_B \: r_+(H_*(B))
\to H^c_* R_+(B)$ is injective and strictly filtration-preserving,
and that its completion
\[
\omega_B^\wedge \: R_+(H_*(B)) \overset{\cong}\longto H^c_* R_+(B)
\]
is an isomorphism of filtered graded $\bF_p$-vector spaces.  More
precisely, $\omega_B$ takes $u^i t^r \otimes x \in r_+(H_*(B))$ to a
class detected by a unit times $u^i t^{r+(p-1)|x|/2} \otimes x^{\otimes
p}$ in the $C_p$-Tate spectral sequence converging to $H^c_* R_+(B)$.
This explains how the given filtration on $R_+(H_*(B))$ is derived
from the Tate filtration on~$H^c_* R_+(B)$.  If $B$ is a ring spectrum,
then $\omega_B^\wedge$ is an $\bF_p$-algebra isomorphism.

Next, when $M_*$ is any right $\sA$-module, we define a right $\sA$-action
on $r_+(M_*)$ by the formulas
\begin{align}
  \P^s_*(t^r \otimes x) &= \sum_k \binom{-1-r-s(p-1)}{s-pk}
                          t^{r+(s-k)(p-1)} \otimes \P^k_*(x) \\
                        &\qquad -\sum_k \binom{-1-r-s(p-1)}{s-pk-1}
                          u t^{-1+r+(s-k)(p-1)} \otimes \P^k_*\beta_*(x) \nonumber \\
\P^s_*(u t^r \otimes x) &= \sum_k \binom{-1-r-s(p-1)}{s-pk}
	  u t^{r+(s-k)(p-1)} \otimes \P^k_*(x) \\
\beta_*(u^i t^r \otimes x) &=
	\begin{cases}
		0 & \text{for $i=0$} \\
		t^{r+1} \otimes x & \text{for $i=1$}
	\end{cases}\\
\intertext{for $p$ odd, and by}
  \Sq^s_*(u^r \otimes x)
  &= \sum_k \binom{-1-r-s}{s-2k} u^{r+s-k} \otimes \Sq^k_*(x)
\end{align}
for $p=2$.  These respect the filtration on $r_+(M_*)$, and induce a
right $\sA$-action on $R_+(M_*)$.  (With the appropriate sign conventions,
this is dual to the left $\sA$-action on the cohomological $C_p$-Singer
construction, denoted $R_+(M)$ in~\cite{Sin81}, \cite{LNR12}*{\S 3.1}
and $T(M)$ in~\cite{AGM85}.)  After a series of explicit calculations,
we deduce in Proposition~\ref{prop:omega-b-is-a-linear} that these
operations satisfy that the homomorphisms
\[
R_+(H_*(B)) \overset{\omega_B^\wedge}\longto
	H^c_* R_+(B) \longto H_* \Sigma^{1+q} D_{C_p}(\Sigma^{-q} B)
\]
are right $\sA$-linear for each $q\ge0$, where the Steenrod operations
on the right-hand side are given by the classical Nishida relations.
In particular, we conclude in Proposition~\ref{prop:r-plus-is-monoidal}
that $\omega_B^\wedge$ is an isomorphism of relatively bounded below
complete filtered right $\sA$-modules, and of right $\sA$-module algebras
when $B$ is a ring spectrum.
See also Remark~\ref{rem:vandermonde} regarding a direct
algebraic proof of the Cartan formula for $r_+$, i.e., that $M_*
\mapsto r_+(M_*)$ is lax symmetric monoidal as an endofunctor of right
$\sA$-modules, which implies the corresponding statement for $R_+$.

\subsection{Singer's $\epsilon$-homomorphism}

Suppose $M_*$ is a bounded below right $\sA$-module.  In
Section~\ref{sec:singers-epsilon} we define a natural $\sA$-linear
homomorphism
\[
  \epsilon \: M_* \longto r_+(M_*)
\]
by the formulas
\begin{equation}
  \epsilon(x) = \sum_j (-1)^j \Bigl( t^{-j(p-1)} \otimes \P^j_*(x)
     + ut^{-1-j(p-1)} \otimes \P^j_* \beta_*(x) \Bigr)
\end{equation}
for $p$ odd and
\begin{equation}
  \epsilon(x) = \sum_j u^{-j} \otimes \Sq^j_*(x)
\end{equation}
for $p=2$.  (This dualizes to the $\Tor$-equivalence from the cohomological
$C_p$-Singer construction denoted~$d \: R_+(M) \to M$ in~\cite{Sin81}
and $\epsilon \: T(M) \to M$ in~\cite{AGM85}, cf.~\cite{LNR12}*{\S 3.1}.)
We prove in Proposition~\ref{prop:epsilon-is-epsilon} that
the $H$-based Tate diagonal map~$\epsilon^H_B$ and Singer's
$\epsilon$-homomorphism are compatible under the comparison
homomorphism $\omega_B$, i.e., that
\[
  \omega_B \circ \epsilon
    = (\epsilon^H_B)_* \: H_*(B) \longto H^c_* R_+(B) \,,
\]
confirming that the Tate diagonal is a spectrum-level lift of~$\epsilon$.

\subsection{A residual differential}
Suppose now that $(M_*, \sigma)$ is a differential right
$\sA$-module, satisfying that $\beta_*(x) = 0$ for all $x \in M_*$.
In Section~\ref{sec:residual-differential-on-rplus} we define a
degree~$+1$ filtration-shifting right $\sA$-linear differential
$\bar\sigma \: S r_+(M_*) \to r_+(M_*)$ by the formulas
\begin{align}
\bar\sigma(t^r \otimes x) &= t^r \otimes \sigma(x) \\
\bar\sigma(u t^r \otimes x) &= t^r \otimes x - u t^r \otimes \sigma(x)
\end{align}
for $p$ odd, and
\begin{align}
\bar\sigma(u^{2r} \otimes x) &= u^{2r} \otimes \sigma(x) \\
\bar\sigma(u^{2r+1} \otimes x)
	&= u^{2r} \otimes x + u^{2r+1} \otimes \sigma(x)
\end{align}
for $p=2$.  It induces a similar differential
$\bar\sigma \: S R_+(M_*) \to R_+(M_*)$ by passing to completions.  If
$(M_*, \sigma)$ is a differential graded right $\sA$-module algebra,
then $(r_+(M_*), \bar\sigma)$ is a filtered differential graded right
$\sA$-module algebra, and likewise for $(R_+(M_*), \bar\sigma)$.
Moreover, we show in Lemma~\ref{lemma:sigmaone-exact} that
$\im(\bar\sigma) = \ker(\bar\sigma)$ inside $r_+(M_*)$, as well as
in~$R_+(M_*)$.

We then define a strictly filtration-preserving right $\sA$-module
homomorphism $f_+ \: c_+(M_*; \sigma) \to r_+(M_*)$ by
\begin{equation}
  f_+(t^r \otimes x) = t^r \otimes x - ut^r \otimes \sigma(x)
\end{equation}
for $p$ odd, and
\begin{equation}
  f_+(t^r \otimes x) = u^{2r} \otimes x + u^{2r+1} \otimes \sigma(x)
\end{equation}
for $p=2$.  In each case the image of $f_+$ equals $\ker(\bar\sigma)
\subset r_+(M_*)$.  Passing to completions, we obtain a strictly
filtration-preserving right $\sA$-module homomorphism
\begin{equation}
  F_+ := f_+^\wedge \: C_+(M_*; \sigma) \to R_+(M_*) \,,
\end{equation}
inducing an isomorphism from the $\bT$-Singer construction
$C_+(M_*; \sigma)$ to the cycles $\ker(\bar\sigma) \subset R_+(M_*)$.
This explains how the given filtration on the $\bT$-Singer
construction is determined by the one on~$R_+(M_*)$.  For $M_*$
bounded below it exhibits $C_+(M_*; \sigma)$ as a relatively bounded
below complete filtered right $\sA$-module, and as a right
$\sA$-module algebra when $(M_*, \sigma)$ is a differential graded
right $\sA$-module algebra.

\subsection{Outline of proof}
To prove the main theorem, it remains to construct the relatively bounded
below complete filtered right $\sA$-module morphism~$\Theta_{MU}$
from~\eqref{eq:ThetaMU}, in such a way that it is an isomorphism
of differential graded algebras.  At the less structured level of
topological right $\sA$-modules and continuous homomorphisms,
\if \anonym
  the first task is achieved
\else
  we already achieved the first task
\fi
in~\cite{LNR11}*{Thm.~2.2}, taking $\Theta_{MU} =
\Phi_{MU}^{-1}$.  In order to also account for multiplicative structure,
without getting confounded by the many possible tensor products for
topological vector spaces, we shall reprove this result at the filtered
\if \anonym
  level.  This also yields
\else
  level, and can now give
\fi
a simpler argument for the pro-isomorphisms involved.

In Subsection~\ref{sec:1-skeleton-approx} we assemble a commutative
diagram~\eqref{eq:epsilon-cyctomic-zeroskeleton}
\[
\xymatrix{
r_+(H_*(B)) \ar[rr]^-{\omega_B} \ar[d]_-{r_+(\eta_*)}
	& & H^c_* R_+(B) \ar[d]^-{\eta^t_*} \\
r_+(H_*(THH(B))
	& H_*(B) \ar[ul]_-{\epsilon} \ar[d]^-{\eta_*}
		\ar[ur]^-{(\epsilon^H_B)_*}
	& H^c_* THH(B)^{tC_p} \\
	& H_* THH(B) \ar[ul]_-{\epsilon} \ar[ur]^-{\gamma^H_*}  \,,
}
\]
where $\gamma^H \: H \wedge THH(B) \to (H \wedge THH(B))^{tC_p}$ is
the cyclotomic structure map generically denoted $\varphi_p \: X \to
X^{tC_p}$.  In Subsection~\ref{sec:ext-iso} we use this to construct
a diagram~\eqref{eq:f-and-g-mu}
\[
\xymatrix@C+1pc{
r_+(H_*(MU)) \otimes E \ar[r]^-{\omega_{MU} \otimes 1}
	\ar[d]^-{f := r_+(\eta) \cdot \epsilon}
	& H^c_* R_+(MU) \otimes E \ar[d]^-{g := \eta^t_* \cdot \gamma^H_*} \\
r_+(H_* THH(MU)) & H^c_* THH(MU)^{tC_p}
}
\]
of relatively bounded below filtered right $\sA$-module algebras,
where $E := E(\sigma m_\ell \mid \ell\ge1) \subset H_* THH(MU)$.
(The homomorphisms $f$ and $f_+$ are not related.)  In~\cite{LNR11}
\if \anonym
  the authors
\else
  we
\fi
treated~$E$ as being concentrated in filtration~$0$, but
\if \anonym
  here~$E$ is given
\else
  here we instead give~$E$
\fi
the filtration pulled back from the Tate filtration on $H^c_*
THH(MU)^{tC_p}$ along $\gamma^H_*|_E$.  This has the advantage that
both $\omega_{MU} \otimes 1$ and~$g = \eta^t_* \cdot \gamma^H_*$ become
strictly filtration-preserving homomorphisms that induce isomorphisms on
filtration quotients, and therefore become isomorphisms after passage to
completions.

Moreover, we show in Lemma~\ref{lemma:f} that $f = r_+(\eta)
\cdot \epsilon$ is an unfiltered isomorphism, with a pro-inverse, so
that its completion $f^\wedge$ is also an unfiltered isomorphism.
However, $f$ does decrease the filtration of the classes
$\sigma m_\ell \in E$, so $f^{-1}$ is not filtration-preserving,
and $f^\wedge$ is not a filtered isomorphism.

At this point we can set $\Theta_{MU} := f^\wedge \circ
(\omega_{MU} \ctensor 1)^{-1} \circ (g^\wedge)^{-1}$, where
$\omega_{MU} \ctensor 1 = (\omega_{MU} \otimes 1)^\wedge$.  The only
information missing is that $\Theta_{MU}$ takes the topologically
defined differential~$\bar\sigma$ on $H^c_* THH(MU)^{tC_p}$ to
the algebraically defined differential~$\bar\sigma$.  This is
verified by direct calculation in Lemma~\ref{lemma:thm-6.4} and
Proposition~\ref{prop:theta-is-dga-morphism}, where the former is
equivalent to~\cite{LNR11}*{Thm.~6.4} and hinges on the specific
$\sA_*$-coaction on $H_*(BP) \subset H_*(MU)$, leading to the
relation $\epsilon(\xibar_k) = 1 \otimes \xibar_k + t^{-(p-1)} \cdot
\epsilon(\xibar_{k-1}^p)$.

\begin{remark}
  The composite
  \begin{multline*}
    g \circ (\omega_{MU} \otimes 1) \circ f^{-1}
	\: r_+(H_* THH(MU))
	\cong \hat H^{-*}(C_p; \bF_p) \otimes H_*(MU) \otimes E \\
	\longto H^c_* THH(MU)^{tC_p}
  \end{multline*}
  has the form of an inverse Cartier isomorphism~$C^{-1}$
  \cite{Kat70}*{Thm.~7.2}, being linear over $\hat H^{-*}(C_p; \bF_p)
  \cong H^c_* S^{tC_p}$, taking $m_\ell \in H_*(MU)$ to a class detected in
  the $C_p$-Tate spectral sequence by a unit times $m_\ell^p$, and taking
  $\sigma m_\ell \in E$ to a class detected by a unit times $m_\ell^{p-1}
  \sigma m_\ell$.  In both cases the unit is $t^{\ell(p-1)}$.  Hence the
  completed inverse~$\Theta_{MU}$ is formally like the original Cartier
  isomorphism~$C$~\cite{Car57}.
\end{remark}

\subsection{Acknowledgments}

We thank Tyler Lawson for correspondence about the
Alexander--Whitney map discussed in Section~\ref{sec:aw}.

The authors would like to thank the Isaac Newton Institute for
Mathematical Sciences, Cambridge, for support and hospitality during
the programme Equivariant homotopy theory in context, where work on
this paper was undertaken. This work was supported by EPSRC grant
EP/Z000580/1 and the Simons Foundation, Award
SFI-MPS-T-Institutes-00006117.

\section{Filtered modules, comodules and algebras}

Let $\sA^*$ denote the mod~$p$ Steenrod algebra~\cite{SE62}, equipped with the
usual product $\phi\: \sA^*\otimes \sA^*\to \sA^*$, coproduct
$\psi\: \sA^*\to \sA^*\otimes \sA^*$ and conjugation $\chi
\: \sA^* \to \sA^*$.  We often write $\sA$ instead
of~$\sA^*$.  We write $\sA_*$ for the dual~\cite{Mil58} of $\sA$, and also let
$\phi$, $\psi$ and~$\chi$ denote the product, coproduct and conjugation of $\sA_*$,
respectively.

As an algebra, $\sA_* = P(\bar\zeta_k \mid k\ge1)$ with
$|\bar\zeta_k| = 2^k-1$ for $p=2$, while
$\sA_* = E(\bar\tau_k \mid k\ge0) \otimes P(\bar\xi_k \mid k\ge1)$
with $|\bar\tau_k| = 2p^k-1$ and $|\bar\xi_k| = 2p^k-2$ for~$p$ odd.
Here
$\bar\zeta_k = \chi(\zeta_k)$ for $p=2$, and
$\bar\xi_k = \chi(\xi_k)$ and $\bar\tau_k = \chi(\tau_k)$ for~$p$ odd,
denote
the conjugates of the usual generators.  The mod~$2$ Steenrod algebra
is generated by the elements $\Sq^{2^n}$ for $n\geq 0$.  For $p>2$,
the Bockstein~$\beta$ together with the elements $\P^{p^n}$ generate
$\sA$.  For each integer $n\geq 0$, we let $\sA(n)$ be the sub-Hopf
algebra generated by the elements
$\Sq^{2^n}, \Sq^{2^{n-1}}, \ldots, \Sq^1$ when $p=2$, and by
$\P^{p^{n-1}}, \P^{p^{n-2}}, \ldots, \P^1, \beta$ when $p>2$.

Let $\sC$ be a category. The category $\gr\sC$ of
\emph{$\bZ$-graded objects in $\sC$} consists of $\bZ$-indexed
families of objects and maps in $\sC$, i.e., is equal to the functor
category $\sC^{\bZ}$ where $\bZ$ is considered to be a discrete
category.  We write $|x|$ to indicate the degree of a (homogeneous)
object $x$ in $\gr\sC$.  Limits and colimits in $\gr\sC$ are taken
degreewise.

\begin{definition}\label{dfn:finite-type-module}
  Let $R$ be a noetherian commutative ring.  A graded $R$-module is of
  \emph{finite type over $R$} if it is finitely generated over $R$ in
  each degree.
\end{definition}

Let $\Rmod$ be the category of left $R$-modules and $R$-linear homomorphisms.
The graded tensor product $\otimes = \otimes_R$ over~$R$ makes
$\grRmod$ a symmetric monoidal category, with unit $R$ considered as a
graded $R$-module concentrated in degree zero.

Let $\bbgrRmod$ denote the full subcategory of $\grRmod$ consisting
of graded $R$-modules that are bounded below.  This means
that for each $M_*$ in $ \bbgrRmod$ there exists an integer $k$ such
that $M_q=0$ for all $q < k$, i.e., that $M_*$ is $(k-1)$-connected
and $k$-connective.

\subsection{Bounded below modules and comodules over the Steenrod
  algebra}\label{sec:a-modules}

Consider the case $R=\bF_p$.  A bounded below right $\sA$-module is an
object $M_*$ in $ \bbgrfpmod$ together with a morphism
$\rho\: M_*\otimes \sA\longto M_*$ making the following diagrams
commute:
$$
\xymatrix@C16mm{
  M_*\otimes \sA\otimes \sA \ar[r]^{\rho\otimes 1}\ar[d]^{1\otimes \phi}
  & M_*\otimes \sA \ar[d]^{\rho}
  & M_*\otimes\bF_p \ar[r]^{1 \otimes \eta} \ar[dr]_\cong
  & M_*\otimes \sA\ar[d]^\rho\\
  M_*\otimes \sA \ar[r]^\rho
  & M_* && M_*\rlap{\,.}
}
$$
Alternatively, we may specify the structure map $\rho$ by its right
adjoint $\tilde\rho\: M_*\to \Hom(\sA, M_*)$, making the following
diagrams commute:
\begin{equation*}
  \xymatrix@C11mm{
    M_* \ar[r]^{\tilde\rho}\ar[d]^-{\tilde\rho}
    & \Hom(\sA, M_*) \ar[d]^{\Hom(\phi, 1)}
    & M_* \ar[r]^-{\tilde\rho} \ar[dr]_\cong
    & \Hom(\sA, M_*)\ar[d]^-{\Hom(\eta, 1)}\\
    \Hom(\sA, M_*) \ar[rd]_{\Hom(1,\tilde\rho)\hspace{5mm}}
    & \Hom(\sA\otimes\sA, M_*)\ar[d]^{\cong}
    && \Hom(\bF_p, M_*)\rlap{\,.} \\
    & \Hom(\sA, \Hom(\sA, M_*))
  }
\end{equation*}

The category of bounded below right $\sA$-modules and
$\sA$-module maps is denoted by $\bbamod$.  It becomes symmetric monoidal when
$M_* \otimes N_*$ is equipped with the right $\sA$-action given by the composite
\begin{multline*}
  M_* \otimes N_*
  \overset{\tilde\rho \otimes \tilde\rho}\longto \Hom(\sA, M_*) \otimes \Hom(\sA, N_*)\\
  \overset{\otimes}\longto \Hom(\sA \otimes \sA, M_* \otimes N_*)
  \overset{\Hom(\psi, 1)}\longto \Hom(\sA, M_* \otimes N_*) \,.
\end{multline*}

A bounded below left $\sA_*$-comodule is an object $M_*$ in
$ \bbgrfpmod$ together with a morphism
$\nu\: M_*\longto \sA_*\otimes M_*$ making the following diagrams
commute:
\begin{equation*}
  \xymatrix@C16mm{
    M_* \ar[r]^{\nu}\ar[d]^-{\nu}
    & \sA_*\otimes M_* \ar[d]^{\psi\otimes 1}
    & M_* \ar[r]^-{\nu} \ar[dr]_\cong
    & \sA_*\otimes M_* \ar[d]^{\epsilon\otimes 1}\\
    \sA_*\otimes M_* \ar[r]^-{1\otimes \nu}
    & \sA_*\otimes \sA_*\otimes M_* && \bF_p\otimes M_*\,.
  }
\end{equation*}
The category of bounded below left $\sA_*$-comodules and
$\sA_*$-comodule maps is denoted by $\bbacomod$.  It becomes symmetric
monoidal when $M_* \otimes N_*$ is equipped with the left
$\sA_*$-coaction given by the composite
\[
M_* \otimes N_*
  \overset{\nu \otimes \nu}\longto \sA_* \otimes M_* \otimes \sA_* \otimes N_*
  \overset{(23)}\longto \sA_* \otimes \sA_* \otimes M_* \otimes N_*
  \overset{\phi \otimes 1}\longto \sA_* \otimes M_* \otimes N_* \,.
  \]

% Line 15: Disse er "strong" symmetric monoidal, men la oss ikke nevne det.
There are strong symmetric monoidal forgetful functors from both $\bbamod$
and $\bbacomod$ to $\bbgrfpmod$.

\subsection{An isomorphism of symmetric monoidal categories}
\label{sec:isocat}
Let $V_*$ and $M_*$ be graded $\bF_p$-vector spaces, and let
$V^* = \Hom(V_*, \bF_p)$ denote the $\bF_p$-linear dual of~$V_*$.
There is a natural injective morphism of graded $\bF_p$-vector spaces
\begin{equation}
  \label{eq:iotamods}
  \iota\: V_*\otimes M_* \longto \Hom(V^*, M_*) \,,
\end{equation}
defined as the right adjoint of the composite
$$
V_*\otimes M_*\otimes V^*
\overset{(123)}\longto
V^*\otimes V_* \otimes M_*
\overset{\ev\otimes 1}\longto
\bF_p\otimes M_* \cong M_*\,,
$$
explicitly given by
$$
\iota(v\otimes m) = \{ f\mapsto (-1)^{(|v|+|m|)|f|} f(v)\cdot m \}\,.
$$
\begin{lemma}
  \label{lemma:iotamods}
  Let $V_*$ and $M_*$ be graded $\bF_p$-vector spaces that are bounded
  below, and assume that $V_*$ of finite type over $\bF_p$.  Then
  $\iota$ is an isomorphism.
\end{lemma}
\begin{proof}
  The homomorphism $\iota$ restricted to any
  given total degree~$q$ is equivalent to the composite
  \begin{equation}
    \label{eq:sum-prod}
    \bigoplus_{n} V_n\otimes M_{q-n}
    \longto
    \prod_{n} V_n\otimes M_{q-n}
    \overset{\prod \iota_{q,n}}\longto
    \prod_n \Hom(V^{-n}, M_{q-n}) \,,
  \end{equation}
  where the first homomorphism is the canonical injection.

  By hypothesis, there exists integers $k$ and $\ell$ such that $V_*$
  is $k$-connective and $M_*$ is $\ell$-connective.  This implies that
  the summands and factors of \eqref{eq:sum-prod} vanish unless~$n$ is
  in the finite range $k\leq n\leq q-\ell$.  Since the indexing set
  can be reduced to a finite one, it follows that first inclusion is
  an isomorphism.

  Finally, since $V_n$ is finite dimensional over $\bF_p$ for
  each~$n$, each $\iota_{q,n}$, and therefore their product, is an
  isomorphism.
\end{proof}

Any left $\sA_*$-comodule $M_*$ gives rise to a right $\sA$-module
having the same underlying graded $\bF_p$-vector space, with structure
map $\rho \: M_*\otimes \sA\to M_*$ left adjoint to the composite
$$
\xymatrix{
  \tilde\rho\: M_*\ar[r]^-\nu
  & \sA_*\otimes M_* \ar[r]^-\iota
  & \Hom(\sA, M_*)\,.  }
$$
In general, not all right $\sA$-modules arise this way.  However, if
$V_*=\sA_*$ in \eqref{eq:iotamods}, then it is implied by Lemma
\ref{lemma:iotamods} that passing from a bounded below left $\sA_*$-comodule
to a right $\sA$-module defines an isomorphism of categories
\begin{equation}
  \label{eq:isocat}
  \bbamod\cong \bbacomod\,.
\end{equation}
It is clearly strong symmetric monoidal.

\subsection{Filtered graded $R$-modules} \label{sec:filtered-r-modules}
Let $R$ be a noetherian commutative ring, such as $\bZ$, $\bZ_{(p)}$,
$\bZ_p$ or~$\bF_p$ for a prime~$p$.  We let $\filgrRmod$ be the
category consisting of objects $(P_*, \{F_nP_*\}_{n\in\bZ})$ where
$P_*$ is a graded $R$-module, and
$$
\ldots\subset F_{n-1}P_* \subset F_nP_*\subset \ldots \subset P_*
$$
is an ascending filtration of $P_*$.  We often omit the filtration
from the notation, and simply say $P_*$ instead of
$(P_*,\{F_nP_*\}_{n\in\bZ})$.  The morphisms of $\filgrRmod$ are the
filtration-preserving maps of graded $R$-modules.

We say that a morphism $f\:P_*\to Q_*$ of filtered graded
$R$-modules is \emph{strictly filtration-preserving} if
$f(F_n P_* \setminus F_{n-1} P_*) \subset F_n Q_* \setminus F_{n-1}
Q_*$, or equivalently if
$\bar f \: F_n P_*/F_{n-1} P_* \to F_n Q_*/F_{n-1} Q_*$ is injective
for each~$n$.  Note that a morphism $f\: P_*\to Q_*$ is strictly
filtration-preserving if~$P_*$ has the pullback filtration, i.e.,
if $F_nP_* = f^{-1}F_nQ_*$ for each $n\in\bZ$.

We say that the filtration $\{F_nP_*\}_n$ is \emph{exhaustive} if
$\colim_n F_nP_* = P_*$, \emph{Hausdorff} if $\lim_n F_nP_*=0$, and
\emph{complete} if $\rlim_n F_nP_*=0$.

The following characterization of isomorphisms in $\filgrRmod$
is elementary.
\begin{lemma}
  \label{lemma:filisos}
  A morphism $f\:P_*\to Q_*$ in $\filgrRmod$ is an isomorphism if and
  only if $f$ is an isomorphism in $\grRmod$ that also induces an
  isomorphism of quotients $P_*/F_nP_*\cong Q_*/F_nQ_*$ for every~$n$.
  \qed
\end{lemma}

\begin{definition}
  We say that an ascending filtration of graded $R$–modules
  $$
  \ldots\subset F_{n-1}P_* \subset F_nP_*\subset \ldots \subset P_*
  $$
  is \emph{relatively bounded below}, or \emph{rbb} for short, if there
  exists an integer $k$ such that $P_*/F_nP_*$ is $(n+k)$-connected for
  each~$n$.  In other words, $F_nP_q = P_q$ for all $n\geq q-k$.

  We let $\rbbfilgrRmod$ be the full subcategory of $\filgrRmod$
  consisting of objects $(P_*, \{F_nP_*\}_{n\in\bZ})$ where $P_*$ is a
  graded (unbounded) $R$-module, and $\{F_nP_*\}_{n}$ is a relatively
  bounded below filtration of $P_*$.
\end{definition}

Note that any relatively bounded below filtration is necessarily
exhaustive since any given homogeneous submodule $P_q$ is eventually
contained in some $F_nP_*$.

We can identify $\bbgrRmod$ with the full subcategory of
$\rbbfilgrRmod$ consisting of bounded below graded $R$-modules
$M_*$ having the \emph{discrete filtration} given by $F_nM_* = M_*$
for all $n\geq 0$ and $F_nM_*=0$ otherwise.  If $M_*$ is
$k$-connective, then the discrete filtration has the property that
$M_*/F_nM_*$ is $(n+k)$-connected for each~$n$.

\subsection{Towers of bounded below graded $R$-modules}
Let
$$
\ldots \longto Q[n-1]_* \longto Q[n]_*
\longto Q[n+1]_*\longto \ldots
$$
be a tower of graded $R$-modules and let $Q_*=\lim_n Q_*[n]$.
The notation~$[n]$ here refers to the indexing
of the tower, not a degree shift.  Then
$Q_*$ is a graded $R$-module filtered by the \emph{kernel filtration}
given by $F_nQ_* = \ker (j_{n}\: Q_* \to Q_*[n+1])$, where $j_{n}$ is
the canonical map from the limit to the $(n+1)$-th stage of the tower.

Suppose that there exists an integer $k$ such that $Q[n+1]_*$ is
$(n+k)$-connected for each $n$.  Since $j_{n}$ induces an injective
map $Q_*/F_nQ_* \to Q[n+1]_*$, it follows that also $Q_*/F_nQ_*$ is
$(n+k)$-connected for each $n$, and therefore that the kernel
filtration of $Q_*$ is relatively bounded below.

\begin{lemma}\label{lemma:f-iso}
  Let
  $$
  \{ f[n] \}_n\: \{P[n]_*\}_n \longto \{Q[n]_*\}_n  %_{n\in\bZ}
  $$
  be a map of towers of graded $\bF_p$-vector spaces.  Denote the
  limits of the two towers by $P_*$ and $Q_*$, respectively.

  Assume that $f=\lim_n f[n]\: P_*\to Q_*$ is an isomorphism in
  $\grfpmod$, and that $P[n]_*$ and $Q[n]_*$ are bounded below for
  each $n$.  Then the induced map
  $$
  \lim_n V_*\otimes P[n]_* \longto \lim_n V_*\otimes Q[n]_*
  $$
  is an isomorphism in $\grfpmod$, for any bounded below graded
  $\bF_p$-vector space $V_*$ of finite type over $\bF_p$.
\end{lemma}

\begin{proof}
  There is a commutative diagram
  $$
  \xymatrix@C25mm{
    \Hom (V^*, P_*) \ar[r]^{\Hom(1, f)} \ar[d]^{\cong}
    & \Hom (V^*, Q_*) \ar[d]^{\cong} \\
    \lim_n \Hom (V^*, P[n]_*) \ar[r]^{\lim_n \Hom(1, f[n])}
    & \lim_n \Hom (V^*, Q[n]_*) \\
    \lim_n V_*\otimes P[n]_* \ar[r]^{\lim_n 1\otimes f[n]} \ar[u]_{\lim_n\iota}
    & \lim_n V_*\otimes Q[n]_*\rlap{\,.} \ar[u]_{\lim_n\iota}
  }
  $$
  The upper horizontal map is an isomorphism by the assumption on~$f$.
  The lemma follows since both maps labeled $\lim_n \iota$ are
  isomorphisms by Lemma \ref{lemma:iotamods}.
\end{proof}

\subsection{Completion}\label{sec:completion}
We define the \emph{completion} of a filtered graded $R$-module $P_*$
to be the limit
$$
P^\wedge_* = \lim_{n} P_*/F_nP_* \,,
$$
equipped with the kernel filtration
\begin{equation}
  \label{eq:filtration-completion}
  F_nP^\wedge_* := \ker\left( j_n\: P^\wedge_* \longto P_*/F_nP_*
  \right) \,.
\end{equation}
For each~$n$, the restricted filtration
$\{F_mP_* \cap F_nP_*\}_{m\in\bZ}$ gives $F_nP_*$ the structure of a
filtered graded $R$-module, and it makes sense to consider its
completion $(F_nP_*)^\wedge$.  Moreover, each inclusion
$F_nP_*\subset P_*$ induces an inclusion
$(F_nP_*)^\wedge \subset P_*^\wedge$ with image equal to
$F_n(P_*^\wedge)$, which is natural in~$n$.  Thus,
$(P_*^\wedge, \{(F_nP_*)^\wedge\}_n)$ is naturally isomorphic to
$(P_*^\wedge, \{F_n(P_*^\wedge)\}_n)$ in $\filgrRmod$.  For this
reason, we usually omit writing parentheses when discussing the
filtration \eqref{eq:filtration-completion}.

The natural \emph{completion homomorphism}
$$
  c\: P_* \longto P^\wedge_*
$$
is a morphism of filtered graded $R$-modules.  For each~$n$ the
composite
$$
  P_* \overset{c}\longto P_*^\wedge \overset{j_n}\longto P_*/F_n P_*
$$
is surjective, which implies that $c$ induces an isomorphism of
filtration quotients
\begin{equation}\label{eq:completion-quotient-iso}
  P_*/F_n P_* \overset{\cong}\longto P_*^\wedge/F_n P_*^\wedge\,.
\end{equation}
It follows that whenever~$P_*$ in $\rbbfilgrRmod$ is
relatively bounded below, then so is~$P^\wedge_*$, when equipped with
the kernel filtration.  In particular, the filtration of~$P^\wedge_*$
is then exhaustive, i.e., $\colim_n F_n P_*^\wedge = P_*^\wedge$.

We say that~$P_*$ in $\filgrRmod$ is \emph{complete Hausdorff} if~$c$
is an isomorphism in $\grRmod$.  By Lemma~\ref{lemma:filisos}, it
follows from the isomorphisms \eqref{eq:completion-quotient-iso} that~$P_*$
is complete Hausdorff if and only if~$c$ is an isomorphism in
$\filgrRmod$.  The completion~$P_*^\wedge$, equipped with its kernel
filtration, is always complete Hausdorff.

By the exact sequence
$$
0\longto \lim_n F_nP_* \longto P_* \overset{c}{\longto}
P^\wedge_* \longto \rlim_n F_nP_* \longto 0
$$
we see that $P_*$ is complete Hausdorff if and only if its
filtration is complete and Hausdorff.

Finally, since every discretely filtered graded $R$-module is complete
Hausdorff, the full subcategory $\grRmod$ of $\filgrRmod$ consists
entirely of complete Hausdorff objects.

\subsection{Symmetric monoidal structure}\label{sec:convolution}

We now let $R=\bF_p$ for a prime $p$.  Let $P_*$ and
$Q_*$ in $ \rbbfilgrfpmod$ be relatively bounded below filtered graded
$\bF_p$-vector spaces.  We endow the graded tensor product
$P_* \otimes Q_*$ with the \emph{convolution filtration} given by
$$
F_n(P_*\otimes Q_*) = \sum_{a+b=n} F_aP_* \otimes F_bQ_* \,,
$$
where the sum is the internal sum in $P_*\otimes Q_*$.  Since the
filtrations of $P_*$ and $Q_*$ are both exhaustive, any elementary
tensor $x\otimes y$ is contained in $F_aP_*\otimes F_bQ_*$ for some
$a$~and~$b$, which implies that the convolution filtration is
exhaustive.
% hausdorff.tex: P\otimes Q Hausdorff hvis P og Q begge er Hausdorff

The convolution product makes $\rbbfilgrfpmod$ a symmetric monoidal
category with unit object $\bF_p$.  A crucial part of this claim is
that the convolution filtration of a tensor product of relatively
bounded below filtered graded $\bF_p$-vector spaces is again relatively
bounded below.  This is ensured by the following lemma.
\begin{lemma} \label{lemma:tensor-rbb}
  Let $P_*$ and $Q_*$ be relatively bounded below filtered graded
  $\bF_p$-vector spaces. Let $k$ and $\ell$ be integers such that
  $P_*/F_nP_*$ is $(n+k)$-connected and $Q_*/F_nQ_*$ is
  $(n+\ell)$-connected for all~$n$.  Then
  $(P_*\otimes Q_*)/F_n(P_*\otimes Q_*)$ is $(n+k+\ell)$-connected for
  all~$n$.  In particular, $P_*\otimes Q_*$ is relatively bounded below.
\end{lemma}
\begin{proof}
  To shorten notation, we write $F_n = F_n(P_*\otimes Q_*)$,
  $F'_a = F_aP_*$ and $F''_b = F_bQ_*$.  The quotient
  $F'_{a}/F'_{a-1}$ is contained in $P_*/F'_{a-1}$, and is therefore
  $(a+k)$-connective for each $a$.  Likewise $F''_{b}/F''_{b-1}$ is
  $(b+\ell)$-connective for each $b$.  Then
  $$
  F_{n}/F_{n-1} \cong \bigoplus_{a+b=n} F'_{a}/ F'_{a-1}
  \otimes F''_{b}/F''_{b-1}
  $$
  is a direct sum of $(a+k+b+\ell) = (n+k+\ell)$-connective graded
  $\bF_p$-vector spaces, and is therefore itself
  $(n+k+\ell)$-connective.  By induction it follows that
  $F_{m}/F_{n-1}$ is $(n+k+\ell)$-connective for every $m \geq n$.
  Since the convolution filtration is exhaustive it follows that
  $$
  \colim_m F_{m}\cong P_* \otimes Q_*
  $$
  and that $(P_* \otimes Q_*)/F_{n-1}$ is $(n+k+\ell)$-connective.
\end{proof}

We define the completed tensor product $P_*\ctensor Q_*$ as the
completion $(P_* \otimes Q_*)^\wedge$.  In the special case of
$V_*\otimes P_*$ where $V_*$ lies in $\bbgrfpmod$, the convolution
filtration is given by
$$
F_n(V_*\otimes P_*) = \sum_{a+b=n}F_aV_* \otimes F_bP_* =
V_*\otimes F_nP_*\,,
$$
since $F_aV_*=0$ for $a<0$ and $F_aV_*=V_*$ otherwise.  Therefore
\begin{equation}
  \label{eq:ctensor-discrete}
  V_*\ctensor P_* \cong \lim_n \,\bigl(V_*\otimes P_*/F_nP_*\bigr)\,.
\end{equation}
It follows from~\eqref{eq:completion-quotient-iso} that
\begin{equation}
  \label{eq:ctensor-completion}
  V_* \ctensor P_* \cong V_* \ctensor P^\wedge_*\,.
\end{equation}

For any two $V_*, W_* $ in $ \bbgrfpmod$, there are compatible natural
isomorphisms
\begin{align*}
  V_*\otimes \frac{W_*\ctensor P_*}{F_n(W_*\ctensor P_*)}
  &\cong V_*\otimes (W_*\otimes P_*/F_nP_*) \\
  &\cong (V_*\otimes W_*)\otimes P_*/F_nP_* \,.
\end{align*}
Taking limits, these isomorphisms give rise to a natural coherent
isomorphism
\begin{equation}
  \label{eq:associso}
  V_*\ctensor (W_*\ctensor P_*) \cong (V_*\otimes W_*)\ctensor P_*\,.
\end{equation}

\begin{lemma}\label{lemma:thin-layers}
  Let $f\:P_*\to Q_*$ be a morphism in $\rbbfilgrfpmod$.  If $f$
  induces an isomorphism
  $$
  \frac{F_{n}P_*}{F_{n-1}P_*} \overset\cong\longto
  \frac{F_{n}Q_*}{F_{n-1}Q_*}
  $$
  for each $n$, then the completion
  $f^\wedge\: P^\wedge_*\to Q^\wedge_*$ is an isomorphism in
  $\rbbfilgrfpmod$.
\end{lemma}
\begin{proof}
  By induction on $m$, it follows from the hypothesis that $f$ induces
  isomorphisms
  $$
  \frac{F_{m}P_*}{F_{n-1}P_*} \longto
  \frac{F_{m}Q_*}{F_{n-1}Q_*}
  $$
  for every $m\geq n$.  For any relatively bounded below filtered
  graded $\bF_p$-vector space, the filtration of any fixed
  degree stabilizes at a finite stage.  Thus, passing to the colimit
  as $m\to \infty$, it follows that $f$ induces an isomorphism
  \begin{equation}
    \label{eq:l1iso}
    \frac{P_*}{F_{n-1}P_*} \overset\cong\longto
    \frac{Q_*}{F_{n-1}Q_*}
  \end{equation}
  for each $n$.  Taking limits over $n$, we obtain that
  $f^\wedge\:P^\wedge_*\to Q_*^\wedge$ is an isomorphism of graded
  $\bF_p$-vector spaces.  By the natural isomorphisms
  \eqref{eq:completion-quotient-iso} together with \eqref{eq:l1iso},
  it follows that $f^\wedge$ induces isomorphisms of filtration
  quotients.  An application of Lemma~\ref{lemma:filisos} then
  concludes the proof.
\end{proof}

\begin{lemma} \label{lemma:convolution-filtration-isos} Assume that
  $f\: P_*\to Q_*$ and $g\: P_*'\to Q_*'$ are morphisms in
  $\rbbfilgrfpmod$ that induce isomorphisms of filtration quotients
  $F_nP_*/F_{n-1}P_* \cong F_nQ_*/F_{n-1}Q_*$ and
  $F_nP_*'/F_{n-1}P_*' \cong F_nQ_*'/F_{n-1}Q_*'$ for each $n\in\bZ$.

  Then $f\otimes g\: P_*\otimes Q_*\to P_*'\otimes Q_*'$ induces
  an isomorphism
  \[
    \frac{F_n(P_*\otimes Q_*)}{F_{n-1}(P_*\otimes Q_*)}
    \overset\cong\longto
    \frac{F_n(P_*'\otimes Q_*')}{F_{n-1}(P_*'\otimes Q_*')}
  \]
  for each $n\in\bZ$.
\end{lemma}
\begin{proof}
  Consider the following commutative diagram, where both horizontal
  homomorphisms are induced by $f\: P_*\to P_*'$ and $g\:Q_*\to Q_*'$,
  and the vertical isomorphisms are induced by the inclusions
  $F_aP_*\otimes F_bQ_* \subset F_{n}(P_*\otimes Q_*)$ as $a$ and $b$
  range over all integers such that $a+b=n$.

  \[
    \xymatrix@C18mm{
      \displaystyle \bigoplus_{a+b=n} \frac{F_aP_*}{F_{a-1}P_*}\otimes \displaystyle\frac{F_bQ_*}{F_{b-1}Q_*}
      \ar[r]^-\cong
      \ar[d]^\cong
      &
      \displaystyle \bigoplus_{a+b=n} \frac{F_aP_*'}{F_{a-1}P_*'}\otimes \frac{F_bQ_*'}{F_{b-1}Q_*'}
      \ar[d]^\cong\\
      \displaystyle\frac{F_n(P_*\otimes Q_*)}{F_{n-1}(P_*\otimes Q_*)}
      \ar[r]
      & \displaystyle\frac{F_n(P_*'\otimes Q_*')}{F_{n-1}(P_*'\otimes Q_*')}
      \,.
    }
  \]
  The upper horizontal homomorphism is an isomorphism by hypothesis,
  so the lemma follows.
\end{proof}

We have the following generalization of the isomorphism
\eqref{eq:ctensor-completion}:
\begin{lemma}\label{lemma:completions-before-or-after-tensor}
  Let $P_*$ and $Q_*$ be relatively bounded below filtered graded
  $\bF_p$-vector spaces.  The completion homomorphisms induce an
  isomorphism
  $$
  c\ctensor c \: P_*\ctensor Q_*
  \overset\cong\longto
  P^\wedge_* \ctensor Q^\wedge_*
  $$
  in $\rbbfilgrfpmod$.
\end{lemma}
\begin{proof}
  Each completion homomorphism induces isomorphisms
  \eqref{eq:completion-quotient-iso} of filtration quotients.  It
  follows from Lemma~\ref{lemma:convolution-filtration-isos} that
  $c\otimes c\: P_*\otimes Q_*\to P^\wedge_*\otimes Q_*^\wedge$
  satisfies the hypothesis of Lemma~\ref{lemma:thin-layers}, which
  implies the lemma.
\end{proof}

Let $\rbbfilgrfpmod^\wedge$ be the full subcategory of
$\rbbfilgrfpmod$ consisting of complete Hausdorff objects.  The
completed tensor product makes this a symmetric monoidal category, and
completion yields a strong symmetric monoidal functor with monoidal
structure map
$(c\ctensor c)^{-1}\: P^\wedge_* \ctensor Q^\wedge_* \cong P_*\ctensor
Q_*$ provided by Lemma~\ref{lemma:completions-before-or-after-tensor}.
\begin{equation}
  \label{eq:completion-and-dimenticare}
  \xymatrix{
    (\rbbfilgrfpmod, \otimes, \bF_p)
    \ar@<1mm>[r]^-{c}
    &
    \ar@<1mm>[l]^-{d}
    (\rbbfilgrfpmod^\wedge, \ctensor, \bF_p)
  }
\end{equation}
The forgetful functor $d$ is lax symmetric monoidal
with structure map given by the completion homomorphism
$c\: P_* \otimes Q_* \to P_*\ctensor Q_*$.

\subsection{Filtered modules and comodules over the Steenrod
  algebra}\label{sec:filamods}

Consider $M_*$ in $\grfpmod$ and $P_*$ in $\filgrfpmod$.  Then
$\Hom(M_*, P_*)$ is a graded $\bF_p$-vector space, with filtration
given by $F_n\Hom(M_*, P_*) = \Hom(M_*, F_nP_*)$.
Since $\Hom(M_*, -)$ is an exact endofunctor on $\grfpmod$,
there is an isomorphism of graded $\bF_p$-vector spaces
$$
\Hom(M_*, P_*)/\Hom(M_*, F_nP_*) \cong \Hom(M_*, P_*/F_nP_*)\,.
$$
In particular, $\Hom(M_*, P_*)$ is relatively bounded below
if $M_*$ is bounded above and $P_*$ is relatively bounded below.

A relatively bounded below filtered right $\sA$-module is an object
$P_*$ in $\rbbfilgrfpmod$, together with a right $\sA$-action specified
by the right adjoint $\tilde\rho\: P_*\to \Hom(\sA, P_*)$ making the
following diagrams commute in $\rbbfilgrfpmod$:
$$
\xymatrix@C11mm{
  P_* \ar[r]^{\tilde\rho}\ar[d]^-{\tilde\rho}
  & \Hom(\sA, P_*) \ar[d]^{\Hom(\phi, 1)}
  & P_* \ar[r]^-{\tilde\rho} \ar[dr]_\cong
  & \Hom(\sA, P_*)\ar[d]^-{\Hom(\eta, 1)}\\
  \Hom(\sA, P_*) \ar[rd]_{\Hom(1,\tilde\rho)\hspace{5mm}}
  & \Hom(\sA\otimes\sA, P_*)\ar[d]^{\cong}
  && \Hom(\bF_p, P_*) \rlap{\,.}\\
  & \Hom(\sA, \Hom(\sA, P_*))
}
$$
The completion $P^\wedge_*$ is then also a (relatively bounded below
filtered) right $\sA$-module, with structure map
$\tilde\rho \: P^\wedge_* \to \Hom(\sA, P^\wedge_*)$ the limit
over~$n$ of the homomorphisms
\[
P_*/F_n P_* \longto \Hom(\sA, P_*/F_n P_*)
\]
induced by the filtration-preserving homomorphism $\tilde\rho \: P_*
\to \Hom(\sA, P_*)$.

The category of relatively bounded below filtered right $\sA$-modules
and filtration-preserving $\sA$-module homomorphisms is denoted by
$\rbbfilamod$.  It becomes symmetric monoidal when $P_* \otimes Q_*$
is equipped with the convolution filtration from
Subsection~\ref{sec:convolution} and the right $\sA$-action from
Subsection~\ref{sec:a-modules}.

We let $\rbbfilamodc$ be the full subcategory of complete Hausdorff
objects in $\rbbfilamod$.  For brevity, we sometimes refer to these
objects as \emph{rbb complete right $\sA$-modules}, leaving `Hausdorff filtered' implicit.  Note that for $P_*$ and $Q_*$ both complete
Hausdorff, the convolution filtration on $P_* \otimes Q_*$ is not
complete Hausdorff in general.  However, the induced filtration on
$P_* \ctensor Q_*$ is complete Hausdorff, and the completed tensor
product makes the full subcategory of complete Hausdorff objects
symmetric monoidal.

It follows from Lemma~\ref{lemma:tensor-rbb} that $\sA_*\otimes P_*$
is an rbb filtered graded $\bF_p$-vector space for every $P_*$ in
$\rbbfilgrfpmod$.  By the discussion in
Subsection~\ref{sec:completion}, it follows that also
$\sA_*\ctensor P_*$ is in $\rbbfilgrfpmod$.

A relatively bounded below complete Hausdorff filtered left
$\sA_*$-comodule is a complete Hausdorff object $P_*$ in
$\rbbfilgrfpmod$ together with a \textit{complete left
  $\sA_*$-coaction} $\nu\: P_*\longto \sA_*\ctensor P_*$ making the
following diagrams commute in $\rbbfilgrfpmod$:
$$
\xymatrix@C16mm{
  P_* \ar[r]^{\nu}\ar[d]^{\nu}
  & \sA_*\ctensor P_* \ar[d]^{\psi\ctensor 1}
  & P_* \ar[r]^{\nu} \ar[dr]_\cong
  & \sA_*\ctensor P_* \ar[d]^{\epsilon\ctensor 1}\\
  \sA_*\ctensor P_* \ar[rd]_{1\ctensor \nu}
  &  (\sA_*\otimes \sA_*)\ctensor P_* \ar[d]^\cong
  && \bF_p\ctensor P_* \rlap{\,.} \\
  & \sA_*\ctensor (\sA_*\ctensor P_*)
}
$$
The unlabeled isomorphism in the left-hand diagram is the isomorphism
\eqref{eq:associso}.  The diagonal arrow in the right-hand diagram is
the completion homomorphism, which is an isomorphism since we assume
that $P_*$ is complete Hausdorff.

The category of relatively bounded below complete Hausdorff filtered
left $\sA_*$-comodules and filtration-preserving $\sA_*$-comodule
homomorphisms is denoted by $\rbbfilacomodc$.  For brevity, we
sometimes refer to these objects as \emph{rbb complete left $\sA_*$-comodules},
leaving `Hausdorff filtered' implicit.  It is symmetric monoidal
under the completed tensor product when we let $P_*\ctensor Q_*$ have
the $\sA_*$-coaction map equal to the completion of the composite
\[
  P_* \otimes Q_* \overset{\nu \otimes \nu}\longto \sA_* \ctensor P_*
  \otimes \sA_* \ctensor Q_* \longto (\sA_* \otimes \sA_*) \ctensor
  (P_* \otimes Q_*) \overset{\phi \ctensor 1}\longto \sA_* \ctensor
  (P_* \otimes Q_*) \,.
\]
Here the middle homomorphism is the colimit over~$a$ and~$b$ of the
limit over~$c$ of the composite homomorphisms
\begin{align*}
  \sA_* \ctensor F_a P_* \otimes \sA_* \ctensor F_b Q_*
  &\longto \sA_* \otimes
    \frac{F_a P_*}{F_{a-c} P_*} \otimes \sA_* \otimes \frac{F_b
    Q_*}{F_{b-c} Q_*}\\
  &\longto \sA_* \otimes \sA_* \otimes \frac{F_{a+b}(P_*
    \otimes Q_*)}{F_{a+b-c}(P_* \otimes Q_*)} \,.
\end{align*}

There are symmetric monoidal forgetful functors from
$\rbbfilamod$, $\rbbfilamodc$ and $\rbbfilacomodc$ to $\rbbfilgrfpmod$.
The first is strong while the latter two are lax.

\subsection{Another isomorphism of symmetric monoidal categories}
\label{sec:filisocat}
Let $V_*$ be a bounded below graded $\bF_p$-vector space of finite
type over $\bF_p$, with dual $V^*$ as in Subsection \ref{sec:isocat}.
For every $P_*$ in $\rbbfilgrfpmod$ there is a natural injective
homomorphism
\begin{equation*}
  \hat\iota\: V_*\ctensor P_* \longto \Hom(V^*, P^\wedge_*)
\end{equation*}
defined as the limit of the compatible injections
\begin{equation*}
  \iota_n\: V_*\otimes P_*/F_nP_* \longto \Hom(V^*, P_*/F_nP_*)\,.
\end{equation*}
\begin{lemma}
  \label{lemma:iotafilmods}
  The map $\hat\iota$ is an isomorphism.
\end{lemma}
\begin{proof}
  For each $n$, Lemma \ref{lemma:iotamods} implies that $\iota_n$ is
  an isomorphism since $P_*/F_nP_*$ is bounded below.  It follows that
  the limit $\hat\iota$ is also an isomorphism.
\end{proof}

\begin{proposition}\label{prop:categoryiso}
  There is an isomorphism of symmetric monoidal categories
  \begin{equation}
    \label{eq:category-iso}
    \rbbfilamodc \cong \rbbfilacomodc\,.
  \end{equation}
\end{proposition}
\begin{proof}
  Any complete Hausdorff $P_*$ in $ \rbbfilacomodc$ gives rise to a
  relatively bounded below complete Hausdorff filtered right
  $\sA$-module having the same underlying complete Hausdorff filtered
  graded $\bF_p$-vector space and adjoint structure map equal to the
  composite
  $$
  \xymatrix{
    \tilde\rho\: P_*\ar[r]^-\nu & \sA_*\ctensor P_* \ar[r]^-{\hat\iota}
    & \Hom(\sA, P^\wedge_*)\,.
  }
  $$
  Letting $V_*=\sA_*$ in Lemma \ref{lemma:iotafilmods} implies that
  passing from a complete Hausdorff left $\sA_*$-comodule to a complete
  Hausdorff right $\sA$-module in this way defines an isomorphism of
  categories~\eqref{eq:category-iso}.  It is clearly strong symmetric
  monoidal.
\end{proof}

We have the following diagram of symmetric monoidal categories:
\begin{equation}
  \label{eq:amods-and-acomods}
  \begin{aligned}
    \xymatrix{
    (\rbbfilamod, \otimes, \bF_p) \ar@<1mm>[r]^-{c}
    &
      (\rbbfilamodc, \ctensor, \bF_p) \cong
      (\rbbfilacomodc, \ctensor, \bF_p)\,.
      \ar@<1mm>[l]^-d
      }
  \end{aligned}
\end{equation}
Proposition~\ref{prop:categoryiso} provides the strong symmetric
monoidal isomorphism on the right.  The completion functor~$c$ and the
forgetful functor~$d$ are the same functors as
in~\eqref{eq:completion-and-dimenticare} on underlying categories of
graded $\bF_p$-vector spaces.  The claim that their symmetric monoidal
structure maps are morphisms in the category of right $\sA$-modules
relies on the fact that the completion homomorphism~$c\: P_*\to P_*^\wedge$
is a morphism in $\rbbfilamod$, for each rbb filtered right
$\sA$-module~$P_*$.

\subsection{Filtered algebras}\label{sec:algebras}

For brevity, in the remainder of this section let $\sC=\rbbfilamod$
denote the category of rbb filtered right $\sA$-modules, and let
$\sC^\wedge$ denote the full subcategory of rbb complete
right $\sA$-modules.  Likewise, let $\sD^\wedge = \rbbfilacomodc$ denote the
category of rbb complete left $\sA_*$-comodules.  With this notation,
$(\sC, \otimes, \bF_p)$, $(\sC^{\wedge}, \ctensor, \bF_p)$ and
$(\sD^\wedge, \ctensor, \bF_p)$, are the symmetric monoidal categories
discussed in Subsection \ref{sec:filamods}.  We usually omit the unit
object $\bF_p$ from the notation.  For any symmetric monoidal category
$(\mathscr{M}, \otimes)$, we write $\alg(\mathscr{M}, \otimes)$ for
the category of monoids in~$\mathscr{M}$.

An object of $\alg(\sC, \otimes)$ is an rbb filtered right
$\sA$-module that is also a graded $\bF_p$-algebra such that the
algebra structure maps are morphisms in $\sC$.  We refer to an object
of $\alg(\sC, \otimes)$ as an \textit{rbb filtered right $\sA$-module
  algebra}.  Likewise, we refer to an object of
$\alg(\sC^\wedge, \ctensor)$ as an \textit{rbb complete right
  $\sA$-module algebra}, and an object of
$\alg(\sD^\wedge, \ctensor)$ as an \textit{rbb complete left
  $\sA_*$-comodule algebra}.

The following is an immediate consequence of Proposition
\ref{prop:categoryiso}.
\begin{proposition}\label{prop:alg-category-iso}
  There is an isomorphism of categories
  $$
  \alg(\sC^\wedge, \ctensor) \cong \alg(\sD^\wedge, \ctensor)\,.
  $$\qed
\end{proposition}

Suppose given any morphism $f\: P_*\otimes Q_*\to R_*$ in $\sC$.  Then
there are morphisms
\begin{equation}
  \label{eq:complete-pairing}
  \xymatrix{
    P_*^\wedge \ctensor Q_*^\wedge
    &P_*\ctensor Q_*
    \ar[l]^-\cong_-{c\ctensor c}
    \ar[r]^-{f^\wedge}
    &R_*^\wedge
  }
\end{equation}
of complete Hausdorff objects in $\sC^\wedge$, where the left-hand
isomorphism of \eqref{eq:complete-pairing} is the natural isomorphism
of Lemma~\ref{lemma:completions-before-or-after-tensor}.  Thus, the
completion functor $c\:\sC\to \sC^\wedge$ promotes monoids in
$(\sC,\otimes)$ to monoids in $(\sC^\wedge, \ctensor)$, and we have
the following diagram of monoid categories:
\begin{equation}
  \label{eq:algebras}
  \begin{aligned}
    \xymatrix{
    \alg(\sC, \otimes) \ar@<1mm>[r]^-{c}
    &
      \alg(\sC^\wedge, \ctensor) \cong
      \alg(\sD^\wedge, \ctensor)\,.
      \ar@<1mm>[l]^-d
      }
  \end{aligned}
\end{equation}

\subsection{Filtered differential graded algebras}\label{sec:dgas}

Let $P_*$ be a filtered graded $\bF_p$-vector space.  For
$i\in \bZ$, we write $\leftsuspension^i P_*$ for the $i$-th (left) suspension of $P_*$,
given by $(\leftsuspension^i P_*)_q = P_{q-i}$ and with filtration given by
$F_n (\leftsuspension^i P_*) = \leftsuspension^i F_{n} P_*$.  It is clear that
$\leftsuspension^i$ commutes with completion.

For $a\in \bZ$, we write $\sh_a P_*$ for the filtered graded
$\bF_p$-vector space having the same underlying graded vector space $P_*$,
and with filtration given by $F_n (\sh_a P_*) = F_{n+a} P_*$.  It is
clear that the $a$-th downward shift $\sh_a$ commutes with completion.

A \emph{filtered differential graded $\bF_p$-algebra} $(P_*, \sigma)$ is a filtered
graded $\bF_p$-algebra $(P_*, \mu)$ together with a filtration-shifting
$\bF_p$-linear differential
\begin{equation}
  \label{eq:dga-diff}
  \sigma\: \leftsuspension P_* \longto \sh_aP_*\,,
\end{equation}
for some fixed $a$, satisfying the Leibniz rule
\begin{equation}
  \label{eq:leibniz}
  \sigma(xy) = \sigma(x)y + (-1)^{|x|}x\sigma(y)
\end{equation}
for every pair of homogeneous elements $x,y\in P_*$.  Alternatively,
we can rephrase this by saying that the diagram
\begin{equation}
  \label{eq:leibniz-square}
  \begin{aligned}
    \xymatrix@C15mm{
    \leftsuspension (P_*\otimes P_*)
    \ar[r]^-{\leftsuspension\mu}
    \ar[d]^{\sigma\otimes 1+1\otimes\sigma}
    & \leftsuspension P_*
      \ar[d]^\sigma\\
    \sh_a (P_*\otimes P_*)
    \ar[r]^-{\sh_a \mu}
    & \sh_aP_*
    }
  \end{aligned}
\end{equation}
should commute in the category of filtered graded $\bF_p$-vector
spaces.  Here, the left-hand vertical map is given by
$$
(\sigma\otimes 1 + 1\otimes \sigma)(x\otimes y) =
\sigma(x)\otimes y + (-1)^{|x|}x\otimes \sigma(y)
$$
for every pair of homogeneous elements $x,y\in P_*$.

For any right $\sA$-module $M_*$ we consider the (left) suspension
$\leftsuspension M_*$ as a right $\sA$-module in the graded sense.
Explicitly, the action of $a\in \sA^k$ on an element
$\leftsuspension x\in \leftsuspension M_*$ is given by
$(\leftsuspension x)\cdot a = (-1)^{k} \leftsuspension(x\cdot a)$.
The definition of differential graded algebras can then be applied to
any of the symmetric monoidal categories $(\sC, \otimes)$,
$(\sC^\wedge, \ctensor)$ and $(\sD^\wedge, \ctensor)$ discussed so
far.  We denote these categories of differential graded algebras by
$\dgalg(\sC,\otimes)$, $\dgalg(\sC^\wedge,\ctensor)$, and
$\dgalg(\sD^\wedge, \ctensor)$, respectively.

An object of $\dgalg(\sC, \otimes)$ is an rbb filtered right
$\sA$-module algebra that is also a differential graded algebra such
that the differential is a morphism in $\sC$.  In particular, we have
that $\sigma(x)\cdot \beta = -\sigma(x\cdot \beta)$ at odd primes.

We refer to an object of $\dgalg(\sC, \otimes)$ as an \textit{rbb
  filtered differential graded right $\sA$-module algebra}.  Likewise,
we refer to an object of $\dgalg(\sC^\wedge, \ctensor)$ as an
\textit{rbb complete differential graded right $\sA$-module algebra},
and an object of $\dgalg(\sD^\wedge, \ctensor)$ as an \textit{rbb
  complete differential graded left $\sA_*$-comodule algebra}.

The following is an immediate consequence of Proposition
\ref{prop:alg-category-iso}.
\begin{proposition}\label{prop:dgalg-category-iso}
  There is an isomorphism of categories
  $$
  \dgalg(\sC^\wedge, \ctensor) \cong \dgalg(\sD^\wedge, \ctensor)\,.
  $$
  \qed
\end{proposition}

The completion functor
$c\:\alg(\sC, \otimes)\to \alg(\sC^\wedge, \otimes)$ promotes an object
$(P_*, \sigma)$ in $\dgalg(\sC,\otimes)$ to an object in
$\alg(\sC^\wedge, \ctensor)$ by completing $(P_*, \mu)$ as in
Subsection~\ref{sec:algebras}, and replacing $\sigma$ by its
completion $\sigma^\wedge$.  Since suspension and filtration-shifting
commutes with completion, this yields a well-defined differential
$\sigma^\wedge\: \leftsuspension P_*^\wedge\to \sh_a P_*^\wedge$.  The one thing
to check is that $\sigma^\wedge$ satisfies the Leibniz rule.  Indeed,
there is a commutative diagram
\begin{equation*}
  \begin{aligned}
    \xymatrix@C19mm{
    \leftsuspension (P_*^\wedge\ctensor P_*^\wedge)
    \ar[d]^{\sigma^\wedge\ctensor 1 + 1\ctensor\sigma^\wedge}
    & \leftsuspension (P_*\ctensor P_*)
      \ar[l]^\cong_{\leftsuspension (c\ctensor c)}
      \ar[r]^{\leftsuspension\mu^\wedge}
      \ar[d]^{\sigma\ctensor 1+1\ctensor\sigma}
    & \leftsuspension P_*^\wedge
      \ar[d]^{\sigma^\wedge}\\
    \sh_a (P_*^\wedge\ctensor P_*^\wedge)
    & \sh_a (P_*\ctensor P_*)
      \ar[l]^\cong_{\sh_a (c\ctensor c)}
      \ar[r]^{\sh_a \mu^\wedge}
    & \sh_aP_*^\wedge \,,
    }
  \end{aligned}
\end{equation*}
where the right-hand square is the completion
of~\eqref{eq:leibniz-square}.  The top and bottom horizontal composite
morphisms, from left to right, are the suspended and the shifted
product of $c(P_*, \mu)$ in $\alg(\sC^\wedge, \ctensor)$,
respectively.  The statement that the outer square commutes thus
asserts that $\sigma^\wedge$ satisfies the Leibniz rule.

We have the following diagram of categories, analogous
to~\eqref{eq:algebras}:
\begin{equation}
  \label{eq:dgalgebras}
  \begin{aligned}
    \xymatrix{
    \dgalg(\sC, \otimes)
    \ar@<1mm>[r]^-{c}
    &
    \ar@<1mm>[l]^-{d}
      \dgalg(\sC^\wedge, \ctensor) \cong
      \dgalg(\sD^\wedge, \ctensor)\,.
      }
  \end{aligned}
\end{equation}

\section{The $G$-Tate construction}\label{sec:tatefiltration}

\begin{definition}
  A spectrum $X$ is \emph{bounded below} if $\pi_*(X)$ is bounded
  below as a graded abelian group.  It is of \emph{finite type} if
  $\pi_*(X)$ is of finite type as a graded $\bZ$-module.
\end{definition}
Note that if $X$ is a spectrum such that~$X/p$ is bounded below,
then~$X/p$ is of finite type if and only if the mod~$p$ homology group
$H_q(X;\bF_p)$ is finite for each $q \in \bZ$.

Equivariantly, we will work with $G$-spectra in the sense of
\cite{Sch18}*{Sec.~3.1}, where usually~$G$ is either the circle
group~$\bT$ or any of its finite subgroups.  As a model for~$\bT$ we
take the group of complex units of length one.

Let $G$ be a closed subgroup of the circle group~$\bT$, and~$X$ any
spectrum with $G$-action.  The $G$-Tate construction on~$X$ is the
fixed point spectrum
\[
  X^{tG} = \bigl[ \wEG \wedge F(EG_+, X) \bigr]^G \,.
\]
Here $EG = S(\infty\bC)$ is a free and contractible $G$-CW space, and
$\wEG = S^{\infty \bC}$ is the mapping cone of the collapse map
$EG_+ \to S^0$.  Since~$G$ acts freely on~$EG$, the genuinely
$G$-equivariant homotopy type of $F(EG_+, X)$ depends only on~$X$ as a
naive $G$-spectrum.  In particular it makes sense to form $G$-fixed
points as stated.

Assigning the $G$-Tate construction $X^{tG}$ to a spectrum $X$
with $G$-action defines a lax symmetric monoidal functor on the
homotopy category of $G$-spectra, with pairing
\begin{equation}
  \label{eq:tate-pairing}
  \mu_{X,Y}\: X^{tG}\wedge Y^{tG}\longto (X\wedge Y)^{tG}\,,
\end{equation}
for any $G$-spectra $X$ and $Y$, see \cite{HR24}*{Sec.~6.2}, and
unit
\begin{equation}
  \label{eq:tate-unit}
  \eta\: S\longto  S^{tG} \,,
\end{equation}
given by composing the map $S\to S^{G}$ given by the
tom~Dieck splitting with the map
$R^h\circ \Gamma_1\: S^{G}\to S^{hG}\to S^{tG}$.

As in~\cite{GM95}, \cite{HM03} and~\cite{HR24}*{Ch.~6}, we consider
two related filtrations of the $G$-Tate construction~$X^{tG}$.  These
arise from the following space- and spectrum-level filtrations of $EG$
and $\wEG$.

For $G = \bT$ and $m\ge0$ we filter $EG$ by
$F_{2m-2} EG = F_{2m-1} EG = S(m\bC)$, so that
$F_{2m} EG/F_{2m-1} EG \cong \bT_+ \wedge S^{2m}$.  We let $F_n \wEG$
be the mapping cone of $F_{n-2} EG_+ \to S^0$, giving a filtration
of~$\wEG$, so that $F_{2m} \wEG = F_{2m+1} \wEG = S^{m\bC}$.

For $G$ finite cyclic we keep $F_{2m-1} EG = S(m\bC)$, and define
$F_{2m} EG$ by adjoining a free $G$-$2m$-cell in such a way that
$F_n EG/F_{n-1} EG \cong G_+ \wedge S^n$ for all $n\ge0$.  In this
case we let $F_n \wEG$ be the mapping cone of $F_{n-1} EG_+ \to S^0$.
This means that $F_{2m} \wEG = S^{m\bC}$ and
$F_n \wEG/F_{n-1} \wEG \cong G_+ \wedge S^n$ for all $n\ge1$.

Passing to $G$-spectra, we let $\tilde E_n = \Sigma^\infty F_n \wEG$
for all $n\ge0$.  Following Greenlees~\cite{Gre87}, we use
Spanier--Whitehead duality to extend this to negative indices.  For
$G = \bT$ we define $\tilde E_{-n} = F(\tilde E_{1+n}, S)$, while for
$G$ finite cyclic we set $\tilde E_{-n} = F(\tilde E_n, S)$.  For
$G = \bT$ the resulting spliced sequence $\{\tilde E_n\}_{n \in \bZ}$
is given by $\tilde E_{2m} = \tilde E_{2m+1} = S^{m\bC}$ for all
$m \in \bZ$, while for $G$ finite cyclic we have
$\tilde E_{2m} = S^{m\bC}$ and
$\tilde E_n/\tilde E_{n-1} \simeq G_+ \wedge S^n$ for all $n \in \bZ$.

Note that in the case $G = \bT$ these conventions agree with those
of~\cite{GM95}, with trivial filtration quotients in odd integer
gradings, while in~\cite{HR24} these trivial subquotients are omitted.
This means that the $\bT$-Tate spectral sequence $\hat E^r$-terms
from~\cite{HR24} appear as $\hat E^{2r-1} = \hat E^{2r}$-terms here.  For
finite~$G$, all conventions agree.

For finite cyclic $G \subset \bT$ the identity maps $EG \to E\bT$ and
$\widetilde{E\bT} \to \wEG$ are filtration-preserving and
$G$-equivariant.

\subsection{The Greenlees--May filtration}\label{sec:greenlees-may-filtration}
The filtration $\{GM_n(X)^G\}_{n \in \bZ}$ is given by the fixed point
spectra
\[
GM_n(X)^G = \bigl[ \tilde E_n \wedge F(EG_+, X) \bigr]^G \,.
\]
It satisfies $\holim_n GM_n(X)^G \simeq *$ and
$\hocolim_n GM_n(X)^G \simeq X^{tG}$, so after applying homotopy groups the
resulting Tate spectral sequence is conditionally convergent to the
colimit $\pi_* X^{tG}$, filtered by the images
\[
F_n \pi_* X^{tG} = \im \bigl( \pi_* GM_n(X)^G
	\longto \pi_* X^{tG} \bigr) \,.
\]
We call the ascending filtration
\begin{equation}
  \label{eq:tate-filtration-homotopy}
  \{F_n \pi_* X^{tG}\}_{n \in \bZ}
\end{equation}
the \emph{Tate filtration} of~$\pi_* X^{tG}$.  It is exhaustive, and
the associated \emph{Greenlees--May Tate spectral sequence} has $\hat E^2$-term
given by
\begin{equation}
  \label{eq:gmtate}
  \hat E^2_{n,m} = \hat H^{-n}(G; \pi_m X) \,,
\end{equation}
and is conditionally convergent to the colimit $\pi_* X^{tG}$.

When $G = \bT$, the differential
$d^2 \: \hat H^{-n}(G; \pi_m X) \to \hat H^{-n-2}(G; \pi_{m+1} X)$ is given
by the operator $\sigma$ associated to the $\bT$-action for
$n \equiv 0 \mod 4$, and by $\sigma + \eta$ for $n \equiv 2 \mod 4$,
where $\eta \in \pi_1(S)$ is the Hopf map.

We also use the notation
\[
X^{tG}[n] = \bigl[ \wEG/\tilde E_{n-1} \wedge F(EG_+, X) \bigr]^G
\]
for the cofiber of $GM_{n-1}(X)^G \to X^{tG}$, which we call the
$n$-th \emph{truncation} of~$X^{tG}$.  Again, the notation $[n]$
refers to the truncation index, not a suspension.  Letting $n \in \bZ$
vary we get a tower of truncated Tate spectra
\begin{equation}
  \label{eq:tate-tower}
  X^{tG} \longto \dots \longto X^{tG}[n]
  \longto X^{tG}[n+1] \longto \dots \longto {*} \,,
\end{equation}
with $X^{tG} \simeq \holim_n X^{tG}[n]$ and
$\hocolim_n X^{tG}[n] \simeq *$.  The Tate filtration can then also be
expressed by the kernels
\begin{equation} \label{eq:kernel-tate-filtration}
  F_n \pi_* X^{tG} =
  \ker \bigl( \pi_* X^{tG} \longto \pi_* X^{tG}[n+1] \bigr) \,.
\end{equation}
These constructions suffice for additive considerations, but the
Greenlees filtration $\{\tilde E_n\}_{n \in \bZ}$ appears not to admit
a product structure that is sufficiently coherent to provide a product
structure on the Greenlees--May Tate spectral sequence.

\subsection{The Hesselholt--Madsen filtration}\label{sec:hesselholt-madsen}
To get a multiplicative construction, one can instead use the
filtration $\{HM_n(X)^G\}_{n \in \bZ}$ given by the fixed point
spectra
\[
HM_n(X)^G = \bigl[ \colim_{a+b \le n}
	F_a \wEG \wedge F(EG/F_{-b-1} EG, X) \bigr]^G
\]
of the convolution product
of the filtrations $\{F_a \wEG\}_a$ and $\{F(EG/F_{-b-1} EG, X)\}_b$.
Then $\holim_n HM_n(X)^G \simeq *$ by~\cite{BM24}*{Lem.~3.16}, and
clearly $\hocolim_n HM_n(X)^G \simeq X^{tG}$, so the associated
Hesselholt--Madsen Tate spectral sequence is also conditionally
convergent to the colimit $\pi_* X^{tG}$, filtered by the images
\[
F'_n \pi_* X^{tG} = \im \bigl( \pi_* HM_n(X)^G
	\longto \pi_* X^{tG} \bigr) \,.
\]
By~\cite{HR24}*{Thm.~6.18} the Hesselholt--Madsen Tate spectral
sequence is symmetric monoidal in~$X$.  In particular, given any
pairing~$\mu \: X \wedge Y \to Z$ of spectra with $G$-action, the
pairing $\mu_* \: \pi_* X^{tG} \otimes \pi_* Y^{tG} \to \pi_* Z^{tG}$
induced by~\eqref{eq:tate-pairing} is filtered by pairings
\[
\mu'_{a,b} \: F'_a \pi_* X^{tG} \otimes F'_b \pi_* Y^{tG}
	\longto F'_{a+b} \pi_* Z^{tG}
\]
for all $a, b \in \bZ$.
Moreover, the induced pairing of associated graded objects is
compatible with the pairing of $\hat E^\infty$-terms.

There is a natural map of filtrations $GM_*(X)^G \to HM_*(X)^G$ that
induces a morphism of conditionally convergent spectral sequences, and
by~\cite{HR24}*{Prop.~6.31} this becomes an isomorphism starting at
the $\hat E^3$-term for $G = \bT$ and at the $\hat E^2$-term for $G$ finite.  It
follows that the induced homomorphisms of $\hat E^\infty$- and
$R\hat E^\infty$-terms are isomorphisms.  For bounded below~$X$, it then
follows from~\cite{Boa99}*{Thm.~7.2} that the inclusions
$F_n \pi_* X^{tG} \subset F'_n \pi_* X^{tG}$ are equalities, for all
$n \in \bZ$.  Hence, for bounded below spectra with $G$-action~$X$,
$Y$, $Z$ and a pairing~$\mu$ as above, the abutment pairing $\mu_*$
also respects the Tate filtrations, in the sense that there are
compatible pairings
\[
\mu_{a,b} \: F_a \pi_* X^{tG} \otimes F_b \pi_* Y^{tG}
        \longto F_{a+b} \pi_* Z^{tG}
\]
for all $a, b \in \bZ$.

\subsection{Completion with respect to the Tate filtration}

For each integer~$n$, the map in
\eqref{eq:kernel-tate-filtration} factors naturally as
\[
  \pi_* X^{tG} \overset{c_n}\longto \frac{\pi_* X^{tG}}{F_n \pi_* X^{tG}}
  \overset{i_n}\longto \pi_* X^{tG}[n+1] \,.
\]
Passing to limits over~$n$ produces a factorization
\[
\pi_* X^{tG} \overset{c}\longto \bigl( \pi_* X^{tG} \bigr)^\wedge
        \overset{i}\longto \lim_n \pi_* X^{tG}[n] \,,
\]
where $c$ is the completion homomorphism with respect to the Tate
filtration.  Note that $i$ is injective since it is a limit of
injective homomorphisms.  The kernel filtration
$$
F_n(\lim_n\pi_*X^{tG}[n]) = \ker \bigl( \lim_n\pi_*X^{tG}[n] \longto
\pi_*X^{tG}[n+1] \bigr)
$$
of the limit makes $i$ a homomorphism of filtered graded abelian
groups.

\begin{lemma}\label{lemma:completetatefiltration}
  Let $X$ be a spectrum with $G$-action.  The canonical homomorphism
  \begin{equation}
    \label{eq:i-iso}
    i\:(\pi_* X^{tG})^\wedge \longto \lim_n \pi_*X^{tG}[n]
  \end{equation}
  is an isomorphism of complete Hausdorff filtered graded abelian groups, and there is a
  short exact sequence of graded abelian groups
  \begin{equation}
    \label{eq:lim-lim1}
    0\longto
    \rlim_n \pi_{*+1}X^{tG}[n]\longto
    \pi_*X^{tG} \overset{c}\longto
    (\pi_* X^{tG})^\wedge\longto 0\,,
  \end{equation}
  where $c$ is the completion homomorphism.  In particular, $c$ is an
  isomorphism of filtered graded abelian groups if and only if
  $\rlim_n \pi_* X^{tG}[n]=0$.
\end{lemma}
\begin{proof}
  Consider the commutative diagram
  \begin{equation}
    \label{eq:limseq}
    \xymatrix{
      0\ar[r]
      &\rlim_n \pi_{*+1}X^{tG}[n]\ar[r]
      &\pi_*X^{tG} \ar[r]^-j \ar[dr]_c
    &\lim_n \pi_*X^{tG}[n]\ar[r]
    &0 \\
    &&& (\pi_* X^{tG})^\wedge \ar[u]_-i\rlap{\,,}
  }
  \end{equation}
  where the top row is the Milnor $\lim\!-\!\rlim$ short exact
  sequence associated to $\holim_n X^{tG}[n] \simeq X^{tG}$.
  Since~$j$ is surjective it follows that~$i$ is surjective, hence an
  isomorphism of graded abelian groups.

  We have already noted that the completion homomorphism~$c$ induces
  isomorphisms \eqref{eq:completion-quotient-iso} of filtration
  quotients.  We claim that the same is true for the
  homomorphism~$j$, and therefore also for~$i$.
  Lemma~\ref{lemma:filisos} then implies that~$i$ is an isomorphism of
  filtered graded abelian groups, and we get the short exact
  sequence~\eqref{eq:lim-lim1} from~\eqref{eq:limseq}.

  To prove the claim, note first that since $j$ is
  filtration-preserving and surjective, it induces surjections of
  filtration quotients.  Secondly, since the Tate filtration of
  $\pi_*X^{tG}$ is the pullback of the kernel filtration of
  $\lim_n \pi_*X^{tG}[n]$ via the homomorphism $j$, it follows that
  $j$ induces injections of filtration quotients.  Thus,~$j$ induces
  isomorphisms of filtration quotients.
\end{proof}

\begin{lemma}\label{lemma:tatefiltration}
  Let $X$ be a spectrum with $G$-action.  The Tate filtration of
  $\pi_*X^{tG}$ is exhaustive and complete.  It is Hausdorff if and
  only if $\rlim_n \pi_{*}X^{tG}[n] = 0$.
\end{lemma}
\begin{proof}
  Exhaustiveness follows from commuting colimits and homotopy groups.
  The exact sequence
  \[
    0\longto \lim_n F_n\pi_*X^{tG}\longto \pi_*X^{tG}\overset{c}\longto
    (\pi_*X^{tG})^\wedge\longto \rlim_n F_n\pi_*X^{tG}\longto 0
  \]
  implies that $c$ is injective, resp.~surjective, if and only if the
  Tate filtration is Hausdorff, resp.~complete.  The lemma then
  follows from the exact sequence in Lemma
  \ref{lemma:completetatefiltration}.
\end{proof}

\begin{lemma} \label{lemma:strong-convergence-homotopy} Suppose that
  $X$ is bounded below.  The Greenlees--May Tate spectral sequence is
  strongly convergent if one of the following conditions hold:
  \begin{enumerate}
  \item $R\hat E^{\infty}=0$.
  \item $G = \bT$ and $\pi_*(X)$ is finite in each
    degree.
  \item $G$ is finite cyclic and $\pi_*(X)$ is finitely generated in
    each degree.
  \item The spectral sequence collapses at a finite stage.
  \end{enumerate}
  Note that strong convergence of the Greenlees--May Tate spectral
  sequence implies that the Tate filtration of $\pi_*X^{tG}$ is complete
  Hausdorff and exhaustive.
\end{lemma}
\begin{proof}
  For bounded-below $X$, \eqref{eq:gmtate} is a half-plane spectral
  sequence with entering differentials, and the condition
  $R\hat E^\infty = 0$ ensures that the spectral sequence is strongly
  convergent by~\cite{Boa99}*{Thm.~7.3}.  The lemma follows since any
  one of the conditions (2)--(4) implies (1).
\end{proof}

Applying homotopy groups to the tower \eqref{eq:tate-tower}
produces a spectral sequence isomorphic to the Greenlees--May Tate
spectral sequence \eqref{eq:gmtate}.  It is conditionally convergent
to the limit
$$
\lim_n \pi_* X^{tG}[n] \cong (\pi_* X^{tG})^\wedge\,,
$$
which is generally not isomorphic to $\pi_* X^{tG}$.  Under the
assumptions of Lemma~\ref{lemma:strong-convergence-homotopy} it
follows from \cite{Boa99}*{Thm.~7.3 and Thm.~7.4} that both spectral
sequences converge strongly, and that $\rlim_n \pi_{*}X^{tG}[n]=0$.
But in this case, the abutments are isomorphic by the short exact
sequence of Lemma~\ref{lemma:completetatefiltration}.

\subsection{Continuous mod $p$ homology of the Tate construction}
\label{sec:continuous-homology-of-tate}

Let $E$ be an $E_\infty$ ring spectrum with trivial action of $G$,
and let $Y$ be an $E$-module spectrum with $G$-action such that the
actions of $G$ and $E$ commute.  The $G$-Tate construction on $Y$ has
the structure of an $E$-module, with structure map given by
$$
E\wedge Y^{tG} \overset{\kappa}\longto (E\wedge Y)^{tG}
\overset{\phi^{tG}}\longto
Y^{tG}\,.
$$
Here $\kappa$ is the canonical map that interchanges smashing with $E$
and taking homotopy colimits and limits, and $\phi\: E\wedge Y\to Y$ is
the $G$-equivariant $E$-module structure map of~$Y$.

As in the $S$-based case, assigning the $G$-Tate construction
$Y^{tG}$ to an $E$-module~$Y$ with $G$-action defines a lax
symmetric monoidal functor with pairing
\begin{equation}
  \label{eq:E-tate-pairing}
  \mu^E_{Y,Y'}\: Y^{tG}\wedge_E Y'\,^{tG}\longto (Y\wedge_E Y')^{tG}\,,
\end{equation}
induced from \eqref{eq:tate-pairing} by passing to smash products over
$E$, and unit
\begin{equation}
  \label{eq:E-tate-unit}
  \eta^E\: E\simeq E\wedge S \overset{1\wedge \eta}\longto
  E\wedge S^{tG}
 \overset{\kappa}\longto E^{tG}\,.
\end{equation}

That $\eta^E$ satisfies the properties needed to be a lax monoidal
unit follows from the commutativity of the following diagram
\[
  \xymatrix{
    E\wedge_E Y^{tG}
    \ar[r]^-{\eta^E\wedge 1}
    & E^{tG}\wedge_E Y^{tG}
    \ar[r]^-{\mu^E_{E,Y}}
    & (E\wedge_E Y)^{tG} \\
    E\wedge Y^{tG}
    \ar[r]^-{\eta^E\wedge 1}
    \ar[u]
    & E^{tG}\wedge Y^{tG}
    \ar[r]^-{\mu_{E,Y}}
    \ar[u]
    & (E\wedge Y)^{tG}
    \ar[u] \\
    S\wedge Y^{tG}
    \ar[r]^-{\eta^S\wedge 1}
    \ar[u]_{\eta\wedge 1}
    & S^{tG}\wedge Y^{tG}
    \ar[r]^-{\mu_{S,Y}}
    \ar[u]_{\eta^{tG}\wedge 1}
    & (S\wedge Y)^{tG}
    \ar[u]_{(\eta\wedge 1)^{tG}} \,,
  }
\]
where the vertical unlabeled morphisms are the canonical maps from the
smash product over~$S$ to the smash product over~$E$.  The left- and
right-hand vertical compositions are the canonical equivalences
induced by unit $\eta\: S \to E$, and the composition across the lower
row agrees with the monoidal equivalence
$S\wedge Y^{tG} \simeq Y^{tG}$ since $Y\mapsto Y^{tG}$ is a monoidal
functor.  Associativity (resp.~symmetry) of the $E$-based
pairing~$\mu^E$ follows from associativity (resp.~symmetry) of the
$S$-based pairing~$\mu$ by way of passing from smash products over~$S$ to
smash products over~$E$.

For the remainder of this subsection, fix a prime $p$ and let
$E=H=H\bF_p$ be the mod~$p$ Eilenberg--MacLane spectrum with the
trivial $G$-action.  We will consider induced $H$-modules
$Y=H\wedge X$ and their $G$-Tate constructions.
\begin{definition}
  For any spectrum~$X$ with $G$-action, define the \emph{continuous
    mod $p$ homology} of $X^{tG}$ to be the homotopy groups
  $$
  H_*^c(X^{tG}) = \pi_* (H\wedge X)^{tG}\,.
  $$
\end{definition}
By the $H$-module structure of $(H\wedge X)^{tG}$, the continuous mod
$p$ homology $H_*^c(X^{tG})$ has the structure of a graded
$\bF_p$-vector space.  Furthermore, any map $a\:H\to \Sigma^k H$,
representing an element in $\sA^k$, induces a degree-shifting
homomorphism
$$
H_*^c(X^{tG}) = \pi_*(H\wedge X)^{tG} \longto
\pi_*(\Sigma^k H\wedge X)^{tG} \cong H_{*-k}^c(X^{tG})\,.
$$
By ``one of the best-kept secrets of stable homotopy theory''
\cite{Boa82}*{p.~203}, this defines a left action of $\sA$ on
$H_*^c (X^{tG})$, and we obtain a right action by letting $a\in\sA$
left-act through the conjugate element $\chi(a)$.

\begin{proposition}\label{prop:homotopy-homology}
  Let $X$ be a spectrum with $G$-action.  There are natural homotopy
  equivalences
  \begin{equation*}
    (H\wedge X)^{tG}[n] \simeq H\wedge (X^{tG}[n])
  \end{equation*}
  that are compatible for varying~$n$.
\end{proposition}
\begin{proof}
  See the first half of the proof of \cite{LNR12}*{Prop. 4.16}.
\end{proof}

From Proposition~\ref{prop:homotopy-homology} we get that the
canonical structure maps of the Tate tower~\eqref{eq:tate-tower} for
$H\wedge X$ induce homomorphisms $H^c_* (X^{tG}) \to H_* (X^{tG}[n])$,
which are homomorphism of right $\sA$-modules. It follows that
$F_n H^c_* (X^{tG}) = F_n \pi_*(H \wedge X)^{tG}$, the $n$-th Tate
filtration for the $G$-spectrum $H \wedge X$, is a right
$\sA$-submodule and that $H^c_* (X^{tG})$ is a filtered right
$\sA$-module.

For $p$ an odd prime, let $\P^k\: H\to \Sigma^{2k(p-1)}H$ represent the
reduced power $\P^k\in \sA$ for each $k\geq 0$.  The homotopy
commutative diagram
\[
  \xymatrix@C32mm{
    H\wedge H
    \ar[r]^-{\mu}
    \ar[d]^{\bigvee_{i+j=k} \P^i\wedge \P^j}
    &
    H
    \ar[d]^{\P^k}
    \\
    \bigvee_{i+j=k} \Sigma^{|\P^i|} H \wedge \Sigma^{|\P^j|}H
    \ar[r]^-{\sum_{i+j=k} (\Sigma^{|\P^k|}\mu) \circ (23)}
    & \Sigma^{|\P^k|} H
}
\]
represents the Cartan formula $\P^k(xy) = \sum_{i+j=k}\P^i(x)\P^j(y)$ in
cohomology, and induces the homotopy commutative diagram of spectra
\begin{equation}
  \label{eq:cartan-continuous-homology}
  \begin{aligned}
    \xymatrix{
    (H\wedge X)^{tC_p} \wedge (H\wedge X')^{tC_p}
    \ar[r]^-{\mu^{H}_{H\wedge X, H\wedge X'}}
    \ar[d]
    &
      (H\wedge X\wedge X')^{tC_p}
      \ar[d]
    \\
    \bigvee_{i+j=k} \Sigma^{|\P^i|} (H\wedge X)^{tC_p} \wedge \Sigma^{|\P^j|}(H\wedge X')^{tC_p}
    \ar[r]
    & \Sigma^{|\P^k|} (H\wedge X\wedge X')^{tC_p} \,.
      }
  \end{aligned}
\end{equation}
Passing to homotopy groups, and tensor products over $\bF_p$,
diagram~\eqref{eq:cartan-continuous-homology} gives rise to a Cartan
formula in continuous homology.  Similar diagrams exist for the
Bockstein~$\beta$, and for the Steenrod squares~$\Sq^k$ when $p=2$.
It follows that the homomorphism
\begin{equation}
  \label{eq:tate-pairing-homology}
  (\mu^H_{H\wedge X, H\wedge Y})_*\:
  H_*^c(X^{tC_p})\otimes H_*^c(Y^{tC_p})
  \longto
  H_*^c((X\wedge Y)^{tC_p})
\end{equation}
induced by the monoidal structure is $\sA$-linear when its domain is
given the diagonal right action by~$\sA$.

By the discussion in Subsection~\ref{sec:hesselholt-madsen}, there is
a set of $\sA$-linear homomorphisms
 \[
   \mu_{a,b} \: F_a H^c_*(X^{tG}) \otimes F_b H^c_*(Y^{tG}) \longto
   F_{a+b} H^c_*((X\wedge X')^{tG})
\]
compatible with \eqref{eq:tate-pairing-homology}, making it a filtered
morphism.

Thus, the assignment $X\mapsto H_*^c(X^{tG})$ is a lax symmetric
monoidal functor from the homotopy category of spectra with
$G$-action to the category of filtered right $\sA$-modules.

\begin{lemma}
  \label{lemma:compact-hausdorff-tate}
  Let $X$ be a spectrum with $G$-action such that $X/p$ is
  $k$-connective for some integer~$k$.  Then $H_*(X^{tG}[n])$ is
  $(n+k)$-connective for each $n$, and $H_*^c(X^{tG})$ endowed with
  the Tate filtration is relatively bounded below.

  If $X/p$ is also of finite type, then $H_*(X^{tG}[n])$ and
  $H^c_*(X^{tG})/F_n H^c_*(X^{tG})$ are of finite type for
  each~$n$.
\end{lemma}
\begin{proof}
  Restricting the Tate tower \eqref{eq:tate-tower} to
  $$
  X^{tG}[n] \longto X^{tG}[n+1] \longto \dots \longto {*} \,,
  $$
  and applying mod $p$ homology, gives rise to a truncated Tate
  spectral sequence, with $\hat E^r_{s,*}=0$ for $s<n$, converging
  conditionally to $H_*(X^{tG}[n])$.  Assuming that $X/p$ is
  $k$-connective, it follows that $H_*(X)$ is $k$-connective as a
  graded module, and therefore that $\hat E^r_{s,t}=0$ if $s<n$ or
  $t<k$.  Therefore, there are only finitely many nonzero
  differentials originating and arriving at any given bidegree, so the
  spectral sequence converges strongly.  Also, the abutment
  $H_*(X^{tG}[n])$ and the quotient
  \[
    H^c_* (X^{tG}) / F_{n-1} H^c_* (X^{tG})\cong \im\!\left( H^c_* (X^{tG})
      \longto H_* (X^{tG}[n]) \right)
  \]
  are $(n+k)$-connective.  Thus, when $X/p$ is bounded below,
  $H^c_* (X^{tG})$ is a relatively bounded below filtered right
  $\sA$-module.

  If $X/p$ is of finite type, then each term of the truncated Tate
  spectral sequence is finite dimensional as an $\bF_p$-vector space
  in each bidegree, and it follows that $H_*(X^{tG}[n])$, and
  therefore also $H^c_* (X^{tG}) / F_{n-1} H^c_* (X^{tG})$, is of finite
  type.
\end{proof}

\begin{lemma}\label{lemma:continuous-homology-is-inverse-limit}
  Let $X$ be a spectrum with $G$-action.  The canonical homomorphism
  $$
  i\: H^c_*(X^{tG})^\wedge \longto \lim_n H_*(X^{tG}[n])
  $$
  is an isomorphism of complete Hausdorff filtered right
  $\sA$-modules, and there is a short exact sequence
  \begin{equation*}
    0\longto
    \rlim_n H_{*+1}(X^{tG}[n])\longto
    H^c_*(X^{tG}) \overset{c}\longto
    H_*^c(X^{tG})^\wedge\longto 0\,,
  \end{equation*}
  where~$c$ is the completion homomorphism.  In particular,~$c$ is an
  isomorphism of filtered right~$\sA$-modules if and only if
  $\rlim_n H_{*}(X^{tG}[n])=0$.
\end{lemma}
\begin{proof}
  The lemma follows by replacing $X$ by $H\wedge X$ in Lemma
  \ref{lemma:completetatefiltration} and applying the equivalence of
  Proposition \ref{prop:homotopy-homology} to identify
  $\pi_* (H\wedge X)^{tG}[n] \cong H_* (X^{tG}[n])$.
\end{proof}

The homological analogue of Lemma \ref{lemma:tatefiltration} is the
following:
\begin{lemma}\label{lemma:homology-tatefiltration}
  Let $X$ be a spectrum with $G$-action.  The Tate filtration of
  $H^c_*(X^{tG})$ is exhaustive and complete.  It is Hausdorff if and
  only if $\rlim_n H_{*}(X^{tG}[n]) = 0$.  \qed
\end{lemma}

Replacing $X$ by $H\wedge X$ in~\eqref{eq:gmtate}, we get a
\emph{homological $G$-Tate spectral sequence} converging
conditionally to $H_*^c(X^{tG})$.  It has $\hat E^2$-term given by
\begin{equation}
  \label{eq:homological-tate-ss}
  \hat{E}^2_{n,m} = \hat{H}^{-n}(G; H_m(X))\,.
\end{equation}

\begin{lemma} \label{lemma:strong-convergence-homology} Suppose that
  $X/p$ is bounded below.  The homological $G$-Tate spectral sequence
  is strongly convergent if one of the following conditions hold:
  \begin{enumerate}
  \item $R\hat E^{\infty}=0$.
  \item $X/p$ is of finite type.
  \item The spectral sequence collapses at a finite stage.
  \end{enumerate}
  Note that strong convergence of the homological $G$-Tate spectral
  sequence implies that the Tate filtration of $H^c_*(X^{tG})$ is
  complete Hausdorff and exhaustive.
\end{lemma}
\begin{proof}
  When $X$ is replaced by $H\wedge X$, conditions (2) and (3) of Lemma
  \ref{lemma:strong-convergence-homotopy} are both satisfied if $X/p$
  is of finite type.
\end{proof}

We summarize some of the conclusions in this subsection.
\begin{proposition}\label{prop:continuous-homology-is-monoidal}
  The assignment $X\mapsto H_*^c(X^{tG})$ is a lax symmetric monoidal
  functor from the homotopy category of spectra with $G$-action to
  the category of filtered right $\sA$-modules.

  If $X/p$ is bounded below then $H_*^c(X^{tG})$ is relatively bounded
  below.  If, in addition, one of the criteria (1)--(3) of
  Lemma~\ref{lemma:strong-convergence-homology} is satisfied, then the
  Tate filtration of $H_*^c(X^{tG})$ is complete Hausdorff.

  In particular, if $X$ is an $E_1$ ring spectrum with $G$-action, and
  $X/p$ is bounded below and of finite type, then $H_*^c(X^{tG})$ is
  an rbb filtered right $\sA$-module algebra, i.e., an object of
  $\alg(\sC, \otimes)$.  If $X/p$ is also of finite type, then
  $H_*^c(X^{tG})$ is also complete.

  \qed
\end{proposition}

By Proposition~\ref{prop:alg-category-iso}, $H_*^c(X^{tG})$ can also
be viewed as an rbb complete left $\sA_*$-comodule algebra if $X$ is
an $E_1$ ring spectrum with $G$-action and $X/p$ is bounded below and
of finite type.

\section{A limit of Adams spectral sequences for $X^{tG}$}\label{sec:limit-of-adams-ss}
Let~$G$ be a compact Lie group, and $X$ a spectrum with $G$-action
such that $X/p$ is bounded below and of finite type.  For
each~$n$, the classical Adams spectral sequence
$$
E^{s,t}_2(X^{tG}[n]) = \Ext_{\sA_*}^{s,t}(\bF_p, H_*(X^{tG}[n]))
\Longrightarrow \pi_{t-s}(X^{tG}[n])^\wedge_p
$$
computes the homotopy groups of the $n$-th Tate truncation,
cf.~\cite{LNR12}*{Def.~4.3}.  We show that there is a
\textit{limit Adams spectral sequence}
$$
E_2^{s,t}(X^{tG}) =
\cExt_{\sA_*}^{s,t}(\bF_p, H^c_*(X^{tG}))
\Longrightarrow
\pi_{t-s}(X^{tG})^\wedge_p\,,
$$
where the $E_2$-term is the cohomology of the \textit{continuous cobar
  complex} of Definition~\ref{dfn:complete-cobar-complex}.

Additively, and in the cohomological setting, this result can be found
in \cite{LNR11}*{Prop.~2.2}.  This time around, we prove the result
working exclusively in mod~$p$ homology, and we also analyze the
multiplicative structure of the limit Adams spectral sequence.

In Subsection~\ref{sec:cext} we define the \textit{continuous
  $\Ext$-groups} of an rbb complete left $\sA_*$-comodule.  We construct the
limit Adams spectral sequence in Subsection~\ref{sec:limit-adams}, and
discuss its multiplicative structure in
Subsection~\ref{sec:limit-adams-multiplicativity}.  In
Subsection~\ref{sec:cext-ext} we identify the continuous
$\Ext$-groups of an rbb complete left $\sA_*$-comodule~$P_*$ with the
algebraic $\Ext$-groups of $P_*$ thought of as a right $\sA$-module.

In Subsection \ref{sec:lass-is-multiplicative} we identify the
multiplicative structure on the $E_2$-term of the limit Adams spectral
sequence with with the cup product in algebraic $\Ext$-groups, via the
isomorphism from Subsection~\ref{sec:cext-ext}.

\subsection{Continuous $\Ext$ of rbb complete left $\sA_*$-comodules}\label{sec:cext}
\begin{definition}\label{dfn:complete-cobar-complex}
  For $P_*$ in $\rbbfilacomodc$, let
  $\widehat{C}_{\sA_*}^*(\bF_p, P_*)$ be the un-normalized cobar
  complex
  \[
    \widehat{C}_{\sA_*}^s(\bF_p, P_*) = \sA_*^{\otimes s}\ctensor P_*\,,
  \]
  with coboundary $\delta$ the usual alternating sum of coface maps, induced by the
  unit $\eta \: \bF_p \to \sA_*$, the coproduct
  $\psi \: \sA_* \to \sA_* \otimes \sA_*$ and the coaction
  $\nu \: P_* \to \sA_* \ctensor P_*$.  See
  \cite{Ra86}*{Def.~A1.2.11}.  We define the \emph{continuous}
  $\sA_*$-comodule $\Ext$ of~$\bF_p$ and~$P_*$ to be the cohomology
  $\cExt^s_{\sA_*}(\bF_p, P_*) = H^s( \widehat{C}_{\sA_*}^*(\bF_p,
  P_*), \delta )$.
\end{definition}
We avoid the notation $\widehat{\Ext}$, since this already means
``stable $\Ext$'' in the terminology of~\cite{AV07}*{\S1.4}, being
related to ordinary $\Ext$ in the same way as Tate cohomology is
related to group cohomology.

A morphism $f\:P_*\to Q_*$ in $\rbbfilacomodc$ induces a map of graded
$\bF_p$-vector spaces
\begin{equation}
  \label{eq:extiso}
  f_*\:\cExt_{\sA_*}^s (\bF_p, P_*)\longto
  \cExt_{\sA_*}^s (\bF_p, Q_*)\,.
\end{equation}
\begin{proposition}\label{prop:ext-iso}
  Let $f\:P_*\to Q_*$ be a morphism in $\rbbfilacomodc$.  If $f$ is an
  isomorphism of (unfiltered) graded $\bF_p$-vector spaces, then
  \eqref{eq:extiso} is an isomorphism for each~$s$.
\end{proposition}
Note that the hypothesis on $f$ does not imply that $f$ is an
isomorphism in $\rbbfilacomodc$, since its inverse might not be
filtration-preserving.
\begin{proof}
  The isomorphism $\hat\iota$ of Lemma \ref{lemma:iotafilmods} is
  natural with respect to morphisms $f\:P_*\to Q_*$.  Since $P_*$ and
  $Q_*$ are complete Hausdorff, there are natural isomorphisms
  $P_* \cong P_*^\wedge$ and $Q_* \cong Q_*^\wedge$, and we get a
  commutative square
  \[
    \xymatrix@C20mm{
      \Hom(\sA^{\otimes s}, P_*) \ar[r]^{\Hom(1, f)}
      & \Hom(\sA^{\otimes s}, Q_*)\\
      \widehat{C}^s_{\sA_*}(\bF_p, P_*) \ar[r] \ar[u]^\cong_{\hat\iota}
      &\widehat{C}^s_{\sA_*}(\bF_p, Q_*) \ar[u]^\cong_{\hat\iota}\rlap{\,.}
   }
 \]
 The top horizontal map is an isomorphism by hypothesis.  Thus,~$f$
 induces a map of cobar complexes that is an isomorphism in each
 codegree $s$.  It follows that \eqref{eq:extiso} is an isomorphism
 for each $s$.
\end{proof}

\begin{lemma}\label{lemma:cobarlimit}
  For each $P_*$ in $\rbbfilacomodc$, the canonical homomorphisms
  $P_*\to P_*/F_nP_*$ induce a natural isomorphism of complexes
  \begin{equation}
    \label{eq:cobarlimit}
    \widehat{C}_{\sA_*}^*(\bF_p, P_*) \overset\cong\longto
    \lim_n C_{\sA_*}^*(\bF_p, P_*/F_nP_*)\,.
  \end{equation}
\end{lemma}
\begin{proof}
  For each $n\in \bZ$, let $j_n\: P_*\to P_*/F_nP_*$ be the canonical
  projection.  Let $\widehat{C}^\bullet_{\sA_*}(\bF_p, P_*)$ be the
  pre-cosimplicial graded $\bF_p$-vector space with associated
  (un-normalized) Moore cochain complex
  $\widehat{C}^*_{\sA_*}(\bF_p, P_*)$.  Each coface homomorphism
  $d^i\: \widehat{C}_{\sA_*}^s(\bF_p, P_*)\to
  \widehat{C}_{\sA_*}^{s+1}(\bF_p, P_*)$ is filtration-preserving.
  (The interesting case here is $i=s+1$, when $d^{s+1}$ is induced by
  the complete coaction $\nu\: P_*\to \sA_*\ctensor P_*$.)  Thus, for
  each~$0\leq i \leq s+1$ we get a commutative square
  \begin{equation}
    \begin{aligned}
      \label{eq:ttrunc}
      \xymatrix{
      \sA_*^{\otimes s}\ctensor P_*
      \ar[r]^-{1\ctensor j_n} \ar[d]^{d^i}
      & \sA_*^{\otimes s}\otimes P_*/F_nP_*
        \ar[d]^{d_n^i}
      \\
      \sA_*^{\otimes (s+1)}\ctensor P_*
      \ar[r]^-{1\ctensor j_n}
      & \sA_*^{\otimes (s+1)}\otimes P_*/F_nP_* \,,
        }
    \end{aligned}
  \end{equation}
  where we have implicitly used the canonical isomorphism
  $\sA_*^{\otimes s}\ctensor P/F_nP_* \cong \sA_*^{\otimes s}\otimes
  P/F_nP_*$ to identify the right-hand column.  For a fixed~$n$ the
  homomorphisms~$d^i_n$ of \eqref{eq:ttrunc} are the coface operators
  of $C^\bullet_{\sA_*}(\bF_p, P_*/F_nP_*)$, and the isomorphism
  \eqref{eq:cobarlimit} follows from forming the alternating sum over
  the coface operators and passing to the limit over~$n$, using
  the natural isomorphism~\eqref{eq:ctensor-discrete}.
\end{proof}

\begin{lemma}
  \label{lemma:extlimit}
  Let $X$ be a spectrum with $G$-action such that $X/p$ is bounded
  below and of finite type.  The canonical truncation maps
  $X^{tG}\to X^{tG}[n]$ induce a natural isomorphism of
  complexes
  \begin{equation}
    \label{eq:cobar-tate-limit}
    \widehat{C}_{\sA_*}^*(\bF_p, H_*^c(X^{tG})) \overset\cong\longto
    \lim_n C_{\sA_*}^*(\bF_p, H_*(X^{tG}[n]))\,.
  \end{equation}
\end{lemma}
\begin{proof}
  The limit of the canonical monomorphisms
  \begin{equation}
    \label{eq:cobarx}
    \frac{H_*^c (X^{tG})}{F_{n-1} H_*^c (X^{tG})} \longto H_* (X^{tG}[n])
  \end{equation}
  is an isomorphism by Lemma~\ref{lemma:continuous-homology-is-inverse-limit}.
  Tensoring both sides of~\eqref{eq:cobarx} with $\sA_*^{\otimes s}$
  and passing to the limit, we get a natural isomorphism
  \begin{equation}
    \label{eq:cobary}
    \sA_*^{\otimes s}\ctensor H_*^c(X^{tG}) \overset\cong\longto
    \lim_n \sA_*^{\otimes s}\otimes H_*(X^{tG}[n])
  \end{equation}
  by Lemma~\ref{lemma:f-iso}.  Naturality ensures that
  \eqref{eq:cobary} for various~$s$ will commute with the coface
  operators.  Hence,
  \begin{equation}
    \label{eq:cobar-limits}
    \lim_n C_{\sA_*}^*\bigl(\bF_p, \frac{H_*^c (X^{tG})}{F_{n-1} H_*^c (X^{tG})}\bigr)
    \overset\cong\longto
    \lim_n C_{\sA_*}^*(\bF_p, H_*(X^{tG}[n]))\,.
  \end{equation}
  Applying Lemma~\ref{lemma:cobarlimit} with $P_*=H_*^c(X^{tG})$
  identifies the left-hand side of~\eqref{eq:cobar-limits},
  and the isomorphism~\eqref{eq:cobar-tate-limit} follows.
\end{proof}

\begin{proposition}
  Let $X$ be a spectrum with $G$-action such that $X/p$ is bounded
  below and of finite type.  The canonical truncation maps
  $X^{tG}\to X^{tG}[n]$ induce a natural isomorphism
  \begin{equation}
    \label{eq:cext-tate-limit}
    \cExt^s_{\sA_*}(\bF_p, H^c_*(X^{tG})) \overset\cong\longto
    \lim_n \Ext^s_{\sA_*}(\bF_p, H_*(X^{tG}[n]))
  \end{equation}
  for each $s$.
\end{proposition}

\begin{proof}
  Let $Q[n]_* = H_*(X^{tG}[n])$.  Consider the tower of cobar
  complexes
  \begin{equation}
    \label{eq:limit-cobar}
    \dots \to C^*_{\sA_*}(\bF_p, Q[n]_*)
    \longto C^*_{\sA_*}(\bF_p, Q[n+1]_*) \to \dots
  \end{equation}
  given in codegree~$s$ by
  \begin{equation}
    \label{eq:cobar-bracket-n}
    \dots \to \sA_*^{\otimes s} \otimes Q[n]_*
    \longto \sA_*^{\otimes s} \otimes Q[n+1]_* \to \dots \,.
  \end{equation}
  By Lemma~\ref{lemma:extlimit}, the limit of~\eqref{eq:limit-cobar}
  is isomorphic to the completed cobar complex
  $\widehat C^*_{\sA_*}(\bF_p, H_*(X^{tG}))$.  By
  Lemma~\ref{lemma:compact-hausdorff-tate}, each $Q[n]_*$ is of finite
  type, hence so is each $\sA_*^{\otimes s} \otimes Q[n]_*$.
  Therefore, \eqref{eq:cobar-bracket-n} is a Mittag-Leffler tower for
  each~$s$, and it follows from \cite{Weibel94}*{Thm.~3.5.8} that
  there is a short exact sequence
  \begin{multline}
      \label{eq:cext-lim}
      0
      \longto
      \rlim_n\Ext^{s-1}_{\sA_*}(\bF_p, Q[n]_*)
      \longto
      \cExt^s_{\sA_*}(\bF_p, H_*^c(X^{tG})) \\
      \longto
      \lim_n \Ext^s_{\sA_*}(\bF_p, Q[n]_*)
      \longto
      0
  \end{multline}
  for each~$s$.  Since each $Q[n]_*$ is of finite type, so is
  $\Ext^{s-1}_{\sA_*}(\bF_p, Q[n]_*)$ for each $n$.  Therefore the
  derived limit in~\eqref{eq:cext-lim} vanishes, and the proposition
  follows.
\end{proof}

\subsection{Bousfield--Kan spectral sequences}\label{sec:limit-adams}
Let $\Delta$ be the category with objects
$[q] = \{0 < 1 < \dots < q\}$ for $q\ge0$ and order-preserving
functions, and let $M = \Delta_{\inj}$ be the wide subcategory
of~$\Delta$ containing only the injective morphisms.  For a category
$\sC$, we will refer to objects of the functor categories
$\sC^M = \Fun(M, \sC)$ and $\sC^\Delta = \Fun(\Delta, \sC)$ as
pre-cosimplicial and cosimplicial objects in~$\sC$, respectively.

Let $\Delta^\bullet$ be the cosimplicial space $[q]\mapsto \Delta^q$,
and let $\sk_k\Delta^\bullet$ be the cosimplicial space that in
codegree~$q$ is the~$k$-skeleton of $\Delta^q$.

Let $\Sp$ be the symmetric monoidal category of orthogonal spectra.
For $Y^\bullet$ in $\Sp^M$, the \emph{fat totalization} of~$Y^\bullet$
is the mapping spectrum
$\TOT Y^\bullet = \map_{\Sp^M}(\Delta^\bullet_+, Y^\bullet)$.  When
$Y^\bullet$ is in~$\Sp^\Delta$, the (ordinary) \emph{totalization} is
$\Tot Y^\bullet = \map_{\Sp^\Delta}(\Delta^\bullet_+, Y^\bullet)$.
The inclusion of categories $M \subset \Delta$ is left cofinal (=
initial), so that
$$
\Tot Y^\bullet \simeq \holim_\Delta Y^\bullet
	\overset{\simeq}\longto \holim_M Y^\bullet \simeq \TOT Y^\bullet
$$
is an equivalence and $\pi_*(\Tot Y^\bullet) \cong \pi_*(\TOT Y^\bullet)$.
Returning to the generality of pre-cosimplicial $Y^\bullet$, we let
\begin{align*}
\fil^s \TOT Y^\bullet
  &= \map_{\Sp^M}(\Delta^\bullet/\sk_{s-1} \Delta^\bullet, Y^\bullet) \\
\gr^s \TOT Y^\bullet
  &= \map_{\Sp^M}(\sk_s \Delta^\bullet/\sk_{s-1} \Delta^\bullet, Y^\bullet)
\end{align*}
for each $s\ge0$.
The inclusions $\sk_{s-1} \Delta^\bullet \subset \sk_s \Delta^\bullet$
produce a tower of spectra
\begin{equation} \label{eq:BK-tower}
\dots \longto \fil^{s+1} \TOT Y^\bullet
	\longto \fil^s \TOT Y^\bullet \longto
	\dots \longto \fil^0 \TOT Y^\bullet
	= \TOT Y^\bullet \,,
\end{equation}
with trivial homotopy limit.  The cofiber sequences
$$
\sk_s \Delta^\bullet/\sk_{s-1} \Delta^\bullet
\longto \Delta^\bullet/\sk_{s-1} \Delta^\bullet
\longto \Delta^\bullet/\sk_s \Delta^\bullet
$$
induce (co-)fiber sequences
$$
\fil^{s+1} \TOT Y^\bullet \longto \fil^s \TOT Y^\bullet
	\longto \gr^s \TOT Y^\bullet \,.
$$
The Bousfield--Kan spectral sequence associated to~$Y^\bullet$ is the
spectral sequence derived from the unrolled exact couple obtained by
applying homotopy groups to these cofiber sequences.
For each $s\ge0$ there is an equivalence
$$
\gr^s \TOT Y^\bullet =
\map_{\Sp^M}(\sk_s \Delta^\bullet/\sk_{s-1} \Delta^\bullet, Y^\bullet)
	\overset{\simeq}\longto \Omega^s Y^s \,,
$$
given by restricting to codegree~$s$.  Hence the Bousfield--Kan spectral
sequence has the form
\begin{equation}
  \label{eq:BK-E1}
  E_1^{s,t} = \pi_{t} (Y^s) \Longrightarrow
  \pi_{t-s}(\Tot Y^\bullet) \,,
\end{equation}
with $d^{s,t}_1\: \pi_t (Y^s) \to \pi_t (Y^{s+1})$ given by the
alternating sum $d^{s,t}_1 = \sum_{i=0}^{s+1} (-1)^i d^i_*$, where
$d^i\: Y^s\to Y^{s+1}$ is the $i$-th coface map of $Y^\bullet$.  In
other words, $E^{*,t}_1$ is the (un-normalized) Moore complex
associated with the pre-cosimplicial abelian group $\pi_t(Y^\bullet)$.

It follows from the Milnor $\lim$--$\rlim$ short exact sequence, and
the fact that $\holim_s \fil^s \TOT Y^\bullet$ is trivial, that both
$\lim_s \pi_*(\fil^s \TOT Y^\bullet)$ and
$\rlim_s \pi_*(\fil^s\TOT Y^\bullet)$ are trivial groups.  Thus, the
Bousfield--Kan spectral sequence is a half-plane spectral sequence
with entering differentials, converging conditionally to
$\pi_*(\TOT Y^\bullet)$.  It follows from \cite{Boa99}*{Thm.~7.1}
that \eqref{eq:BK-E1} converges strongly if Boardman's derived
$E_\infty$-term $RE^{s,t}_\infty$ vanishes for each $s$ and $t$.  In
the current setting, strong convergence is implied if $\pi_t(Y^s)$ is
finite for each $s$ and $t$.

\begin{example}[The Adams spectral sequence]\label{ex:adams-sp-seq}
  For a fixed prime~$p$, let $H$ be the mod~$p$ Eilenberg--MacLane
  spectrum, and assume that $X$ is a spectrum such that $X/p$ is
  bounded below and of finite type.  There is a cosimplicial spectrum
  $Y^\bullet$ with
  \begin{equation}
    \label{eq:cobarcomplex}
    [s] \mapsto Y^s = H^{\wedge(1+s)}\wedge X\,,
  \end{equation}
  coface maps induced by the ring spectrum unit $\eta\: S\to H$, and
  codegeneracy maps induced by the associative ring spectrum product
  $H \wedge H \to H$.  Consider its underlying pre-cosimplicial
  spectrum, also denoted~$Y^\bullet$.  Then $(E^*_1, d_1)$ of the
  associated Bousfield--Kan spectral sequence is naturally isomorphic
  to the un-normalized cobar complex
  $(C^*_{\sA_*}(\bF_p,H_*(X)), \delta)$.  Thus,
  \begin{equation}
    \label{eq:adams-ss}
    E_2^{s,t} = \Ext_{\sA_*}^{s,t}(\bF_p, H_*(X))
    \Longrightarrow \pi_{t-s} (\TOT Y^\bullet)
    \cong \pi_{t-s} (\Tot Y^\bullet)\,.
  \end{equation}
  Since $H_*(X)$ is of finite type, $C^*_{\sA_*}(\bF_p, H_*(X))$ is
  finite in each bidegree, implying strong convergence of
  \eqref{eq:adams-ss}.

  It follows from \cite{MNN17}*{Prop.~2.14} that the totalization
  $\Tot Y^\bullet$ is equivalent to the Bousfield~$H$-nilpotent
  completion $X^\wedge_H$.  When $X$ is bounded below, Bousfield shows
  \cite{B79}*{Thm.~6.6 and Prop.~2.5} that
  $X^\wedge_H \simeq X^\wedge_p$ is $p$-completion.  In fact it
  suffices that $X/p$ is bounded below:  If $X/p$ is $k$-connected,
  then $p$ acts invertibly on $\pi_i(X)$ for $i < k$.  The cofiber
  sequence $\tau_{\ge k} X \to X \to \tau_{<k} X$, and Bousfield's
  result for the bounded below spectrum $\tau_{\ge k} X$, reduces us
  to showing that $\tau_{<k} X$ has trivial $p$-completion and
  $H$-nilpotent completion.  The first claim follows from
  $S/p \wedge \tau_{<k} X \simeq *$, which also implies
  $H \wedge \tau_{<k} X \simeq *$ and the second claim.
\end{example}

\begin{proposition}
  \label{prop:inverse-limit-of-adams-ss}
  Let $G$ be a compact Lie group, and $X$ a spectrum with $G$-action
  such that $X/p$ is bounded below and of finite type.  There is a
  limit Adams spectral sequence
  \begin{equation}
    \label{eq:limit-adams}
    E_2^{s,t} = \cExt^{s,t}_{\sA_*}(\bF_p, H_*^c(X^{tG})) \Longrightarrow
    \pi_{t-s} \left((X^{tG})^\wedge_p\right)\,,
  \end{equation}
  converging strongly to the homotopy groups of the $p$-completion of
  the $G$-Tate spectrum $X^{tG}$.
\end{proposition}
\begin{proof}
  By applying the $G$-Tate construction in each codegree of
  \eqref{eq:cobarcomplex}, we get a pre-cosimplicial spectrum
  $Y^\bullet$, equal to
  \begin{equation}
    \label{eq:cobarcomplex-tate}
    [s] \mapsto (H^{\wedge(1+s)}\wedge X)^{tG}\,.
  \end{equation}
  The action of~$G$ on the smash product $H^{\wedge(1+s)}\wedge X$ is
  defined by giving~$H$ the trivial $G$-action and forming smash
  products in the category of spectra with $G$-action.  The spectral
  sequence we are after is the Bousfield--Kan spectral sequence
  associated with $Y^\bullet$, and we start by identifying its $E_1$-
  and $E_2$-terms.

  For each $n\in\bZ$, let $Y^\bullet[n]$ be the pre-cosimplicial
  spectrum
  $$
  [s] \mapsto H^{\wedge (1+s)}\wedge X^{tG}[n]\,.
  $$
  There are natural equivalences
  $(H^{\wedge (1+s)}\wedge X)^{tG}[n] \simeq H^{\wedge (1+s)}\wedge
  X^{tG}[n]$, from which we get a natural map of pre-cosimplicial
  spectra $Y^\bullet\to Y^\bullet[n]$, for each $n$, and these
  assemble into an equivalence $Y^\bullet\simeq \holim_n Y^\bullet[n]$
  by \cite{LNR12}*{Lem.~4.4}.

  For each~$n$ we get a map of towers,
  \begin{equation}
    \label{eq:tate-totalization-limit}
    \begin{aligned}
      \xymatrix{
      \ldots \ar[r]
      & \fil^{s+1}\TOT Y^\bullet \ar[r]\ar[d]
      & \fil^s\TOT Y^\bullet \ar[r]\ar[d]
      & \ldots \ar[r]
      & \TOT Y^\bullet \ar[d] \\
      \ldots \ar[r]
      & \fil^{s+1}\TOT Y^\bullet [n] \ar[r]
      & \fil^s \TOT Y^\bullet [n] \ar[r]
      & \ldots \ar[r]
      & \TOT Y^\bullet [n]\,,
        }
    \end{aligned}
  \end{equation}
  and an induced map of associated spectral sequences
  $E_1(Y^\bullet)\to E_1(Y^\bullet[n])$, which equals
  \begin{equation}
    \label{eq:isos}
    \pi_* ((H^{\wedge (1+s)}\wedge X)^{tG})
    \longto
    \pi_*(H^{\wedge (1+s)}\wedge X^{tG}[n])
    \cong
    \sA_*^{\otimes s}\otimes H_* (X^{tG}[n])
  \end{equation}
  in filtration degree~$s$.  Under the hypotheses on $X$,
  Lemma~\ref{lemma:compact-hausdorff-tate} implies that the right-hand
  side of \eqref{eq:isos} is of finite type, from which it follows
  that $\rlim_n \pi_*(H^{\wedge (1+s)}\wedge X^{tG}[n])=0$ and that
  \eqref{eq:isos} induces an isomorphism after passing to the limit
  over~$n$.

  Therefore, the $E_1$-term of the Bousfield--Kan spectral sequence
  associated with~$Y^\bullet$ is isomorphic to the limit over~$n$ of
  the $E_1$-terms of the (Adams $\cong$) Bousfield--Kan spectral
  sequences associated with the $Y^\bullet[n]$, and we get
  \begin{align}
    \nonumber E_1^{s,*}(Y^\bullet) &\cong \lim_n C_{\sA_*}^s(\bF_p, H_*(X^{tG}[n]))\\
    \label{eq:BK-E1-iso}
   &\cong \widehat{C}_{\sA_*}^s(\bF_p, H^c_*(X^{tG})) \,,
  \end{align}
  where the isomorphism~\eqref{eq:BK-E1-iso} is implied by
  Lemma~\ref{lemma:extlimit}.  The description of the $E_2$-term
  \eqref{eq:limit-adams} now follows from the definition of the
  continuous $\Ext$-groups.

  Next, we consider convergence.  The argument at the beginning of the
  proof of \cite{LNR12}*{Prop.~2.2} shows that taking filtered limits
  is exact in the category of compact Hausdorff abelian groups.
  Note that by the isomorphism \eqref{eq:cobar-tate-limit} and our
  assumptions on $H_*(X)$, the graded abelian group
  $\widehat C^s_{\sA_*}(\bF_p, H^c_*(X^{tG}))$ is a limit of finite
  groups in each degree. In particular, it is a compact Hausdorff
  abelian group in each degree.  The category of compact Hausdorff
  abelian groups is an abelian category that is Pontryagin dual to the
  category of discrete abelian groups, so kernels, images and
  quotients are also compact Hausdorff.  It follows that $E_r^{s,t}$
  is compact Hausdorff abelian for each $r$, $s$ and $t$.  Therefore,
  Boardman's $RE_\infty$ is trivial by exactness of the limit functor
  in the category of compact Hausdorff abelian groups, and the
  spectral sequence converges strongly by \cite{Boa99}*{Thm.~7.1}.

  Finally, to describe the abutment $\pi_*(\TOT Y^\bullet)$, we note
  that since fat totalization commutes with homotopy limits, there is
  a natural homotopy equivalence
  $ \TOT Y^\bullet \simeq \holim_n \TOT Y^\bullet[n]$.  As in
  Example~\ref{ex:adams-sp-seq}, %now with $X^{tG}[n]$ instead of $X$,
  the fat totalization $\TOT Y^\bullet[n]$ is equivalent to
  the~$p$-completion $X^{tG}[n]^\wedge_p$, and we get
  $\TOT Y^\bullet \simeq \holim_n (X^{tG}[n]^\wedge_p)$, which is
  equivalent to $(X^{tG})^\wedge_p$ since completion commutes with
  homotopy limits.
\end{proof}

\subsection{Multiplicative structure}
\label{sec:limit-adams-multiplicativity}
A pairing of pre-cosimplicial spectra
$m^\bullet \: Y_1^\bullet \wedge Y_2^\bullet \to Y_3^\bullet$,
together with the diagonal embedding
$\diag^\bullet\: \Delta^\bullet \to \Delta^\bullet \times
\Delta^\bullet$, induces a pairing
\begin{equation}
  \label{eq:tot-pairing}
  \TOT Y^\bullet_1 \wedge \TOT Y^\bullet_2 \longto \TOT Y^\bullet_3 \,,
\end{equation}
given by
\begin{equation*}
  \begin{aligned}
    \xymatrix{
    \map_{\Sp^M}(\Delta^\bullet_+, Y_1^\bullet)
    \wedge \map_{\Sp^M}(\Delta^\bullet_+, Y_2^\bullet)
    \ar[d]^\wedge \\
    \map_{\Sp^M}((\Delta^\bullet\times \Delta^\bullet)_+,
    Y_1^\bullet\wedge Y_2^\bullet)
    \ar[d]^{\map_{\Sp^M}(\diag^\bullet, m^\bullet)} \\
    \map_{\Sp^M}(\Delta^\bullet_+, Y_3^\bullet) \,.
    }
  \end{aligned}
\end{equation*}
By working in the category $\Sp^\Delta$ instead of $\Sp^M$, the same
construction produces a pairing
$$
\Tot Y_1^\bullet \wedge \Tot Y_2^\bullet \longto \Tot Y_3^\bullet
$$
of totalizations of cosimplicial spectra, and the two pairings are
compatible along the natural equivalence $\Tot \to \TOT$.

Our next goal is to refine~\eqref{eq:tot-pairing} into a pairing that
respects the filtrations \eqref{eq:BK-tower} for $Y_1^\bullet$,
$Y_2^\bullet$ and $Y_3^\bullet$, i.e., to produce a set of
compatible pairings
\begin{equation}
  \label{eq:filtered-tot-pairing}
  \fil^a \TOT Y_1^\bullet \wedge \fil^b \TOT Y_2^\bullet
  \longto
  \fil^{a+b} \TOT Y_3^\bullet\,,
\end{equation}
for all non-negative integers~$a$ and~$b$, which in the case $a=b=0$
is homotopic to the unfiltered pairing~\eqref{eq:tot-pairing}.

The obstacle is that $\diag^n\: \Delta^n\to \Delta^n\times \Delta^n$
does not respect the skeleton filtration of its domain and codomain,
when $\Delta^n \times \Delta^n$ is equipped with the product CW
structure (so that its skeleton filtration is the convolution product
of the skeleton filtrations of its two factors).  To sidestep this, we
use the following lemma, which we prove in Subsection~\ref{sec:aw}
below.
\begin{lemma}[The Alexander--Whitney diagonal approximation]
  \label{lemma:aw}
  There is a map of pre-cosimplicial spaces
  \begin{equation}
    \label{eq:AW}
    \AW^\bullet\: \Delta^\bullet
    \longto \Delta^\bullet\times \Delta^\bullet \,,
  \end{equation}
  which restricts to a map of skeleta
  $$
  \sk_k \AW^\bullet\: \sk_k\Delta^\bullet
  \longto \sk_k(\Delta^\bullet\times \Delta^\bullet)
  $$
  for each~$k\geq 0$, such that $\AW^\bullet$ is homotopic to the
  diagonal embedding $diag^\bullet$ through a pre-cosimplicial
  convex homotopy.
\end{lemma}

Using Lemma~\ref{lemma:aw}, we adjust \eqref{eq:tot-pairing} into a
filtration-preserving pairing by replacing the diagonal embedding
$\diag^\bullet$ by the homotopic Alexander--Whitney diagonal
approximation $\AW^\bullet$, as in
\begin{equation}
  \label{eq:tot-pairing-details}
  \begin{aligned}
    \xymatrix{
    \map_{\Sp^M}(\Delta^\bullet/\sk_{a-1} \Delta^\bullet, Y_1^\bullet)
    \wedge
    \map_{\Sp^M}(\Delta^\bullet/\sk_{b-1} \Delta^\bullet, Y_2^\bullet)
    \ar[d]^{\wedge} \\
    \map_{\Sp^M}(\Delta^\bullet/\sk_{a-1} \Delta^\bullet
	\wedge \Delta^\bullet/\sk_{b-1} \Delta^\bullet,
	Y_1^\bullet \wedge Y_2^\bullet)
    \ar[d]^{\map_{\Sp^M}(\overline\AW^{\,\bullet}_{a,b}, m^\bullet)} \\
    \map_{\Sp^M}(\Delta^\bullet/\sk_{a+b-1} \Delta^\bullet,
    Y_3^\bullet)\,.
    }
  \end{aligned}
\end{equation}
Here the last map is induced by the pairing $m^\bullet$ and the
quotient
\begin{equation}
  \label{eq:AW-quotient}
  \overline{\AW}^{\,\bullet}_{a,b} \:
  \frac{\Delta^\bullet}{\sk_{a+b-1} \Delta^\bullet}
  \longto
  \frac{\Delta^\bullet \times \Delta^\bullet}
  {\sk_{a+b-1} (\Delta^\bullet \times \Delta^\bullet)}
  \longto
  \frac{\Delta^\bullet}{\sk_{a-1} \Delta^\bullet}
  \wedge \frac{\Delta^\bullet}{\sk_{b-1} \Delta^\bullet}
\end{equation}
of the skeleton-preserving Alexander--Whitney map, where the first map
is the quotient $\AW^\bullet/\sk_{a+b-1}\AW^\bullet$ and the second
map is the surjection induced by the inclusion
$$
\sk_{a+b-1}(\Delta^\bullet\times \Delta^\bullet) =
\bigcup_{k+\ell=a+b-1}
\sk_{k}\Delta^\bullet \times \sk_{\ell}\Delta^\bullet \subset
(\sk_{a-1}\Delta^\bullet \times \Delta^\bullet) \cup
(\Delta^\bullet \times \sk_{b-1}\Delta^\bullet) \,.
$$
The composition \eqref{eq:tot-pairing-details} is the desired pairing
\eqref{eq:filtered-tot-pairing}.  It is homotopic to the unfiltered
pairing~\eqref{eq:tot-pairing} when $a=b=0$, since
$\AW^\bullet\simeq \diag^\bullet$.

The filtration-preserving pairing \eqref{eq:filtered-tot-pairing}
leads to a pairing of spectral sequences, converging to the pairing
$m_* \: \pi_*(\TOT Y_1^\bullet) \otimes \pi_*(\TOT Y_2^\bullet) \to
\pi_*(\TOT Y_3^\bullet)$, e.g.~by \cite{HR24}*{Thm.~4.27}.  At the
$E_1$-terms, it is induced by the pairing
$$
\gr^a \TOT Y_1^\bullet \wedge \gr^b \TOT Y_2^\bullet \longto \gr^{a+b}
\TOT Y_3^\bullet
$$
of associated graded spectra, given by
\begin{equation}
  \begin{aligned}
    \label{eq:BK-pairing}
    \xymatrix{
    \map_{\Sp^M}
	\bigl(\ds\frac{\sk_a \Delta^\bullet}{\sk_{a-1} \Delta^\bullet},
	Y_1^\bullet\bigr)
    \wedge
    \map_{\Sp^M}
	\bigl(\ds\frac{\sk_b \Delta^\bullet}{\sk_{b-1} \Delta^\bullet},
	Y_2^\bullet\bigr)
    \ar[r]^-{\simeq} \ar[d]^-{\wedge}
	& \Omega^a Y_1^a \wedge \Omega^b Y_2^b \\
    \map_{\Sp^M}
	\bigl(\ds\frac{\sk_a \Delta^\bullet}{\sk_{a-1} \Delta^\bullet}
	\wedge
	\ds\frac{\sk_b \Delta^\bullet}{\sk_{b-1} \Delta^\bullet},
	Y_1^\bullet \wedge Y_2^\bullet\bigr)
	\ar[d]^-{\map_{\Sp^M}(\sk_{a+b} \overline{\AW}^\bullet_{a,b}, m^\bullet)} \\
    \map_{\Sp^M}
	\bigl(\ds\frac{\sk_{a+b} \Delta^\bullet}{\sk_{a+b-1} \Delta^\bullet},
	Y_3^\bullet\bigr)
	\ar[r]^-{\simeq}
	& \Omega^{a+b} Y_3^{a+b}
      }
  \end{aligned}
\end{equation}
where the last vertical map is induced by the pairing $m^\bullet$
and the subquotient
\begin{equation}
  \label{eq:AW-subquotient}
  \sk_{a+b} \overline{\AW}^{\,\bullet}_{a,b}\:
  \frac{\sk_{a+b} \Delta^\bullet}{\sk_{a+b-1} \Delta^\bullet}
  \longto
  \frac{\sk_{a+b}(\Delta^\bullet \times \Delta^\bullet)}
  {\sk_{a+b-1} (\Delta^\bullet \times \Delta^\bullet)}
  \longto
  \frac{\sk_a \Delta^\bullet}{\sk_{a-1} \Delta^\bullet}
  \wedge \frac{\sk_b \Delta^\bullet}{\sk_{b-1} \Delta^\bullet}
\end{equation}
of $\AW^\bullet$, obtained by restricting \eqref{eq:AW-quotient} to the
$(a+b)$-skeleta $\sk_{a+b} \Delta^\bullet$
and $\sk_{a+b}(\Delta^\bullet\times\Delta^\bullet)$.

For $a,b\ge0$ we write $\alpha\: [a]\to [a+b]$ and
$\beta\: [b]\to [a+b]$ in $\Delta$ for the front $a$-face monomorphism
and the back $b$-face monomorphism, respectively.  For a
pre-cosimplicial object $Y^\bullet \in \sC^M$ and any monomorphism
$\mu\: [p]\to [q]$, we simply write $\mu\: Y^p\to Y^q$ instead of
$Y(\mu)$.
\begin{lemma}
  \label{lemma:BK-E1-pairing}
  Passing to homotopy groups, the pairing \eqref{eq:BK-pairing}
  induces the classical Alexander--Whitney homomorphism
  \begin{multline}
    \pi_{*+a} Y_1^a \otimes \pi_{*+b} Y_2^b
    \overset{\alpha_* \otimes \beta_*}\longto
    \pi_{*+a} Y_1^{a+b} \otimes \pi_{*+b} Y_2^{a+b}\\
    \overset{\wedge}\longto 
    \pi_{*+a+b} (Y_1^{a+b} \wedge Y_2^{a+b})
    \overset{m^{a+b}_*}\longto
    \pi_{*+a+b} Y_3^{a+b} \,.
  \end{multline}
\end{lemma}
\begin{proof}
  Given $f^a \: \Delta^a/\partial \Delta^a \to Y_1^a$ and
  $g^b \: \Delta^b/\partial \Delta^b \to Y_2^b$ in $\Omega^a Y_1^a$
  and $\Omega^b Y_2^b$, these extend up to contractible choices to
  pre-cosimplicial maps
  $f^\bullet \: \sk_a \Delta^\bullet/\sk_{a-1} \Delta^\bullet \to
  Y_1^\bullet$ and
  $g^\bullet \: \sk_b \Delta^\bullet/\sk_{b-1} \Delta^\bullet \to
  Y_2^\bullet$, using the pre-cosimplicial structures on the codomains,
  thus lifting $f^a\wedge g^b$ over the top homotopy equivalence
  of~\eqref{eq:BK-pairing}.  Then, to compute the pairing of $f^a$ and
  $g^b$, we need to evaluate $f^\bullet\wedge g^\bullet$ precomposed
  with $\sk_{a+b} \overline{\AW}^{\,\bullet}_{a,b}$, but only in
  codegree~$a+b$.  There is a commutative diagram
  \begin{equation}
    \begin{aligned}
      \label{eq:AW-d1}
      \xymatrix@C12mm{
      \ds \frac{\Delta^{a+b}}{\partial \Delta^{a+b}}
      \ar[rr]^-{\sk_{a+b}\overline\AW^{a+b}_{a,b}}
      \ar@/_4mm/[rrd]_-{\simeq}
      &&
         \ds\frac{\sk_a\Delta^{a+b}}{\sk_{a-1} \Delta^{a+b}}
         \wedge \ds \frac{\sk_b\Delta^{a+b}}{\sk_{b-1} \Delta^{a+b}}
         \ar[r]^-{f^{a+b}\wedge g^{a+b}}
      &
        Y_1^{a+b}\wedge Y_2^{a+b}
      \\
      &&
         \ds\frac{\Delta^a}{\partial \Delta^a}
         \wedge \ds\frac{\Delta^b}{\partial \Delta^b}
         \ar[u]_{\alpha\wedge \beta}
         \ar[r]^-{f^a\wedge g^b}
      &
        Y_1^a\wedge Y_2^b\,,
        \ar[u]_{\alpha\wedge \beta}
         }
    \end{aligned}
  \end{equation}
  where the unlabeled homotopy equivalence has degree~$1$.
  Lemma~\ref{lemma:AW-subcomplex} below gives the left-hand triangle,
  and the right-hand square commutes since $f^\bullet$ and
  $g^\bullet$ are maps of pre-cosimplicial objects.  From
  \eqref{eq:AW-d1} we see that $(f^a, g^b)$ in
  $\Omega^a Y_1^a \wedge \Omega^b Y_2^b$ maps, up to homotopy, to
  $$
  \Delta^{a+b}/\partial \Delta^{a+b} \simeq \Delta^a/\partial \Delta^a
  \wedge \Delta^b/\partial \Delta^b
  \overset{\alpha f^a \wedge \beta g^b}\longto
  Y_1^{a+b} \wedge Y_2^{a+b} \overset{m^{a+b}}\longto Y_3^{a+b}
  $$
  in $\Omega^{a+b} Y_3^{a+b}$, and the lemma follows.
\end{proof}

\subsubsection{The Alexander--Whitney diagonal approximation}
\label{sec:aw}
For $q\geq 0$, let $\Delta^q$ be the standard topological $q$-simplex, given by
the set of points $x=(x_0, \ldots, x_q)\in \bR^{q+1}$ with barycentric
coordinates satisfying $0\leq x_i\leq 1$ and $x_0+\ldots+x_q=1$.  Let
$\Delta[q]$ be the standard simplicial $q$-simplex with $k$-simplices
$\{(i_0\leq i_1\leq \ldots\leq i_k)\mid 0\leq i_s\leq q\}$.  There is
a natural homeomorphism $|\Delta[q]| \cong \Delta^q$ from its
geometric realization to the standard $q$-simplex, sending the
$0$-simplex $(i)$ to the point in $\Delta^q$ with $i$-th
barycentric coordinate equal to~$1$.  The $2$-fold edgewise
subdivision \cite{BHM93}*{\S1} of $\Delta[q]$ is a simplicial set
with $k$-simplices $(\sd_2\Delta[q])_k = \Delta[q]_{2k+1}$.  There is
a natural homeomorphism $D_2\: |\sd_2\Delta[q]| \cong |\Delta[q]|$,
given by sending each $0$-simplex $(i_0\leq i_1)$ in
$(\sd_2\Delta[q])_0$ to the midpoint of the edge connecting the
vertices $(i_0)$ and~$(i_1)$ in $|\Delta[q]|\cong \Delta^q$, and extending
affine linearly.

For each~$q$, define
\begin{equation}
  \label{eq:aw-dfn}
  \AW^q\: \Delta^q \underset\cong{\overset{D_2}\longleftarrow} |\sd_2\Delta[q]| \overset{\aw^q}\longto |\Delta[q]\times \Delta[q]| \cong \Delta^q\times \Delta^q\,,
\end{equation}
where $\aw^q$ is the topological realization $|\aw^q_\bullet|$
of a simplicial map
$\aw^q_\bullet \: \sd_2 \Delta[q] \to \Delta[q] \times \Delta[q]$,
taking the $0$-simplex $(i_0 \le i_1)$ to the $0$-simplex
$(i_0, i_1)$, and more generally taking the $k$-simplex
$(i_0 \leq \dots \leq i_{2k+1})$ to the $k$-simplex
$(i_0 \leq \dots \leq i_k, i_{k+1} \leq \dots \leq i_{2k+1})$.  Since
both~$D_2$ and $\aw^q$ are natural in $[q]$, \eqref{eq:aw-dfn}
defines a cosimplicial map
$\AW^\bullet\: \Delta^\bullet\to \Delta^\bullet\times
\Delta^\bullet$, and this is the Alexander--Whitney map of
Lemma~\ref{lemma:aw}.

For all non-negative integers~$a$ and $b$, let
$K^{a+b}_a\subset \sd_2 \Delta[{a+b}]$ be the simplicial subset with
$k$-simplices
$\{(i_0\leq \ldots \leq i_{2k+1}) \mid i_k\leq a\leq i_{k+1}\}$.
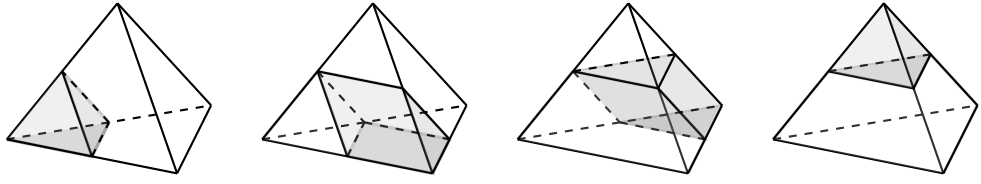
\begin{figure}[h]
  \label{fig:AW}
  \centering
  \begin{tikzpicture}[thick,scale=3]

    \definecolor{blue0}{RGB}{255,255,255}

    \definecolor{green0}{RGB}{100,100,100}
    \definecolor{green1}{RGB}{120,120,120}
    \definecolor{green2}{RGB}{140,140,140}
    \definecolor{green3}{RGB}{180,180,180}

    \coordinate (v0) at (0, 0);
    \coordinate (v1) at (1*0.75, -0.2*0.75);
    \coordinate (v2) at (1.2*0.75, 0.2*0.75);
    \coordinate (v3) at (0.65*0.75, 0.8*0.75);

    \coordinate (v01) at ($0.5*(v0)+0.5*(v1)$);
    \coordinate (v02) at ($0.5*(v0)+0.5*(v2)$);
    \coordinate (v03) at ($0.5*(v0)+0.5*(v3)$);
    \coordinate (v12) at ($0.5*(v1)+0.5*(v2)$);
    \coordinate (v13) at ($0.5*(v1)+0.5*(v3)$);
    \coordinate (v23) at ($0.5*(v2)+0.5*(v3)$);

    % Delta^3
    \draw[thick] (v0) -- (v1);
    \draw[thick] (v1) -- (v2);
    \draw[dashed] (v0) -- (v2);
    \draw[thick] (v0) -- (v3);
    \draw[thick] (v1) -- (v3);
    \draw[thick] (v2) -- (v3);

    % K_0^3
    \draw[thick] (v01) -- (v03);
    \draw[dashed] (v01) -- (v02);
    \draw[dashed] (v02) -- (v03);
    \draw[fill=green1,very thick,opacity=0.2] (v0) -- (v01) -- (v02);
    \draw[fill=green2,very thick,opacity=0.2] (v01) -- (v02) -- (v03);
    \draw[fill=green3,very thick,opacity=0.2] (v0) -- (v01) -- (v03);

    \begin{scope}[shift={(1.5*0.75,0)}]
      \coordinate (v0) at (0, 0);
      \coordinate (v1) at (1*0.75, -0.2*0.75);
      \coordinate (v2) at (1.2*0.75, 0.2*0.75);
      \coordinate (v3) at (0.65*0.75, 0.8*0.75);

      \coordinate (v01) at ($0.5*(v0)+0.5*(v1)$);
      \coordinate (v02) at ($0.5*(v0)+0.5*(v2)$);
      \coordinate (v03) at ($0.5*(v0)+0.5*(v3)$);
      \coordinate (v12) at ($0.5*(v1)+0.5*(v2)$);
      \coordinate (v13) at ($0.5*(v1)+0.5*(v3)$);
      \coordinate (v23) at ($0.5*(v2)+0.5*(v3)$);
      % Delta^3
      \draw[thick] (v0) -- (v1);
      \draw[thick] (v1) -- (v2);
      \draw[dashed] (v0) -- (v2);
      \draw[thick] (v0) -- (v3);
      \draw[thick] (v1) -- (v3);
      \draw[thick] (v2) -- (v3);

      % K_1^3
      \draw[thick] (v03) -- (v01);
      \draw[thick] (v03) -- (v13);
      \draw[thick] (v13) -- (v12);
      \draw[dashed] (v02) -- (v12);
      \draw[dashed] (v03) -- (v02);
      \draw[dashed] (v01) -- (v02);
      \draw[fill=blue0,very thick,opacity=0.2] (v0) -- (v03) -- (v01);
      \draw[fill=green1,very thick,opacity=0.2] (v01) -- (v02) -- (v12) -- (v1);
      \draw[fill=green3,very thick,opacity=0.2] (v03) -- (v13) -- (v1) -- (v01);
      \draw[fill=green2,very thick,opacity=0.2] (v13) -- (v12) -- (v1);

   \end{scope}
   \begin{scope}[shift={(3.0*0.75,0.0)}]
     \coordinate (v0) at (0, 0);
     \coordinate (v1) at (1*0.75, -0.2*0.75);
     \coordinate (v2) at (1.2*0.75, 0.2*0.75);
     \coordinate (v3) at (0.65*0.75, 0.8*0.75);

     \coordinate (v01) at ($0.5*(v0)+0.5*(v1)$);
     \coordinate (v02) at ($0.5*(v0)+0.5*(v2)$);
     \coordinate (v03) at ($0.5*(v0)+0.5*(v3)$);
     \coordinate (v12) at ($0.5*(v1)+0.5*(v2)$);
     \coordinate (v13) at ($0.5*(v1)+0.5*(v3)$);
     \coordinate (v23) at ($0.5*(v2)+0.5*(v3)$);

     % Delta^3
     \draw[thick] (v0) -- (v1);
     \draw[thick] (v1) -- (v2);
     \draw[dashed] (v0) -- (v2);
     \draw[thick] (v0) -- (v3);
     \draw[thick] (v1) -- (v3);
     \draw[thick] (v2) -- (v3);

     % K_2^3
     \draw[dashed] (v03) -- (v23);
     \draw[thick] (v03) -- (v13);
     \draw[thick] (v13) -- (v12);
     \draw[dashed] (v02) -- (v12);
     \draw[dashed] (v03) -- (v02);
     \draw[thick] (v13) -- (v23);

     \draw[fill=blue0,very thick,opacity=0.2] (v0) -- (v03) -- (v13)  -- (v12) -- (v1);

     \draw[fill=green1,opacity=0.2] (v03) -- (v23) -- (v13);
     \draw[fill=green1,opacity=0.2] (v02) -- (v2) -- (v12);
     \draw[fill=green2,opacity=0.2] (v13) -- (v23) -- (v2) -- (v12);
     \draw[fill=green0,opacity=0.2] (v03) -- (v13) -- (v12) -- (v02);

   \end{scope}
   \begin{scope}[shift={(4.5*0.75,0.0)}]
      \coordinate (v0) at (0, 0);
      \coordinate (v1) at (1*0.75, -0.2*0.75);
      \coordinate (v2) at (1.2*0.75, 0.2*0.75);
      \coordinate (v3) at (0.65*0.75, 0.8*0.75);

     \coordinate (v01) at ($0.5*(v0)+0.5*(v1)$);
     \coordinate (v02) at ($0.5*(v0)+0.5*(v2)$);
     \coordinate (v03) at ($0.5*(v0)+0.5*(v3)$);
     \coordinate (v12) at ($0.5*(v1)+0.5*(v2)$);
     \coordinate (v13) at ($0.5*(v1)+0.5*(v3)$);
     \coordinate (v23) at ($0.5*(v2)+0.5*(v3)$);
     % Delta^3
     \draw[thick] (v0) -- (v1);
     \draw[thick] (v1) -- (v2);
     \draw[dashed] (v0) -- (v2);
     \draw[thick] (v0) -- (v3);
     \draw[thick] (v1) -- (v3);
     \draw[thick] (v2) -- (v3);
     % K_3^3
     \draw[thick] (v03) -- (v13);
     \draw[dashed] (v03) -- (v23);
     \draw[thick] (v13) -- (v23);

     \draw[fill=blue0,very thick,opacity=0.2] (v0) -- (v03) -- (v13)  -- (v23) -- (v2) --(v1);

     \draw[fill=green1,opacity=0.2] (v03) -- (v23) -- (v13);
     \draw[fill=green3,opacity=0.2] (v03) -- (v3) -- (v13);
     \draw[fill=green2,opacity=0.2] (v3) -- (v23) -- (v13);
   \end{scope}
 \end{tikzpicture}
 \caption{The subcomplexes $|K_0^3|, |K_1^3|, |K_2^3|$ and $|K_3^3|$ of
   $\Delta^3$.}
\end{figure}
As before, let $\alpha\: \Delta^a\to \Delta^{a+b}$ and
$\beta\: \Delta^b\to \Delta^{a+b}$ be the front $a$-face and back
$b$-face of $\Delta^{a+b}$, respectively.  There is a commutative
diagram
\begin{equation}
  \begin{aligned}
    \label{eq:AW-K}
    \xymatrix{
    |K^{a+b}_a| \ar[r]^-{\text{incl}} \ar[d]^-{\cong}
	& |\sd_2 \Delta[{a+b}]| \ar[r]^-{D_2}_-{\cong} \ar[d]^-{|\aw_\bullet^{a+b}|}
	& \Delta^{a+b} \ar[d]^{\AW^{a+b}}\\
    |\Delta[a] \times \Delta[b]| \ar[r]^-{|\alpha \times \beta|}
	& |\Delta[{a+b}] \times \Delta[{a+b}]| \ar[r]^-{\cong}
	& \Delta^{a+b} \times \Delta^{a+b} \,.
      }
  \end{aligned}
\end{equation}
By construction, the left-hand vertical map is a homeomorphism and
preserves the order of the~$0$-simplices.  It is therefore an
orientation-preserving map.  Let~$c$ be a non-negative integer.  Since
$\Delta^{c} = \bigcup_{i=0}^c |K_i^{c}|$, diagram~\eqref{eq:AW-K}
implies that $\AW^{c}$ restricts to the CW product $c$-skeleton
\begin{multline}
    \sk_{c}\AW^{c}\: \sk_{c}\Delta^{c} = \Delta^{c} \longto
  \sk_{c}(\Delta^{c}\times \Delta^{c}) = \bigcup_{a+b=c}
  \sk_a\Delta^{c} \times \sk_b\Delta^{c}\,,
\end{multline}
and induces the subquotient map
\begin{multline}
  \label{eq:AW-sub-quotient}
  \sk_{c} \overline\AW^{c}_{a,b}\: \frac{\Delta^{c}}{\partial \Delta^{c}}
  = \frac{\sk_{c}\Delta^{c}}{\sk_{{c}-1}\Delta^{c}}
  \longto
  \frac{\sk_a\Delta^{c}\times \sk_b\Delta^{c}}{\sk_{{c}-1}(\sk_a\Delta^{c}\times \sk_b\Delta^{c})} \\
  \cong \frac{\sk_a\Delta^{c}}{\sk_{a-1}\Delta^{c}}\wedge
  \frac{\sk_b\Delta^{c}}{\sk_{b-1}\Delta^{c}}
\end{multline}
for each pair of non-negative integers $a,b$ such that $a+b=c$.

\begin{lemma}
  \label{lemma:AW-subcomplex}
  Let $a,b\geq 0$ and $c=a+b$.  The subquotient of the Alexander--Whitney map
  factors as
  \begin{equation}
    \begin{aligned}
      \label{eq:AW-front-face-backface}
      \xymatrix@C9mm{
      \ds \frac{\Delta^{c}}{\partial \Delta^{c}}
      \ar[rr]^-{\sk_{c}\overline\AW^{c}_{a,b}}
      \ar@/_4mm/[rrd]_-{\simeq}
      &&
         \ds \frac{\sk_a\Delta^{c}}{\sk_{a-1} \Delta^{c}}
         \wedge \ds\frac{\sk_b\Delta^{c}}{\sk_{b-1} \Delta^{c}}\\
      &&
         \ds \frac{\Delta^a}{\partial \Delta^a}
         \wedge \ds \frac{\Delta^b}{\partial \Delta^b}
         %\rlap{\ $\cong S^a\wedge S^b$}
         \ar[u]_{\alpha\wedge \beta} \rlap{\,,}
         }
    \end{aligned}
  \end{equation}
  where the unlabeled homotopy equivalence has degree~$1$.
\end{lemma}
\begin{proof}
  Consider the subquotient of the Alexander--Whitney map
  \begin{equation}
    \label{eq:AW-sub-quotient-wedge}
    \sk_{c} \overline\AW^{c}_{a,b}\: \frac{\Delta^{c}}{\partial \Delta^{c}}
    \longto
    \frac{\sk_a\Delta^{c}}{\sk_{a-1}\Delta^{c}}\wedge
    \frac{\sk_b\Delta^{c}}{\sk_{b-1}\Delta^{c}}
    \overset\cong\longleftarrow
    \bigvee_{\mu, \nu} \frac{\Delta^a}{\partial \Delta^a}\wedge
    \frac{\Delta^b}{\partial \Delta^b}\,,
  \end{equation}
  where we have identified the codomain as a wedge of spheres indexed by
  all monomorphisms $\mu\: [a]\to [{c}]$ and $\nu\:[b]\to [{c}]$ in
  $\Delta$, for our fixed integers $a,b$.  Then the image
  of~\eqref{eq:AW-sub-quotient-wedge} is contained in the wedge summand
  corresponding to the front~$a$-face $\alpha$ and the back~$b$-face
  $\beta$, and we get a factorization
  \begin{equation}
    \label{eq:aw2}
    \sk_{c}\overline\AW^c_{a,b}\:
    \frac{\Delta^{c}}{\partial \Delta^{c}}\longto
    \frac{\Delta^a}{\partial \Delta^a}\wedge
    \frac{\Delta^b}{\partial \Delta^b}
    \overset{\alpha\wedge \beta}\longto
    \frac{\sk_a\Delta^{c}}{\sk_{a-1}\Delta^{c}}\wedge
    \frac{\sk_b\Delta^{c}}{\sk_{b-1}\Delta^{c}}\,,
  \end{equation}
  which is diagram~\eqref{eq:AW-front-face-backface}.  The first map
  of \eqref{eq:aw2} equals the orientation-preserving homeomorphism
  of~\eqref{eq:AW-K} when restricted to the interior of
  $K_a^{c}\subset \Delta^{c}$, and sends the complement to the
  basepoint of
  $\Delta^a/\partial\Delta^a\wedge \Delta^b/\partial \Delta^b$.  It is
  therefore a degree~$1$ homotopy equivalence.
\end{proof}

\subsection{Continuous and algebraic $\Ext$ groups}\label{sec:cext-ext}

By the isomorphism of categories of
Proposition~\ref{prop:categoryiso}, we can regard any $P_*$ in
$\rbbfilacomodc$ as an object of $\rbbfilamodc$.  Furthermore, we may
forget the filtration of $P_*$ and compute $\Ext_\sA^s(\bF_p, P_*)$ in
the category of right $\sA$-modules.
\begin{proposition}\label{prop:left-right-ext}
  Let $P_*$ in $\rbbfilacomodc$ be an rbb complete left $\sA_*$-comodule.
  Then for each $s$ there is a natural isomorphism
  \begin{equation}
    \label{eq:right-left-ext}
    \cExt_{\sA_*}^s (\bF_p, P_*) \cong \Ext_\sA^s(\bF_p, P_*)\,,
  \end{equation}
  where $\Ext$ on the right-hand side is computed in the category of
  (unfiltered) right $\sA$-modules.
\end{proposition}
\begin{proof}
  By Lemma~\ref{lemma:cext-ext} below there are isomorphisms of
  cochain complexes
  $$
  C^*_{\sA_*}(\bF_p, P_*/F_nP_*) \cong \Hom_\sA(B_*^\sA(\bF_p,
  \sA), P_*/F_nP_*)
  $$
  and, by taking limits over $n$,
  \begin{equation}
    \label{eq:chain-cochain-iso}
    \widehat C^*_{\sA_*}(\bF_p, P_*) \cong \Hom_\sA(B_*^\sA(\bF_p,
    \sA), P_*^\wedge)\,.
  \end{equation}
  By definition,
  $$
  \cExt_{\sA_*}^s(\bF_p, P_*) = H^s(\widehat C^*_{\sA_*}(\bF_p, P_*)) \,,
  $$
  which is isomorphic to
  $$
  H^s(\Hom_\sA(B_*^\sA(\bF_p, \sA), P_*^\wedge)) = \Ext_\sA^s(\bF_p, P^\wedge_*)\,.
  $$
  Since $P_*$ is complete Hausdorff, $P_*^\wedge \cong P_*$ and the
  proposition follows.
\end{proof}

The following discussion, including Lemma~\ref{lemma:cext-ext}, is
valid for $k=\bF_p$, $R_*=\sA_*$ and $W_* \subset H_*(X^{tG}[n])$ for any
spectrum $X$ such that $X/p$ is bounded below and of finite type.

Let $k$ be a field and~$R_*$ a coaugmented graded $k$-coalgebra.
Let $R = \Hom(R_*, k)$ be the dual augmented graded $k$-algebra.  We
assume that $R_*$ is connective of finite type over~$k$.  Let $W_*$ be
a bounded below left $R_*$-comodule with coaction
$\nu\: W_*\to R_*\otimes W_*$.  By composing $\nu$ with the isomorphism
$$
\iota\: R_*\otimes W_* \cong \Hom(R, W_*)\,,
$$
we also consider $W_*$ to be a right $R$-module with $R$-module
structure map $W_*\otimes R\to W_*$ adjoint to $\iota\circ \nu$.

We may consider the cobar complex $C^*_{R_*}(k, W_*)$,
\begin{align*}
  &W_*
    \longto
    R_*\otimes W_*
    \longto
    R_*\otimes R_*\otimes W_*
    \longto
    \ldots\,, \\
  \intertext{where the coboundary is induced by
  the coaugmentation $\epsilon^*\: k\to R_*$,
  the coproduct $R_*\to R_*\otimes R_*$,
  and the coaction $W_*\to R_*\otimes W_*$.
  When $W_* = R_*$
  with coaction equal to the coproduct, this defines the
  right $R_*$-comodule cobar resolution $C_{R_*}^*(k, R_*)$
  of~$k$.  Its $k$-linear dual $\Hom(C_{R_*}^*(k, R_*), k)$
  is canonically isomorphic to the right $R$-module bar
  resolution $B^R_*(k, R)$,}
  &R
    \longleftarrow
    R\otimes R
    \longleftarrow
    R\otimes R\otimes R
    \longleftarrow\ldots \,,
\end{align*}
with boundary induced by the augmentation $\epsilon\: R\to k$ and
the product $R\otimes R\to R$.

\begin{lemma}\label{lemma:cext-ext}
  There is a natural isomorphism of cochain complexes
  \begin{equation}
    \label{eq:cobar-cobar}
    \iota\:C^*_{R_*}(k, W_*) \overset\cong\longto \Hom_R(B_*^R(k,R), W_*) \,.
  \end{equation}
  The element $[r_1|r_2|\dots|r_s]w \in C_{R_*}^s(k, W_*)$ is
  mapped by $\iota$ to the $R$-linear homomorphism
  \begin{equation}
    \label{eq:iota-elements}
    \bigl\{
    [a_1|a_2|\dots|a_s]m \mapsto
    \sum_j (\pm) a_1(r_1)a_2(r_2)\dots a_s(r_s)\phi_j(m)\cdot w_j
    \bigr\} \,.
  \end{equation}
  Here $w\in W_*$ and
  $\nu(w) = \sum_j \phi_j\otimes w_j \in R_*\otimes W_*$ is its image
  under the coaction.  The homomorphism \eqref{eq:iota-elements} is
  determined by its value on $R$-module generators.  Since
  $\nu(w) \equiv 1 \otimes w$ modulo $\ker(\epsilon) \otimes W_*$, we
  get
  \begin{equation}
    \label{eq:iota-generators}
    \bigl\{
    [a_1|a_2|\dots|a_s]1 \mapsto
    (\pm) a_1(r_1)a_2(r_2)\dots a_s(r_s)\cdot w
    \bigr\} \,.
  \end{equation}
  Let $c_i = |a_i|(|w|+|r_s| + \dots + |r_i|)$.  Then the sign $(\pm)$
  in \eqref{eq:iota-generators} equals $(-1)^{c_1+c_2+\dots +c_s}$.
\end{lemma}
\begin{proof}
  Let $V_*$ be a right $R_*$-comodule, with coaction
  $\mu \: V_* \to V_* \otimes R_*$.  The $k$-linear dual
  $V = \Hom(V_*, k)$ is a right $R$-module, with action
  $\mu^* \: V \otimes R \to V$.  We assume that $V_*$ is bounded below
  and of finite type.  The natural isomorphism $\iota$ restricts to
  the cotensor product over $R_*$, with image equal to the $R$-linear
  homomorphisms from $V$ to $W_*$.  Precisely, there is a $k$-linear
  isomorphism of vertical equalizer diagrams:
  \[
  \xymatrix@C15mm{
    V_*\otimes R_*\otimes W_* \ar[r]_\cong
    & \Hom(V\otimes R, W_*)
    \rlap{\ $\cong \Hom(V, \Hom(R,W_*))$ }
    \\
    V_*\otimes W_*
    \ar[r]^\iota_\cong \ar@<5pt>[u]^{\mu\otimes 1}
    \ar[r]^\iota_\cong \ar@<-5pt>[u]_{1\otimes \nu}
    & \Hom(V, W_*)
    \ar@<5pt>[u]^{\Hom(\mu^*, 1)}
    \ar@<-5pt>[u]
    \ar@/_2mm/[]!<4ex,1ex>;[u]!<20ex,0ex>_{\ \ \Hom(1, \iota\circ \nu)}
    \\ %----
    V_*\square_{R_*} W_*
    \ar[r]^\iota_\cong
    \ar@{ >->}[]!<0ex,12pt>;[u]
    & \Hom_R(V, W_*)\,.
    \ar@{ >->}[]!<0ex,12pt>;[u]
  }
  \]
  Here $\Hom_R(V,W_*)$ is the graded vector subspace of $R$-linear
  homomorphisms $V\to W_*$ with the respect to the right action on $V$
  and $W_*$.

  Applying this with $V_* = C_{R_*}^s(k, R_*)$, for each $s\ge0$, the
  required isomorphism is now given by the composite
  \begin{equation}
    \label{eq:iota-cobar-bar}
    C_{R_*}^*(k, W_*) {\overset\cong\longto} C_{R_*}^*(k, R_*) \cotensor_{R_*} W_*
    \underset{\iota}{\overset\cong\longto}
    \Hom_R(B^R_*(k, R), W_*) \,,
  \end{equation}
  where the left-hand isomorphism is induced by the coaction
  $\nu\: W_*\to R_*\otimes W_*$, which restricts to an isomorphism
  $W_* \cong R_* \cotensor_{R_*} W_*$.
%  (The composite $\iota\circ \nu\: W_*\to \Hom_R(R,W_*)$ is the
%  canonical isomorphism for which evaluation on $1$ is an inverse.
%  Since $\iota$ and $\iota\circ\nu$ are isomorphisms, it follows that
%  the coproduct $\nu$ induces an isomorphism
%  $\nu\: W_*\to R_*\square_{R_*} W_*$.)
  Naturality of the restricted~$\iota$ ensures that the coboundary on
  the left corresponds to the dual of the boundary on the right.
\end{proof}

\subsection{Multiplicativity of the limit Adams spectral sequence}
\label{sec:lass-is-multiplicative}
For each $i=1,2,3$, let $X_i$ be a spectrum with $G$-action such that
$X_i/p$ is bounded below and of finite type.  Like in the proof of
Proposition~\ref{prop:inverse-limit-of-adams-ss}, let $Y_i^\bullet$
be the pre-cosimplicial spectrum
$$
[s] \mapsto (H^{\wedge (1+s)}\wedge X_i)^{tG} \,,
$$
with fat totalization $\TOT Y_i^\bullet \simeq (X_i^{tG})^\wedge_p$.
A~$G$-equivariant pairing $m\: X_1\wedge X_2\to X_3$ gives rise to a
pre-cosimplicial pairing
$m^\bullet\: Y_1^\bullet\wedge Y_2^\bullet\to Y_3^\bullet$, which,
by the discussion in
Subsection~\ref{sec:limit-adams-multiplicativity}, induces a pairing
of limit Adams spectral sequences, converging to the pairing
\begin{equation}
  \label{eq:abutment-pairing}
  m_*\: \pi_* ((X_1^{tG})^\wedge_p)
  \otimes \pi_* ((X_2^{tG})^\wedge_p)
  \longto \pi_* ((X_3^{tG})^\wedge_p)\,.
\end{equation}
On $E_1$-terms, \eqref{eq:abutment-pairing} can be identified with
a pairing of complete cobar complexes,
\begin{equation}
  \label{eq:cobarcup}
  \wC^*_{\sA_*}(\bF_p, H_*^c(X_1^{tG}))
  \otimes \wC^*_{\sA_*}(\bF_p, H_*^c(X_2^{tG}))
  \longto
  \wC^*_{\sA_*}(\bF_p, H_*^c(X_3^{tG}))\,,
\end{equation}
described by Lemma~\ref{lemma:BK-E1-pairing}.  Indeed, the
homomorphism $\alpha_*$ in Lemma~\ref{lemma:BK-E1-pairing} is the
front coface operator
$$
\wC_{\sA_*}^a(\bF_p, H^c_*(X_1^{tG}))
	\overset{\alpha_*}\longto \wC_{\sA_*}^{a+b}(\bF_p, H^c_*(X_1^{tG}))
$$
equal to
$$
1 \ctensor \nu^{(b)} \:
\sA_*^{\otimes a} \ctensor H^c_*(X_1^{tG})
	\overset{\alpha_*}\longto \sA_*^{\otimes (a+b)} \ctensor H^c_*(X_1^{tG})
\,,
$$
where
$\nu^{(b)} \: H^c_*(X_1^{tG}) \to \sA_*^{\otimes b} \ctensor
H^c_*(X_1^{tG})$ denotes the $b$-fold iterated complete left
$\sA_*$-coaction.  The homomorphism $\beta_*$ is the back coface
operator
$$
\wC_{\sA_*}^b(\bF_p, H^c_*(X_2^{tG}))
        \overset{\beta_*}\longto \wC_{\sA_*}^{a+b}(\bF_p, H^c_*(X_2^{tG}))
$$
equal to
$$
\eta^{(a)} \ctensor 1 \:
\sA_*^{\otimes b} \ctensor H^c_*(X_2^{tG})
        \overset{\beta_*}\longto \sA_*^{\otimes (a+b)} \ctensor H^c_*(X_2^{tG})
\,,
$$
where $\eta^{(a)} \: \bF_p \to \sA_*^{\otimes a}$ is the $a$-fold
iterated unit map.  The pairing
$$
m^{a+b}_* \circ \wedge \: \wC_{\sA_*}^{a+b}(\bF_p, H^c_*(X_1^{tG}))
	\otimes \wC_{\sA_*}^{a+b}(\bF_p, H^c_*(X_2^{tG}))
	\longto \wC_{\sA_*}^{a+b}(\bF_p, H^c_*(X_3^{tG}))
$$
equals the filtration-preserving pairing
$$
\cdot \: \sA_*^{\otimes (a+b)} \ctensor H^c_*(X_1^{tG})
\otimes \sA_*^{\otimes (a+b)} \ctensor H^c_*(X_2^{tG})
\longto \sA_*^{\otimes (a+b)} \ctensor H^c_*(X_3^{tG})
$$
uniquely extending the algebraic pairing
\begin{multline}
  \sA_*^{\otimes (a+b)} \otimes H^c_*(X_1^{tG})
  \otimes \sA_*^{\otimes (a+b)} \otimes H^c_*(X_2^{tG})
  \underset{\cong}{\overset{\text{shuffle}}\longto} \\
  (\sA_*\otimes \sA_*)^{\otimes (a+b)}
  \otimes H^c_*(X_1^{tG}) \otimes H^c_*(X_2^{tG})
  \overset{\phi^{\otimes (a+b)}\otimes m_*}\longto
  \sA_*^{\otimes (a+b)} \otimes H^c_*(X_3^{tG}) \\
  \overset{c}\longto
  \sA_*^{\otimes (a+b)} \ctensor H^c_*(X_3^{tG}) \,,
\end{multline}
where the last morphism is completion.  Passing to
cohomology,~\eqref{eq:cobarcup} induces the \textit{continuous cup
  product}
\begin{equation}
  \label{eq:continuous-cup}
  \cExt_{\sA_*}^*(\bF_p, H^c_*(X_1^{tG}))
  \otimes
  \cExt_{\sA_*}^*(\bF_p, H^c_*(X_2^{tG}))
  \overset{\cup}\longto
  \cExt_{\sA_*}^*(\bF_p, H^c_*(X_3^{tG})) \,.
\end{equation}
Specializing to the case of a ring spectrum $X$ with $G$-action,
we summarize the above in the following proposition.

\begin{proposition}\label{prop:ext-cup}
  Let $X$ be a ring spectrum with $G$-action, such that $X/p$ is
  bounded below and of finite type.  The limit Adams spectral sequence
  of Proposition~\ref{prop:inverse-limit-of-adams-ss} is
  multiplicative, i.e., it is an algebra spectral sequence, and the
  product on the $E_2$-term is given by the continuous cup
  product~\eqref{eq:continuous-cup}.
\end{proposition}

The following lemma justifies the naming of~\eqref{eq:continuous-cup}.

\begin{lemma}
  Let $P_*$, $Q_*$ in $\rbbfilacomodc$ be rbb complete left
  $\sA_*$-comodules.  The isomorphism of
  Proposition~\ref{prop:left-right-ext} identifies the external
  continuous cup product
  \begin{equation}
    \label{eq:external-continuous-cup}
    \cExt_{\sA_*}^*(\bF_p, P_*)
    \otimes
    \cExt_{\sA_*}^*(\bF_p, Q_*)
    \overset{\cup}\longto
    \cExt_{\sA_*}^*(\bF_p, P_*\otimes Q_*)
  \end{equation}
  with the normal external cup (=cross) product
  \begin{equation}
    \label{eq:external-cup}
    \Ext_{\sA}^*(\bF_p, P_*)
    \otimes
    \Ext_{\sA}^*(\bF_p, Q_*)
    \overset{\cup}\longto
    \Ext_{\sA}^*(\bF_p, P_*\otimes Q_*) \,.
  \end{equation}
\end{lemma}
\begin{proof}
  Let $a,b$ be non-negative integers, and consider the diagram
  \begin{equation}
    \label{eq:cobar-bar-product}
    \begin{aligned}
      \xymatrix{
      \wC^a_{\sA_*}(\bF_p, P_*) \otimes
      \wC^b_{\sA_*}(\bF_p, Q_*)
      \ar[r]^-{\hat\iota\otimes \hat\iota}_-{\cong}
      \ar[d]^{\alpha_*\otimes \beta_*}
      & \Hom_\sA(B_a, P_*^\wedge) \otimes
        \Hom_\sA(B_b, Q_*^\wedge)
      \ar[d]^{\otimes} \\
      \wC^{a+b}_{\sA_*}(\bF_p, P_*) \otimes
      \wC^{a+b}_{\sA_*}(\bF_p, Q_*)
      \ar[d]^{\cdot}
      & \Hom_\sA(B_a\otimes B_b, P_*^\wedge\otimes Q_*^\wedge)
      \ar[d]^{\Hom(\AW, c)}\\
      \wC^{a+b}_{\sA_*}(\bF_p, P_*\otimes Q_*)
      \ar[r]^-{\hat\iota}_-\cong
      & \Hom_\sA(B_{a+b}, P_*\ctensor Q_*) \,,
      }
    \end{aligned}
  \end{equation}
  where $B_* := B^{\sA}_*(\bF_p, \sA)$, the horizontal
  isomorphisms~$\hat\iota$ are the
  isomorphisms~\eqref{eq:chain-cochain-iso} in the proof of
  Proposition~\ref{prop:left-right-ext}, and
  $\AW\: B_* \to B_* \otimes B_*$
  is the Alexander--Whitney diagonal approximation.  Passing to
  cohomology, the left-hand column yields the continuous cup
  product~\eqref{eq:external-continuous-cup}, and the right-hand
  column yields the normal cup product~\eqref{eq:external-cup}.  Thus,
  the lemma will follow from showing that the diagram commutes.

  Each homomorphism of~\eqref{eq:cobar-bar-product} is a morphism of
  rbb filtered graded $\bF_p$-vector spaces, where the filtrations are
  inherited from those on~$P_*$ and $Q_*$ and the convolution filtration of $P_*\otimes Q_*$.
  Therefore, for each pair of integers $(\ell, k)$, with $k\geq 0$,
  there is a corresponding diagram obtained by restricting to
  filtration~$\ell$, followed by taking the quotient by filtration
  $\ell-k$.  Diagram~\eqref{eq:cobar-bar-product} commutes by
  Lemma~\ref{lemma:f-mnk} below if this ``subquotient diagram'' commutes for
  each~$\ell$ and~$k$, when pre-composed with the canonical
  homomorphism
  \[
    \wC^a_{\sA_*}(\bF_p, \ds \frac{F_m P_*}{F_{m-k}P_*}) \otimes
    \wC^b_{\sA_*}(\bF_p, \ds \frac{F_n Q_*}{F_{n-k}Q_*})
    \overset{i_{m,n,k}}\longto
    \ds \frac{
      F_{\ell} \bigl(\wC^a_{\sA_*}(\bF_p, P_*) \otimes
      \wC^b_{\sA_*}(\bF_p, Q_*)\bigr)}
    {F_{\ell-k} \bigl(\wC^a_{\sA_*}(\bF_p, P_*) \otimes
      \wC^b_{\sA_*}(\bF_p, Q_*)\bigr)} \,,
  \]
  for each pair of integers~$(m,n)$ such that $m+n=\ell$.  Naturality
  of diagram~\eqref{eq:cobar-bar-product} in $P_*$ and $Q_*$ then
  reduces the problem to the case where $P_*$ is replaced by the
  subquotient $F_mP_*/F_{m-k}P_*$ and $Q_*$ is replaced by
  $F_nQ_*/F_{n-k}Q_*$.  In this case, both $P_*$ and $Q_*$ have
  finite-length filtrations and every completed tensor product
  in~\eqref{eq:cobar-bar-product} is canonically isomorphic to a
  regular algebraic one.  For instance, the complete cobar complex
  $\wC^a_{\sA_*}(\bF_p, P_*)$ is canonically isomorphic to the
  algebraic cobar complex
  $C^a_{\sA_*}(\bF_p, P_*) \cong \sA_*^{\otimes a}\otimes P_*$,
  and~\eqref{eq:cobar-bar-product} becomes
  \begin{equation}
    \label{eq:algebraic-cobar-bar-product}
    \begin{aligned}
      \xymatrix@C-5mm{
      (\sA_*^{\otimes a}\otimes P_*) \otimes
      (\sA_*^{\otimes b}\otimes Q_*)
      \ar[r]^-{\iota\otimes \iota}_-{\cong}
      \ar[d]^{\alpha_*\otimes \beta_*}
      & \Hom_\sA(\sA^{\otimes (a+1)}, P_*) \otimes
        \Hom_\sA(\sA^{\otimes (b+1)}, Q_*)
      \ar[d]^{\otimes} \\
      (\sA_*^{\otimes (a+b)}\otimes P_*) \otimes
      (\sA_*^{\otimes (a+b)}\otimes Q_*)
      \ar[d]^{\phi^{\otimes (a+b)}\circ \text{shuffle}}
      & \Hom_\sA(\sA^{\otimes (a+1)}\otimes \sA^{\otimes (b+1)}, P_*\otimes Q_*)
      \ar[d]^{\Hom(\AW, 1)}\\
      \sA_*^{\otimes (a+b)}\otimes P_*\otimes Q_*
      \ar[r]^-\iota_-\cong
      & \Hom_\sA(\sA^{\otimes (a+b+1)}, P_*\otimes Q_*) \,.
      }
    \end{aligned}
  \end{equation}
  In this form, the left-hand vertical pairing is given by
  \begin{equation}
    \label{eq:lhs-bar-cobar}
    [ \varphi_1 | \cdots | \varphi_a ] p
    \otimes
    [ \psi_1 | \cdots | \psi_b ] q \mapsto
    \sum_k (\pm) [ \varphi_1 | \cdots | \varphi_a |
    \theta_{1,k} \psi_1 | \cdots | \theta_{b,k} \psi_b ] p_{k} q \,,
  \end{equation}
  where
  $$
  \nu^{(b)}(p) = \sum_k
  [ \theta_{1,k} | \cdots | \theta_{b,k} ] p_{k}
  $$
  is the $b$-fold iterated coaction applied to $p$, and~$(\pm)$ is the
  sign arising from the shuffle map.  This is comparable with
  \cite{Ra86}*{A1.2.15}, where a formula is also given for the sign.

  On the right-hand side, the Alexander--Whitney diagonal
  approximation
  \begin{multline}
    \AW\: B^{\sA}_{a+b}
    \longto
    B_{a+b}^{\sA\otimes \sA}(\bF_p, \sA\otimes \sA)\\
    \longto \bigl(B_{*} \otimes
    B_{*}\bigr)_{a+b}
    \cong \bigoplus_{i+j=a+b} B_i \otimes B_j \,,
  \end{multline}
  is the homomorphism that first maps
  $[r_1 | \dots | r_{a+b}] m \in \sA^{\otimes (a+b)} \otimes \sA =
  B_{a+b}$ to the element
  \begin{equation}
    \label{eq:aw-element}
    \sum [r'_1 \otimes r''_1 | \dots | r'_{a+b}\otimes r''_{a+b}] m'\otimes m''
  \end{equation}
  in $B^{\sA \otimes \sA}_{a+b}(\bF_p, \sA \otimes \sA)$, where
  $\psi(r_i) = \sum r'_i \otimes r''_i$ for $1 \le i \le a+b$ and
  $\psi(m) = \sum m' \otimes m''$. Thereafter,~\eqref{eq:aw-element} is
  mapped to the element having $(i,j)$-component given by
  $\alpha^* \otimes \beta^*$, namely
  \begin{equation}
    \label{eq:rhs-bar-cobar}
    \sum (\pm) [r'_1 | \dots | r'_i] (r'_{i+1} \cdots r'_{a+b} m')
    \otimes \epsilon(r''_1 \cdots r''_i) [r''_{i+1} | \dots | r''_{a+b}]
    m'' \,.
  \end{equation}
  Again, the sign indicated by
  $(\pm)$ in the above formula arises from the necessary
  transpositions of elements.

  Lemma~\ref{lemma:cext-ext} gives an explicit formula describing the
  isomorphisms $\iota$, and together with~\eqref{eq:lhs-bar-cobar}
  and~\eqref{eq:rhs-bar-cobar}, it is possible to check that the
  diagram commutes.  The computation is elementary, and we omit the
  details.
\end{proof}

Let $P_*$, $Q_*$ and $W_*$ be rbb filtered graded left
$R$-modules, and let $f \: P_*\otimes Q_*\to W_*$ be a
morphism in $\rbbfilgrRmod$.  As usual, the domain of~$f$ has the
convolution filtration.  Then, for all integers $m$, $n$ and
$k\geq 0$, $f$ induces a homomorphism
\begin{equation}
  \label{eq:f-mnk}
  f_{m,n,k} \: \ds \frac{F_mP_*}{F_{m-k}P_*} \otimes \ds
  \frac{F_nQ_*}{F_{n-k}Q_*}
  \overset{i_{m,n,k}}\longto
  \ds \frac{F_{m+n}(P_*\otimes Q_*)}{F_{m+n-k}(P_*\otimes Q_*)}
  \overset{\overline{F_{m+n}f}}\longto
  \ds \frac{F_{m+n}W_*}{F_{m+n-k}W_*} \,,
\end{equation}
where $i_{m,n,k}$ is induced by the inclusion
$F_mP_*\otimes F_nQ_* \subset F_{m+n}(P_*\otimes Q_*)$.

The following is an instance of the fact that a continuous function to
a Hausdorff space is determined by its restriction to a dense subset.

\begin{lemma}\label{lemma:f-mnk}
  With notation as above, assume that $f_{m,n,k}=0$ for each triple
  $(m,n,k)$, and that $W_*$ is Hausdorff.  Then $f$ is the zero
  homomorphism.
\end{lemma}
\begin{proof}
  Fix an integer~$\ell$.  The sum
  \[
    \bigoplus_{m+n=\ell} \ds \frac{F_mP_*}{F_{m-k}P_*} \otimes \ds
    \frac{F_nQ_*}{F_{n-k}Q_*}
    \longto
    \ds \frac{F_{\ell}(P_*\otimes Q_*)}{F_{\ell-k}(P_*\otimes Q_*)} \,,
  \]
  induced by $i_{m,n,k}$, is surjective.  It follows that
  \[
    \overline{F_{\ell,k} f}\:
    \ds \frac{F_{\ell}(P_*\otimes Q_*)}{F_{\ell-k}(P_*\otimes Q_*)}
    \longto
    \ds \frac{F_{\ell}W_*}{F_{\ell-k}W_*}
  \]
  is trivial for each pair of integers $(\ell, k)$ with $k\geq 0$.
  Passing to the limit over $k$, and the colimit over $\ell$, we
  conclude that $f^\wedge$ is trivial.  Since $W_*$ is assumed to be
  Hausdorff, the completion homomorphism $c\:W_*\to W_*^\wedge$ is
  injective, from which it follows that $f=0$.
\end{proof}

\section{The residual circle action on $X^{tC_p}$}
\label{sec:residual-circle-action}

Let $X$ be a spectrum with an action $\lambda \: \bT_+ \wedge X \to X$
of the circle group $\bT$.  The $C_p$-Tate construction
$X^{tC_p} = [\widetilde {E\bT}\wedge F(E\bT_+, X) ]^{C_p}$ is the
$C_p$-fixed points of a spectrum with $\bT$-action, and has therefore
a residual action of the quotient group $\bT/C_p$,
$$
\bar\lambda\: \bT/C_{p+}\wedge X^{tC_p} \longto X^{tC_p}\,.
$$
The multiplication $H\wedge H\to H$ together with $\bar\lambda$ give
rise to an action of $\pi_*(H\wedge \bT/C_{p+}) = H_*(\bT/C_p)
\cong \bF_p\{\bar e_0, \bar e_1\}$
on $\pi_*((H\wedge X)^{tC_p}) = H_*^c(X^{tC_p})$, yielding a
differential $\bar\sigma\: \leftsuspension H_*^c(X^{tC_p})\to H_*^c(X^{tC_p})$.
%Explicitly, $\bar\sigma(x) = \bar\lambda_*(\bar{e}_1\otimes x)$, where
%$\bar{e}_1\in H_1(\bT/C_p)$ is the fundamental class.
It is a morphism of right $\sA$-modules in the sense that it graded
commutes with the action of $\sA$ as discussed
in~Subsection~\ref{sec:dgas}.

Note that while the restricted $C_p$-action on
$\widetilde {E\bT}\wedge F(E\bT_+, X)$ preserves the Greenlees--May
filtration, the full $\bT$-action only preserves the even-indexed
filtration terms, since it involves acting on $\widetilde {E\bT}$ as a
$\bT$-space.  However, $\bar\lambda$ can not increase Tate filtration
by more than~$1$, so the differential
$$
\bar\sigma\: \leftsuspension H_*^c(X^{tC_p})\longto \sh_1 H_*^c(X^{tC_p})
$$
is a morphism of filtered right $\sA$-modules,
cf.~Subsection~\ref{sec:dgas}.

Assume that $\mu\: X\wedge Y\to Z$ is a pairing of spectra with left
$\bT$-actions, meaning that the diagram
\begin{equation}
  \label{eq:circle-pairing}
  \begin{aligned}
    \xymatrix@C17mm{
    \bT_+\wedge \bT_+ \wedge X\wedge Y \ar[d]^{(23)}_\cong
    & \bT_+\wedge X\wedge Y
      \ar[r]^{1\wedge \mu}
      \ar[l]_-{\Delta\wedge 1\wedge 1}  \ar[d]\ar[d]^{\lambda_{X\wedge Y}}
    & \bT_+\wedge Z \ar[d]^{\lambda_Z}\\
    \bT_+\wedge X\wedge \bT_+\wedge Y \ar[r]^-{\lambda_X\wedge \lambda_Y}
    & X\wedge Y \ar[r]^\mu
    & Z
      }
  \end{aligned}
\end{equation}
commutes.  Then the following Leibniz identity holds in mod~$p$
homology,
\begin{equation*}
  (\lambda_{Z})_*(e_1\otimes xy) =
  (\lambda_X)_*(e_1\otimes x)\,y + (-1)^{|x|}x\,(\lambda_Y)_*(e_1\otimes y)\,,
\end{equation*}
where $e_1\in H_1(\bT)$ is the fundamental class and $x$ and~$y$ are
homogeneous classes in $H_*(X)$ and $H_*(Y)$, respectively.  When
$\mu$ is the multiplication of a ring spectrum with a compatible
$\bT$-action, and $\sigma(x) = \lambda_*(e_1\otimes x)$ is the
differential in homology induced by the circle action, we have
the Leibniz formula
\begin{equation}
  \label{eq:leibniz-homology}
  \sigma(xy) =\sigma(x)y + (-1)^{|x|}x\sigma(y)\,,
\end{equation}
making $(H_*(X), \sigma)$ a differential graded algebra.

The following proposition applies to $X=\THH(B)$ where $B$ is any
$E_2$ ring spectrum.
\begin{proposition}\label{prop:tate-dga}
  Let $X$ be an $E_1$ ring spectrum with left action
  $\lambda\:\bT_+\wedge X\to X$, compatible with the multiplication map
  $\mu\: X\wedge X\to X$ in the sense that diagram
  \eqref{eq:circle-pairing} commutes with $X=Y=Z$.  Then
  $(H_*^c(X^{tC_p}), \bar\sigma)$ is a filtered
  differential graded right $\sA$-module algebra, where
  the differential
  $$
  \bar\sigma\: \leftsuspension H_*^c(X^{tC_p}) \longto \sh_1 H_*^c(X^{tC_p})
  $$
  given by $\bar\sigma(x) = \bar\lambda_*(\bar e_1\otimes x)$ is
  induced by the residual circle action on $X^{tC_p}$.
\end{proposition}
\begin{proof}
  By Proposition~\ref{prop:continuous-homology-is-monoidal} it follows
  that $H_*^c(X^{tC_p})$ is a filtered right $\sA$-module algebra.  We
  must show that the differential~$\bar\sigma$ satisfies the Leibniz
  formula.

  To shorten notation, let $t(X)$ be the $\bT$-spectrum
  $\widetilde{E\bT}\wedge F(E\bT_+, X)$ and let $\bar\bT=\bT/C_p$, so that
  $X^{tC_p}$ is the $\bar\bT$-spectrum
  $t(X)^{C_p} \simeq F(\bar\bT_+, t(X))^\bT$.  The residual circle action
  $\bar\lambda\: \bar\bT_+\wedge X^{tC_p}\to X^{tC_p}$ is then given by
  letting $\bar\bT$ act by multiplication in the domain of
  $F(\bar\bT_+, -)^{\bT}$, according to the formula
  $(z, f) \mapsto \{x\mapsto f(x\cdot z)\}$.

  There is a $\bT$-equivariant map
  $p\:t(X)\wedge t(Y)\to t(X\wedge Y)$, unique up to $\bT$-equivariant
  homotopy, which together with the pairing $\mu$ induces a pairing
  $\bar\mu\: X^{tC_p}\wedge Y^{tC_p}\to Z^{tC_p}$ given by the
  composite
  \[
    \xymatrix{
      F(\bar\bT_+, t(X))^\bT \wedge F(\bar\bT_+, t(Y))^\bT
      \ar[r]^-\wedge
      \ar[d]_{\bar\mu}
      & F(\bar\bT_+\wedge \bar\bT_+, t(X)\wedge t(Y))^\bT
      \ar[d]^{F(\Delta, 1)}\\
      F(\bar\bT_+, t(Z))^\bT
      & F(\bar\bT_+, t(X)\wedge t(Y))^\bT
      \ar[l]^-{F(1, t(\mu)\circ p)}\rlap{\,.}
    }
  \]
  The pairing $\bar\mu$ is compatible with the residual
  $\bar\bT$-action in the sense that the following diagram,
  corresponding to \eqref{eq:circle-pairing}, commutes:
  \[
    \hspace{-2cm}
    \xymatrix{
      \bar\bT_+\wedge X^{tC_p}\wedge Y^{tC_p}
      \ar[r]^-{1\wedge \bar\mu}
      \ar[d]_{(\bar\lambda \wedge \bar\lambda) \circ (23)\circ (\Delta\wedge1\wedge1)}
      & \bar\bT_+\wedge Z^{tC_p}
      \ar[d]^{\bar\lambda}\\
      X^{tC_p} \wedge Y^{tC_p}
      \ar[r]^-{\bar\mu}
      &
      Z^{tC_p}\rlap{\,.}
    }
  \]
  It follows that $\bar\lambda$ induces a differential satisfying the
  Leibniz formula~\eqref{eq:leibniz-homology}.
\end{proof}

\subsection{The mod~$p$ homology of the extended power
  construction}\label{sec:continuous-homology-of-extended-powers}
For any subgroup $G\subset \Sigma_p$ of the symmetric group, the
$G$-extended power construction on a spectrum~$B$ is the homotopy
orbit spectrum
$$
D_G(B) = EG\ltimes_G B^{\wedge p}\,.
$$
\begin{lemma}[\cite{BMMS86}*{Cor.~I.2.3}] \label{lemma:grouphomology}
  For any spectrum $B$ and any subgroup $G\subset \Sigma_p$ of the
  symmetric group, there is a natural isomorphism
  $$
  H_*(D_G(B))\cong H_*(G; H_*(B)^{\otimes p}) \,,
  $$
  where $G$ permutes the $p$ copies of $H_*(B)$.
  \qed
\end{lemma}
Choose $\sB$ to be a homogeneous vector space basis for $H_*(B)$, so
that $H_q(B)$ is spanned by $\sB_q$ for each $q\in\bZ$.  Assuming
$G=C_p$, we follow \cite{BMMS86}*{\S II.5}.  The mod $p$ homology of the
$C_p$-extended power construction splits additively as
$$
H_*(D_{C_p}(B)) \cong
\bF_p\{ e_0 \otimes x_1 \otimes \dots \otimes x_p \}
\oplus \bF_p\{ e_j \otimes x^{\otimes p} \mid j \geq 0 \}\,,
$$
where the $x_i$ and $x$ range over $\sB$, the $x_i$ are not all equal,
and only one representative is taken from each $C_p$-orbit of the
tensors $x_1 \otimes \dots \otimes x_p$.  The grading is determined by
$|e_j| = j$.  In particular, $H_*(BC_p)=\bF_p\{e_j \mid j\geq 0\}$.

In the formulas that follow, we will use the convention that any
term involving $e_j$ for $j<0$ should be read as $0$.

The following Nishida formulas for the action of the Steenrod
operations in the mod $p$ homology of the $C_p$-extended power
construction on $B$ can be found in \cite{LMS86}*{Thm.~VIII.3.1}.
%((Lewis-May-Steinberger cite
%\cite{May70}*{Thm. 9.4}, but the latter theorem does not include a
%formula for the homology Bockstein.))
\begin{multline}
  \label{eq:ext-power-steenrod-ops-homology-beta}
  P_*^s(e_j\otimes x^{\otimes p}) =
  \sum_k \binom{[\frac{j}{2}] + m(q-2s)}{s-pk}e_{j-2(s-pk)(p-1)}\otimes \P^k_*(x)^{\otimes p} \\
  + \delta(j)\alpha(q) \sum_k \binom{[\frac{j+1}{2}] + m(q-2s) - 1}{s-pk-1}e_{j+p-2(s-pk)(p-1)}\otimes \P^k_*\beta(x)^{\otimes p}
\end{multline}
\begin{equation}
  \label{eq:ext-power-bockstein-homology-beta}
  \beta(e_j\otimes x^{\otimes p}) =
  \begin{cases}
    e_{j-1}\otimes x^{\otimes p} & \text{when $j$ is even,}\\
    0 & \text{when $j$ is odd.}
  \end{cases}
\end{equation}
Here the coefficient~$\delta(j)$ is
equal to~$1$ when~$j$ is odd, and~$0$ otherwise.
Furthermore, $m=(p-1)/2$ and $\alpha(q) = -(-1)^{mq}\cdot m!$.  The summations
range over all integers~$k$, and we use the convention that
$\binom{a}{0} = 1$, $\binom{a}{b} = 0$ when $b<0$, and
$\binom{a}{b} = a(a-1)\cdots(a-b+1)/b!$ for $b\geq 1$.

In \eqref{eq:ext-power-steenrod-ops-homology-beta}, $\P^s_*$ denotes
the right action of $\P^s \in \sA$ on homology, which dualizes to the
left action of $\P^s$ on cohomology, and is equal to the left action
by~$\chi(\P^s)$ on homology.  Similarly, the right action of~$\beta$
on homology is denoted~$\beta_*$ and agrees up to a twist isomorphism
with the left action by~$\chi(\beta)$.  The Bockstein~$\beta$
appearing in
\eqref{eq:ext-power-steenrod-ops-homology-beta}--\eqref{eq:ext-power-bockstein-homology-beta}
denotes the homology Bockstein acting from the left on homology.  The
relation between the right and left action of~$\beta$ on an element
$x\in H_q(B)$ is given by
$\beta\cdot x = (-1)^q x\cdot \chi(\beta) = -(-1)^q x\cdot \beta =
-(-1)^q \beta_*(x)$.  Thus, rewriting
\eqref{eq:ext-power-steenrod-ops-homology-beta}--\eqref{eq:ext-power-bockstein-homology-beta}
using $\beta_*$ instead of $\beta$ yields
\begin{multline}
  \label{eq:ext-power-steenrod-ops}
  P_*^s(e_j\otimes x^{\otimes p}) =
  \sum_k \binom{[\frac{j}{2}] + m(q-2s)}{s-pk}e_{j-2(s-pk)(p-1)}\otimes \P^k_*(x)^{\otimes p} \\
  + \delta(j)(-1)^{q+1}\alpha(q) \sum_k \binom{[\frac{j+1}{2}] + m(q-2s) - 1}{s-pk-1}e_{j+p-2(s-pk)(p-1)}\otimes \P^k_*\beta_*(x)^{\otimes p}
\end{multline}
%and
\begin{equation}
  \label{eq:ext-power-bockstein}
  \beta_*(e_j\otimes x^{\otimes p}) =
  \begin{cases}
    (-1)^{q+1}e_{j-1}\otimes x^{\otimes p} & \text{when $j$ is even,}\\
    0 & \text{when $j$ is odd.}
  \end{cases}
\end{equation}
Due to the sign $(-1)^{q+1}$ appearing twice,
\eqref{eq:ext-power-steenrod-ops}--\eqref{eq:ext-power-bockstein} are
slightly more complicated than
\eqref{eq:ext-power-steenrod-ops-homology-beta}--\eqref{eq:ext-power-bockstein-homology-beta}.
However, using $\beta_*$ instead of $\beta$ will lead to simpler
formulas
%\eqref{eq:steenrod-singer-tr}--\eqref{eq:steenrod-singer-beta}
in Subsection~\ref{sec:right-action-of-A-on-the-homological-singer-construction}
when we describe the right action of $\sA$ on
$H_*^c((B^{\wedge p})^{tC_p})$.

Finally, when $p=2$, the right action of the mod~$2$ Steenrod algebra
is given by the formula
\begin{equation}
  \label{eq:ext-power-squaringops}
  \Sq^s_*(e_j\otimes x^{\otimes 2}) = \sum_k \binom{j+q-s}{s-2k}
  e_{j-s+2k}\otimes \Sq^k_*(x)^{\otimes 2} \,.
\end{equation}

\subsection{The continuous mod~$p$ homology of $S^{tC_p}$}
\label{sec:continuous-homology-of-singer}

Let $S$ be the sphere spectrum with the trivial $\bT$-action.  The
spectrum $\widetilde{E\bT}\wedge F(E\bT_+, S)$ is then a $\bT$-spectrum,
and $S^{tC_p} = [\widetilde{E\bT}\wedge F(E\bT_+, S)]^{C_p}$ is a
spectrum with a residual $\bT/C_p$-action
$\bar\lambda\: \bT/C_{p+}\wedge S^{tC_p}\to S^{tC_p}$, compatible with
the product on $S^{tC_p}$ and giving rise to a differential
$\bar\sigma\: \leftsuspension H_*^c(S^{tC_p})\to \sh_1 H_*^c(S^{tC_p})$.  By
Proposition \ref{prop:tate-dga}, $(H_*^c(S^{tC_p}), \bar\sigma)$ is a
filtered differential graded right $\sA$-module algebra.  In the
remainder of this section, we make this structure explicit.

The homological $C_p$-Tate spectral sequence converging
to $H_*^c(S^{tC_p})$ has
\[
  \hat E^2_{s,*} = \hat H^{-s}(C_p;\bF_p) \,,
\]
given by the Tate cohomology of $C_p$ with trivial coefficients.
There is an isomorphism of filtered graded algebras
$$
\hat{H}^{-*}(C_p; \bF_p) = \begin{cases}
E(u) \otimes P(t^{\pm1}) & \text{for $p$ odd,}\\
P(u^{\pm1}) & \text{for $p=2$,}
\end{cases}
$$
where $u \in \hat H^1$ is a class of homological degree~$-1$ and
$t \in \hat H^2$ is a class of homological degree~$-2$.  The
filtration~$n$ part of this algebra equals the subspace in degrees
$* \le n$.  Hereafter, when $p$ is odd, whenever we write $u^i t^r$ we
assume that $i \in \{0,1\}$ and $r \in \bZ$.  To describe the algebra
structure, we can choose any non-zero classes
$u, t\in H_{*}^c(S^{tC_p})$ in degrees~$-1$ and $-2$, respectively.
However, to be able to describe the differential on $H_*^c(S^{tC_p})$
we now proceed to make a specific choice of~$u$ and~$t$.

Let $\partial\: S^{tC_p} \longto \Sigma D_{C_p}S$ be
the connecting map of the Puppe sequence associated with the
norm--restriction cofiber sequence
$$
S_{hC_p} \overset{N^h}\longto S^{hC_p} \overset{R^h}\longto S^{tC_p} \,.
$$
The induced homomorphism on continuous homology
$\partial_*\: H^c_*(S^{tC_p})\to H_* (\Sigma D_{C_p}S)$ is an
isomorphism in degrees $*\geq 1$.  Moreover, $\partial$ is a map of
$\bT/C_p$-spectra, where the $\bT/C_p$-action on
$\Sigma D_{C_p}S\cong \Sigma BC_{p+}\cong \Sigma E\bT/C_{p+}$ is
induced by the $\bT$-action on $E\bT$.  Restricting the action to the
0-dimensional $\bT$-cell $S(\bC) \subset E\bT$ yields the map
\[
  \bT/C_{p+}\wedge \Sigma S(\bC)/C_{p+} \cong
  \Sigma (\bT/C_{p+}\wedge S(\bC)/C_{p+}) \overset{\Sigma \mu}\longto
  \Sigma S(\bC)/C_{p+}\,,
\]
where $\mu$ is the map induced by complex multiplication.  In
particular, $\bar\sigma\: H_1(\Sigma BC_p)\to H_2(\Sigma BC_p)$ is an
isomorphism, and it follows that so is
$\bar\sigma\: H^c_1(S^{tC_p})\to H_2^c(S^{tC_p})$.

Formula \eqref{eq:ext-power-bockstein} ensures that
$\beta_*\: H_{2k}(D_{C_p}S)\to H_{2k-1}(D_{C_p}S)$ is an isomorphism
for each $k\geq 1$, which implies that
$\beta_*\: H_{2k+1}(S^{tC_p})\to H_{2k}(S^{tC_p})$ is an isomorphism
for each $k\geq 1$.

In summary, we have the following commutative diagram, where every
homology group listed is isomorphic to $\bF_p$:
\[
  \xymatrix{
    H_0^c(S^{tC_p})
    \ar[d]_{\partial_*}
    & H_1^c(S^{tC_p})
    \ar[l]_{\beta_*}
    \ar[r]_{\cong}^{\bar\sigma}
    \ar[d]_{\partial_*}^{\cong}
    & H_2^c(S^{tC_p})
    \ar[d]_{\partial_*}^{\cong}
    & H_3^c(S^{tC_p})
    \ar[l]^{\cong}_{\beta_*}
    \ar[d]_{\partial_*}^{\cong}\\
    0
    & H_1(\Sigma D_{C_p}S)
    \ar[l]_{\beta_*}
    \ar[r]_{\cong}^{\bar\sigma}
    & H_2(\Sigma D_{C_p}S)
    & H_3(\Sigma D_{C_p}S)
    \ar[l]^{\cong}_{\beta_*} \,.
  }
\]
We claim that the upper left-hand Bockstein is an isomorphism, too.
This follows from the multiplicative structure.  Let $p>2$, and
consider any choice of non-zero classes $u \in H_{-1}^c(S^{tC_p})$ and
$t\in H_{-2}^c(S^{tC_p})$.  Then, since $t^{-1}$ is in the image
of~$\beta_*$, it follows from the fact that $\beta_*$ is a
derivation and a differential that $\beta_*(t^r)=0$ for each $r\in\bZ$.
Therefore,
$$
\beta_*(ut^{-1}) = \beta_*(ut^{-2}\cdot t) =
\beta_*(ut^{-2})\cdot t\,,
$$
which is non-zero.  We conclude that
$\beta_*\: H_1^c(S^{tC_p})\to H_0^c(S^{tC_p})$ is an isomorphism.
Replacing~$t$ by $u^2$, this argument also applies when $p=2$.

\begin{lemma}\label{lemma:u-and-t}
  Let $p>2$.  There is a unique pair of non-zero classes
  $u\in H^c_{-1}(S^{tC_p})$ and $t\in H_{-2}^c(S^{tC_p})$ such that
  $\beta_*(u) = t$, $\beta_*(t) = 0$, $\bar\sigma(u) = 1$ and
  $\bar\sigma(t) = 0$.

  Likewise, for $p=2$ the (unique) non-zero class
  $u\in H^c_{-1}(S^{tC_2})$ satisfies $\beta_*(u) = u^2$,
  $\beta_*(u^2) = 0$, $\bar\sigma(u) = 1$ and $\bar\sigma(u^2) = 0$.
\end{lemma}
\begin{proof} Let $p>2$.  Since
  $\beta_*\: H_1^c(S^{tC_p})\to H_0^c(S^{tC_p})$ is an isomorphism, we
  take $ut^{-1}\in H_1^c(S^{tC_p})$ to be the non-zero element
  such that $\beta_*(ut^{-1})=1$ is the algebra unit in
  $H^c_0(S^{tC_p})$.

  We then take $t^{-1}\in H_2^c(S^{tC_p})$ to be
  $t^{-1} = \bar\sigma(ut^{-1})$.  Since
  $\bar\sigma\: H_1^c(S^{tC_p}) \to H_2^c(S^{tC_p})$ is an
  isomorphism, $t^{-1}$ is also non-zero.

  Since $t^{-1}$ is in the image of $\bar\sigma$, it follows from the
  fact that $\bar\sigma$ is a derivation and a differential that
  $\bar\sigma(t^r)=0$ for each $r\in\bZ$.  The class $u$ is the
  product $ut^{-1}\cdot t$, and the Leibniz formula for $\bar\sigma$
  implies that $\bar\sigma(u) = \bar\sigma(ut^{-1})\cdot t$, which is
  equal to $1$ by our choices.

  We saw that $\beta_*(t) = 0$ in the discussion preceding the lemma,
  and $\beta_*(u) = t$ follows by the Leibniz rule.

  The case of $p=2$ can be proven by the same argument, with $t$ replaced
  by~$u^2$.

\end{proof}

\begin{lemma}\label{lemma:ut-images}
  Let $u\in H_{-1}^c(S^{tC_p})$ and $t\in H_{-2}^c(S^{tC_p})$ be the classes
  appearing in Lemma~\ref{lemma:u-and-t}.  Then
  $\partial_*\: H_*^c(S^{tC_p})\to H_* (\Sigma D_{C_p}S)$ is equal to
  the homomorphism given by
  $u^it^r \mapsto (-1)^{r}\Sigma e_{-1-i-2r}$, up to multiplication by
  a fixed unit.  As usual, $e_j=0$ for $j<0$.
\end{lemma}
\begin{proof}
  We first choose the orientation of $\bar\bT$ so that
  $\bar\lambda\: \bar\bT \times BC_{p}\to BC_{p}$ induces the
  differential $\bar\sigma\: \leftsuspension H_*(BC_{p})\to H_*(BC_{p})$
  sending $\leftsuspension e_0$ to $e_1$.

  In mod~$p$ homology, the diagonal map
  $\Delta\: BC_p\to BC_p\times BC_p$ is given
  by $\Delta_*( e_{n}) = \sum_{i+j=n}e_{i}\otimes e_{j} $, and
  interacts with the differential $\bar\sigma$ according to the
  co-Leibniz formula
  $\Delta_*\circ \bar\sigma = (\bar\sigma\otimes 1 +
  1\otimes\bar\sigma)\circ \Delta_*$.

  The differential
  $\bar\sigma\:H_{2r-1}^c(S^{tC_p})\to H_{2r}^c(S^{tC_p})$ takes
  $ut^{-r}$ to $t^{-r}$ for all $r\in \bZ$, which implies that
  $\bar\sigma\: H_{2r}(BC_p)\to H_{2r+1}(BC_p)$ is an isomorphism for
  each $r\geq 0$.  Thus $\bar\sigma(e_{2r+1})=0$, and
  $\bar\sigma(e_{2r}) = c_{2r}\cdot e_{2r+1}$ where each $c_{2r}$ is
  a unit.  It follows from the co-Leibniz formula by induction on
  $r$ that $c_{2r}=1$ for all $r\geq 0$.  Therefore,
  $\bar\sigma(e_{2r}) = e_{2r+1}$ for each $r\geq 0$.

  Similarly, the Bockstein $\beta_*$ sends $e_{2r+2}$ to $-e_{2r+1}$ for
  each $r\geq 0$, by \eqref{eq:ext-power-bockstein}.  In summary, we
  have that
  \begin{equation}
    \label{eq:l2ids}
    \bar\sigma(\Sigma e_{2r}) = \Sigma e_{2r+1}
    \text{~~and~~}
    \beta_*(\Sigma e_{2r+2}) = -\Sigma e_{2r+1}
  \end{equation}
  for each $r\geq 0$.  Note that
  $\bar\sigma \Sigma = \Sigma\bar\sigma$ since $\bar\lambda$ is a left
  action, and that $\beta_* \Sigma = \Sigma \beta_*$ since
  $\beta \Sigma = -\Sigma \beta$ where $\beta$ denotes the left action
  of $\beta$.

  Recall that $\partial_*$ is an isomorphism in positive degrees.
  Since $\partial_*$ also commutes with both $\beta_*$ and
  $\bar\sigma$, the explicit formulas \eqref{eq:l2ids} force the
  identity $\partial_*(u^it^r) = (-1)^r\Sigma e_{-1-i-2r}$, up to
  multiplication by some common unit, for all $r<0$.
\end{proof}
According to \cite{HM03}*{Add.~4.2.2}, the unit in question is $+1$,
assuming that all sign conventions agree.

\begin{lemma}
  \label{lemma:singer-an-primitives}
  Let $n\geq 0$.  For $p>2$, multiplication with $t^{p^{n}}$ defines
  an $\sA(n)$-linear isomorphism
  $H_{*}^c(S^{tC_p})\to H_{*-2p^n}^c(S^{tC_p})$.  For $p=2$,
  multiplication with $u^{2^{n+1}}$ defines an $\sA(n)$-linear
  isomorphism $H_{*}^c(S^{tC_2})\to H_{*-2^{n+1}}^c(S^{tC_2})$.
\end{lemma}
\begin{proof} Let $p>2$. By the Cartan formula, the lemma will follow
  from the fact that $\sA(n)$ acts trivially on $t^{p^n}$, i.e.,
  that $t^{p^n}$ is $\sA(n)_*$-comodule primitive.
  Since $t\in H_*^c(S^{tC_p})$ is invertible, multiplication by
  $t^{p^n}$ defines an $\bF_p$-linear isomorphism, and it suffices to
  check that $\sA(n)$ acts trivially on the class $t^{-p^n}$ in
  positive degree $2p^n$.

  By Lemma \ref{lemma:ut-images}, the $\sA$-linear homomorphism
  $\partial_*\: H_*(S^{tC_p})\to H_{*}(\Sigma D_{C_p}S)$ is an
  isomorphism when restricted to positive degrees, sending
  $t^{-p^{n}}$ to $-\Sigma e_{-1+2p^{n}}$, up to multiplication by a
  unit.  The degree of both $\beta_*(t^{-p^{n}})$ and
  $\P^{p^k}_*(t^{-p^{n}})$ is positive if $k < n$.  Therefore, $\sA(n)$
  acts trivially on $t^{-p^n}$ if and only if it acts trivially on
  $\Sigma e_{-1+2p^{n}}\in H_*(\Sigma D_{C_p}S)$.  It follows from
  \eqref{eq:ext-power-bockstein} that~$\beta_*$ acts trivially
  $\Sigma e_{-1+2p^{n}}$.  By \eqref{eq:ext-power-steenrod-ops} we
  have
  \begin{equation}
    \label{eq:a-n-primitive}
    \P^s_*(\Sigma e_{-1+2p^{n}}) = \binom{p^{n} - 1 - s(p-1)}{s}
    \Sigma e_{-1 + 2p^{n} - 2s(p-1)}
  \end{equation}
  in $H_*(\Sigma D_{C_p}S)$.  When $s=p^k$, the binomial coefficient in
  \eqref{eq:a-n-primitive} equals
  \[
    \binom{ \bigl(p^{n-1} + p^{n-2} + \ldots + p^{k+1} + p^{k-1} +
      \ldots + 1\bigr)(p-1)}{p^k}\,,
  \]
  which is congruent to~$0$ mod~$p$ by Lucas' theorem if $k<n$.  The
  claim that $\sA(n)$ acts trivially on $t^{-p^n}$ follows since
  $\sA(n)$ is generated, as an algebra, by $\beta$ and $\P^{p^k}$
  for $0\leq k < n$.

  The argument above can be repeated in the case of $p=2$, with $t$
  replaced by $u^2$, $\beta_*$ replaced by $\Sq^1_*$, $\P^s_*$ replaced
  by $\Sq_*^{2s}$, and using formula \eqref{eq:ext-power-squaringops}
  instead of \eqref{eq:ext-power-steenrod-ops}--
  \eqref{eq:ext-power-bockstein}.
\end{proof}

\begin{proposition}
  \label{prop:continuous-homology-of-s-tcp}
  There is an isomorphism of filtered differential graded right
  $\sA$-module algebras
  \begin{align*}
    H_*^c(S^{tC_p})
    \cong
    \begin{cases}
      E(u)\otimes P(t^{\pm 1}) & \text{for $p>2$,} \\
      P(u^{\pm 1}) & \text{for $p=2$,}
    \end{cases}
  \end{align*}
  where $u$ has degree~$-1$ and~$t$ has degree $-2$.  The differential
  acts by $\bar\sigma(u)=1$ and $\bar\sigma(t)=0$, and the filtration
  is given by $F_n H_*^c(S^{tC_p}) = H_{*\leq n}^c(S^{tC_p})$ for each
  integer~$n$.  Moreover, the right $\sA$-module structure is given by
  the formulas
  \begin{align}
    \P^s_*(u^it^r)
    &= \binom{-1-r-s(p-1)}{s}u^it^{r+s(p-1)}
      \label{eq:singer-sphere-steenrod} \\
    \intertext{and}
    \beta_*(u^it^r)
    &=
      \begin{cases}
        0 &  \text{for $i=0$,} \\
        t^{r+1} & \text{for $i=1$}
      \end{cases}\label{eq:singer-sphere-bockstein}
      \intertext{when $p>2$, and by}
      \Sq^s_*(u^r) &= \binom{-1-r-s}{s}u^{r+s} \label{eq:singer-sphere-squares}
  \end{align}
  when $p=2$.
\end{proposition}
\begin{proof}
  Let $p>2$.  Let $n\ge0$ be so large that $\P^s \in \sA(n)$.  Since
  multiplication with $t^{-p^{n}}$ is an $\sA(n)$-linear isomorphism
  by Lemma~\ref{lemma:singer-an-primitives}, we can assume that the
  degrees of $u^it^r$ and $\P^s_*(u^it^r)$ are positive.

  By Lemma \ref{lemma:ut-images}, the $\sA$-linear homomorphism
  $\partial_*\: H_*(S^{tC_p})\to H_{*}(\Sigma D_{C_p}S)$ is an
  isomorphism when restricted to positive degrees.  Up to
  multiplication by a unit, $\partial_*$ sends $u^it^{r}$ to
  $(-1)^r \Sigma e_{-1-i-2r}$.

  Formula \eqref{eq:ext-power-steenrod-ops} applied to the case of
  $H_*(D_{C_p}S)$ yields
  $$
  \P^s_*(\Sigma e_{-1-i-2r}) = \binom{-1-r-s(p-1)}{s}\Sigma
  e_{-1-i-2r-2s(p-1)}\,,
  $$
  implying \eqref{eq:singer-sphere-steenrod} since $\partial_*$ is
  $\sA(n)$-linear.  Note that the suspension does not introduce a sign
  in these formulas, since the degree of $\P^s$ is even.

  Formula \eqref{eq:singer-sphere-bockstein} is implied by
  Lemma~\ref{lemma:u-and-t} and the fact that $\beta_*$ is a
  derivation, but can also be derived from comparing with the action
  of $\beta_*$ in $H_*(\Sigma D_{C_p}S)$ given by
  \begin{equation*}
    \beta_*(\Sigma e_j) =
    \begin{cases}
      -\Sigma e_{j-1} & \text{when $j$ is even,}\\
      0 & \text{when $j$ is odd,}
    \end{cases}
  \end{equation*}
  using $\sA(n)$-linearity of $\partial_*$.

  The argument above can be repeated in the case of $p=2$ with $t$
  replaced by $u^2$, $\beta_*$ replaced by $\Sq^1_*$, $\P^s_*$ replaced
  by $\Sq_*^{2s}$, and using formula \eqref{eq:ext-power-squaringops}
  instead of \eqref{eq:ext-power-steenrod-ops}--
  \eqref{eq:ext-power-bockstein}.
\end{proof}

Proposition \ref{prop:continuous-homology-of-s-tcp} identifies
$H_*^c(S^{tC_p})$ with the $\bF_p$-dual of the $C_p$-version of the 
\emph{algebraic Singer construction} on the right $\sA$-module $\bF_p$.  See
\cite{LNR12}*{Sec.~3}.

\section{The topological $C_p$-Singer construction}\label{sec:topsinger}

Let $G\subset \Sigma_p$ be any subgroup of the symmetric group, and
$B$ any spectrum.  As in \cite{BMMS86}*{\S II.3}, there is a tower of
spectra
\begin{multline} \label{eq:tower-extended-powers} \dots \longto
  \Sigma^{n+1} D_G(\Sigma^{-n-1} B)
  \overset{\Sigma^{n}\!\Delta}\longto %{\xrightarrow{\hspace*{5mm}}}
  \Sigma^{n} D_G(\Sigma^{-n} B) \longto \dots\\
  \longto
  \Sigma D_G(\Sigma^{-1} B)
  \overset{\Delta}\longto
  D_G(B)\,,
\end{multline}
where the maps are induced by the $p$-fold diagonal
$S^1 \to (S^1)^{\wedge p}$.  We are interested in the case where
$G=C_p$ is the cyclic group with $p$ elements.  By
\cite{LNR12}*{Prop. 5.7} there are natural homotopy equivalences
\begin{equation}
  \label{eq:tate-extended-power}
  (B^{\wedge p})^{tC_p}[1-n(p-1)] \simeq \Sigma^{1+n} D_{C_p}(\Sigma^{-n}B)\,,
\end{equation}
compatible for varying $n$.  The induced equivalence of homotopy
limits, combined with the first part of Lemma
\ref{lemma:continuous-homology-is-inverse-limit}, identifies
$c\: H_*^c((B^{\wedge p})^{tC_p}) \to H_*^c((B^{\wedge
  p})^{tC_p})^{\wedge}$ with the natural homomorphism
\begin{equation}
  \label{eq:completion-tate}
  c\: H_*^c((B^{\wedge p})^{tC_p}) \longto \lim_n H_* (\Sigma^{1+n} D_{C_p}(\Sigma^{-n}B))\,.
\end{equation}
It will follow from Proposition~\ref{prop:singer-tate-filtration} that
\eqref{eq:completion-tate} is an isomorphism whenever $B/p$ is bounded
below.

Miller observed \cite{BMMS86}*{\S II.5} that if one applies mod $p$
cohomology to the tower of extended powers
\eqref{eq:tower-extended-powers} when $G=\Sigma_p$, the colimit of the
resulting left $\sA$-modules is isomorphic to a desuspension of the
$\Sigma_p$-version of the algebraic Singer construction $R_+(H^*(B))$
on the left $\sA$-module $H^*(B)$.  See \cite{Sin81}, \cite{LS82}.
This fact motivates the following definition.
\begin{definition}[\cite{LNR12}*{Def.~5.8}] \label{dfn:topological-singer}
  The \emph{topological Singer construction} on a spectrum $B$
  is the $C_p$-Tate construction
  $$
  R_+(B) = (B^{\wedge p})^{tC_p}\,.
  $$
  Here $B^{\wedge p}$ is the $p$-fold smash power of $B$, considered
  as a spectrum with $C_p$ acting by cyclic permutation of the smash
  factors.
\end{definition}

The calculations that identify the continuous cohomology of $R_+(B)$
with the algebraic Singer construction on the left $\sA$-module
$H^*(B)$ were performed in \cite{LNR12}, and explicit formulas for the
action of the mod~$p$ Steenrod algebra on $H_*^c(R_+(B))$ were derived
by dualizing from cohomology to homology.  These dualizations depended
on the assumption that $B/p$ is bounded below and of finite type.

In the present work we revisit the topological Singer construction,
since we need to discuss the multiplicative structure of
$H_*^c(R_+(B))$.  Also, by working exclusively in mod~$p$ homology we
arrive at explicit formulas for the right action of $\sA$ on
$H_*^c(R_+(B))$, without the assumption that $B/p$ is of finite type.

\begin{remark}
  Let $B=S$.  The $p$-fold product map $S^{\wedge p}\to S$ is a
  $C_p$-equivariant equivalence, inducing an equivalence of spectra
  $R_+(S) = (S^{\wedge p})^{tC_p}\simeq S^{tC_p}$.  Proposition
  \ref{prop:continuous-homology-of-s-tcp} thus describes $H_*^c(R_+(B))$ as
  a filtered differential graded right $\sA$-module algebra in the
  case of $B=S$.
\end{remark}

\subsection{The $E$-based Singer construction}\label{sec:e-based-singer}

Consider the case $G = C_p$, let $E$ be an $E_\infty$ ring
spectrum, and let $M$ be an $E$-module spectrum.  Then
$M^{\wedge_E p}$ is an $E$-module with $C_p$ acting by cyclic
permutations on the smash factors, in the symmetric monoidal category
of $E$-modules.
The $C_p$-equivariant shuffle isomorphism
$M^{\wedge_E p}\wedge_E N^{\wedge_E p} \cong (M\wedge_E N)^{\wedge_E
  p}$ gives the assignment $M\mapsto M^{\wedge_E p}$ the structure of
a strong symmetric monoidal functor.

The \emph{$E$-based topological Singer construction} on an $E$-module
$M$ is the $C_p$-Tate construction
$$
R_+^E(M) = (M^{\wedge_E p})^{tC_p} =
(M\wedge_E \cdots \wedge_E M)^{tC_p} \,.
$$
It is a lax symmetric monoidal endofunctor on the homotopy category of
$E$-modules, being the composition of a strong and a lax symmetric monoidal functor.

We will only be interested in the $E$-based Singer construction on
induced $E$-modules, i.e., those of the form $M=E\wedge B$ for $B$ a
spectrum.  We make the following generalization of
Definition~\ref{dfn:topological-singer}.
\begin{definition}\label{dfn:E-rplus}
  The \emph{$E$-based topological Singer construction} on a spectrum
  $B$ is the $C_p$-Tate construction
  $$
  R_+^E(E\wedge B) = ((E\wedge B)^{\wedge_E p})^{tC_p}\,.
  $$
\end{definition}
Note that there is a natural equivalence of $E$-modules
$R_+^E(E\wedge B) \simeq (E\wedge B^{\wedge p})^{tC_p}$.

The assignment $B\mapsto E\wedge B$ is a strong symmetric monoidal
functor, and it follows that $B\mapsto R_+^E(E\wedge B)$ is a lax symmetric
monoidal functor from the homotopy category of spectra
to the homotopy category of $E$-modules.

Suppose temporarily that $E=S$.  Associated to a spectrum $B$, the
structure map of the $S$-based topological Singer construction, now
often known as the \emph{Tate diagonal}, is the map of spectra
$\epsilon_B\: B\to R_+(B)$ given by the composite map
\begin{equation}
  \label{eq:tate-diagonal}
  \xymatrix@C7mm{
    \epsilon_B\: B
    & [\widetilde{EC}_p\wedge B^{\wedge p}]^{C_p}
    \ar[l]_-{\ \simeq}
    \ar[r]^-{\hat\Gamma_1}
    & [\widetilde{EC}_p\wedge F(EC_{p+}, B^{\wedge p})]^{C_p} = R_+(B)\,,
  }
\end{equation}
as in \cite{LNR11}*{Def.~5.10} and \cite{NS18}*{\S III.1}.  The Tate
diagonal is a monoidal natural transformation of lax symmetric
monoidal endofunctors on the homotopy category of spectra.  In
particular, the monoidal unit
$\eta^S\: S\longto S^{tC_p} \simeq R_+(S)$ equals
$\epsilon_S$, the Tate diagonal for the sphere spectrum.

For a general $E$, the Tate diagonal for $B$
induces an \emph{$E$-based Tate diagonal} of $E$-modules
\begin{equation}
  \label{eq:e-based-tate-diagonal}
  \epsilon_B^E\: E\wedge B
  \overset{1\wedge \epsilon_B}\longto E\wedge R_+(B)
  \overset{\kappa}\longto (E\wedge B^{\wedge p})^{tC_p}
  \simeq R_+^E(E\wedge B) \,.
\end{equation}
Both $1\wedge \epsilon_B$ and $\kappa$ are monoidal natural transformations,
and it follows that the composition $\epsilon_B^E$ is a monoidal
natural transformation of lax symmetric monoidal functors in $B$.  In
particular, the monoidal unit
$\eta^E\: E\longto E^{tC_p} \simeq R^E_+(E)$ equals
$\epsilon^E_S$, the $E$-based Tate diagonal for the sphere spectrum.

\begin{remark}
  It is known from \cite{LNR12}*{Thm.~5.13} and \cite{NS18}*{Thm.~1.7}
  that the ($S$-based) Tate diagonal $\epsilon_B\: B\to R_+(B)$ is a
  $p$-adic homotopy equivalence for each bounded below spectrum $B$.
  In general, however, the $E$-based Tate diagonal is not an
  equivalence, which can be seen by considering the induced
  homomorphism in homotopy in the case of $E=H$ and $B=S$.
\end{remark}

\begin{remark}
  The natural transformation \eqref{eq:e-based-tate-diagonal} is natural in the
  $S$-module $B$, but generally not in the induced $E$-module
  $E\wedge B$.  One could try to construct a natural transformation
  $1\to R_+^E$ of functors of $E$-modules by forming the composite
  \begin{equation}
    \label{eq:non-linear}
    M\overset{\epsilon_M}\longto R_+(M)
    \overset{\wedge_S \to \wedge_E}\longto R_+^E(M)\,,
  \end{equation}
  where the last map is the natural map induced by passing from smash
  products over $S$ to smash products over $E$.  The spectrum map
  \eqref{eq:non-linear} would be natural with respect to $E$-module
  maps $M\to M'$, but it is generally not $E$-linear.

  For instance, in the case $E = M = H$, \cite{LNR12}*{(3.9),
    Prop.~5.14} shows that the composite
  $H \to R^H_+(H) \simeq H^{tC_p} \to \Sigma H_{hC_p}$ is nonzero in
  homology, and can therefore not be $H$-linear.  See also
  \cite{NS18}*{Thm.~III.1.10}, as well as
  \cite{Wilson2023}*{Warning~2.3.4}.
\end{remark}

\subsection{The diagonal equivalence}

Let $K$ be a finite CW complex with a chosen base point.  The diagonal
$\Delta\: K\to K^{\wedge p}$ is $C_p$-equivariant and induces a map
$\Delta^E_K\: R^E_+(E\wedge B)\wedge K\to R^E_+(E\wedge B\wedge K)$
for any spectrum $B$.  Indeed, $\Delta^E_K$ is the $C_p$-fixed points
of the composite map
\medskip
\begin{align}
  \nonumber \widetilde{EC}_p\wedge F(EC_{p+}, E\wedge B^{\wedge p})\wedge K
  &\overset{1\wedge \alpha}\longto
    \widetilde{EC}_p\wedge F(EC_{p+}, E\wedge B^{\wedge p}\wedge K) \\
  \label{eq:delta-finite-complex}  &\hspace{-3.3mm}\overset{1 \wedge F(1, 1 \wedge \Delta)}\longto
    \widetilde{EC}_p\wedge F(EC_{p+}, E\wedge B^{\wedge p}\wedge K^{\wedge p})\\
  \nonumber &\hspace{-6.0mm}\overset{1\wedge F(1, \text{shuffle})}\longto
    \widetilde{EC}_p\wedge F(EC_{p+}, E\wedge (B\wedge K)^{\wedge p})\,,
\end{align}
where $\alpha$ is the canonical map of spectra
$F(A, X)\wedge Y \to F(A, X\wedge Y)$, or more generally
$F(A, X)\wedge F(B, Y) \to F(A\wedge B, X\wedge Y)$, for spectra
$A,B,X$ and $Y$.

Furthermore, in the case of the suspension spectrum of $K$, the Tate
diagonal~$\epsilon_K\: K\to R_+(K)$ is realized as the fixed points of
the following $C_p$-equivariant map of spectra,
\begin{multline*}
%  \label{eq:epsilon-K}
  K\cong S^0\wedge K \overset{i\wedge \Delta}
  \longto \widetilde{EC}_p\wedge K^{\wedge p} \\
  \cong \widetilde{EC}_p\wedge F(S^0, K^{\wedge p})
  \overset{1\wedge F(c,1)}\longto \widetilde{EC}_p\wedge F(EC_{p+},
  K^{\wedge p})\,,
\end{multline*}
where
$i\: S^0\to \widetilde{EC}_p$ is the cofiber of the collapse map
$c\: EC_{p+}\to S^0$.  Likewise, the $E$-based Tate diagonal
$\epsilon^E_K$
is the $C_p$-fixed points of the composite
\begin{multline}
  \label{eq:pre-epsilon-K-E}
  e^E_K\: E\wedge K\cong S^0\wedge E\wedge K
  \overset{i\wedge 1\wedge \Delta}\longto
  \widetilde{EC}_p\wedge E\wedge K^{\wedge p}\\
  \cong \widetilde{EC}_p\wedge F(S^0, E\wedge K^{\wedge p})
  \overset{1\wedge F(c,1)}\longto
  \widetilde{EC}_p\wedge F(EC_{p+}, E\wedge K^{\wedge p})\,.
\end{multline}

\begin{lemma}\label{lemma:delta-epsilon}
  Let $B$ be a spectrum, and $K$ a finite based CW complex. The map
  $\Delta^E_K$ agrees with the composite map
  \begin{multline}
    \label{eq:monoid-delta}
    R^E_+(E \wedge B) \wedge K \cong
    R^E_+(E\wedge B)\wedge_E (E\wedge K) \\
    \overset{1\wedge
      \epsilon^E_{K}}\longto R^E_+(E\wedge B)\wedge_E R^E_+(E\wedge K)
    \overset\mu\longto R^E_+(E\wedge B\wedge K)\,,
  \end{multline}
  where $\mu = \mu^E_{E \wedge B^{\wedge p}, E \wedge K^{\wedge p}}$.
\end{lemma}
\begin{proof}
  For any $C_p$-spectrum $X$, let $t(X)$ be the Tate $C_p$-spectrum
  $\widetilde{EC}_p\wedge F(EC_{p+}, X)$, and consider the following
  diagram of $C_p$-spectra:
  \begin{equation}
    \label{eq:monoid-delta-diagram}
    \xymatrix@C5mm{
      t(E\wedge B^{\wedge p})\wedge K
      \ar[r]^-{1\wedge \alpha}
      \ar[d]^-\cong
      &
      t(E\wedge B^{\wedge p}\wedge K)
      \ar[d]^-{t(1\wedge 1\wedge \Delta)}
      \\
      t(E\wedge B^{\wedge p})\wedge_E E\wedge K
      \ar[d]^-{1\wedge e^E_K}
      &t(E\wedge B^{\wedge p}\wedge K^{\wedge p})
      % &\hspace{-10mm}((B\wedge K)^{\wedge p})
      \\
      t(E\wedge B^{\wedge p})\wedge_E t(E\wedge K^{\wedge p})
      \ar[r]^-{\alpha}
      & \widetilde{EC}_p^{\wedge 2}\wedge
      F((EC_{p+})^{\wedge 2}, E\wedge B^{\wedge p}\wedge K^{\wedge p})
      \ar[u]_{N\wedge F(\Delta, 1)}\,.
    }
  \end{equation}
  Here,~$N$ is a $C_p$-equivariant pairing
  $N \: \widetilde{EC}_p \wedge \widetilde{EC}_p \to \widetilde{EC}_p$
  extending the fold map
  $\widetilde{EC}_p \cup_{S^0} \widetilde{EC}_p \to \widetilde{EC}_p$, as
  in~\cite{HR24}*{\S6.1}.  The composition of
  $N\wedge F(\Delta, 1)\circ \alpha$, followed by the map
  \begin{equation}
    \label{eq:rplus-monoidal-shuffle}
    t(E\wedge B^{\wedge p}\wedge K^{\wedge p}) \cong
    t(E\wedge (B\wedge K)^{\wedge p})
  \end{equation}
  induced by the $C_p$-equivariant shuffle
  $B^{\wedge p}\wedge K^{\wedge p}\cong (B\wedge K)^{\wedge p}$,
  produces the pairing $\mu$ of \eqref{eq:monoid-delta}, after passing
  to $C_p$-fixed points.

  Using the description of $\epsilon^E_K$ given
  by~\eqref{eq:pre-epsilon-K-E} and the unitality of~$N$, one can
  verify that \eqref{eq:monoid-delta-diagram} commutes.  Comparing the
  two paths around the diagram, followed by the map
  \eqref{eq:rplus-monoidal-shuffle}, and finally passing to
  $C_p$-fixed points, proves the lemma, since the upper composite
  becomes $\Delta^E_K$ and the lower becomes
  $\mu\circ (1\wedge \epsilon^E_K)$.
\end{proof}

\begin{lemma}\label{lemma:delta-is-a-homotopy-equivalence}
  Let $B$ be a spectrum and $K$ a finite based CW complex.  The map
  \begin{equation}
    \label{eq:delta-k-h}
    \Delta^E_K\: R_+^E(E\wedge B)\wedge K \longto R_+^E(E\wedge B\wedge K)
  \end{equation}
  is an equivalence.
\end{lemma}
\begin{proof}
  We claim that each of the three maps
  in~\eqref{eq:delta-finite-complex} induces an equivalence after
  passing to $C_p$-fixed points.  This is immediate
  for~$1\wedge \alpha$ since $K$ is a finite complex, and for
  $1\wedge F(1, \text{shuffle})$ since it is a $C_p$-equivariant
  isomorphism.

  For the middle map, we note that the $C_p$-Tate construction is an
  exact functor that vanishes on induced $C_p$-spectra.  Therefore, the
  diagonal $\Delta\: K\to K^{\wedge p}$ induces an equivalence
  $$
  (1\wedge \Delta)^{tC_p}\: (X\wedge K)^{tC_p}
  \overset\simeq\longto
  (X\wedge K^{\wedge p})^{tC_p}
  $$
  for any spectrum $X$ with $C_p$-action.  Applying this to the case
  of $X=E\wedge B^{\wedge p}$ proves the claim, and the lemma follows.
\end{proof}

\begin{proposition}\label{prop:diag-q}
  Let $B$ be a spectrum.  The composite map
  \begin{equation}
    \label{eq:diag-q}
  R_+^E(E\wedge B)\wedge_E (E\wedge S^q) \overset{1\wedge \epsilon^E_{S^q}}\longto
  R_+^E(E\wedge B) \wedge_E R_+^E(E\wedge S^q) \overset\mu\longto
  R_+^E(E\wedge B\wedge S^q)
  \end{equation}
  is a homotopy equivalence for each $q\in \bZ$.
\end{proposition}
\begin{proof}
  Assume first that $q\geq 0$.  In the case of $K=S^q$, the statement
  of Lemma \ref{lemma:delta-epsilon} is that \eqref{eq:diag-q} is
  homotopic to $\Delta^E_{S^q}$.  The proposition then follows from
  Lemma \ref{lemma:delta-is-a-homotopy-equivalence}.

  Let $q$ be any integer, and let $n\geq0$ be an integer such that
  $q+n\geq 0$.  Consider the diagram
  \[
    \begin{aligned}
      \xymatrix@C14mm{
      R_+^E(E\wedge B)\wedge_E (E\wedge S^q)\wedge_E (E\wedge S^n)
      \ar[r]^-{f_q\wedge 1}
      \ar[dd]^\cong
      & R_+^E(E\wedge B\wedge S^q)\wedge_E (E\wedge S^n)
        \ar[d]^-{1\wedge \epsilon^E_{S^n}}
        \\
      & R_+^E(E\wedge B\wedge S^q)\wedge_E R_+^E(E\wedge S^n)
        \ar[d]^\mu \\
      R_+^E(E\wedge B)\wedge_E (E\wedge S^{q+n})
      \ar[r]^-{f_{q+n}}_-\simeq
      & R_+^E(E\wedge B\wedge S^{q+n}) \,.
      }
    \end{aligned}
  \]
  Here, $f_k$ is the composite map \eqref{eq:diag-q} for $q=k$.
  Commutativity is implied by the fact that $\epsilon^E$ is a monoidal natural
  transformation of lax symmetric monoidal functors.  The lower horizontal
  map is an equivalence by the first part of this proof, since
  $q+n\geq 0$.  The right-hand vertical composite map is equivalent to
  $\Delta_{S^n}^E\: R_+^E(E\wedge B\wedge S^q)\wedge S^n\to
  R_+^E(E\wedge B\wedge S^{q+n})$ by Lemma \ref{lemma:delta-epsilon},
  and is therefore an equivalence by Lemma
  \ref{lemma:delta-is-a-homotopy-equivalence}.  It follows that the
  top horizontal map, as well as its $n$-th desuspension, is an
  equivalence.
\end{proof}

\begin{corollary}\label{cor:epsilon-suspensions}
  Let $B$ be a spectrum.  The map $\epsilon^E_B$ commutes with
  suspensions.  More precisely, for each $n \in \bZ$ there is a
  commutative square
  \begin{equation*}
    \xymatrix@C11mm{
      \Sigma^n(E\wedge B)
      \ar[r]^{\Sigma^n \epsilon_B^E}
      \ar[dr]_{\epsilon_{\Sigma^n B}^E}
      & \Sigma^n R_+^E(E\wedge B)
      \ar[d]^{\simeq}\\
      & R_+^E(\Sigma^n(E\wedge B))
    }
  \end{equation*}
  where the vertical equivalence is the composite map of
  Proposition~\ref{prop:diag-q} for $q=n$.
\end{corollary}
\begin{proof}
  Consider the following diagram:
  \[
    \xymatrix@C18mm{
      & R_+^E(E\wedge B) \wedge_E E\wedge S^n
      \ar[d]^{1\wedge \epsilon^E_{S^n}}\\
      E\wedge B\wedge_E E\wedge S^n
      \ar[ur]^{\epsilon^E_B\wedge 1}
      \ar[r]^-{\epsilon^E_B \wedge \epsilon^E_{S^n}}
      \ar[dr]_{\epsilon^E_{B\wedge S^n}}
      & R_+^E(E\wedge B)\wedge_E R_+^E(E\wedge S^n)
      \ar[d]^{\mu}\\
      & R_+^E(E\wedge B\wedge S^n) \rlap{\,.}
    }
  \]
  The lower triangle commutes since $\epsilon^E$ is a monoidal
  natural transformation.  The upper triangle clearly commutes, and
  the right hand vertical composite is the homotopy equivalence
  of Proposition~\ref{prop:diag-q}.   The outer triangle is the
  commutative diagram of the lemma.
\end{proof}

\subsection{The continuous mod~$p$ homology of the topological Singer
  construction}
\label{sec:continuous-mod-p-homology-of-r-plus}

For a bounded below ring spectrum $B$, we discuss the
structure of $H_*^c(R_+(B))$ as a filtered right $\sA$-module algebra.

Let $E=H$ be the mod~$p$ Eilenberg--MacLane spectrum.  The functors
$B\mapsto H\wedge B$ and $B\mapsto R_+^H(H\wedge B)$ are lax symmetric
monoidal in $B$, and the $H$-based Tate diagonal
$\epsilon_B^H\: H\wedge B\to R_+^H(H\wedge B)$ is a monoidal natural
transformation of lax symmetric monoidal functors.  Passing to homotopy
groups, this monoidal structure implies that
$\pi_*(R_+^H(H\wedge B))\cong H_*^c(R_+(B))$ is a left module over the
$\bF_p$-algebra $H_*^c(S^{tC_p})$. It is also a right module over the
mod~$p$ Steenrod algebra, by the discussion in
Subsection~\ref{sec:continuous-homology-of-tate}.  The $H$-based Tate diagonal
$\epsilon_B^H$ from \eqref{eq:e-based-tate-diagonal} induces a homomorphism of right $\sA$-modules
$$
(\epsilon^H_B)_*\: H_*(B) \longto H_*^c(R_+(B))\,,
$$
which is a homomorphism of (graded commutative) $\bF_p$-algebras if $B$ is a
(homotopy commutative) ring spectrum.

For any integer~$q$ and spectrum~$B$, the composite map
\begin{equation}
  \label{eq:l4eq}
  R_+^H(H\wedge B)\wedge_H (H\wedge S^q) \overset{1\wedge \epsilon^H_{S^q}}\longto
  R_+^H(H\wedge B) \wedge_H R_+^H(H\wedge S^q) \overset\mu\longto
  R_+^H(H\wedge B\wedge S^q)
\end{equation}
of Proposition~\ref{prop:diag-q} induces an isomorphism
\begin{equation}
  \label{eq:l5eq}
  H_*^c(R_+(B))\otimes H_*(S^q)
  \overset\cong\longto H_*^c(R_+(B\wedge S^q))
\end{equation}
in continuous mod~$p$ homology.  For non-negative~$q$,
Lemma~\ref{lemma:delta-epsilon} says that \eqref{eq:l4eq} agrees with
$$
\Delta_{S^q}^H\: R^H_+(H\wedge B)\wedge S^q \longto R_+^H(H\wedge B\wedge S^q) \,.
$$
Passing to homotopy groups, the latter map induces an $\sA$-linear homomorphism in
continuous mod~$p$ homology.  By an argument similar to the one in the
proof of Proposition~\ref{prop:diag-q}, we see that~\eqref{eq:l5eq} is
$\sA$-linear for all integers~$q$.

\subsubsection{Additive structure of $H_*^c(R_+(B))$}

Let $B$ be a spectrum.  The homological $C_p$-Tate spectral sequence
\eqref{eq:homological-tate-ss} associated with $X=B^{\wedge p}$ has
$\hat E^2$-term given by
\begin{equation}
  \label{eq:singer-ss}
  \hat{E}^2_{s,*}(B) = \hat{H}^{-s}(C_p; H_*(B)^{\otimes p})\,,
\end{equation}
and converges conditionally to $H_*^c (R_+(B))$.  The group $C_p$ acts
on the coefficients $H_*(B)^{\otimes p}$ by cyclic permutations.  For
any graded $\bF_p$-vector space $M_*$, there is an isomorphism of
bigraded $\bF_p$-vector spaces
$$
\hat H^{-*}(C_p; M_*^{\otimes p}) \cong
\hat H^{-*}(C_p; \bF_p) \otimes
\bF_p\{x^{\otimes p} \mid x \in \sB \} \,,
$$
where~$x$ ranges over an $\bF_p$-basis $\mathscr{B}$ for $M_*$.

Let $x\: S^q\to H\wedge B$ represent a class in $H_q(B)$, and let
$\bar x \: H \wedge S^q \to H \wedge B$ be its unique extension to a
map of $H$-modules.  Passing to homotopy groups, $\bar x$ induces a
homomorphism $\bar x_*\: H_*(S^q) \longto H_*(B)$ sending
$\iota_q\mapsto x$, where $\iota_q=\Sigma^q 1\in H_q(S^q)$ is the
fundamental class.  By naturality of $R^H_+(-)$ in $H$-modules, we get
a filtration-preserving map
\begin{equation}
  \label{eq:xp}
  R^H_+(\bar{x})\: R^H_+(H\wedge S^q) \longto R^H_+(H\wedge B)
\end{equation}
inducing the map of homological $C_p$-Tate spectral sequences
\begin{equation}
  \label{eq:x-power-p}
  \hat{H}^{-*}(C_p; H_*(S^q)^{\otimes p}) \longto
  \hat{H}^{-*}(C_p; H_*(B)^{\otimes
    p})
\end{equation}
that sends $u^it^r\otimes \iota_q^{\otimes p}$ to
$u^it^r\otimes x^{\otimes p}$ for each $i\in \{0,1\}$ and $r\in \bZ$.

\begin{proposition}\label{prop:singer-tate-filtration}
  Let $B$ be a spectrum.  The homological $C_p$-Tate spectral
  sequence~\eqref{eq:singer-ss} for $B^{\wedge p}$
  converges conditionally to $H_*^c(R_+(B))$, and collapses with
  \begin{equation}
    \label{eq:e2-tate}
    \hat{E}^\infty_{s,*}(B) = \hat{E}^2_{s,*}(B) = \hat H^{-s}(C_p;
    H_*(B)^{\otimes p})\,.
  \end{equation}
  If $B/p$ is bounded below, then the spectral sequence converges
  strongly and, in particular, the Tate filtration of $H_*^c(R_+(B))$
  is complete Hausdorff and exhaustive.
\end{proposition}
\begin{proof}
  Since $\hat E^2(S^q)$ is concentrated on the horizontal line
  $(*,pq)$, it follows that the homological $C_p$-Tate spectral
  sequence for $X=(S^q)^{\wedge p}$ collapses at the $\hat E^2$-term.
  For a spectrum $B$, any
  given $u^it^r\otimes x^{\otimes p} \in\hat E^2(B)$ is an infinite
  cycle by naturality with respect to \eqref{eq:x-power-p}.  Thus, the
  spectral sequence collapses with $\hat E^\infty(B) = \hat E^2(B)$.

  If $B/p$ is bounded below, then $B^{\wedge p}/p$ is bounded below
  and the remaining part of the proposition follows from
  Lemma~\ref{lemma:strong-convergence-homology}.
\end{proof}

\subsubsection{Multiplicative structure of $H_*^c(R_+(B))$}
\label{sec:multiplicative-structure-of-rplus}

For any spectrum $B$, let
\begin{equation}
  \label{eq:epsilon-bar}
  \bar\epsilon_B^H\: H^{tC_p}\wedge_H (H\wedge B) \longto R_+^H(H\wedge B)
\end{equation}
be the unique extension of the $H$-based Tate diagonal
$\epsilon^H_{B}\: H\wedge B\to R_+^H(H\wedge B)$ along $\eta^H$ to a
map of $H^{tC_p}$-modules.

For any $H$-module $M$ and any map of spectra $x\: S^q\to M$
representing a class in $\pi_q(M)$, let $\bar x \: H \wedge S^q \to M$
be the unique extension of~$x$ along $\eta: S\to H$ to a map of
$H$-modules.

\begin{definition}
  \label{dfn:omega}
  For any $H$-module $M$, define
  \begin{equation*}
    \omega^H_M\: H_*^c(S^{tC_p}) \otimes \pi_*(M) \longto \pi_*(R^H_+(M))
  \end{equation*}
  to be the homomorphism that makes the diagram
  \begin{equation*}
      \xymatrix@C15mm{
      H_*^c(S^{tC_p}) \otimes \pi_*(M)
      \ar[r]^-{\omega^H_M}
      & \pi_*(R^H_+(M))\\
      H_*^c(S^{tC_p})\otimes H_*(S^q)
      \ar[u]_{1\otimes \bar x_*}
      \ar[r]^-{(\bar\epsilon_{S^q}^H)_*}
      & H_*^c(R_+(S^q))
        \ar[u]_{R_+^H(\bar x)_*}
        }
  \end{equation*}
  commute for each $x\in \pi_q(M)$.  For our purposes, we are
  primarily interested in the situation where $M$ is an induced
  $H$-module, $M=H\wedge B$.  In this case, we will shorten the
  notation and write $\omega_B$ instead of $\omega_{H\wedge B}^H$.
\end{definition}
The homomorphism~$\omega_B$ will play a key role in our algebraic
modeling of~$H^c_*(R_+(B))$.  The linear dual of~$\omega_B$ will
correspond to the cohomological homomorphism~$\omega$ defined
in~\cite{LNR12}*{Thm.~5.2}.
\begin{lemma}
  The homomorphism $\omega^H_M$ is well-defined and is natural in the
  $H$-module $M$.
\end{lemma}
\begin{proof}
  Let $e: \pi_*(M)\to \pi_*(R_+^H(M))$ be the function
  that makes the diagram
  \begin{equation*}
    \xymatrix@C15mm{
      \pi_*(M)
      \ar[r]^-{e}
      & \pi_*(R^H_+(M))\\
      H_*(S^q)
      \ar[u]_{\bar x_*}
      \ar[r]^-{(\epsilon_{S^q}^H)_*}
      & H_*^c(R_+(S^q))
      \ar[u]_{R_+^H(\bar x)_*}
    }
  \end{equation*}
  commute for each $x\in\pi_q(M)$.  Then $\omega_M^H$ of
  Definition~\ref{dfn:omega} is the unique extension of~$e$ along
  $\eta^H_*$ as a morphism of $H_*^c(S^{tC_p})$-modules.  Thus, we
  must argue that~$e$ is a well-defined homomorphism, natural in the
  $H$-module $M$.

  To show that~$e$ is a homomorphism it suffices to show that
  $R_+^H(\bar x) + R_+^H(\bar y) = R_+^H(\overline{x+y})$ in the
  homotopy category of spectra, for arbitrary maps $x,y\: S^q\to M$.

  Generally, for any endofunctor $F$ on the homotopy category of
  $E$-modules, let
  $$
  F(X)\vee F(X) \overset\alpha\longto F(X\vee X)
  \overset\beta\longto  F(X)\vee F(X)
  $$
  be the canonical maps associated with the inclusions
  $i_s\: X\to X\vee X$ and projections $p_s\: X\vee X\to X$, for
  $s=1,2$.  For any $X$ we have that $\beta\circ \alpha = 1$.  If~$F$
  is the functor $M\mapsto R^H_+(M)$, and if $K$ is a finite
  CW-spectrum, then the domain and codomain of
  $\alpha_*\: \pi_*(R_+^H(H\wedge K)\vee R_+^H(H\wedge K)) \longto
  \pi_*(R_+^H(H\wedge (K \vee K)))$ have the same finite dimension
  over $\bF_p$ in each degree, and $\alpha_*$ is therefore an
  isomorphism and $\alpha$ and $\beta$ are (homotopy) inverses.

  Given any two morphisms $f,g\: H\wedge S^q\to M$ of $H$-modules,
  consider the following diagram:
  \begin{equation*}
    \label{eq:rplus-additive}
    \begin{aligned}
      \xymatrix@C4mm{
      &  R_+^H(H\wedge S^q)
        \ar[rrrr]^-{R_+^H(f+g)}
        \ar`l[ldd]`[dd]_{\Delta} [dd]
        \ar[d]^{R_+^H(\Delta)}
      &&&& R_+^H(M)\\
      & R_+^H(H\wedge S^q \vee H\wedge S^q)
        \ar@<-1mm>[d]_\beta
        \ar[rrrr]^-{R_+^H(f\vee g)}
      &&&&
        R_+^H(M \vee M)
           \ar[u]_{R_+^H(\nabla)}\\
      & R_+^H(H\wedge S^q)\vee R_+^H(H\wedge S^q)
        \ar@<-1mm>[u]_{\alpha}
      \ar[rrrr]^-{R_+^H(f)\vee R_+^H(g)}
      &&&& R_+^H(M) \vee R_+^H(M)\rlap{\ \ \ \ \ \,.}
        \ar[u]_\alpha
        \ar@<-1pt> `r/5pt[u] `[uu]_{\nabla} [uu]
        }
    \end{aligned}
  \end{equation*}
  The left-hand and right-hand parts of the diagram commute
  by universal properties, and the upper inner square commutes by the
  addition rule for maps of spectra.  The lower inner square involving
  the two vertical maps labeled $\alpha$ is easily seen to commute by
  restriction to the individual summands.  At the left-hand side, the
  maps $\alpha$ and~$\beta$ are inverses, and it follows that the
  perimeter of the diagram commutes.  We conclude that
  $R_+^H(f + g) = R_+^H(f) + R_+^H(g)$.

  Let $x,y\: S^q\to M$ represent classes in $\pi_q(M)$.  The
  commutativity of
  \begin{equation*}
    \label{eq:xbar-additive}
    \begin{aligned}
      \xymatrix@C8mm{
      H\wedge S^q
      \ar[r]^-{1\wedge \Delta}
      \ar[dr]_-\Delta
      & H\wedge (S^q\vee S^q)
        \ar[r]^-{1\wedge (x\vee y)}
        \ar[d]^\beta
      & H\wedge (M \vee M)
        \ar[r]^-{1\wedge \nabla}
      & H\wedge M\ar[d]\\
      & (H\wedge S^q) \vee (H\wedge S^q)
        \ar[r]^-{\bar x\vee \bar y}
      & M\vee M
        \ar[r]^-\nabla
      & M\,,
      }
    \end{aligned}
  \end{equation*}
  shows that $\overline{x+y} = \bar x + \bar y$.  In summary,
  $R^H_+(\overline{x+y}) = R^H_+(\bar x + \bar y) = R^H_+(\bar x) +
  R_+^H(\bar y)$, and it follows that~$e$, and thus $\omega^H_M$, is
  additive.

  Naturality with respect to $H$-module maps $f\: M\to M'$ follows
  readily from the definition of $\omega^H_M$ and the fact that
  $\overline{f\circ x} = f\circ \bar x$ in the homotopy category of
  $H$-modules, for any map $x\: S^q\to M$.
\end{proof}

Let $M$ be an $H$-module.  The functor sending~$M$ to
$H_*^c (S^{tC_p}) \otimes \pi_*(M)$ is lax symmetric monoidal, with
unit
\begin{equation}
  \label{eq:algebraic-rplus-monoidal-unit}
  \eta^H_*\: \bF_p \longto H_*^c(S^{tC_p}) \,,
\end{equation}
and pairing
\begin{multline}
  \label{eq:algebraic-rplus-monoidal-pairing}
  (\mu_{S,S*} \otimes \wedge) \circ (23)\: H_*^c (S^{tC_p}) \otimes \pi_*(M)\otimes
  H_*^c (S^{tC_p}) \otimes \pi_*(N)\\
  \longto
  H_*^c (S^{tC_p}) \otimes \pi_*(M\wedge_H N)\,,
\end{multline}
where $\wedge\: \pi_*(M)\otimes \pi_*(N)\to \pi_*(M\wedge_H N)$ is the
lax monoidal structure map for passing from $H$-modules to
homotopy groups.

From now on, for any two spectra~$B$ and~$B'$ we abbreviate
the name of the lax symmetric monoidal pairing
$$
R^H_+(H\wedge B)\wedge_H R^H_+(H\wedge B')\longto R_+^H(H\wedge B\wedge B')
$$
by letting
$\mu_{B,B'} = \mu^H_{H \wedge B^{\wedge p}, H \wedge {B'}^{\wedge
    p}}$.  Thus, the pairing $\mu_{S,S*}$ appearing
in~\eqref{eq:algebraic-rplus-monoidal-pairing} is the algebra product
on $H_*^c(S^{tC_p})$ discussed in Section
\ref{sec:continuous-homology-of-singer}.

\begin{lemma}\label{lemma:omega-is-monoidal}
  For $H$-modules $M$, the assignment $M\mapsto \omega^H_M$ is a
  monoidal natural transformation between the lax symmetric monoidal
  functors $M\mapsto H^c_*(S^{tC_p})\otimes \pi_*(M)$ and
  $M\mapsto \pi_*(R^H_+(M))$.

  In particular, for induced $H$-modules $M=H\wedge B$ and
  $M'=H\wedge B'$, the following diagram of graded $\bF_p$-vector spaces
  commutes
  \begin{equation}
    \label{eq:omega-multiplicative}
    \xymatrix@C20mm{
      H^c_*(S^{tC_p})\otimes H_*(B) \otimes H^c_*(S^{tC_p})\otimes H_*(B')
      \ar[d]^-{(\mu_{S,S*}\otimes \wedge)\circ (23)}
      \ar[r]^-{\omega_B\otimes \omega_{B'}}
      &
      H_*^c(R_+(B)) \otimes H_*^c(R_+(B'))
      \ar[d]^-{\mu_{B,B'*}}
      \\
      H^c_*(S^{tC_p})\otimes H_*(B\wedge B')
      \ar[r]_{\omega_{B\wedge B'}}
      &
      H_*^c(R_+(B\wedge B')) \,.
    }
  \end{equation}
\end{lemma}
\begin{proof}
  The unit $\eta^H\: H\longto H^{tC_p}\simeq R^H_+(H)$ of the lax
  symmetric monoidal functor $M\mapsto R^H_+(M)$ equals
  $\epsilon^H_S$, the $H$-based Tate diagonal for the sphere spectrum.
  Since $\omega_S = (\bar\epsilon_S^H)_*$, the fact that $\omega^H$
  preserves the monoidal unit amounts to the commutativity of the
  diagram
  \[
  \xymatrix@C15mm{
    H\wedge_H H
    \ar[r]^-{\simeq}
    \ar[d]^-{\epsilon_S^H\wedge 1}
    & H \ar[d]^-{\epsilon_S^H}\\
    R_+^H(H)\wedge_H H \ar[r]^-{\bar\epsilon_S^H}
    & R_+^H(H)\,.
    }
  \]
  Therefore, unitality of $\omega^H$ follows from the fact that
  $\epsilon_B^H$ is a monoidal natural transformation
  $H\wedge B \to R_+^H(H\wedge B)$.

  Let $M$ and $M'$ be $H$-modules, together with maps $x\:S^q\to M$
  and $x'\:S^{q'}\to M'$.  To show that $\omega^H$ is multiplicative,
  we must show that the middle square of the following diagram
  commutes:
  \begin{equation}
    \begin{aligned}
      \label{eq:omega-H-multiplicative}
      \xymatrix@C2mm{
      &\txt{
      $H_*^c(S^{tC_p})\otimes H_*(S^q)$\\
      $\otimes$\\
      $H_*^c(S^{tC_p})\otimes H_*(S^{q'})$}
      \ar[rrrr]^-{(\bar\epsilon_{S^q}^H)_* \otimes (\bar\epsilon_{S^{q'}}^H)_*}
      \ar[d]^-{1\otimes \bar{x}_*\otimes 1\otimes {\bar{x}'}_*}
      \ar`l[ld]`[ddd][ddd]
      &&&&
        H_*^c(R_+(S^q))\otimes H_*^c(R_+(S^{q'}))
        \ar[d]^{R^H_+(\bar{x})_*\otimes R^H_+(\bar{x}')_*}
        \ar`r[rd]`d[ddd]_{\mu_{S^q, S^{q'}*}} [ddd]
      \\
      &\txt{$H_*^c(S^{tC_p})\otimes \pi_*(M)$\\
      $\otimes$\\
      $H_*^c(S^{tC_p})\otimes \pi_*(M')$}
      \ar[rrrr]^-{\omega^H_M \otimes \omega^H_{M'}}
      \ar[d]^-{(\mu_{S,S*}\otimes \wedge)\circ (23)}
      &&&&
        \pi_*(R^H_+(M))\otimes \pi_*(R^H_+(M'))
        \ar[d]^{\mu^H_{M,M'*}} &
      \\
      &H_*^c(S^{tC_p})\otimes \pi_*(M\wedge_H M')
      \ar[rrrr]^-{\omega^H_{M\wedge_H M'}}
      &&&&
        \pi_*(R^H_+(M\wedge_H M'))\\
      &H_*^c(S^{tC_p})\otimes H_*(S^q\wedge S^{q'})
      \ar[u]_{1\otimes (\bar{x}\wedge \bar{x}')_*}
      \ar[rrrr]^-{(\bar\epsilon_{S^q\wedge S^{q'}}^H)_*}
      &&&&
        H_*^c(R_+(S^q\wedge S^{q'}))  \,.
        \ar[u]_{R^H_+(\bar{x}\wedge \bar{x}')_*}
        }
    \end{aligned}
  \end{equation}
  Here, the unlabeled left-hand morphism is the composite
  \begin{multline*}
    H_*^c(S^{tC_p})\otimes H_*(S^q)\otimes
    H_*^c(S^{tC_p})\otimes H_*(S^{q'})
    \overset{(23)}\longto\\
    H_*^c(S^{tC_p})\otimes H_*^c(S^{tC_p}) \otimes H_*(S^q)\otimes H_*(S^{q'})
    \overset{\mu_{S,S*}\otimes \wedge}\longto
    H_*^c(S^{tC_p}) \otimes H_*(S^q\wedge S^{q'}) \,.
  \end{multline*}
  The top square commutes by the definition of $\omega^H_M$
  and~$\omega^H_{M'}$.  The lower square commutes by the definition of
  $\omega^H_{M \wedge_H M'}$, after identifying
  $S^{q+q'}\cong S^{q}\wedge S^{q'}$.  It is clear from the setup that
  the left-hand sub-diagram commutes, and the right-hand sub-diagram
  commutes since $R_+^H(-)$ is a lax symmetric monoidal endofunctor on
  the homotopy category of $H$-modules.  The outer square is induced
  by the diagram
  \[
    \xymatrix{
      R_+^H(H)\wedge_H (H\wedge S^q) \wedge_H R_+^H(H)\wedge_H (H\wedge S^{q'})
      \ar[r]^-{\bar\epsilon_{S^q}^H\wedge \bar\epsilon_{S^{q'}}^H}
      \ar[d]^{(\mu_{S,S}\wedge 1\wedge 1)\circ (23)}
      &
      R_+^H(H\wedge S^q) \wedge_H R_+^H(H\wedge S^{q'})
      \ar[dd]^{\mu_{S^q,S^{q'}}}
      \\
      R_+^H(H) \wedge_H (H\wedge S^{q}) \wedge_H (H\wedge S^{q'})
      \ar[d]^{\simeq}
      \\
      R_+^H(H) \wedge_H (H\wedge S^{q}\wedge S^{q'})
      \ar[r]^-{\bar\epsilon^H_{S^q\wedge S^{q'}}}
      &
      R_+^H(H\wedge S^q\wedge S^{q'}) \,,
    }
  \]
  and commutes since $\epsilon^H$ (resp.~$\bar\epsilon^H$) is a
  monoidal natural transformation between $B\mapsto H\wedge B$
  (resp.~$B \mapsto R^H_+(H) \wedge B$) and
  $B\mapsto R^H_+(H\wedge B)$.

  Since~$x$ and~$x'$ were arbitrary, this proves that the middle square
  of diagram~\eqref{eq:omega-H-multiplicative} commutes, and the lemma
  follows.
\end{proof}

\begin{corollary}
  If $B$ is a ring spectrum, then $\omega_B$ is a morphism of graded
  $\bF_p$-algebras.\qed
\end{corollary}

\begin{lemma}
  \label{lemma:omega-suspensions}
  Let $B$ be a spectrum.  The homomorphism $\omega_B$ commutes with
  suspensions.  More precisely, for each $n \in \bZ$ there is a
  commutative square of graded $\bF_p$-vector spaces,
  \begin{equation}
    \begin{aligned}
      \label{eq:omega-suspensions}
      \xymatrix@C11mm{
      \Sigma^n \bigl(H_*^c(S^{tC_p})\otimes H_*(B)\bigr)
      \ar[r]^-{\Sigma^n\omega_B}
      \ar[d]^{\cong}
      & \Sigma^n H_*^c(R_+(B))
        \ar[d]^\cong\\
      H_*^c(S^{tC_p})\otimes H_*(B\wedge S^n)
      \ar[r]^-{\omega_{\Sigma^n\!B}}
      & H_*^c(R_+(B\wedge S^n)) \,,
        }
    \end{aligned}
  \end{equation}
  natural in~$B$.
\end{lemma}
\begin{proof}
  The left-hand isomorphism of \eqref{eq:omega-suspensions} is induced
  by the suspension isomorphism
  $\Sigma^n H_*(B) \cong H_*(B\wedge S^n)$, while the right-hand
  isomorphism is induced by the homotopy equivalence
  of Proposition~\ref{prop:diag-q} for $E=H$ and $q=n$.

  To show that \eqref{eq:omega-suspensions} commutes, it suffices by
  Definition~\ref{dfn:omega} to restrict to the case of $B=S^q$.  In
  this case, \eqref{eq:omega-suspensions} is obtained by passing to
  homotopy groups in the following diagram:
  \begin{equation}
    \label{eq:l3diag}
    \begin{aligned}
      \xymatrix{
      R^H_+(H)\wedge_H (H\wedge S^q) \wedge S^n
      \ar[r]^-{\bar\epsilon_{S^q}^H\wedge 1}
      \ar[dd]^\cong
      %\ar@/_5mm/[ddr]_{\bar\epsilon^H_{S^{q+n}}}
      & R^H_+(H\wedge S^q) \wedge S^n
        \ar[d]^{1\wedge \epsilon^H_{S^n}} \\
      & R^H_+(H\wedge S^q)\wedge_H R_+^H(H\wedge S^n)
        \ar[d]^{\mu_{S^q, S^n}}\\
      R^H_+(H)\wedge_H (H\wedge S^q\wedge S^n)
      \ar[r]_-{\bar\epsilon^H_{S^{q+n}}}
      & R_+^H(H\wedge S^{q+n})\,.
      }
    \end{aligned}
  \end{equation}
  Commutativity of~\eqref{eq:l3diag} is a consequence of~$\epsilon^H$
  being a monoidal natural transformation.
\end{proof}

\section{The homological $C_p$-Singer construction}\label{sec:cp-singer}

Recall from Subsection~\ref{sec:continuous-homology-of-singer} the
structure of $H^c_*(S^{tC_p}) \cong \hat H^{-*}(C_p; \bF_p)$ as a
graded commutative $\bF_p$-algebra.
\begin{definition}\label{dfn:tate-filtration-rplus}
  For a graded $\bF_p$-vector space $M_*$, let
  $$
  r_+(M_*) = H_*^c(S^{tC_p})\otimes M_*\,,
  $$
  and define the \emph{Tate filtration} of $r_+(M_*)$ to be the
  ascending Hausdorff filtration given by the span
  \begin{equation*}
    F_n r_+(M_*) = \< u^it^r \otimes x
    \mid \text{$-i-2r -|x|(p-1) \leq n$, $x \in M_*$} \>
  \end{equation*}
  for each $n\in \bZ$.

  The \emph{homological Singer construction} on $M_*$ is the
  completion
  $$
  R_+(M_*) = r_+(M_*)^\wedge
  $$
  of $r_+(M_*)$ with respect to the Tate filtration.
\end{definition}
Note that $r_+(H_*(B))$ is the domain of the homomorphism
$\omega_B$ defined in
Subsection~\ref{sec:multiplicative-structure-of-rplus}.

The following lemma is elementary, and applies, in particular, for
$M = \bF_p$.
\begin{lemma}\label{lemma:rbb-rplus}
  If $M_*$ is bounded above, then the Tate filtration on $r_+(M_*)$ is
  discrete in each degree and $r_+(M_*) = R_+(M_*)$.

  If $M_*$ is bounded below, then the Tate filtration is relatively bounded
  below.
  \qed
\end{lemma}

\begin{proposition}
  \label{prop:omega-B}
  Let $B$ be a spectrum such that $B/p$ is bounded below.
  The homomorphism $\omega_B\: r_+(H_*(B))\to H_*^c(R_+(B))$ is
  injective and strictly filtration-preserving.  The completion
  $$
  \omega^\wedge_B \: R_+(M_*) \overset{\cong}\longto
  H^c_*(R_+(B))
  $$
  of $\omega_B$ with respect to the Tate filtration is an isomorphism.

  For each $x\in H_q(B)$, the class
  $u^it^r\otimes x\in r_+(H_*(B))$ is mapped by $\omega_B$ to a
  class detected by a non-zero multiple of
  $u^it^{r+mq}\otimes x^{\otimes p}$ in the $C_p$-Tate spectral
  sequence converging strongly to $H_*^c(R_+(B))$.
  Here $m = (p-1)/2$ for $p$ odd, and $t^{r+mq}$ is to be
  interpreted as~$u^{2r+q}$ for $p=2$.
\end{proposition}
\begin{proof}
  The homological $C_p$-Tate spectral sequence
  \begin{equation}
    \label{eq:homtatess}
    \hat E^{\infty} = \hat E^{2} = \hat H^{-*}(C_p; H_*(B)^{\otimes p}) \Longrightarrow
    H_*^c(R_+(B))
  \end{equation}
  collapses and converges strongly by Proposition~\ref{prop:singer-tate-filtration},
  using the assumption that $B/p$ is bounded below.  By
  Proposition~\ref{prop:diag-q}, the induced homomorphism
  $$
  (\bar\epsilon_{S^q}^H)_*\: H_*^c(S^{tC_p})\otimes H_*(S^q)
  \longto H_*^c(R_+(S^q))
  $$
  is an isomorphism, and since the homological Tate spectral sequence
  converging to $H_*^c(R_+(S^q))$ is concentrated in
  bidegrees $(*,pq)$, the class
  $(\bar\epsilon^H_{S^q})_*(u^it^r\otimes \iota_q)$ must be detected by
  some non-zero multiple of $u^it^{r+mq}\otimes\iota_q^{\otimes p}$.
  It follows that
  $\omega_B(u^it^r\otimes x) = R_+^H(\bar x)_*
  (\bar\epsilon^H_{S^q})_*(u^it^r\otimes \iota_q)$ is detected by the same non-zero
  multiple of $u^it^{r+mq}\otimes x^{\otimes p}$ in the $C_p$-Tate
  spectral sequence.  This shows the last part of the proposition.

  The Tate filtration on~$r_+(M_*)$ is so defined that $\omega_B$ maps
  it to the Tate filtration on~$H^c_*(R_+(B))$, specified in
  Subsection~\ref{sec:greenlees-may-filtration}.  Moreover, $\omega_B$
  induces an isomorphism of associated graded vector spaces.
  Since~\eqref{eq:homtatess} converges strongly, it follows that
  $\omega_B$ induces an isomorphism after completion.
\end{proof}

\subsection{The right action of the Steenrod algebra on $H_*^c(R_+(B))$}
\label{sec:right-action-of-A-on-the-homological-singer-construction}
For any spectrum~$B$, recall that
$$
\Delta \: \Sigma D_{C_p}(B) \longto D_{C_p}(\Sigma B)
$$
is the map induced by the diagonal embedding
$S^1\to (S^1)^{\wedge p}$.  Its effect in mod~$p$ homology is
described in the proof of \cite{BMMS86}*{Lem.~II.5.6}, using a chain
level computation from \cite{May70}*{p.~166--167}.  The latter
computation yields the formula
\begin{equation}
  \label{eq:delta-in-homology}
  \Delta_* \Sigma (e_{j}\otimes x^{\otimes p}) =
  (-1)^{j+1}\alpha(q)\cdot e_{j-(p-1)}\otimes (\Sigma x)^{\otimes p}
\end{equation}
for $x \in H_q(B)$. As before, $\alpha(q) = - (-1)^{mq} \cdot m!$ is a
unit mod~$p$, which is to be interpreted as~$1$ when $p=2$.  Note that
the formula for $\Delta_*$ appearing in \cite{BMMS86}*{p.~47} is
mistakenly stated to hold for $x \in H_{q-1}(B)$, and therefore
implicitly differs by a sign~$(-1)^m$
from~\eqref{eq:delta-in-homology}.  Our formula is the one that
follows from~\cite{May70}.

For each $q\geq 0$, let
\begin{equation}
  \label{eq:delta-orbits}
  \Delta^q \: \Sigma^q D_{C_p}(B) \longto D_{C_p}(\Sigma^q B)
\end{equation}
be the $q$-fold iteration of $\Delta$, i.e., $\Delta^0 = 1$ and
$\Delta^q = \Delta\circ \Sigma\Delta^{q-1}$ for $q>0$.

\begin{lemma}\label{lemma:delta-maps-homology}
  Let $q\geq 0$.  For $B=S$, the homomorphism
  in mod~$p$ homology induced by \eqref{eq:delta-orbits} is given by
  $$
  \Delta^q_* \Sigma^q (e_{j}\otimes 1^{\otimes p}) =
  c_{q,j}\cdot e_{j-q(p-1)}\otimes \iota_q^{\otimes p}\,,
  $$
  where $\iota_q = \Sigma^q 1\in H_q(S^q)$ is the fundamental class, and
  \begin{equation}
    \label{eq:coefficient}
    c_{q,j} =
    \begin{cases}
      1 & \text{for $q \equiv 0$ mod $4$} \\
      (-1)^{j}m! & \text{for $q \equiv 1$ mod $4$}\\
      -1 & \text{for $q \equiv 2$ mod $4$}\\
      -(-1)^{j}m! & \text{for $q \equiv 3$ mod $4$} \,.
    \end{cases}
  \end{equation}
  In each case, $c_{q,j}$ is to be interpreted as~$1$ for $p=2$.
\end{lemma}
\begin{proof}
  By definition, $\Delta^0$ is the identity map, which implies the
  lemma in the case of $q=0$.

  For $q\geq 1$, we iterate~\eqref{eq:delta-in-homology} $q$ times,
  starting with $B = S$ and $x = 1 \in H_0(S)$, to obtain the formula
  \begin{equation}
  \Delta^q_* \Sigma^q (e_{j}\otimes 1^{\otimes p}) =\\
  (-1)^{q(j+1)}\alpha(0)\alpha(1)\cdots\alpha(q-1)\cdot e_{j-q(p-1)}\otimes \iota_q^{\otimes p}\,.
  \end{equation}
  We note that $\alpha(r)$ is two-periodic in~$r$, and that
  \begin{equation*}
    % \label{eq:alpha-alpha}
    \alpha(r)\alpha(r+1) = (-1)^m (m!)^2 \equiv (p-1)!
    \equiv -1~~~\text{mod $p$}\,.
  \end{equation*}
  Here, the middle congruence uses that $r~\equiv~-(p-r)$ mod~$p$, and
  the last congruence is provided by Wilson's theorem.  It follows
  that $c_{q,j} = (-1)^{q(j+1)}\alpha(0)\cdots\alpha(q-1)$ satisfies
  $c_{q+2,j}\equiv -c_{q,j}$ and is four-periodic in~$q$.  The lemma
  then follows for $q\geq 1$ since
  $c_{1,j}= (-1)^{j+1}\alpha(0) = (-1)^{j}m!$, and
  $c_{2,j}\equiv (-1)^{2(j+1)}\alpha(0)\alpha(1) \equiv -1$.
\end{proof}

When $n=1$, we identify the projection $R_+(B) \to R_+(B)[n]$ onto the
$n$-th Tate truncation with
$\partial(B)\: R_+(B) \to \Sigma D_{C_p}B$, the connecting map of the
Puppe sequence associated with the norm-restriction cofiber sequence
$$
(B^{\wedge p})_{hC_p}
\overset{N^h}\longto (B^{\wedge p})^{hC_p}
\overset{R^h}\longto R_+(B) \,.
$$
In Subsection~\ref{sec:continuous-homology-of-singer} we discussed the
case of $B=S$, and simply wrote $\partial$ for $\partial(S)$.  We now
extend that analysis, assuming first that $B/p$ is connective, and
then generalizing to the case of $B/p$ being bounded below.

\begin{lemma}
  \label{lemma:partial-omega-formula}
  Let $B$ be a spectrum such that $B/p$ is connective. The composite
  homomorphism
  $$
  r_+(H_*(B))
  \overset{\omega_B}\longto
  H_*^c(R_+(B))
  \overset{\partial(B)_*}\longto
  H_*(\Sigma D_{C_p}(B))
  $$
  is given for~$p$ odd by
  \begin{equation}
    \label{eq:partial-omega-formula}
    \partial(B)_*(\omega_B(u^it^r\otimes x)) =
    (-1)^{r+q}c_{q,1-i} \Sigma (e_{-1-i-2r-q(p-1)}\otimes x^{\otimes p})\,,
  \end{equation}
  up to multiplication by a fixed unit.  This fixed unit is the one appearing in Lemma~\ref{lemma:ut-images},
  and does not depend on $i$, $r$, $x$ or~$B$ in
  \eqref{eq:partial-omega-formula}.

  For $p=2$ the formula is
  $$
  \partial(B)_*(\omega_B(u^r \otimes x))
  = \Sigma(e_{-1-r-q} \otimes x^{\otimes 2}) \,.
  $$
\end{lemma}
\begin{proof}
  Let $x\in H_q(B)$ be represented by $x\: S^q\to H\wedge B$,
  for~$q\geq 0$, and consider the following commutative diagram:
  \begin{equation}
    \label{eq:partial-naturality}
    \begin{aligned}
      \xymatrix@C17mm{
      H_*^c(R_+(S))\otimes H_*(B)
        \ar[r]^-{\omega_B}
      & H_*^c(R_+(B))
         \ar@(r, u)[rdd]^{\partial(B)_*} \\
      H_*^c(R_+(S))\otimes H_*(S^q)
      \ar[u]_{1\otimes \bar x_*}
      \ar[r]^-{(\bar\epsilon_{S^q}^H)_* = (\Delta^H_{S^q})_*}
      \ar[d]^{\partial(S)_*\otimes 1}
      & H_*^c(R_+(S^q))
         \ar[u]_{R_+^H(\bar x)_*}
         \ar[d]^{\partial(S^q)_*}
      \\
      H_*(\Sigma D_{C_p}(S))\otimes H_*(S^q)
      \ar[r]^-{\Delta^H_{S^q}[1]_* }
      & H_*(\Sigma D_{C_p}(S^q))
         \ar[r]^-{\Sigma D^H_{C_p}(\bar x)_*}
      & H_*(\Sigma D_{C_p}(B))  \,.
        }
    \end{aligned}
  \end{equation}
  Here, the top square equals the diagram of Definition~\ref{dfn:omega}.
  Moreover, since~$q$ is non-negative, Lemma~\ref{lemma:delta-epsilon}
  applies to the case of $K=S^q$ and $B=S$ and implies that
  $\bar\epsilon_{S^q}^H\simeq \Delta^H_{S^q}$.  The latter map was
  explicitly defined by~\eqref{eq:delta-finite-complex} and respects
  the Tate filtration, thus producing a well-defined map
  $\Delta^H_{S^q}[1]\: (\Sigma D_{C_p}(S))\wedge S^q\to \Sigma
  D_{C_p}(S^q)$ making the lower square
  of~\eqref{eq:partial-naturality} commute.

  Explicitly, $\Delta^H_{S^q}[1]$ is the composite map
  $$
  (\Sigma D_{C_p}(S)) \wedge S^q =
  D_{C_p}(S) \wedge S^1\wedge S^q
  \overset{(23)}\longto
  D_{C_p}(S) \wedge S^q\wedge S^1
  \overset{\Sigma \Delta^q}\longto D_{C_p}(S^q) \wedge S^1 \,,
  $$
  where $\Delta^q$ is the map \eqref{eq:delta-orbits} induced by the
  diagonal embedding $S^q\to (S^q)^{\wedge p}$.  By Lemma
  \ref{lemma:delta-maps-homology} we know the effect of $\Delta^q$ in
  mod~$p$ homology, and we get
  \begin{equation}
    \label{eq:twist-diagonal}
    \Delta^H_{S^q}[1]_*((\Sigma e_{j})\otimes \iota_q)
    = (-1)^{q}c_{q,j}  \Sigma (e_{j-q(p-1)}\otimes \iota_q^{\otimes p})\,.
  \end{equation}

  By Lemma \ref{lemma:ut-images}, the homomorphism $\partial(S)_*$
  is given by the formula
  \begin{equation}
    \label{eq:partial-normalized}
    \partial(S)_*(u^it^r) = (-1)^r\Sigma e_{-1-i-2r}\,,
  \end{equation}
  up to multiplication by a fixed unit.  By the commutativity of
  \eqref{eq:partial-naturality}, together with
  \eqref{eq:twist-diagonal} and \eqref{eq:partial-normalized}, we
  deduce that
  \begin{equation}
    \partial(B)_*(\omega_B(u^it^r\otimes x)) =
    (-1)^{r+q}c_{q,1-i} \Sigma (e_{-1-i-2r-q(p-1)}\otimes x^{\otimes p})\,,
  \end{equation}
  up to multiplication by that fixed unit.  Here we
  have used that the coefficient $c_{q,j}$ is two-periodic in~$j$ to
  make the simplification $c_{q,-1-i-2r} = c_{q,1-i}$.
\end{proof}

For each $n\geq 0$, let
$h(n)\: H_*(\Sigma D_{C_p}(B))\to H_*(\Sigma D_{C_p}(B))$ be the
degree~$-2p^n$ homomorphism mapping $\Sigma(e_j\otimes x^{\otimes p})$ to
$-\Sigma(e_{j-2p^n}\otimes x^{\otimes p})$.
\begin{lemma}\label{lemma:singer-coordinates}
  Let $B$ be a spectrum such that $B/p$ is connective.  For each
  non-negative integer $n$, the homomorphism $h(n)$ is $\sA(n)$-linear
  and there is a natural, injective and $\sA(n)$-linear homomorphism
  $$
  \phi(n)\: H_*^c(R_+(B)) \longto \lim_{h(n)} H_*(\Sigma D_{C_p}(B))\,.
  $$
  For each $\ell\geq 0$, let
  $\phi(n,\ell)\: H_*^c(R_+(B))\to H_*(\Sigma D_{C_p}(B))$ be the
  composition of $\phi(n)$ followed by the projection to the $\ell$-th
  stage of the limit system.  Then $\phi(n,\ell)$ maps
  $\omega_B(u^it^r\otimes x) \in H_*^c(R_+(B))$ to
  \begin{equation}
    \label{eq:phin-projection-formula}
    (-1)^{r+q+\ell} c_{q,1-i} \Sigma(e_{-1-i-2r-q(p-1)+2\ell p^n}\otimes
    x^{\otimes p})
    \in H_*(\Sigma D_{C_p}(B))
  \end{equation}
  for $p$ odd.  When $p=2$, $\phi(n, \ell)$ maps
  $\omega_B(u^r \otimes x) \in H^c_*(R_+(B))$ to
  $$
  \Sigma(e_{-1-r-q+2^{n+1}\ell} \otimes x^{\otimes 2})
  \in H_*(\Sigma D_{C_2}(B)) \,.
  $$
  As before, $e_j=0$ when~$j$ is negative.
\end{lemma}
\begin{proof}
  For a fixed~$n\geq 0$, consider the diagram
  \begin{equation}
    \label{eq:omega-partial}
    \begin{aligned}
      \xymatrix@C18mm{
      r_+(H_*(B)) \ar[r]^{\omega_B}
      \ar[d]^{\cdot t^{p^{n}}}_\cong
      &
        H_*^c(R_+(B))
        \ar[r]^{\partial(B)_*}
       \ar[d]^{\cdot t^{p^{n}}}_\cong
      & H_*(\Sigma D_{C_p}(B))
        \ar[d]^{h(n)}
      \\
      r_+(H_*(B)) \ar[r]^{\omega_B}
      & H_*^c(R_+(B))
        \ar[r]^{\partial(B)_*}
      & H_*(\Sigma D_{C_p}(B))
        \,.
        }
    \end{aligned}
  \end{equation}
  It follows from Lemma~\ref{lemma:partial-omega-formula} that the
  outer rectangle commutes, and the left-hand square commutes since
  $\omega_B$, by construction, is a morphism of
  $H_*^c(R_+(S))$-modules.  The discrete filtration of
  $H_*(\Sigma D_{C_p}(B))$ is complete and Hausdorff, and the
  homomorphism $\partial(B)_*$ is a morphism of filtered graded
  $\bF_p$-vector spaces.  Moreover, the Tate filtration of
  $H_*^c(R_+(B))$ is also complete and Hausdorff by
  Proposition~\ref{prop:singer-tate-filtration}.  Passing to the
  completion, $\omega_B$ becomes an isomorphism by
  Proposition~\ref{prop:omega-B}, and we conclude that the right-hand
  square of~\eqref{eq:omega-partial} commutes since it is already
  complete.

  It follows that $H^c_*(R_+(B))$ maps to the $\ell$-th stage in the
  tower defined by iterations of $h(n)$, sending
  $\omega_B(u^it^r\otimes x)$ to
  $\partial(B)_*(\omega_B(u^it^{r-\ell p^n}\otimes x))$.  The formula
  \eqref{eq:phin-projection-formula} follows
  from~\eqref{eq:partial-omega-formula} of
  Lemma~\ref{lemma:partial-omega-formula}, up to multiplication by a
  fixed unit.  We instead define $\phi(n)$ by normalizing with respect to this
  ambiguity, so that~\eqref{eq:phin-projection-formula} holds on the
  nose.  The kernel of $\partial(B)_*$ consists of the subspace
  $F_0 H_*^c(R_+(B))$ of elements of Tate filtration~$\leq 0$.  Since
  the Tate filtration of $H_*^c(R_+(B))$ is exhaustive and Hausdorff,
  every non-zero element has a well-defined Tate filtration.  Thus,
  any non-zero element multiplied by $t^{-\ell p^n}$ has positive Tate
  filtration for all sufficiently large~$\ell$, and it follows that
  that $\phi(n)$ is injective.

  Suppose $p$ is odd.  By Lucas' theorem, the homomorphism $h(n)$
  commutes with the action of $\sA(n)$ on $H_*(\Sigma D_{C_p}B)$
  specified by
  \eqref{eq:ext-power-steenrod-ops}--\eqref{eq:ext-power-bockstein}.
  Multiplication by $t^{p^n}$ is an $\sA(n)$-linear automorphism of
  $H_*^c(R_+(S))$ by Lemma~\ref{lemma:singer-an-primitives}, and
  therefore an $\sA(n)$-linear automorphism of $H_*^c(R_+(B))$ via the
  left $H_*^c(R_+(S))$-module structure.  Thus, the right-hand square
  of~\eqref{eq:omega-partial} consists of morphisms of right
  $\sA(n)$-modules, and the limit over $h(n)$ is therefore a
  right $\sA(n)$-module.

  For $p=2$, the argument in the previous paragraph can be repeated
  with $t$ replaced by $u^2$, and using
  \eqref{eq:ext-power-squaringops} instead of
  \eqref{eq:ext-power-steenrod-ops}--\eqref{eq:ext-power-bockstein}.
\end{proof}

Let~$M_*$ be a right $\sA$-module, and let $x\in M_q$.  We now show
how that we can define a right action of the Steenrod algebra on
$r_+(M_*)$ by the following formulas:

\begin{align}\label{eq:steenrod-singer-tr}
  \P^s_*(t^r \otimes x)
  & = \sum_k \binom{-1-r-s(p-1)}{s-pk}
  t^{r+(s-k)(p-1)} \otimes \P^k_*(x) \\
  & \qquad -\sum_k \binom{-1-r-s(p-1)}{s-pk-1} \nonumber
    u t^{-1+r+(s-k)(p-1)} \otimes \P^k_*\beta_*(x)\,, \\
    \label{eq:steenrod-singer-utr}
  \P^s_*(u t^r \otimes x)
  &= \sum_k \binom{-1-r-s(p-1)}{s-pk}
  u t^{r+(s-k)(p-1)} \otimes \P^k_*(x)\\
\intertext{and}
\label{eq:steenrod-singer-beta}
  \beta_*(u^i t^r \otimes x) &=
  \begin{cases}
    0 & \text{for $i=0$} \\
    t^{r+1} \otimes x & \text{for $i=1$}
  \end{cases}\\
\intertext{for $p>2$, and}
\label{eq:steenrod-singer-squares}
  \Sq_*^s(u^r \otimes x) &= \sum_{k} \binom{-1-r-s}{s-2k}
  u^{r+s-k} \otimes \Sq_*^k(x)
\end{align}
for $p=2$.  It is elementary to check that these operations respect
the Tate filtration of $r_+(M_*)$.  We warn the reader that
\eqref{eq:steenrod-singer-squares} for $\Sq^{2s}_*(u^{2r+1}\otimes x)$
is not the formula obtained from \eqref{eq:steenrod-singer-utr} by
letting $p=2$ and replacing $\P^s_*$ by $\Sq^{2s}_*$.

\begin{lemma} \label{lemma:omega-b-is-a-linear} Let $B$ be a spectrum
  such that $B/p$ is connective.  For each $n\geq 0$ and $\ell\geq 0$, the composite
  homomorphism
  $$
  \phi(n, \ell) \circ \omega_B\:
  r_+(H_*(B))
  \longto
  H_*(\Sigma D_{C_p}(B))
  $$
  commutes with the action of~$\sA(n)$.
\end{lemma}
\begin{proof}
  Let $p>2$ and $m = (p-1)/2$.  By combining
  \eqref{eq:phin-projection-formula}, \eqref{eq:ext-power-bockstein}
  and the fact that $\beta_*$ commutes with (right) homology suspension, we
  obtain that
  \begin{align}
    \nonumber
    \beta_*(\phi(n,\ell)\circ\omega_B)(u^it^r\otimes x)
    &= \beta_*((-1)^{r+q+\ell}c_{q,1-i} \Sigma (e_{a}\otimes x^{\otimes p}))\\
    &=
      \begin{cases}
        0 & \text{for $i=0$} \\
        (-1)^{r+\ell+1}c_{q,0} \Sigma (e_{a-1}\otimes x^{\otimes p})
        & \text{for $i=1$} \,,
      \end{cases}\label{eq:partial-beta}
  \end{align}
  where $a := -1 - i - 2r - q(p-1) + 2\ell p^n$.  On the other hand, it
  follows from \eqref{eq:steenrod-singer-beta}
  and \eqref{eq:phin-projection-formula} that
  \begin{align}
    (\phi(n,\ell)\circ \omega_B)( \beta_*(u^it^r \otimes x ) )
      &=\begin{cases}
        0 & \text{for $i=0$}\\
        (-1)^{r+1+q+\ell} c_{q,1} \Sigma (e_{a-1}\otimes x^{\otimes p})
        & \text{for $i=1$}  \,.
      \end{cases} \label{eq:beta-partial}
  \end{align}
  Using the identity $c_{q,1}/c_{q,0} = (-1)^q$, which
  can be deduced from equation \eqref{eq:coefficient}, we see that
  \eqref{eq:partial-beta} equals \eqref{eq:beta-partial}.
  We conclude that $\omega_B$ commutes with $\beta_*$.

  For the remainder of the proof we assume that $\ell\geq 0$ and
  $s<p^n$.  We repeatedly use that
  $$
  \binom{a+p^n}{b} \equiv \binom{a}{b} \mod p
  $$
  holds for all integers $a,b$ if $b<p^n$.  Furthermore, we
  still follow the convention that $e_j$ should be read as~$0$ if
  $j<0$.

  Let $i=0$ so that $a=-1-2r - q(p-1) + 2\ell p^n$ is odd.
  By~\eqref{eq:ext-power-steenrod-ops}, the operation $\P^s_*$ sends
  $(\phi(n,\ell)\circ \omega_B) (t^r\otimes x)= (-1)^{r+q+\ell}c_{q,1}
  \Sigma (e_{a}\otimes x^{\otimes p})$ to
  \begin{gather}
    (-1)^{r+q+\ell}c_{q,1}
    % Lines 20, 21: If we need to break these equations over more lines,
    % then between the cdot and the sum may be a good place.
      \sum_{k}\binom{-1-r-qm + \ell p^n + m(q-2s)}{s-pk}
      \Sigma (e_{a-2(s-pk)(p-1)}\otimes \P^k_*(x)^{\otimes p})\nonumber\\
    {-} ({-}1)^{r{+}\ell}c_{q,1}\alpha(q)
    % Lines 20, 21: If we need to break these equations over more lines,
    % then between the cdot and the sum may be a good place.
      \sum_{k}\binom{{-}r{-}qm {+} \ell p^n {+} m(q{-}2s) {-}1}{s-pk-1}
      \Sigma (e_{a{+}p{-}2(s{-}pk)(p{-}1)}\otimes \P^k_*\beta_*(x)^{\otimes p})\nonumber\\
      \intertext{$=$}
      (-1)^{r+q+\ell}c_{q,1}  \sum_{k}\binom{-1-r-s(p-1)}{s-pk}
      \Sigma (e_{a-2(s-pk)(p-1)}\otimes \P^k_*(x)^{\otimes p}) \label{eq:11} \\
    - (-1)^{r+\ell}c_{q,1}\alpha(q)  \sum_{k}\binom{-1-r-s(p-1)}{s-pk-1}
      \Sigma (e_{a+p-2(s-pk)(p-1)}\otimes \P^k_*\beta_*(x)^{\otimes p}) \label{eq:12} \,.
  \end{gather}
  On the other hand,
  \begin{align*}
    \P^s_*(t^r\otimes x)  &=\sum_{k}\binom{-1-r-s(p-1)}{s-pk}t^{r+(s-k)(p-1)}\otimes P_*^k(x)\\
    &\quad -\sum_{k}\binom{-1-r-s(p-1)}{s-pk-1}ut^{-1+r+(s-k)(p-1)}\otimes P_*^k\beta_*(x)
  \end{align*}
  is mapped by $\phi(n,\ell)\circ \omega_B$ to
  \begin{align*}
    &\sum_{k}\binom{-1-r-s(p-1)}{s-pk} (-1)^{r+q+\ell} c_{q-2k(p-1),1} \Sigma (e_{a-2(s-pk)(p-1)}\otimes P_*^k(x)^{\otimes p})\\
    -&\sum_{k}\binom{{-}1{-}r{-}s(p{-}1)}{s-pk-1} (-1)^{r+q+\ell} c_{q{-}1{-}2k(p{-}1),0} \Sigma (e_{a{+}p{-}2(s{-}pk)(p{-}1)}\otimes P_*^k\beta_*(x)^{\otimes p})\,.
  \end{align*}
  Since $c_{q+4,j} = c_{q,j}$, the coefficients $c$ in these sums do not
  depend on the summation index $k$, and the above expression simplifies to
  \begin{align}
    (-1)^{r+q+\ell} c_{q,1} \sum_{k}\binom{-1-r-s(p-1)}{s-pk}  \Sigma (e_{a-2(s-pk)(p-1)}\otimes P_*^k(x)^{\otimes p}) \label{eq:21} \\
    -(-1)^{r+q+\ell} c_{q-1,0} \sum_{k}\binom{-1-r-s(p-1)}{s-pk-1}
    \Sigma (e_{a+p-2(s-pk)(p-1)}\otimes P_*^k\beta_*(x)^{\otimes
    p}) \label{eq:22}\,.
  \end{align}
  The sum \eqref{eq:11} equals the sum \eqref{eq:21}.  The fact that
  \eqref{eq:12} is congruent to \eqref{eq:22} mod~$p$ follows from
  showing that the coefficients of the two sums are congruent, i.e.,
  $$
  -(-1)^{r+\ell} c_{q,1} \alpha(q) \equiv -(-1)^{r+q+\ell} c_{q-1,0}\,,
  $$
  which can be done by a direct computation from the definitions of
  $c_{q,j}$ and $\alpha(q)$, using that $(m!)^2 \equiv -(-1)^m$.

  Now, let $i=1$ so that $a=-2-2r-q(p-1)+2\ell p^n$ is even.  By
  \eqref{eq:ext-power-steenrod-ops}, the operation $\P^s_*$ sends
  $(\phi(n,\ell)\circ \omega_B) (ut^r\otimes x)=
  (-1)^{r+q+\ell}c_{q,0} \Sigma (e_{a}\otimes x^{\otimes p})$ to
  \begin{align}
    \nonumber
    (-1)^{r+q+\ell}c_{q,0} &\sum_{k}\binom{-1-r-qm + \ell p^n + m(q-2s)}{s-pk}
      \Sigma (e_{a-2(s-pk)(p-1)}\otimes \P^k_*(x)^{\otimes p})\\
    = (-1)^{r+q+\ell}c_{q,0} &\sum_{k}\binom{-1-r-s(p-1)}{s-pk}
      \Sigma (e_{a-2(s-pk)(p-1)}\otimes \P^k_*(x)^{\otimes p}) \label{eq:11-i1}\,.
  \end{align}
  On the other hand,
  \begin{equation*}
    \P^s_*(ut^r\otimes x) = \sum_{k}\binom{-1-r-s(p-1)}{s-pk}ut^{r+(s-k)(p-1)}\otimes P_*^k(x)
  \end{equation*}
  is mapped by $\phi(n,\ell)\circ \omega_B$ to
  \begin{equation}\label{eq:21-i1}
    \sum_{k}\binom{-1-r-s(p-1)}{s-pk} (-1)^{r+q+\ell} c_{q-2k(p-1),0} \Sigma (e_{a-2(s-pk)(p-1)}\otimes P_*^k(x)^{\otimes p})\,.
  \end{equation}
  Since $c_{q+4,j} = c_{q,j}$, the coefficients $c$ in
  \eqref{eq:21-i1} do not depend on the summation index $k$, and we
  see that \eqref{eq:21-i1} equals \eqref{eq:11-i1}.

  We conclude that $\phi(n,\ell)\circ \omega_B$ commutes with the
  action of $\beta_*$ and $\P^s_*$, for all $s<p^n$, and the lemma
  follows for the case $p>2$.

  The case of $p=2$ is easier.  When $s < 2^{n+1}$, both
  $\Sq^s_* \circ \phi(n, \ell) \circ \omega_B$ and
  $\phi(n, \ell) \circ \omega_B \circ \Sq^s_*$ map $u^r \otimes x$ to
  $$
  \sum_k \binom{-1-r-s}{s-2k} \Sigma(e_{-1-r-q+2^{n+1}\ell-s+2k}
  \otimes \Sq^k_*(x)^{\otimes 2}) \,.
  $$
\end{proof}

\begin{proposition} \label{prop:omega-b-is-a-linear} Let $B$ be a
  spectrum such that $B/p$ is bounded below.  Then
  $\omega_B\: r_+(H_*(B))\to H_*^c(R_+(B))$ is a morphism of
  relatively bounded below filtered right $\sA$-modules.
\end{proposition}
\begin{proof}
  Assume first that $B/p$ is connective.  For each integer $n\geq 0$, the
  injective composition
  $$
  r_+(H_*(B))
  \overset{\omega_B}\longto
  H_*^c(R_+(B))
  \overset{\phi(n)}\longto
  \lim_{h(n)} H_*(\Sigma D_{C_p}(B))
  $$
  commutes with the action of $\sA(n)$ by
  Lemma~\ref{lemma:omega-b-is-a-linear}.  It follows that
  \eqref{eq:steenrod-singer-tr}--\eqref{eq:steenrod-singer-squares}
  satisfy the Adem relations and that $\omega_B$ is $\sA(n)$-linear
  for each $n$, hence $\sA$-linear.

  The formulas
  \eqref{eq:steenrod-singer-tr}--\eqref{eq:steenrod-singer-squares}
  specifying the right $\sA$-action on $r_+(H_*(B))$ commute
  with suspensions, i.e.,
  $\Sigma \P^s_*(u^it^r\otimes x) = \P^s_*(u^it^r\otimes \Sigma x)$
  and
  $\Sigma \beta_*(u^it^r\otimes x) = \beta_*(u^it^r\otimes \Sigma x)$.
  It follows that the left-hand vertical isomorphism
  of~\eqref{eq:omega-suspensions} is an isomorphism of right
  $\sA$-modules.  Moreover, as noted at the beginning of
  Section~\ref{sec:continuous-mod-p-homology-of-r-plus}, the
  right-hand vertical isomorphism of~\eqref{eq:omega-suspensions} is
  also $\sA$-linear.  Thus Lemma~\ref{lemma:omega-suspensions} implies
  that $\omega_B$ is $\sA$-linear if and only if $\omega_{\Sigma^m B}$
  is $\sA$-linear, for any integer $m$.  Therefore, the proposition
  follows for bounded below $B/p$ by reduction to the connective case.
\end{proof}

Proposition~\ref{prop:omega-b-is-a-linear} says in particular that
\eqref{eq:steenrod-singer-tr}--\eqref{eq:steenrod-singer-squares} give
$r_+(M_*)$, and its completion $R_+(M_*)$, the structure of a
filtered right $\sA$-module when $M_*=H_*(B)$ for some spectrum $B$
such that $B/p$ is bounded below.  We end this subsection by noting that
we do not have to restrict attention to $\sA$-modules $M_*$ that arise as
the homology of spectra.

\begin{proposition}
  Let $M_*$ be a right $\sA$-module.  The action of the Steenrod
  operations specified by
  \eqref{eq:steenrod-singer-tr}--\eqref{eq:steenrod-singer-squares}
  define a filtered right $\sA$-module structure on $r_+(M_*)$.
\end{proposition}
\begin{proof}
  Let $M_*$ be a right $\sA$-module, and assume that $p=2$.  For any
  integer $N$, write $\tau_{*\geq N}M_*$ for the quotient module
  $M_*/M_{*<N}$.

  Let $x\in M_q$.  According to \eqref{eq:steenrod-singer-squares},
  the class $\Sq^s_*(u^r\otimes x)$ is contained in the vector
  subspace $r_+(M_{*\geq q-[s/2]})$.  Therefore, the Adem relation
  for $\Sq^b_* \Sq^a_* (u^r \otimes x)$ holds in $r_+(M_*)$ if
  and only if it holds in $r_+(\tau_{*\geq q-k} M_*)$ for
  $k=[a/2]+[b/2]$.

  Let $F_* = \Sigma^q\sA$ be the $q$-th suspension of the free right
  $\sA$-module on a single generator.  As a graded right $\sA$-module,
  $F_*$ is bounded above, with generator in top degree~$q$.  The
  $\sA$-linear homomorphism $F_* \to \tau_{*\geq q-k}M_*$ sending
  $\Sigma^q 1$ to $x$ factors over $\tau_{*\geq q-k}F_*$.  Let
  $X=D(\text{sk}_{k} H\bF_2)$ be the Spanier--Whitehead dual of the
  $k$-skeleton of the mod~$2$ Eilenberg--MacLane spectrum.  Then
  $H_*(\Sigma^qX)\cong \tau_{*\geq q-k}F_*$.  Moreover, the Adem relation
  for $\Sq^b_* \Sq^a_* (u^r \otimes \Sigma^q 1)$ holds in
  $r_+( H_*(\Sigma^qX))$ by
  Proposition~\ref{prop:omega-b-is-a-linear}, and then, by naturality,
  also in $r_+(\tau_{*\geq q-k}M_*)$.

  The argument for $p>2$ is similar.
\end{proof}

\subsection{Multiplicative structure}
Recall from Subsection~\ref{sec:e-based-singer} that the assignment
$B\mapsto R_+^H(H\wedge B)$ is a lax symmetric monoidal functor.
Passing to homotopy groups, we get that the assignment
$B\mapsto H_*^c(R_+(B))$ is a lax symmetric monoidal functor from the
homotopy category of spectra to the category
filtered right $\sA$-modules.

When restricted to spectra $B$ with $B/p$ bounded below,
$H_*^c(R_+(-)$ takes values in the symmetric monoidal category
$(\sC, \otimes)$ of rbb filtered right $\sA$-modules.  Furthermore,
$H_*^c(R_+(B))$ is complete Hausdorff by
Proposition~\ref{prop:singer-tate-filtration} so we can also think of
$B\mapsto H_*^c(R_+(B))$ as a functor into the symmetric monoidal
category $(\sC^\wedge, \ctensor)$ of rbb complete right $\sA$-modules,
or, equivalently, the symmetric monoidal category
$(\sD^\wedge, \ctensor)$ of rbb complete left $\sA_*$-comodules.

Likewise, the assignment $B\mapsto r_+(H_*(B))$ is a lax symmetric
monoidal functor from the homotopy category of spectra $B$ such that
$B/p$ is bounded below to $(\rbbfilgrfpmod, \otimes)$, with monoidal
structure maps given by
\eqref{eq:algebraic-rplus-monoidal-unit}--\eqref{eq:algebraic-rplus-monoidal-pairing}.
By composing with the completion, resp. completion followed by the
forgetful functor~ \eqref{eq:completion-and-dimenticare}, we can also
regard $B\mapsto R_+(H_*(B))$ as a lax symmetric monoidal functor with
values in $(\rbbfilgrfpmod^\wedge, \ctensor)$, resp.
$(\rbbfilgrfpmod, \otimes)$.

Taking the $\sA$-module structure into account, the following
proposition is a variant of Lemma~\ref{lemma:omega-is-monoidal}.
\begin{proposition}\label{prop:r-plus-is-monoidal}
  When restricted to spectra $B$ with $B/p$ bounded below, the functor
  $B\mapsto r_+ (H_*(B))$ is lax symmetric monoidal with values
  in the category $(\sC, \otimes)$ of rbb filtered right $\sA$-modules.
  The natural transformation
  $\omega$ is a monoidal natural transformation with components
  $\omega_B\: r_+(H_*(B)) \to H_*^c(R_+(B))$.

  In particular, the following diagram of rbb filtered right
  $\sA$-modules commutes for all spectra $B$ and $B'$ such that $B/p$
  and $B'/p$ are bounded below:
  \begin{equation}
    \begin{aligned}
      \label{eq:omega-multiplicative-and-A-linear}
      \xymatrix@C20mm{
      r_+(H_*(B))\otimes r_+(H_*(B'))
      \ar[d]^-{\mu}
      \ar[r]^-{\omega_B\otimes \omega_{B'}}
      &
      H_*^c(R_+(B)) \otimes H_*^c(R_+(B'))
        \ar[d]^-{\mu_{B,B'*}}\\
      r_+ (H_*(B\wedge B'))
        \ar[r]^-{\omega_{B\wedge B'}}
      & H_*^c(R_+(B\wedge B')) \,.
        }
    \end{aligned}
  \end{equation}
  The pairing $\mu$ sends $u^it^r\otimes x \otimes u^jt^s\otimes y$ to
  $(-1)^{|x|j}u^{i+j}t^{r+s}\otimes (x\wedge y)$, and its domain has
  the diagonal action of $\sA$.

  Furthermore, the functor $B\mapsto R_+(H_*(B))$ is lax symmetric
  monoidal with values in the category $(\sC^\wedge, \ctensor)$ of rbb
  complete right $\sA$-modules, and the completion of $\omega$ is a
  monoidal natural transformation.  Composing with the forgetful
  functor, $B\mapsto R_+(H_*(B))$ can also be viewed as a lax symmetric
  monoidal functor into $(\sC, \otimes)$.

  In particular, diagram~\eqref{eq:omega-multiplicative-and-A-linear}
  commutes with $r_+(-)$ replaced by $R_+(-)$ and, possibly,
  $\otimes$ replaced by $\ctensor$.
\end{proposition}
\begin{proof}
  We must show that the monoidal structure maps of
  $B\mapsto r_+(H_*(B))$ are morphisms of right $\sA$-modules.   This
  is clear for the monoidal unit~\eqref{eq:algebraic-rplus-monoidal-unit}
  since it is induced by the composite map of spectra
  $$
  \epsilon_S^H\: H\wedge S
  \overset{1\wedge \epsilon_S}\longto
  H\wedge R_+(S)
  \overset{\kappa}\longto
  (H\wedge S^{\wedge p})^{tC_p} \simeq R^H_+(H\wedge S)\,,
  $$
  where both maps induce $\sA$-linear homomorphisms after passing to
  homotopy.

  Next, note that as a diagram of graded $\bF_p$-vector spaces,
  diagram~\eqref{eq:omega-multiplicative-and-A-linear} equals the
  diagram~\eqref{eq:omega-multiplicative} of
  Lemma~\ref{lemma:omega-is-monoidal}, and is therefore commutative.

  By hypothesis, $B/p$ and $B'/p$ are bounded below, and it follows
  that also $(B\wedge B')/p$ is bounded below.  Thus $\omega_B$,
  $\omega_{B'}$ and $\omega_{B\wedge B'}$ are injective and
  $\sA$-linear by Proposition~\ref{prop:omega-B} and
  Proposition~\ref{prop:omega-b-is-a-linear}, and $\sA$-linearity
  of~$\mu$ follows therefore from $\sA$-linearity of $\mu_{B,B'*}$.

  The remaining part of the proposition follows from the fact that the
  completion and forgetful functors in~\eqref{eq:amods-and-acomods}
  are monoidal.
\end{proof}

\begin{remark}\label{rem:vandermonde}
One can give a purely algebraic proof that $M_*\mapsto
r_+(M_*)$ is lax symmetric monoidal as an endofunctor of right
$\sA$-modules, without presuming that $M_* \cong H_*(B)$ for some
spectrum~$B$ with $B/p$ bounded below.  Given the explicit formulas
\eqref{eq:steenrod-singer-tr}--\eqref{eq:steenrod-singer-squares}, and
the elementwise description of the multiplication~$\mu \: r_+(M_*)
\otimes r_+(N_*) \to r_+(M_* \otimes N_*)$, where the
domain has the diagonal $\sA$-action, this reduces to verifying
certain mod~$p$ congruences involving sums of binomial coefficients.
When $M_*=N_*=\bF_p$, the assertion that $\mu$ is $\sA$-linear amounts
to the mod~$p$ congruences
$$
\sum_{s_1+s_2=s} \binom{-1-r_1-s_1(p-1)}{s_1} \binom{-1-r_2-s_2(p-1)}{s_2}
\equiv \binom{-1-r-s(p-1)}{s}
$$
for $r = r_1 + r_2$ with $r_1, r_2 \ge0$, which are equivalent
to the congruences
$$
\sum_{s_1+s_2=s} \binom{r_1+ps_1}{s_1} \binom{r_2+ps_2}{s_2}
\equiv \binom{r+ps}{s} \,.
$$
For more general $\sA$-modules $M_*$ and $N_*$, one requires
the mod~$p$ congruences
$$
\sum_{s_1+s_2=s} \binom{r_1+ps_1-pk_1}{s_1-pk_1} \binom{r_2+ps_2-pk_2}{s_2-pk_2}
\equiv \binom{r+ps-pk}{s-pk}
$$
for $r = r_1 + r_2$ and $k = k_1 + k_2$.
These all follow from the special case
$$
\sum_{i+j=n} \binom{x + p i}{i} \binom{y + p j}{j}
\frac{x}{x + p i} = \binom{x + y + p n}{n}
$$
of the 19th century Rothe--Hagen identity, see \cite{Gou56}*{(11)}, since
$\binom{x+pi}{i} \frac{x}{x+pi} \equiv \binom{x+pi}{i} \mod p$.  That
identity is in turn a generalization of Vandermonde's convolution formula.
We omit the lengthy but elementary intermediate calculations.
\end{remark}

\section{Singer's $\epsilon$-homomorphism}\label{sec:singers-epsilon}

There is a natural morphism $\epsilon \: M_* \to r_+(M_*)$ defined for
so-called algebraic right $\sA$-modules~$M_*$.  We prove in
Proposition~\ref{prop:epsilon-is-epsilon} that
when
$M_* = H_*(B)$ is the homology of a spectrum~$B$, the composite
$\omega_B \circ \epsilon \: H_*(B) \to H^c_*(R_+(B))$ is equal to the
morphism induced by the spectrum map
$\epsilon_B^H \: H\wedge B \to R^H_+(H \wedge B)$.  Hence Singer's
$\epsilon$-homomorphism models the $H$-based Tate diagonal map.

Let $M_*$ be a left $\sA_*$-comodule with structure map
$\nu\: M_*\to \sA_*\otimes M_*$.  Recall from
Subsection~\ref{sec:isocat} that $\nu$ gives rise to a right action of
$\sA$ on $M_*$, with structure map $M_*\otimes \sA\to M_*$ adjoint to
the composite
$$
M_*\overset\nu\longto \sA_*\otimes M_* \overset{\iota}\longto
\Hom(\sA, M_*) \,,
$$
where~$\iota$ is the injective graded homomorphism \eqref{eq:iotamods}.

An $\bF_p$-linear homomorphism $f\: \sA\to M_*$ is contained in the
image of~$\nu$ if and only if $f$ vanishes on all but finitely many
elements $a\in \sA$.  It follows that a right $\sA$-module structure on
$M_*$ arises from a left $\sA_*$-comodule structure in the way
described above if and only if each element $x\in M_*$ has the
property that $\P^s_*(x) = 0$ for all but finitely many integers~$s$.  A
right $\sA$-module satisfying this property will be called
\emph{algebraic}.  An element~$x\in M_*$ is \emph{algebraic} if it is
contained in an algebraic $\sA$-submodule of $M_*$.  Thus, $x\in M_*$
is algebraic if and only if $\P^s_*(x)=0$ for all but finitely many
integers~$s$.

Examples of algebraic right $\sA$-modules include modules that are
bounded below, and modules that are the mod~$p$ homology of spectra.

In contrast, the next lemma implies that $r_+(M_*)$ is essentially
never an algebraic right $\sA$-module.  When $M_*=\bF_p$, we recall
from \eqref{eq:steenrod-singer-tr}--\eqref{eq:steenrod-singer-squares}
for $B = S$ that the right $\sA$-module structure on $r_+(\bF_p)$ is
given by
formulas~\eqref{eq:singer-sphere-steenrod}--\eqref{eq:singer-sphere-squares}.

\begin{lemma}\label{lemma:alg}
  The vector subspace $\bF_p\{1\}\subset r_+(\bF_p)$ is a right
  $\sA$-submodule and consists of the only algebraic elements in
  $r_+(\bF_p)$.
\end{lemma}
\begin{proof}
  Let $p>2$ and assume first that $r<0$. Then  $-1-r\geq 0$, and
  $$
  \binom{-1-r-p^k(p-1)}{p^k} \equiv
  \binom{p^k}{p^k} = 1
  $$
  by Lucas' theorem for all $k>0$ such that $p^k > -1-r$.
  It follows from
  \eqref{eq:singer-sphere-steenrod} that $\P^{p^k}_* (u^it^r) \neq 0$
  for all sufficiently large~$k$, and $u^it^r$ is therefore not
  algebraic when $-i-2r > 0$.

  Assume next that $r>0$ and that
  $r \equiv a_kp^k \neq 0 \mod p^{k+1}$.  Then
  \[
    \binom{-1-r-p^k(p-1)}{p^k} \equiv
    \binom{-1 - a_k p^k + p^k}{p^k}
	\equiv -a_k\,,
  \]
  which implies that $\P^{p^k}_*(u^it^r)\neq 0$ for $i=0,1$ and
  $r>0$. It follows from \eqref{eq:singer-sphere-steenrod} that
  $\beta_*(u)=t$, and we conclude that $u^it^r$ is never
  $\sA_*$-comodule primitive when $-i-2r < 0$.  By iteration, this
  implies that $u^it^r$ supports infinitely many non-trivial Steenrod
  operations.

  Finally, since
  $$
  \binom{-1-p^k(p-1)}{p^k} \equiv
  \binom{-1+p^k}{p^k} \equiv 0\,,
  $$
  we get that $\P^{p^k}_* (1) = 0$ for all $k\geq 0$.  Also
  $\beta_*(1) = 0$.  The class $1 \in r_+(\bF_p)$ is therefore
  $\sA_*$-comodule primitive and, in particular, algebraic.

  The case~$p=2$ follows from the same argument as above, with $t$
  replaced by $u^2$ and by using \eqref{eq:singer-sphere-squares}
  instead of
  \eqref{eq:singer-sphere-steenrod}--\eqref{eq:singer-sphere-bockstein}.
\end{proof}

Let $Q_*$ be a right $\sA$-module.  If the only algebraic element of
$Q_*$ is $0$, we say that $Q_*$ is \emph{totally non-algebraic}.
Lemma~\ref{lemma:alg} says that the quotient $r_+(\bF_p)/\bF_p\{1\}$
is totally non-algebraic.  If~$M_*$ is algebraic and $Q_*$ is totally
non-algebraic, then any morphism $f\:M_*\to Q_*$ must be trivial.

\begin{lemma}\label{lemma:non-algebraic}
  If $Q_*$ is a totally non-algebraic right $\sA$-module, then the
  tensor product $Q_*\otimes N_*$ is totally non-algebraic for any
  right $\sA$-module $N_*$.
\end{lemma}
\begin{proof}
  The degree filtration of $N_*$ induces a filtration of
  $Q_*\otimes N_*$ which is exhaustive and Hausdorff.  Thus, for any
  non-trivial element $z\in Q_*\otimes N_*$, there exists an integer
  $q$ such that $z\in Q_*\otimes N_{*\leq q}$, and such that $z$ maps
  non-trivially to the filtration quotient
  $Q_*\otimes N_{*\leq q}/N_{*\leq q-1}$.  It follows from the Cartan
  formula that $Q_*\otimes N_{*\leq q}/N_{*\leq q-1}$ is totally
  non-algebraic.  Therefore, $z$ must be totally non-algebraic.  Since
  $z$ was an arbitrary non-trivial element, the lemma follows.
\end{proof}

Let $M_*$ be an algebraic right $\sA$-module.  Singer's
$\epsilon$-homomorphism is the natural morphism of right $\sA$-modules
\begin{equation}
  \label{eq:singer-epsilon}
  \epsilon \: M_* \to r_+(M_*)
\end{equation}
given by
\begin{equation}
	\label{eq:epsilon-def-p-even}
\epsilon(x) = \sum_{j} u^{-j} \otimes \Sq_*^j(x)
\end{equation}
for $p=2$, and
\begin{align}
  \label{eq:epsilon-def-p-odd}
  \begin{aligned}
    \epsilon(x) =
    &\sum_{j}
      (-1)^{j} t^{-j(p-1)} \otimes \P^j_*(x)\\
    &+\sum_{j}
      (-1)^j u t^{-1-j(p-1)} \otimes \P^j_*\beta_*(x)
  \end{aligned}
\end{align}
for~$p$ odd.  Note that our assumption on $M_*$ implies that these
sums are finite, and therefore that~$\epsilon$ is well-defined.
See~\cite{S1980}, \cite{LS82} and~\cite{LNR12}*{\S3.2.1}.

\begin{lemma}
  The Singer homomorphism $\epsilon \: M_* \to r_+(M_*)$ is right
  $\sA$-linear.
\end{lemma}

\begin{proof}

  Let $p=2$.  To show that \eqref{eq:singer-epsilon} is $\sA$-linear,
  we apply \eqref{eq:steenrod-singer-squares} to
  \eqref{eq:epsilon-def-p-even}, yielding
  \begin{align*}
    \Sq^s_* (\epsilon(x))
    &= \sum_{j} \sum_{k}\binom{-1+j-s}{s-2k} u^{-j+s-k} \otimes \Sq_*^k\Sq_*^j(x)\\
    &= \sum_{n}  u^{s-n} \otimes \sum_{k}\binom{-1+n-k-s}{s-2k}  \Sq_*^k\Sq_*^{n-k}(x) \,.
  \end{align*}
  On the other hand,
  $$
  \epsilon(\Sq^s_* (x))
  = \sum_{j} u^{-j} \otimes \Sq_*^j\Sq^s_*(x) \,.
  $$
  By comparing terms, with $j=n-s$, we see that $\epsilon$ and
  $\Sq^s_*$ commute if and only if
  \begin{align*}
    \Sq_*^j\Sq^s_* = \sum_{k}\binom{-1+j-k}{s-2k}  \Sq_*^k\Sq_*^{j+s-k}
  \end{align*}
  for all integers $s$ and $j$, or, equivalently,
  \begin{equation}\label{eq:epsilon-a-linearity-condition}
    \Sq^s\Sq^j = \sum_{k}\binom{-1+j-k}{s-2k}  \Sq^{j+s-k}\Sq^k
  \end{equation}
  in $\sA$ for all integers $s$ and $j$.  When $s<2j$, this is an Adem
  relation.  However, it is a result of Bullett--Macdonald
  \cite{BM1982} that \eqref{eq:epsilon-a-linearity-condition} also holds as
  a relation in the mod~2 Steenrod algebra for all integers $s$ and
  $j$, without the restriction that $s<2j$.
  
  The proof in the case of $p>2$ follows a similar route, again
  relying on Bullett--Macdonald \cite{BM1982} to say that the usual
  formulas giving the mod~$p$ Adem relations for $\P^s\P^j$ and
  $\P^s \beta \P^j$ are valid for all integers $s$ and $j$.
\end{proof}

\subsection{Untwisting}
Let $C$ be a spectrum such that $C/p$ is bounded below, and let
$N_*=H_*(C)$.  Define $\bar\epsilon$ to be the composite
\begin{equation}
  \label{eq:epsilonbar}
  \bar\epsilon\: r_+(\bF_p) \otimes N_*
  \overset{1\otimes \epsilon}\longto
  r_+(\bF_p)\otimes r_+(N_*)
  \overset\mu\longto
  r_+(N_*)\,,
\end{equation}
where $\epsilon$ is the $\sA$-linear
homomorphism~\eqref{eq:singer-epsilon}, the domain and codomain of
$1\otimes\epsilon$ have the diagonal $\sA$-action, and~$\mu$ is the
multiplication map sending $u^it^r\otimes u^jt^s\otimes x$ to
$u^{i+j}t^{r+s}\otimes x$.  Since $C/p$ is bounded below,
Proposition~\ref{prop:r-plus-is-monoidal} implies that $\mu$, and
therefore the composite homomorphism $\bar\epsilon$, is $\sA$-linear.
\begin{lemma}
  \label{lemma:barepsilon}
  Let $C$ be a spectrum such that $C/p$ is bounded below, and let
  $N_*=H_*(C)$.  Then
  $\bar\epsilon\: r_+(\bF_p)\otimes N_*\to r_+(N_*)$ is an
  isomorphism of right $\sA$-modules.
\end{lemma}
\begin{proof}
  Filter $N_*$ by degree, so that $F_q N_* = N_{*\leq q}$ for each
  $q\in\bZ$.  Every inclusion $F_{q-1}N_*\subset F_{q}N_*$ is a
  morphism of right $\sA$-modules, and induces filtrations of
  $r_+(\bF_p)\otimes N_*$ and $r_+(N_*)$.

  For each $x\in N_q$ we have
  $\bar\epsilon(u^it^r\otimes x) \equiv u^it^r\otimes x$ modulo the
  submodule $r_+(N_{*<q})$.  Consequently, the homomorphism
  $\bar\epsilon$ is a morphism of filtered right $\sA$-modules and it
  induces isomorphisms of filtration quotients.  Since $N_*$ is bounded
  below, $\bar\epsilon$ restricts to an isomorphism of right
  $\sA$-modules
  $$
  F_q (\bar\epsilon)\: r_+(\bF_p) \otimes N_{*\leq q}
  \overset\cong\longto r_+(N_{*\leq q})
  $$
  for each $q\in\bZ$.  The lemma then follows by passing to the
  colimit over~$q$.
\end{proof}

\subsection{$\Hom$-isomorphisms}

Let $N_*$ be a right $\sA$-module, and let the tensor product
$r_+(\bF_p)\otimes N_*$ have the diagonal action by~$\sA$.  Since
$1\in r_+(\bF_p)$ is $\sA_*$-comodule primitive, the inclusion
$i\: N_*\to r_+(\bF_p)\otimes N_*$ given by $x\mapsto 1\otimes x$
is a morphism of right $\sA$-modules.

\begin{lemma}
  \label{lemma:hom-iso}
  Let $M_*$ and $N_*$ be right $\sA$-modules, and assume that $M_*$ is
  algebraic. Then
  $$
  i_* \: \Hom_{\sA}(M_*, N_*)
  \overset\cong\longto
  \Hom_{\sA}(M_*, r_+(\bF_p)\otimes N_*)
  $$
  is an isomorphism.
\end{lemma}
\begin{proof}
  The morphism $i\: N_*\to r_+(\bF_p)\otimes N_*$ is injective.  Thus,
  by left-exactness of $\Hom_\sA(M_*, -)$, it suffices to show that
  $i_*$ is surjective.  Take any~$f$ in the codomain of~$i_*$, and
  consider the following short exact sequence of right $\sA$-modules:
  \begin{equation*}
    \label{eq:ses}
    0\longto
    N_* \overset{i}\longto
    r_+(\bF_p)\otimes N_* \overset{q}\longto
    \frac{r_+(\bF_p)}{\bF_p\{1\}} \otimes N_*
    \longto 0\,.
  \end{equation*}
  Lemma~\ref{lemma:alg} and Lemma~\ref{lemma:non-algebraic} imply
  that the codomain of~$q$ is totally non-algebraic.  Under the assumption
  that $M_*$ is algebraic, the composite $q\circ f$ must therefore be
  trivial.  Thus~$f$ lifts over~$i$, and the lemma follows.
\end{proof}

\begin{proposition}
  \label{prop:hom-equivalence}
  Let $B$ and $C$ be spectra, and assume that $C/p$ is bounded below.
  Let $M_*=H_*(B)$ and $N_*= H_*(C)$.  Then
  $$
  \epsilon_*\: \Hom_\sA(M_*, N_*) \overset\cong\longto
  \Hom_\sA(M_*, r_+(N_*))
  $$
  is an isomorphism.
\end{proposition}
\begin{proof}
  As a right $\sA$-module, $M_*$ is algebraic since it is the mod~$p$
  homology of a spectrum.  The lemma now follows by applying
  Lemma~\ref{lemma:hom-iso} and Lemma~\ref{lemma:barepsilon} to the
  two homomorphisms of the factorization
  \[
    \epsilon_* \: \Hom_\sA(M_*, N_*) \overset{i_*}\longto
    \Hom_\sA(M_*, r_+(\bF_p)\otimes N_*) \overset{\bar\epsilon_*}\longto
    \Hom_\sA(M_*, r_+(N_*)) \,.
  \]
\end{proof}

\subsection{Finite factorization}

Let $B$ be a spectrum, and consider the morphism of
right $\sA$-modules
$$
(\epsilon_B^H)_*\: H_*(B)\longto
H_*^c(R_+(B))
$$
induced by the $H$-based Tate diagonal
$\epsilon_B^H\: H\wedge B\to R_+^H(H\wedge B)$.

\begin{lemma}\label{lemma:finite-faktorization}
  There is a unique natural $\sA$-linear homomorphism
  $\phi_{H_*B}\: H_*(B)\to r_+(H_*(B))$ such that $(\epsilon^H_B)_*$
  factors as
  \begin{equation}
      \label{eq:epsilon-comp}
      (\epsilon^H_B)_* \:
      H_*(B)
      \overset{\phi_{H_* B}}\longto
      r_+(H_*(B))
      \overset{\omega_B}\longto
      H^c_*(R_+(B)) \,.
    \end{equation}
    Furthermore, $\phi_{H_* B}$ is compatible
    with the homology suspension, in the sense that
    $\phi_{H_*\Sigma^n B} = \Sigma^n \phi_{H_* B}$.
\end{lemma}
\begin{proof}

  For each spectrum $B$ we have natural morphisms
  $$
  \xymatrix{
    H_*(B) \ar[r]^-{(\epsilon^H_B)_*}
	& H^c_*(R_+(B)) \\
    r_+(H_*(B)) \ar@{ >->}[r]^-{c} \ar[ur]^-{\omega_B}
	& R_+(H_*(B)) \ar[u]_-{\omega_B^\wedge}
  }
  $$
  of right $\sA$-modules.  When $B = F$ is finite, the morphisms $c$,
  $\omega_F$ and $\omega_F^\wedge$ are isomorphisms.  By passing to
  the colimit over all finite spectra~$F$ over~$B$, we get natural
  morphisms
  $$
  \xymatrix{
    H_*(B) \ar[r]
	& \colim_{F \to B} H^c_*(R_+(F)) \ar[r]
	& H^c_*(R_+(B)) \\
    r_+(H_*(B)) \ar[r]^-{\cong} \ar[ur]^-{\cong}
	& \colim_{F \to B} R_+(H_*(F)) \ar[u]^-{\cong} \ar@{ >->}[r]
	& R_+(H_*(B)) \ar[u]_-{\omega_B^\wedge} \,.
  }
  $$
  This uses that $\colim_{F \to B} H_*(F) \cong H_*(B)$, and
  that~$r_+$ preserves colimits.  Therefore, we have a unique natural
  morphism $\phi_{H_* B} \: H_*(B) \to r_+(H_*(B))$ of right
  $\sA$-modules such that
  $$
  \xymatrix{
    H_*(B) \ar[r]^-{(\epsilon^H_B)_*} \ar[d]_-{\phi_{H_* B}}
	& H^c_*(R_+(B)) \\
    r_+(H_*(B)) \ar@{ >->}[r]^-{c}
    \ar[ru]^{\omega_B}
	& R_+(H_*(B))) \ar[u]_-{\omega_B^\wedge}
  }
  $$
  commutes.

  Finally, we note that both $\epsilon^H_B$ and $\omega_B$ are
  compatible with suspension in $B$, see
  Corollary~\ref{cor:epsilon-suspensions} and
  Lemma~\ref{lemma:omega-suspensions}, hence so is $\phi_{H_* B}$.
\end{proof}

\begin{lemma}\label{lemma:l1}
  Let $B=\Sigma^q H$.  Then $\phi_{H_* B}$ of 
  Lemma~\ref{lemma:finite-faktorization} equals
  $\epsilon\: \Sigma^q\sA_*\to r_+(\Sigma^q\sA_*)$, and
  $(\epsilon^H_{\Sigma^q H})_*$ factors as
  \begin{equation*}
      (\epsilon^H_{\Sigma^q H})_* \:
      \Sigma^q \sA_*
      \overset{\epsilon}\longto
      r_+(\Sigma^q\sA_*)
      \overset{\omega_{\Sigma^qH}}\longto
      H^c_*(R_+(\Sigma^qH)) \,.
    \end{equation*}
\end{lemma}
\begin{proof}
  Since $\Sigma^q \phi_{H_*B} = \phi_{\Sigma^q H_*B}$, it suffices to
  consider the case of $q=0$.  Since $\sA_*$ is a bounded below right
  $\sA$-module and $\Hom_{\sA}(\sA_*, \sA_*) \cong \bF_p$, it follows
  from Proposition~\ref{prop:hom-equivalence} that
  $\phi_{\sA_*} = \epsilon$, up to multiplication by a scalar.  This
  scalar must equal~$1$ since multiplicativity of $\epsilon^H_H$ implies
  that $\phi_{\sA_*}(1)=1$.
\end{proof}

The proof of the following result is essentially the proof of
\cite{LNR12}*{Prop.~5.12}, re-written for homology instead of
cohomology.
\begin{proposition}\label{prop:epsilon-is-epsilon}
  Let $B$ be any spectrum.  Then $(\epsilon_B^H)_*$ factors as
  \begin{equation}
    (\epsilon_B^H)_* \: H_*(B)
    \overset{\epsilon}\longto
    r_+(H_*(B))
    \overset{\omega_B}\longto
    H_*^c(R_+(B)) \,.
  \end{equation}
\end{proposition}
\begin{proof}
  Let $q$ be any integer and $f \: B\to \Sigma^q H$ any map.  Consider
  the diagram
  \begin{equation}
    \label{eq:naturality-epsilon}
    \begin{aligned}
      \xymatrix@C12mm@R6mm{
      H_*(B)
      \ar[rr]^{(\epsilon^H_{B})*}
      \ar@<-0.5pt>[rd]^-{\phi_{H_*B}}
      \ar@<-4.5pt>[rd]_-{\epsilon}
      \ar[ddd]^{(1\wedge f)_*}
      && H_*^c(R_+(B))
         \ar[ddd]^{R^H_+(1\wedge f)_*} \\
      & r_+(H_*(B))
        \ar[ru]^{\omega_{B}}
        \ar[d]^{r_+(f_*)} \\
      & r_+(H_*(\Sigma^q H))
        \ar[rd]^{\omega_{\Sigma^q H}}\\
      H_*(\Sigma^q H)
      \ar[ru]^-{\epsilon}
      \ar[rr]_{(\epsilon^H_{\Sigma^q H})_*}
      && H_*^c(R_+(\Sigma^q H))\,.
         }
    \end{aligned}
  \end{equation}
  We want to show that $\phi_{H_*B} = \epsilon$.  The outer square and
  the left- and right-hand trapezoids of
  \eqref{eq:naturality-epsilon} commute since $\epsilon^H_B$,
  $\epsilon$, and $\omega_B$ are natural in~$B$.  The upper triangle
  commutes by Lemma~\ref{lemma:finite-faktorization}, and the lower
  triangle commutes by Lemma~\ref{lemma:l1}.

  The injectivity of $\omega_{\Sigma^q H}$ (from
  Proposition~\ref{prop:omega-B}) then implies that the image of
  $\epsilon - \phi_{H_*B}$ is contained in the kernel of $r_+(f_*)$.
  Since $q$ and $f$ were arbitrarily chosen, the universal coefficient
  theorem then implies that $\phi_{H_*B} = \epsilon$.
\end{proof}

\section{A residual differential on $R_+(M_*)$} \label{sec:residual-differential-on-rplus}

Recall from Subsection~\ref{sec:dgas} the (left) suspension operator
$S$ and the filtration-shift operator $\sh_a$.  Let $M_*$ be a right
$\sA$-module equipped with a differential
$\sigma\: \leftsuspension M_*\to M_*$ that is also a morphism in the
category of right $\sA$-modules.  Recall that this means that $\sigma$
commutes with even degree operations and anticommutes with odd degree
operations in $\sA$, and in particular that
$\beta_*\sigma = -\sigma\beta_*$.

Define a morphism
\begin{equation}
  \label{eq:sigmaone}
  \bar\sigma\: \leftsuspension r_+(M_*) \to \sh_{1}r_+(M_*)
\end{equation}
of filtered graded $\bF_p$-vector spaces by the formulas
\begin{align}
  \bar\sigma(t^r\otimes x)
  &= t^r\otimes \sigma(x) \label{eq:sigmaone-tr}\\
  \bar\sigma(ut^r\otimes x)
  &= t^r\otimes x - ut^r\otimes \sigma(x) \label{eq:sigmaone-utr}
    \intertext{for $p>2$, and}
  \bar\sigma(u^{2r}\otimes x)
  &= u^{2r}\otimes \sigma(x)  \label{eq:sigmaone-tr-even-prime} \\
  \bar\sigma(u^{2r+1}\otimes x)
  &= u^{2r}\otimes x + u^{2r+1}\otimes \sigma(x) \label{eq:sigmaone-utr-even-prime}
\end{align}
for $p=2$.  It is clear from
\eqref{eq:sigmaone-tr}--\eqref{eq:sigmaone-utr-even-prime} that $\bar\sigma$ is a
differential, and that it shifts Tate filtration
(Definition~\ref{dfn:tate-filtration-rplus}) by $+1$.  As noted in
Subsection~\ref{sec:dgas}, $\bar\sigma$ induces a differential
\begin{equation}
  \label{eq:sigmaone-completed}
  \bar\sigma^\wedge \: \leftsuspension R_+(M_*) \to \sh_{1}R_+(M_*)\,,
\end{equation}
by passing to the completion with respect to the Tate filtration.  We
will usually simply write $\bar\sigma$ for both \eqref{eq:sigmaone} and
\eqref{eq:sigmaone-completed}, relying on the context to decide whether we
are in the completed case or not.

Note that when $M_*=\bF_p$, the differential specified
by~\eqref{eq:sigmaone-tr}--\eqref{eq:sigmaone-utr-even-prime} agrees
with the formulas for the differential induced by the residual circle
action on $S^{tC_p}$ specified in Lemma~\ref{lemma:u-and-t}.  More
precisely,
$\omega_S\: r_+(\bF_p) \to H_*^c(R_+(S))\cong H_*^c(S^{tC_p})$
is an isomorphism of differential graded algebras.

\begin{lemma}\label{lemma:sigmaone-alinear}
  Let $M_*$ be a right $\sA$-module and
  $\sigma\: \leftsuspension M_*\to M_*$ a differential that is a
  morphism of right $\sA$-modules.  If $\beta_*$ acts trivially
  on $M_*$, then the differential \eqref{eq:sigmaone} and its
  completion \eqref{eq:sigmaone-completed} are both morphisms of
  filtered right $\sA$-modules.
\end{lemma}
\begin{proof}
  Let $p>2$.  A direct calculation using
  \eqref{eq:steenrod-singer-tr}--\eqref{eq:steenrod-singer-utr} and
  \eqref{eq:sigmaone-tr}--\eqref{eq:sigmaone-utr} shows that
  \[
    (\P^s_*\bar\sigma - \bar\sigma \P^s_*) (u^it^r\otimes x)
    = (-1)^i\sum_k \binom{-1-r-s(p-1)}{s-pk-1}u^it^{-1+r+(s-k)(p-1)}
    \otimes \P^k_* \beta_* (x)\,.
  \]
  Thus, each $\P^s_*$ commutes with $\bar\sigma$ if $\beta_*=0$.  The
  relation $\beta_*\bar\sigma=-\bar\sigma\beta_*$ follows immediately
  from \eqref{eq:steenrod-singer-beta} and
  \eqref{eq:sigmaone-tr}--\eqref{eq:sigmaone-utr}, without the
  assumption that $\beta_*(x)=0$ in~$M_*$.

  The proof for $p=2$ is similar, with $t$ replaced by $u^2$ and using
  \eqref{eq:steenrod-singer-squares} instead of
  \eqref{eq:steenrod-singer-tr}--\eqref{eq:steenrod-singer-beta}.
\end{proof}

\begin{proposition}
  Let $(M_*, \sigma)$ be a differential graded right $\sA$-module
  algebra.  If $\beta_*$ acts trivially on $M_*$, then
  $(r_+(M_*), \bar\sigma)$ is a filtered differential graded
  right $\sA$-module algebra and its completion
  $(R_+(M_*), \bar\sigma)$ is a complete differential graded right
  $\sA$-module algebra.  If $M_*$ is bounded below, then these
  filtered algebras are relatively bounded below (rbb).
\end{proposition}
\begin{proof}
  By a direct calculation one can verify that $\bar\sigma$ satisfies
  the Leibniz formula \eqref{eq:leibniz}.  Thus,
  $(r_+(M_*), \bar\sigma)$ is a filtered differential graded
  $\bF_p$-algebra.

  Lemma \ref{lemma:sigmaone-alinear} ensures that $\bar\sigma$ is also
  morphism in the category of filtered right $\sA$-modules, and the
  proposition follows for $(r_+(M_*),\bar\sigma)$.  By the
  discussion following Proposition~\ref{prop:dgalg-category-iso}, the
  completed object $(R_+(M_*), \bar\sigma)$ is therefore a complete filtered
  differential graded right $\sA$-module algebra.
\end{proof}

\begin{lemma}\label{lemma:sigmaone-exact}
  The image and the kernel of the differential
  $\bar\sigma \: \leftsuspension r_+(M_*) \to \sh_1 r_+(M_*)$ are
  equal as filtered submodules of~$r_+(M_*)$, and are explicitly given
  by
  \begin{equation}
    \label{eq:ker-im-sigmaone}
    \ker(\bar\sigma) = \im(\bar\sigma)
	= \< t^r \otimes x - ut^r \otimes \sigma(x) \mid
	\text{$r \in \bZ$, $x \in M_*$} \> \,.
  \end{equation}
  Similarly, the image and the kernel of the completed differential
  $\bar\sigma \: S R_+(M_*) \to \sh_1 R_+(M_*)$ are equal, and are
  explicitly given by the completion of~\eqref{eq:ker-im-sigmaone}
  with respect to the Tate filtration.
\end{lemma}
\begin{proof}
  It follows from \eqref{eq:sigmaone-tr}--\eqref{eq:sigmaone-utr} and
  $\sigma^2 = 0$ that
  $\bar\sigma(ut^r\otimes \sigma(x)) = t^r\otimes \sigma(x) =
  \bar\sigma(t^r\otimes x)$.  Thus, $\im(\bar\sigma)$ is spanned by
  elements of the form $\bar\sigma(ut^r\otimes x)$, proving the right-hand
  equality of \eqref{eq:ker-im-sigmaone}.

  We proceed to show that $\im(\bar\sigma) = \ker(\bar\sigma)$, first as
  $\sA$-modules, and then as filtered objects.  Since $\bar\sigma$ is a
  differential, we immediately get that
  $\im(\bar\sigma) \subset \ker(\bar\sigma)$.  To show the reverse
  inclusion, consider an element $z = ut^r\otimes x' + t^r\otimes x$
  and its image under~$\bar\sigma$,
  \[
    t^r\otimes x' - ut^r \otimes \sigma(x') +
    t^r\otimes \sigma(x) = t^r \otimes (x'+\sigma(x))- ut^r\otimes
    \sigma(x') \,.
  \]
  It follows that if $z\in \ker(\bar\sigma)$ then $x'=-\sigma(x)$, and
  therefore $z = \bar\sigma(ut^r\otimes x)$.  Since the vector subspace
  $\bF_p\{ut^r, t^r\}\otimes M_* \subset r_+(M_*)$ is invariant
  under $\bar\sigma$ for any fixed $r \in\bZ$, it follows that any
  $z\in \ker(\bar\sigma)$ is a (finite) sum of the form
  $\sum_k \bar\sigma(ut^k\otimes x_k)$.  We conclude that
  $\ker(\bar\sigma) = \im(\bar\sigma)$, as $\sA$-modules.

  For morphisms $f\:A\to B$ and $g\:B\to C$ of filtered abelian
  groups, with components $\{f_n\}_{n \in \bZ}$ and
  $\{g_n\}_{n \in \bZ}$, we consider $\im(f)$ and $\ker(g)$ to have
  filtrations given by $F_n\im(f) = \im(f_n)$ and
  $F_n\ker(g) = \ker(g_n)$.  We always have an equality
  $F_n\ker(g) = \ker(g)\cap F_n B$, but in general only an inclusion
  $F_n\im(f) \subset \im(f)\cap F_nB$.  If $g\circ f=0$ then the
  inclusion $\im(f)\subset\ker(g)$ is a filtered homomorphism, but an
  equality $\im(f)=\ker(g)$ of abelian groups does not in general
  imply that $F_n\im(f)=F_n\ker(g)$ for each $n\in\bZ$.

  When $f=\bar\sigma$, however, we have that if
  $z = \bar\sigma(u t^r \otimes x) \in r_+(M_*)$ lies in Tate
  filtration~$n+1$, then $u t^r \otimes x \in S r_+(M_*)$ lies
  in Tate filtration~$n$.  Hence
  \begin{equation}
    \label{eq:subfiltration}
    F_n \im(\bar\sigma) = \im(\bar\sigma) \cap F_{n+1} r_+(M_*) \,,
  \end{equation}
  and we get that
  \[
    F_n\im(\bar\sigma) = \im(\bar\sigma) \cap F_{n+1}r_+(M_*) =
    \ker(\bar\sigma) \cap F_{n+1}r_+(M_*) =
    F_{n+1}\ker(\bar\sigma)\,,
  \]
  since $\im(f) = \ker(g)$ as unfiltered objects.  This shows the
  first part of the lemma.

  Lemma~\ref{lemma:homology-commutes-with-completion} applies to
  \[
    S^2 r_+(M_*) \overset{S\bar\sigma}\longto S \sh_1
    r_+(M_*) \overset{\sh_1\bar\sigma}\longto \sh_2 r_+(M_*)
  \]
  because of~\eqref{eq:subfiltration}.  Thus,
  $\im(\bar\sigma)=\ker(\bar\sigma)$ implies that
  $\im(\bar\sigma)^\wedge = \ker(\bar\sigma)^\wedge$ and
  $\im(\bar\sigma^\wedge)=\ker(\bar\sigma^\wedge)$.
\end{proof}

\begin{lemma}\label{lemma:homology-commutes-with-completion}
  Let $f\: A\to B$ and $g\: B\to C$ be morphisms of filtered abelian
  groups such that $g\circ f = 0$.  If $\im(f_n) = \im(f) \cap F_nB$
  and $\im(g_n) = \im(g) \cap F_n C$
  for each~$n\in \bZ$, then
  there are canonical isomorphisms $\ker(f)^\wedge \cong \ker(f^\wedge)$,
$\im(f)^\wedge \cong \im(f^\wedge)$, $\ker(g)^\wedge \cong \ker(g^\wedge)$
and
  \begin{equation}
    \label{eq:homology-completion}
    \frac{\ker(g)^\wedge}{\im(f)^\wedge}
    \cong
    \frac{\ker(g^\wedge)}{\im(f^\wedge)}\,.
  \end{equation}
\end{lemma}
\begin{proof}
  Consider
  $$
  \xymatrix{
    F_nK \ar@{ >->}[r]^-{i_n} \ar@{ >->}[d]
	& F_nA \ar@{->>}[r]^{p_n} \ar@{ >->}[d]
	& F_nI \ar@{ >->}[r]^-{j_n} \ar@{ >->}[d]
	& F_nB \ar@{ >->}[d] \\
    K \ar@{ >->}[r]^-i \ar@{->>}[d]
	& A \ar@{->>}[r]^p \ar@{->>}[d]
	& I \ar@{ >->}[r]^-j \ar@{->>}[d]
	& B \ar@{->>}[d] \\
    K/F_nK \ar@{ >->}[r]^-{i^n}
	& A/F_nA \ar@{->>}[r]^{p^n}
	& I/F_nI \ar[r]^-{j^n}
	& B/F_nB
  }
  $$
  with $i \: K = \ker(f) \to A$, $j \: I = \im(f) \to B$, $f = jp$,
  $i_n \: F_nK = \ker(f_n) \to F_nA$,
  $j_n \: F_nI = \im(f_n) \to F_nB$ and $f_n = j_n p_n$.  The three
left-hand columns form a $3 \times 3$ diagram of short exact sequences,
the right-hand column is exact, and $j_n$ and~$j$ are injective.

  The completion $f^\wedge \: A^\wedge = \lim_n A/F_nA \to B^\wedge$
  factors as
  $$
  A^\wedge \overset{p^\wedge}\longto I^\wedge \overset{j^\wedge}\longto
  B^\wedge \,.
  $$
  Here
  $0 \to K^\wedge \overset{i^\wedge}\longto A^\wedge \overset{p^\wedge}\longto I^\wedge \overset{\delta}\longto
  \rlim_n K/F_nK$ is exact and $\rlim_n K/F_nK = 0$ since it is the
  derived limit of a tower of surjective homomorphisms, so $p^\wedge$
  is surjective.

  Since $F_nI = I \cap F_nB$ within~$B$ by hypothesis, we get that
  each $j^n$ is injective, so $j^\wedge = \lim_n j^n$ is injective.
  Hence
  $\ker(f)^\wedge = K^\wedge \cong \ker(p^\wedge) = \ker(f^\wedge)$
  via~$i^\wedge$, and $\im(f)^\wedge = I^\wedge \cong \im(f^\wedge)$
  via~$j^\wedge$.

  Next, we repeat the argument with $g \: B \to C$ in place of
  $f \: A \to B$, using that $\im(g_n) = \im(g) \cap F_n C$ to deduce
  that $\ker(g)^\wedge \cong \ker(g^\wedge)$.

  Let $h \: \im(f) \subset \ker(g)$ be the filtration-preserving
  inclusion.  Since $I \cap F_n \ker(g) = I \cap F_n B = F_n I$ within
  $\ker(g)$, we deduce as for~$j$ that
  $h^\wedge \: \im(f)^\wedge \to \ker(g)^\wedge$ is injective.  It is
  compatible under the isomorphisms above with the inclusion
  $\im(f^\wedge) \subset \ker(g^\wedge)$, and therefore
  $\ker(g)^\wedge/\im(f)^\wedge \cong \ker(g^\wedge)/\im(f^\wedge)$,
  as asserted.
\end{proof}

\subsection{A homological $\bT$-Singer construction}\label{sec:t-singer}

Let $M_*$ be a right $\sA$-module with the property that $\beta_*$
acts trivially, and assume that $\sigma\: \leftsuspension M_*\to M_*$
is a differential and a morphism of right $\sA$-modules.  Define
$$
c_+(M_*; \sigma) = \hat H^{-*}(\bT; \bF_p)\otimes M_*\,,
$$
where $\hat H^{-*}(\bT; \bF_p) = P(t^{\pm1})$, and specify a right
action of the Steenrod operations on $c_+(M_*; \sigma)$ by the
formulas
\begin{align}
  \P^s_*(t^r \otimes x)
  &= \sum_k \binom{-1-r-s(p-1)}{s-pk}
    t^{r+(s-k)(p-1)} \otimes \P^k_*(x) \label{eq:Ps-C-plus} \\
  \beta_*(t^r \otimes x)
  &= -t^{r+1} \otimes \sigma(x) \label{eq:beta-C-plus} \,.
\end{align}
For $p=2$, replace $\P^s_*$ by $\Sq^{2s}_*$ in \eqref{eq:Ps-C-plus},
and $\beta_*$ by $\Sq^1_*$ in \eqref{eq:beta-C-plus}.  The \emph{Tate
  filtration} of $c_+(M_*; \sigma)$ is the ascending filtration
given by
\begin{equation*}
  % \label{eq:singer-tate-filtration-p-odd}
  F_n c_+(M_*; \sigma) =  \<t^r \otimes x \mid
  \text{$-2r -|x|(p-1) \leq n$, $x \in M_*$}
  \>\,.
\end{equation*}
The Tate filtration of $c_+(M_*)$ is relatively bounded below if
$M_*$ is bounded below, cf.~Lemma~\ref{lemma:rbb-rplus}.

Define $f_+\: c_+(M_*; \sigma) \to r_+(M_*)$ by the
formula
\begin{align}
  f_+(t^r \otimes x) &= t^r \otimes x - ut^r \otimes \sigma(x) \label{eq:f+odd} \\
\intertext{for $p$ odd, and}
  f_+(t^r \otimes x) &= u^{2r} \otimes x + u^{2r+1} \otimes \sigma(x) \label{eq:f+two}
\end{align}
for $p=2$.  It is clear that $f_+$ is injective and strictly
filtration-preserving.  Lemma~\ref{lemma:fplus-is-a-linear} below
implies that \eqref{eq:Ps-C-plus}--\eqref{eq:beta-C-plus} satisfy the
Adem relations and that $c_+(M_*; \sigma)$ is a filtered right
$\sA$-module.

Furthermore, the image of $f_+$ equals the kernel of the
residual circle action $\bar\sigma$, described by
Lemma~\ref{lemma:sigmaone-exact}.  To see this, note that
\[
  0 \to c_+(M_*; \sigma)
  \overset{f_+}\longto r_+(M_*)
  \overset{\bar\sigma}\longto r_+(M_*)
\]
is isomorphic as a sequence of graded $\bF_p$-vector spaces to the
direct sum of the sequences
\begin{equation} \label{eq:f+barsigma-r}
  0 \to \bF_p\{t^r\} \otimes M_*
  \overset{f_+}\longto \bF_p\{u t^r , t^r\} \otimes M_*
  \overset{\bar\sigma}\longto \bF_p\{u t^r , t^r\} \otimes M_*
\end{equation}
for $r \in \bZ$.  Here,~$f_+$ is given
by~\eqref{eq:f+odd}--\eqref{eq:f+two}, and $\bar\sigma$ by
\eqref{eq:sigmaone-tr}--\eqref{eq:sigmaone-utr-even-prime}.  It is
elementary to check that~\eqref{eq:f+barsigma-r} is exact for each
$r$.

\begin{lemma}\label{lemma:fplus-is-a-linear}
  The homomorphism $f_+$ commutes with the action of the
  Steenrod operations.
\end{lemma}
\begin{proof}
  Let $p>2$.  Since $\beta_*$ acts trivially on $M_*$, the formulas
  \eqref{eq:steenrod-singer-tr}--\eqref{eq:steenrod-singer-utr} reduce
  to the single formula
  \begin{equation}
    \label{eq:powers-on-m-when-beta-is-trivial}
    P_*^s(u^it^r\otimes x) =
    \sum_k \binom{-1-r-s(p-1)}{s-pk}u^it^{r+(s-k)(p-1)}\otimes \P^k_*(x)\,,
  \end{equation}
  valid for any $u^it^r\otimes x$ in $r_+(M_*)$.  Comparing
  \eqref{eq:powers-on-m-when-beta-is-trivial} to \eqref{eq:Ps-C-plus},
  it is then clear that $f_+$ commutes with $\P^s_*$ for each
  $s\geq 0$.  Moreover,
  $$
  \beta_*(f_+(t^r\otimes x)) =
  \beta_*(t^r\otimes x - ut^r\otimes \sigma (x)) =
  -t^{r+1}\otimes \sigma(x)\,,
  $$
  which equals
  $$
  f_+ (\beta_*(t^r\otimes x)) =
  f_+(-t^{r+1}\otimes \sigma(x)) =
  -t^{r+1}\otimes\sigma(x)\,.
  $$
  The last equality above uses that $\sigma$ is a differential on
  $M_*$.

  Let $p=2$.  We claim that $f_+$ commutes with any even squaring
  operation and that $\Sq^1_* f_+ = f_+\Sq^1_*$.  It follows
  that~$f_+$ commutes with $\Sq^i_*$ for any $i\geq 0$.

  Combining \eqref{eq:steenrod-singer-squares} with the assumption that
  $\Sq^1_*=0$ in $M_*$, we get that $\Sq^{2s}_*(f_+(t^r\otimes x))$
  equals
  \begin{multline}
    \label{eq:squares-on-m-when-beta-is-trivial}
    \Sq^{2s}_*(u^{2r}\otimes x + u^{2r+1}\otimes \sigma(x))
    = \sum_k \binom{-1-2r-2s}{2s-4k} u^{2r+2s-2k}\otimes \Sq^{2k}_*(x)\\
    +\sum_k \binom{-2-2r-2s}{2s-4k} u^{2r+1+2s-2k}\otimes \Sq^{2k}_*(\sigma (x))\,.
  \end{multline}
  By Lucas' theorem, there are congruences
  \[
    \binom{-1-2r-2s}{2s-4k} \equiv
    \binom{-2-2r-2s}{2s-4k} \equiv
    \binom{-1-r-s}{s-2k} \,,
  \]
  thus
  \[
    \Sq^{2s}_*(f_+(t^r\otimes x))
    = \sum_k\binom{-1-r-s}{s-2k} u^{2r+2s-2k}\cdot (1\otimes \Sq^{2k}_*(x) + u\otimes \Sq_*^{2k}(\sigma (x))) \,,
  \]
  which is seen to equal $f_+ (\Sq_*^{2s}(t^r \otimes x))$ by comparing
  with the definition~\eqref{eq:Ps-C-plus} and using that
  $\Sq^{2k}_* \sigma = \sigma \Sq^{2k}_*$.

  Finally, from~\eqref{eq:steenrod-singer-squares}
  and~\eqref{eq:f+two} we get
  \[
    \Sq^1_*( f_+(t^r\otimes x)) = \Sq^1_* (u^{2r}\otimes x + u^{2r+1}\otimes \sigma(x))
    = u^{2r+2}\otimes \sigma(x) \,.
  \]
  According to~\eqref{eq:beta-C-plus} and~\eqref{eq:f+odd}
  \[
    f_+(\Sq^1_* (t^r\otimes x)) = f_+(t^{r+1}\otimes \sigma(x))
    = u^{2r+2}\otimes \sigma(x) \,,
  \]
  where the latter equality depends on the fact that~$\sigma^2=0$.
  We conclude that $\Sq^1_* f_+ = f_+\Sq^1_*$.
\end{proof}

\begin{definition}\label{dfn:homological-C-singer}
  The \emph{homological $\bT$-Singer construction} on $M_*$ is the
  completion
  $$
  C_+(M_*; \sigma) = c_+(M_*; \sigma)^\wedge
  $$
  of $c_+(M_*; \sigma)$ with respect to the Tate filtration.
\end{definition}

As usual, the isomorphism~\eqref{eq:completion-quotient-iso} implies
that the Tate filtration of~$C_+(M_*;\sigma)$ is relatively bounded
below if $M_*$ is bounded below.  Thus, for a bounded below right
$\sA$-module $M_*$ with the property that~$\beta_*$ acts trivially,
the completed homomorphism
\[
  F_+ := f_+^\wedge\: C_+(M_*;\sigma) \longto R_+(M_*)
\]
is an injective morphism in the category $\rbbfilamodc$, and there is
an exact sequence
\begin{equation}
  \label{eq:F+sigmabar}
  0 \to C_+(M_*; \sigma)
  \overset{F_+}\longto R_+(M_*) \overset{\bar\sigma}\longto R_+(M_*)
\end{equation}
from which we identify the image of $F_+$ with the kernel of
$\bar\sigma$.  To see this, we note that as a sequence of graded $\bF_p$-vector spaces,
\eqref{eq:F+sigmabar} is isomorphic to the product over $r \in \bZ$ of
the exact sequences~\eqref{eq:f+barsigma-r}.

When $M_*$ is a right $\sA$-module algebra and $\sigma$ is a
derivation, $c_+(M_*; \sigma)$ inherits an algebra structure
from $r_+(M_*)$, making
$f_+\: c_+(M_*; \sigma)\to r_+(M_*)$ and its
completion $F_+\: C_+(M_*; \sigma)\to R_+(M_*)$ right
$\sA$-module algebra homomorphisms.  Elementwise, this algebra
structure is given by
$$
(t^r\otimes x) \cdot (t^s\otimes y) = t^{r+s}\otimes xy - (-1)^{|x|}u^2t^{r+s}\otimes \sigma(x)\sigma(y) \,.
$$
When $p>2$, the correction term $u^2t^{r+s}\otimes \sigma(x)\sigma(y)$
is zero since the class $u$ is exterior.  However, when $p=2$ we have
$u^2=t$, and the formula describing the multiplication reads
$$
(t^r\otimes x) \cdot (t^s\otimes y) = t^{r+s}\otimes xy +
t^{r+s+1}\otimes \sigma(x)\sigma(y)\,.
$$

\section{The continuous mod~$p$ homology of $\THH(MU)^{t\bT}$}

In~\cite{LNR11}*{Sec.~7.1}
\if \anonym
\else
  we defined
\fi
a homomorphism of right $\sA$-modules
\begin{equation}
  \Phi_{MU}\: R_+(H_*(\THH(MU))) \longto H_*^c (\THH(MU)^{tC_p})
  \if \anonym
  \else
    \,,
  \fi
\end{equation}
\if \anonym
  is defined and is shown to be
\else
  and showed that is
\fi
an isomorphism of unfiltered graded $\bF_p$-vector spaces.
A feature of $\Phi_{MU}$ is that it strictly increases Tate filtration
on certain classes, and is therefore not a morphism of filtered right
$\sA$-modules.  However, we explain in Subsection~\ref{sec:ext-iso}
how $\Theta_{MU} := \Phi_{MU}^{-1}$ is a morphism of rbb complete right
$\sA$-module algebras, or equivalently, of rbb complete left
$\sA_*$-comodule algebras.  It follows from
Proposition~\ref{prop:ext-iso} that $\Theta_{MU}$ induces an isomorphism of
continuous $\Ext$-algebras
$$
\cExt^*_{\sA_*}(\bF_p, H^c_*(\THH(MU)^{tC_p}))
\cong
\cExt^*_{\sA_*}(\bF_p, R_+(H_*(\THH(MU))))\,.
$$
The switch from $\Phi_{MU}$ to $\Theta_{MU}$, and from topologized to
filtered graded $\bF_p$-vector spaces, thus allows us to refine
\if \anonym
  results from~~\cite{LNR11}
\else
  our previous results
\fi
to also account for the multiplicative structure.
As a bonus, our Lemmas~\ref{lemma:f} and~\ref{lemma:omega-and-g} below
significantly simplify the discussion of pro-isomorphisms
from~\cite{LNR11}*{Prop.~7.2}.

For any $E_1$ ring spectrum $B$ the spectrum $THH(B)$ has a
$\bT$-action and an associated differential $\sigma$ acting on
$H_*( THH(B)) $.  If $B$ is an $E_2$ ring spectrum, then $\THH(B)$ is
an $E_1$ ring spectrum with $\bT$-action, and $H_*^c(\THH(B)^{tC_p})$
is a filtered differential graded right $\sA$-module algebra by
Proposition~\ref{prop:tate-dga}.  In Subsection~\ref{sec:thh-mu-dga}
we refine $\Theta_{MU}$ to a morphism of rbb filtered differential graded
right $\sA$-module algebras.  This will enable us to compute the
kernel of $\bar\sigma$ acting on $H_*^c(\THH(MU)^{tC_p})$, and
therefore also $H_*^c(\THH(MU)^{t\bT})$, using
Lemma~\ref{lemma:sigmaone-exact}.  The answer is given in
Theorem~\ref{thm:finale-MU}, in terms of the $\bT$-Singer
construction~$C_+$.

\subsection{Equivariant approximations to $\THH(B)$}\label{sec:1-skeleton-approx}

Let $\eta\: B \to \THH(B)$ be the map induced by inclusion of the
$0$-simplices of $\THH(B)$, and let $\eta_p\: B^{\wedge p}\to \THH(B)$
be induced from the inclusion of the $0$-simplices of the
$p$-fold edgewise subdivision of $\THH(B)$.  Then $\eta_p$ is $C_p$-equivariant,
and is homotopic to the $p$-fold multiplication
$B^{\wedge p}\to B$ followed by the inclusion $\eta\: B\to \THH(B)$.
See \cite{LNR11}*{Sec.~5} for details.
\begin{theorem}[\cite{LNR11}*{Thm.~5.3}]\label{thm:0-1-skeleton-approx}
  There is a commutative diagram of spectra in the stable category,
  \begin{equation}
    \begin{aligned}\label{eq:0-1-skeleton-approx}
      \xymatrix{
      B \ar[d]^{\epsilon_B} \ar[r] \ar@(ru,lu)[rr]^{\eta}
      & \bT_+ \wedge B \ar[d]^{\rho_+ \wedge \epsilon_B} \ar[r]^-{\omega}
      & THH(B) \ar[d]^{\gamma} \\
      R_+(B) \ar[r] \ar@(rd,ld)[rr]_{\eta^t}
      & \bT/C_{p+} \wedge R_+(B) \ar[r]^-{\omega^t}
      & THH(B)^{tC_p} \,.
        }
    \end{aligned}
  \end{equation}
\end{theorem}

In \eqref{eq:0-1-skeleton-approx}, the map $\gamma$ is the cyclotomic
structure map denoted~$\varphi_p$ in~\cite{NS18}, and $\omega$ and
$\omega^t$ are the unique equivariant extensions of $\eta$ and
$\eta^t := \eta_p^{tC_p}$, with respect to the actions of $\bT$ and $\bT/C_p$,
respectively.  The map $\rho\: \bT\to \bT/C_p$ is the $p$-th root
isomorphism with inverse $\rho^{-1}\: [z]\mapsto z^p$.
We note that the maps $\omega$ and $\omega^t$ of
\eqref{eq:0-1-skeleton-approx} are not the same as the homomorphism
$\omega_B$ discussed in
Subsection~\ref{sec:multiplicative-structure-of-rplus}.

Let $X$ be a cyclotomic spectrum with structure map
$\varphi_p\:X\to X^{tC_p}$.  Let $E$ be an $E_\infty$ ring spectrum with
trivial $C_p$-action.  In the same way we constructed the $E$-based
Tate diagonal~\eqref{eq:e-based-tate-diagonal} in
Subsection~\ref{sec:e-based-singer}, we get an ($E$-based) cyclotomic
structure on $E\wedge X$ with structure map
\[
  \varphi_p^E\: E\wedge X \overset{1\wedge \gamma}\longto E\wedge
  X^{tC_p} \overset{\kappa}\longto (E\wedge X)^{tC_p}\,.
\]
In particular, for $E=H$ and $B$ an $\bE_2$-ring spectrum, we have the
following diagram of filtered right $\sA$-module algebras
\begin{equation}
  \begin{aligned}
    \label{eq:epsilon-cyctomic-zeroskeleton}
    \xymatrix{
    r_+(H_*(B)) \ar[rr]^-{\omega_B} \ar[d]_-{r_+(\eta_*)}
	&
	& H^c_*(R_+(B)) \ar[d]^-{\eta^t_*} \\
    r_+(H_*(THH(B))
	& H_*(B) \ar[ul]_-{\epsilon} \ar[d]^-{\eta_*} \ar[ur]^-{(\epsilon^H_B)_*}
	& H^c_*(THH(B)^{tC_p}) \\
	& H_*(THH(B)) \ar[ul]_-{\epsilon} \ar[ur]^-{\gamma^H_*}  \,.
      }
  \end{aligned}
\end{equation}
This follows from Proposition~\ref{prop:epsilon-is-epsilon},
naturality of~$\epsilon$, and compatibility of the Tate diagonal with
the cyclotomic structure map.  It leads to the following diagram of
right $\sA$-module algebras:
\begin{equation}
  \begin{aligned}
    \label{eq:f-and-g}
    \xymatrix@C11mm{
    r_+(H_*(B)) \otimes_{H_*(B)} H_*(\THH(B))
    \ar[d]^-f
    \ar[r]^-{\omega_B\otimes 1}
    & H_*^c (R_+(B))\otimes_{H_*(B)} H_*(\THH(B)) \ar[d]^-g \\
    r_+(H_*(\THH(B)))
    & H_*^c \left(\THH(B)^{tC_p}\right)\rlap{\,.}
      }
  \end{aligned}
\end{equation}
Here, $f=r_+(\eta_*)\cdot \epsilon$ and
$g=\eta^t_*\cdot \gamma^H_*$.  The left-hand tensor product is
formed over $H_*(B)$ using that $r_+(H_*(B))$ is an
$H_*(B)$-algebra via Singer's $\epsilon$-homomorphism.  Likewise, the
right-hand tensor product is formed over $H_*(B)$ using that
$H_*^c(R_+(B))$ is an $H_*(B)$-algebra via the map induced by the
$H$-based Tate diagonal.  Finally, both tensor products use that
$H_*(\THH(B))$ is an $H_*(B)$-algebra via $\eta_*$.

\subsection{An $\Ext$-isomorphism}\label{sec:ext-iso}

Let $B$ be an $E_2$ ring spectrum, such that $B/p$ is bounded below.
The homological $C_p$-Tate spectral sequence for $X=\THH(B)$
converging to $H_*^c(\THH(B)^{tC_p})$ has
$\hat E^2(\THH(B)) = \hat H^{-*}(C_p; H_*(\THH(B)))$.  The
$C_p$-action on the coefficients $H_*(\THH(B))$ is trivial, since it
factors through the action of the connected group $\bT$, and we get an
isomorphism
$\hat E^2(\THH(B)) \cong \hat H^{-*}(C_p; \bF_p) \otimes
H_*(\THH(B))$.  Moreover, the $d^2$-differential satisfies
$d^2(u^it^r \otimes x) = u^it^{r+1} \otimes \sigma(x)$, and we have an
isomorphism
$$
\hat E^3_{**}(\THH(B)) \cong \hat H^{-*}(C_p;\bF_p)\otimes
\ker(\sigma)/\im(\sigma)\,.
$$
See \cite{R98}*{Lem.~3.3} for details.

Let $MU$ be the complex cobordism spectrum, realized as an $E_\infty$
ring spectrum \cite{Ma77}*{IV.2}.  For $p>2$, recall
\cite{Ad74}*{pp.~75--77} the $\sA_*$-comodule algebra isomorphism
\begin{equation}
  \label{eq:homology-of-mu-adams}
  H_*(MU) \cong P(\xibar_k \mid k\ge1) \otimes P(m_\ell \mid \ell \ne
  p^k-1) \,,
\end{equation}
where $m_\ell$ is an $\sA_*$-comodule primitive element of
degree~$2\ell$, for each $\ell\ge1$ not of the form $p^k-1$.  We
define $m_{p^k-1} =\xibar_k$ for each $k\geq 1$, so that we have an
isomorphism
\begin{equation}
  \label{eq:homology-of-mu}
  H_*(MU) \cong P(m_\ell \mid \ell \geq 1) \,.
\end{equation}
The surjection
\begin{equation}
  \label{eq:mu-to-bp}
  H_*(MU)\longto H_*(BP) \cong P(\xibar_k\mid k\geq 1) \subset \sA_*\,,
\end{equation}
sending $m_\ell$ to $\xibar_k$ for $\ell = p^k-1$, and to~$0$
otherwise, is a morphism of right $\sA$-module algebras.  When $p=2$,
the above statements are true after replacing $\xibar_k$ by
$\bar\zeta_k^2$ in \eqref{eq:homology-of-mu-adams}--\eqref{eq:mu-to-bp}.

By \cite{LNR11}*{Lem.~6.2} there is an isomorphism of left
$\sA_*$-comodule algebras
\begin{equation}
  \label{eq:homology-of-thh-of-mu}
  H_* (\THH(MU)) \cong H_*(MU) \otimes E(\sigma m_\ell \mid \ell\geq 1)\,,
\end{equation}
where each exterior generator $\sigma m_\ell$ is left $\sA_*$-comodule
primitive.  It can also be viewed as an isomorphism of right
$\sA$-module algebras, by the isomorphism of symmetric monoidal
categories~\eqref{eq:isocat} discussed in Subsection~\ref{sec:isocat},
since $H_*(\THH(MU))$ is bounded below.

The homological $C_p$-Tate spectral sequence converging to
$H^c_*(THH(MU)^{tC_p})$
collapses at the $\hat E^3$-term, and there is an algebra isomorphism
\begin{equation}
  \label{eq:e-infty-cp-tate-mu}
  \hat E^{\infty}_{**}(\THH(MU))
  \cong \hat H^{-*}(C_p;\bF_p)\otimes
    P(m_\ell^p \mid \ell\geq 1)\otimes
    E(m_\ell^{p-1}\sigma m_\ell \mid \ell\geq 1) \,.
\end{equation}
This collapse result is due to Bruner--Rognes \cite{BR05}*{Thm.~6.4}.
A proof using different methods can be found in
\cite{LNR11}*{Prop.~6.3}.

Consider $H_*(\THH(MU))$ and the sub right $\sA$-module algebra
$E:= E(\sigma m_\ell\mid \ell \geq 1)$ as filtered objects by pulling
back the Tate filtration along the homomorphism
\[
  \gamma^H_*\: H_*(\THH(MU)) \longto H_*^c(\THH(MU)^{tC_p}) \,,
\]
making $\gamma^H_*$ a strictly filtration-preserving morphism in the
sense discussed in~Subsection~\ref{sec:filtered-r-modules}.

\begin{lemma}\label{lemma:filtration-of-e}
  The homomorphism $\gamma_*|_E$ is injective and the Tate filtration
  of an element $\sigma m_{\ell_1}\cdots \sigma m_{\ell_k}$ in the
  monomial basis of~$E$ equals
  $-2(\ell_1+\dots+ \ell_k)(p-1)$.
\end{lemma}
\begin{proof}
  By \cite{LNR11}*{Thm.~6.4}, a unit multiple of
  $\gamma^H_* (\sigma m_\ell) \in H_*^c (\THH(MU)^{tC_p})$ is
  detected at the $\hat E^\infty$-term \eqref{eq:e-infty-cp-tate-mu}
  of the homological $C_p$-Tate spectral sequence by the class
  \begin{equation}
    \label{eq:image-gamma-sigma-mell}
    t^{\ell(p-1)}\otimes m_\ell^{p-1}\sigma m_\ell
    \in \hat E^\infty_{-2\ell(p-1), 2\ell p+1}(\THH(MU)) \,.
  \end{equation}
  In particular, the Tate filtration of
  $\sigma m_\ell\in E\subset H_*(\THH(B))$ equals $-2\ell(p-1)$.

  Since $\gamma^H_*$ is an algebra homomorphism we deduce from
  \eqref{eq:image-gamma-sigma-mell} and the algebra structure of the
  $\hat E^{\infty}$-term~\eqref{eq:e-infty-cp-tate-mu} that
  $\gamma_*|_E$ is injective and that the Tate filtration of the
  monomial $\sigma m_{\ell_1}\cdots \sigma m_{\ell_k}$ equals
  $-2(\ell_1+\dots+ \ell_k)(p-1)$.
\end{proof}
We note that since $E$ is degreewise finite dimensional over $\bF_p$,
this Tate filtration of~$E$ is complete and~$E$ becomes an rbb
complete right $\sA$-module algebra.

Using the isomorphism~\eqref{eq:homology-of-thh-of-mu} we
rewrite~\eqref{eq:f-and-g} for $B = MU$ as the following diagram
\begin{equation}
  \begin{aligned}
    \label{eq:f-and-g-mu}
    \xymatrix@C11mm{
    r_+ (H_*(MU)) \otimes E
    \ar[d]^-{f=r_+(\eta_*)\cdot \epsilon}
    \ar[r]^-{\omega_{MU}\otimes 1}
    & H_*^c (R_+(MU))\otimes E
      \ar[d]^-{g=\eta^t_*\cdot \gamma^H_*} \\
    r_+(H_*(\THH(MU)))
    & H_*^c \left(\THH(MU)^{tC_p}\right)\rlap{\,.}
      }
  \end{aligned}
\end{equation}
Here the tensor products have the convolution
filtrations, which are rbb by Lemma~\ref{lemma:tensor-rbb}.

We proceed by explicitly
describing the morphisms and how they behave with respect to the
different filtrations.

For any graded $\bF_p$-vector space $P_*$ with an exhaustive
Hausdorff filtration $F_nP_*$, and each non-zero $z\in P_*$, we write
$\Fil z$ for the least integer~$n$ such that $z\in F_nP_*$.

Using the isomorphisms~\eqref{eq:homology-of-mu}
and~\eqref{eq:homology-of-thh-of-mu}, the domain and codomain
of~$f$ are both isomorphic to
\begin{equation}
  \label{eq:domain-and-codomain-of-f}
  \hat H^{-*}(C_p; \bF_p) \otimes P(m_\ell \mid \ell\geq 1) \otimes
  E(\sigma m_\ell \mid \ell\geq 1)
\end{equation}
as right $\sA$-module algebras.  Under these identifications, $f$
corresponds to the identity morphism since $r_+(\eta_*)$ is the
inclusion
$1\otimes \eta_*\: H_*^c(S^{tC_p})\otimes H_*(MU)\subset
H_*^c(S^{tC_p})\otimes H_*(\THH(MU))$ and
$f(\sigma m_\ell)=\epsilon(\sigma m_\ell)=1\otimes \sigma m_\ell$ for each
$\ell\geq 1$.  The last equality follows from
\eqref{eq:epsilon-def-p-even}--\eqref{eq:epsilon-def-p-odd} and
the fact that
$\sigma m_\ell\in H_*(\THH(MU))$ is $\sA_*$-comodule primitive.

Note that the domain and codomain of~$f$ differ as filtered objects
since the filtration of $E$ in the domain is given by the pullback
filtration along $\gamma^H_*$, while the injective image of~$E$ in
$r_+(H_*(\THH(MU)))$ has the Tate filtration given by
Definition~\ref{dfn:tate-filtration-rplus}.  Explicitly, for each
$\ell\geq 1$ and $\sigma m_\ell\in E$ we have
$\Fil \sigma m_\ell = -2\ell(p-1)$ by
Lemma~\ref{lemma:filtration-of-e}, which is strictly greater than
$\Fil f(\sigma m_\ell)=-(2\ell+1)(p-1)$.

Table~\ref{tab:f-on-generators} lists the values of~$f$ when applied
to a set of algebra generators of $r_+(H_*(MU))\otimes E$,
together with their Tate filtrations.

\begingroup
\setlength{\tabcolsep}{10pt} % Default value: 6pt
\renewcommand{\arraystretch}{1.5} % Default value: 1
\begin{table}[h!]
$$
  \begin{tabular}{|c|c|c|c|}
    \hline
    $z$ & $\Fil z$ & $f(z)$ & $\Fil f(z)$\\
    \hline
    \hline
    $u^it^r$ & $-i-2r$ & $u^it^r$ & $-i-2r$\\
    \hline
    $m_\ell$ & $-2\ell(p-1)$ & $m_\ell$ & $-2\ell(p-1)$\\
    \hline
    $\sigma m_\ell$ & $-2\ell(p-1)$ & $\sigma m_\ell$ & $-(2\ell+1)(p-1)$ \\
    \hline
  \end{tabular}
$$
  \caption{The homomorphism $f$ \label{tab:f-on-generators}}
\end{table}
\endgroup

\begin{lemma}\label{lemma:f}
  The homomorphism~$f$ of~\eqref{eq:f-and-g-mu} is
  filtration-preserving and an unfiltered isomorphism of right
  $\sA$-module algebras.  The completion $f^\wedge$ is an isomorphism
  of unfiltered graded $\bF_p$-vector spaces.
\end{lemma}
\begin{proof}
  For brevity, let $P_* = r_+(H_*(MU)) \otimes E$ and
  $Q_* = r_+(H_*(THH(MU)))$.  We already pointed out that $f$
  corresponds to the identity homomorphism when $P_*$ and $Q_*$ are
  identified with \eqref{eq:domain-and-codomain-of-f}.  Hence~$f$ is
  an isomorphism of unfiltered objects.

  However, since $P_*$ and $Q_*$ have different filtrations, it
  remains to check that~$f$ is both filtration-preserving and induces
  an unfiltered isomorphism after completion.

  We refer to Table~\ref{tab:f-on-generators} when making the
  following observations: For $z = u^i t^r$ we have
  $|z| = \Fil z = \Fil f(z) = -i-2r$, so
  $|z| - \Fil z = |f(z)| - \Fil f(z) = 0$.

  For $z = m_\ell$ we have $|z| = 2\ell$, $\Fil z = -2\ell(p-1)$ and
  $\Fil f(z) = -2\ell(p-1)$, so
  $|z| - \Fil z = |f(z)| - \Fil f(z) = 2 \ell p$.

  For $z = \sigma m_\ell$ we have $|z| = 2\ell+1$,
  $\Fil z = -2 \ell (p-1)$ and $\Fil f(z) = - (2\ell+1) (p-1)$, so
  $|z| - \Fil z = 2 \ell p + 1$ and
  $|f(z)| - \Fil f(z) = 2 \ell p + p$.

  Since $1 < (2 \ell p + p)/(2 \ell p + 1) < 3/2$ for $\ell\ge1$, we
  have
  \[
    |z| - \Fil z \le |f(z)| - \Fil f(z) \le (3/2)(|z| - \Fil z)
  \]
  in each of these cases.  It follows that these relations also hold
  for any monomial generator $z \in P_*$ and its image $f(z) \in Q_*$.

  In any fixed degree~$d$, it then follows that
  \[
    F_{\lfloor\frac{3n-d}{2}\rfloor} Q_d \subset F_n P_d \subset F_n Q_d
    \subset F_{\lceil\frac{2n+d}{3}\rceil} P_d
  \]
  for each $n \in \bZ$.  In particular, $f$ is filtration-preserving,
  and the induced map
  \[
    f^\wedge_d \: P^\wedge_d = \lim_n P_d/F_n P_d
    \overset{\cong}\longto \lim_n Q_d/F_n Q_d = Q^\wedge_d
  \]
  is an isomorphism for each~$d \in \bZ$.
\end{proof}

To describe $\omega_{MU}\otimes 1$, we recall from
Proposition~\ref{prop:singer-tate-filtration} that the homological
$C_p$-Tate spectral sequence converging to $H_*^c(R_+(MU))$ collapses
with
\[
  \label{eq:e-infty-rplus-mu}
  \hat E^\infty (R_+(MU)) = \hat E^2 (R_+(MU)) =
  \hat H^{-*}(C_p; \bF_p)\otimes P(m_\ell^{\otimes p} \mid \ell\geq
1) \,.
\]
By Proposition~\ref{prop:omega-B}, a unit multiple of the class
$\omega_{MU}(u^it^r\otimes m_\ell)$ is detected by
$u^it^{r+\ell(p-1)}\otimes m_\ell^{\otimes p}$ in Tate filtration
$-i-2r-2\ell(p-1)$.  In other words,
up to a unit $\omega_{MU}(u^it^r\otimes m_\ell) \in
\{u^it^{r+\ell(p-1)}\otimes m_\ell^{\otimes p} \}$.
Table~\ref{tab:omega-1-on-generators} summarizes the situation for
$\omega_{MU}\otimes 1$.

\begingroup
\setlength{\tabcolsep}{10pt} % Default value: 6pt
\renewcommand{\arraystretch}{1.5} % Default value: 1
\begin{table}[h!]
  $$
  \begin{tabular}{|c|c|c|c|}
    \hline
    $z$ & $\Fil z$ & $(\omega_{MU} \otimes 1) (z)$ & $\Fil (\omega_{MU}\otimes 1)(z)$\\
    \hline
    \hline
    $u^it^r$ & $-i-2r$ & $\{u^it^r\}$ & $-i-2r$\\
    \hline
    $m_\ell$ & $-2\ell(p-1)$ & $\{t^{\ell(p-1)}\otimes m_\ell^{\otimes p}\}$ & $-2\ell(p-1)$\\
    \hline
    $\sigma m_\ell$ & $-2\ell(p-1)$ & $\sigma m_\ell$ & $-2\ell(p-1)$ \\
    \hline
  \end{tabular}
  $$
  \caption{The homomorphism $\omega_{MU}\otimes 1$ \label{tab:omega-1-on-generators}}
\end{table}
\endgroup

For any spectrum $B$, the map
$\eta^t = \eta_p^{tC_p} \: R_+(B)\to \THH(B)^{tC_p}$ induces the
morphism of homological $C_p$-Tate spectral sequences that on
$\hat{E}^2$-terms sends the class $u^it^r\otimes x^{\otimes p}$ to
$u^it^r\otimes (\eta_p)_*(x^{\otimes p}) = u^it^r\otimes x^p$.  When
$B=MU$, we deduce from the algebra structure of the
$\hat E^\infty$-term displayed in~\eqref{eq:e-infty-cp-tate-mu} that
$u^it^r\otimes x^p$ survives to the $\hat E^\infty$-term, for each
non-zero class $x\in H_*(MU)$.  The behavior of~$g$ is summarized in
Table~\ref{tab:g-on-generators}.

\begingroup
\setlength{\tabcolsep}{10pt} % Default value: 6pt
\renewcommand{\arraystretch}{1.5} % Default value: 1
\begin{table}[h!]
  $$
  \begin{tabular}{|c|c|c|c|}
    \hline
    $z$ & $\Fil z$ & $g (z)$ & $\Fil g(z)$\\
    \hline
    \hline
    $\{u^it^r\otimes x^{\otimes p}\}$ & $-i-2r$ & $\{u^it^r\otimes x^p\}$ & $-i-2r$\\
    \hline
    $\sigma m_\ell$ & $-2\ell(p-1)$ & $\{t^{\ell(p-1)}\otimes m_\ell^{p-1}\sigma m_\ell\}$ & $-2\ell(p-1)$ \\
    \hline
  \end{tabular}
  $$
  \caption{The homomorphism $g$ \label{tab:g-on-generators}}
\end{table}
\endgroup

\begin{lemma}\label{lemma:omega-and-g}
  Both~$\omega_{MU}\otimes 1$ and~$g$ of diagram~\eqref{eq:f-and-g-mu}
  are strictly filtration-preserving homomorphisms that induce
  isomorphisms of filtration quotients.  In particular, they both
  become isomorphisms in $\rbbfilamodc$ after passing to completions.
\end{lemma}
\begin{proof}
  A morphism $h\: P_*\to Q_*$ of rbb filtered graded $\bF_p$-vector
  spaces that induces isomorphisms
  $F_nP_*/F_{n-1}P_*\cong F_nQ_*/F_{n-1}Q_*$ of filtration quotients
  is in particular strictly filtration-preserving.  Furthermore, the
  completion $h^\wedge \: P_*^\wedge\to Q_*^\wedge$ is an isomorphism
  in $\rbbfilgrRmod$ by Lemma~\ref{lemma:thin-layers}.

  By Proposition~\ref{prop:omega-B}, the homomorphism $\omega_{MU}$
  induces isomorphisms of filtration quotients.  The lemma then
  follows for the filtered homomorphism $\omega_{MU}\otimes 1$ by an
  application of Lemma~\ref{lemma:convolution-filtration-isos} and the
  observation above.

  The homological $C_p$-Tate spectral sequences for
  $R_+(MU)$ and $\THH(MU)^{tC_p}$ both converge strongly, and we
  repeat their $\hat E^\infty$-terms here for convenience:
  \begin{align*}
    \hat E^\infty(R_+(MU))
    &= \hat H^{-*}(C_p; \bF_p) \otimes
      P(m_\ell^{\otimes p}\mid \ell\geq 1)\\
    \hat E^\infty(\THH(MU))
    &= \hat H^{-*}(C_p; \bF_p) \otimes
      P(m_\ell^p\mid \ell\geq 1) \otimes E(m_\ell^{p-1}\sigma m_\ell\mid \ell\geq 1)\,.
  \end{align*}
  It follows from strong convergence that the associated gradeds of
  $H_*^c(R_+(MU))$ and $H_*^c(\THH(MU))$ are given by these
  $\hat E^\infty$-terms.  From the algebra structure together
  with the explicit description of $g$ given in
  Table~\ref{tab:g-on-generators}, it is then clear that $g$ is
  strictly filtration-preserving and that it induces an isomorphism of
  associated graded $\bF_p$-vector spaces.  This proves the lemma for~$g$.
\end{proof}

Lemma~\ref{lemma:f} and Lemma~\ref{lemma:omega-and-g} imply that
\eqref{eq:f-and-g-mu} is a diagram in the category of rbb filtered
right $\sA$-module algebras.  In particular, it makes sense to
consider its completion.
\begin{proposition}\label{prop:theta}
  There is a commutative diagram
  \begin{equation}
    \begin{aligned}
      \label{eq:f-and-g-mu-completed}
      \xymatrix@C11mm{
      R_+ (H_*(MU)) \ctensor E
      \ar[d]^-{f^\wedge}
      \ar[r]^-{\omega_{MU}\ctensor 1}_-\cong
      & H_*^c (R_+(MU))\ctensor E
        \ar[d]^-{g^\wedge}_-\cong \\
      R_+(H_*(\THH(MU)))
      & H_*^c \left(\THH(MU)^{tC_p}\right)
        \ar[l]_-{\Theta_{MU}}
        }
    \end{aligned}
  \end{equation}
  in the category $\alg(\sC^\wedge, \ctensor)$ of rbb complete right
  $\sA$-module algebras.

  Moreover, each morphism in \eqref{eq:f-and-g-mu-completed} is an
  isomorphism of unfiltered graded $\bF_p$-vector spaces and, in
  particular, $\Theta_{MU}$ induces an isomorphism of continuous
  $\Ext$-algebras
  $$
  \cExt^*_{\sA_*}(\bF_p, H_*(\THH(MU)^{tC_p}))
  \cong
  \cExt^*_{\sA_*}(\bF_p, R_+(H_*(\THH(MU))))\,.
  $$
\end{proposition}
\begin{proof}
  It follows from Lemma~\ref{lemma:f} and
  Lemma~\ref{lemma:omega-and-g} that
  diagram~\eqref{eq:f-and-g-mu-completed} exists as a diagram in the
  category $\alg(\sC^\wedge, \ctensor)$, and that all morphisms
  are unfiltered isomorphisms of graded  $\bF_p$-vector spaces.
  Proposition~\ref{prop:ext-iso} then implies that $\Theta_{MU}$ induces an
  isomorphism of continuous $\Ext$-algebras.
\end{proof}

\subsection{The differential graded structure of
  $H_*^c(\THH(MU)^{tC_p})$} \label{sec:thh-mu-dga}

Recall from the discussion in Section~\ref{sec:residual-circle-action}
that $\THH(MU)^{tC_p}$ carries a residual circle action by $\bar\bT = \bT/C_p$,
inducing a differential
$$
\bar\sigma\: \leftsuspension H_*^c(\THH(MU)^{tC_p})
\longto \sh_1 H_*^c (\THH(MU)^{tC_p}) \,.
$$
Proposition~\ref{prop:tate-dga} applied to $X=\THH(MU)$
states that $(H_*^c(\THH(MU)^{tC_p}), \bar\sigma)$ is an rbb filtered
differential graded right $\sA$-module algebra.

\begin{lemma}[\cite{LNR11}*{Thm.~6.4}]
  \label{lemma:thm-6.4}
  The differential $\bar\sigma$ maps
  $\eta^t_*(\omega_{MU}(1\otimes m_\ell))$ to
  $\gamma^H_*(\sigma m_\ell)$, for each $\ell\geq 1$.
\end{lemma}
\begin{proof}
  Consider the following commutative diagram
  \[
    \xymatrix{
      & H_*(MU) \ar[r]^-{\eta_*} \ar[dd]^-{(\epsilon^H_{MU})_*} \ar[dl]_-{\epsilon}
      & H_*(THH(MU)) \ar[dd]^-{\gamma^H_*} \\
      r_+(H_*(MU)) \ar[dr]_-{\omega_{MU}} \\
      & H^c_*(R_+(MU)) \ar[r]^-{\eta^t_*} & H^c_*(THH(MU)^{tC_p})
      \rlap{\,.}
    }
  \]
  The left-hand triangle commutes by
  Proposition~\ref{prop:epsilon-is-epsilon}, and the right-hand square
  commutes by Theorem~\ref{thm:0-1-skeleton-approx}.  The cyclotomic
  structure map $\gamma$ is equivariant with respect to
  $\rho \: \bT \to \bar\bT$.  Thus the differential $\bar\sigma$ maps
  $\gamma^H_*(\eta_*(x)) = \eta^t_*(\omega_{MU}(\epsilon(x)))$ to
  $\gamma^H_*(\sigma(\eta_*(x)))$ for every class $x \in H_*(MU)$.

  We now specialize to $x = m_\ell$, noting that
  $\sigma(\eta_*(m_\ell)) = \sigma m_\ell \in H_*(THH(MU))$.  When
  $\ell \ne p^k-1$, the class $m_\ell$ is $\sA_*$-comodule primitive,
  so $\epsilon(m_\ell) = 1 \otimes m_\ell \in r_+(H_*(MU))$, and the
  lemma follows from the previous paragraph.

  Finally, when $\ell = p^k-1$ and $x = m_\ell = \xibar_k$,
  \if \anonym
    a calculation
  \else
    we calculated
  \fi
  in the proof of~\cite{LNR11}*{Thm.~6.4}
  \if \anonym
    shows
  \fi
  that
  $\epsilon(\xibar_k) = 1 \otimes \xibar_k + t^{-(p-1)} \cdot
  \epsilon(\xibar_{k-1}^p)$.  It follows that
  \[
    \eta^t_*(\omega_{MU}(1 \otimes \xibar_k))
	= \eta^t_*(\omega_{MU}(\epsilon(\xibar_k))) - t^{-(p-1)} \cdot \eta^t_*(\omega_{MU}(\epsilon(\xibar_{k-1}^p)))
  \]
  is mapped by $\bar\sigma$ to $\gamma^H_*(\sigma(\xibar_k)) -
  t^{-(p-1)} \cdot \gamma^H_*(\sigma(\xibar_{k-1}^p))$.  Here we
  use that $\eta^t_* \circ \omega_{MU}$ is $H^c_*(S^{tC_p})$-linear,
  and that $\bar\sigma(t) = 0$.  The lemma then follows by noting that
  $\sigma(\xibar_{k-1}^p) = 0$, since $\sigma$ is a derivation.
\end{proof}

\begin{proposition}\label{prop:theta-is-dga-morphism}
  The morphism
  \[
    \Theta_{MU}\: H_*^c(\THH(MU)^{tC_p})
    \longto
    R_+(H_*(\THH(MU))
  \]
  commutes with the differential $\bar\sigma$ and is therefore
  a morphism in $\dgalg(\sC^\wedge, \ctensor)$, of rbb complete
  differential graded right $\sA$-module algebras.
\end{proposition}

\begin{proof}
  Note that even before completion, the morphism $f$ of diagram
  \eqref{eq:f-and-g-mu} is an isomorphism of unfiltered graded $\bF_p$-vector
  spaces.  We claim that the filtration-increasing, algebra morphism
  $$
  \Phi := g\circ (\omega_{MU}\otimes 1)\circ f^{-1}\:
  r_+(H_*(\THH(MU))) \longto H_*^c(\THH(MU)^{tC_p})
  $$
  is a map of differential graded algebras.  Here the domain has the
  algebraic differential~$\bar\sigma$ from
  Section~\ref{sec:residual-differential-on-rplus}, while the codomain
  has the topological differential~$\bar\sigma$ from
  Section~\ref{sec:residual-circle-action}.

  Since~$f$ is an isomorphism of $\bF_p$-algebras, there is a unique
  differential graded structure on $r_+(H_*(MU))\otimes E$ that
  makes~$f$ a morphism in the category $\dgalg(\sC, \otimes)$ of rbb
  filtered differential graded right $\sA$-module algebras.  With this
  structure, the claim implies that~$g\circ(\omega_{MU}\otimes 1)$ is
  a morphism in $\dgalg(\sC,\otimes)$ and that
  $\Theta_{MU} = f^\wedge \circ (g^\wedge \circ (\omega_{MU}\ctensor
  1))^{-1}$ is a morphism in $\dgalg(\sC^\wedge, \ctensor)$.

  Since $\Phi$ is a morphism of $\bF_p$-algebras, it suffices to show
  that it commutes with~$\bar\sigma$ when applied to a set of algebra
  generators.  From the definition of $\Theta_{MU}$,~$f$ and~$g$, we get
  \begin{align}
    \label{eq:mu-classes}
    \Phi(u^it^r\otimes x) &= \eta^t_*(\omega_{MU}(u^it^{r}\otimes x)) \\
    \label{eq:thh-mu-classes}
    \Phi(1\otimes \sigma m_\ell) &= \gamma^H_*(\sigma m_\ell) \,,
  \end{align}
  for each $x\in H_*(MU)$ and $\ell \geq 1$.

  We have that
  $ \Phi(\bar\sigma(u\otimes 1)) = \Phi(1\otimes 1) = 1\,.  $
  Furthermore,
  $ \bar\sigma(\Phi(u\otimes 1)) = \bar\sigma(u\cdot 1) = 1\,, $ where
  the last equality follows from naturality with respect to $S\to MU$
  and Proposition \ref{prop:continuous-homology-of-s-tcp}.  This
  implies that
  $\Phi(\bar\sigma(u\otimes 1)) = 1 = \bar\sigma(\Phi(u\otimes 1))$.
  Similarly, $\Phi(\bar\sigma(t^r \otimes 1)) = \Phi(0) = 0$, while
  $\bar\sigma(\Phi(t^r \otimes 1)) = \bar\sigma(t^r \cdot 1) = 0$,
  again by naturality along $S \to MU$.  Hence
  $\Phi(\bar\sigma(t^r \otimes 1)) = 0 = \bar\sigma(\Phi(t^r \otimes
  1))$.

  $$
  \xymatrix{
    r_+(H_*(S)) \ar[r]^-{\omega_S} \ar[d]
	& H^c_*(R_+(S)) \ar[r]^-{\eta^t_*} \ar[d]
	& H^c_*(THH(S)^{tC_p}) \ar[d] \\
    r_+(H_*(MU)) \ar[r]^-{\omega_{MU}}
	& H^c_*(R_+(MU)) \ar[r]^-{\eta^t_*}
	& H^c_*(THH(MU)^{tC_p})
  }
  $$

  It follows from \eqref{eq:mu-classes} and Lemma~\ref{lemma:thm-6.4}
  that $\Phi(1\otimes m_\ell)$ is sent by $\bar\sigma$ to
  $\gamma^H_*(\sigma m_\ell)\in H_*^c(\THH(MU)^{tC_p})$.  From
  \eqref{eq:thh-mu-classes}, we then conclude that
  $\Phi(\bar\sigma(1\otimes m_\ell)) =
  \gamma^H_*(\sigma m_\ell) =
  \bar\sigma(\Phi(1\otimes
  m_\ell))$.

  For the remaining cases, we use that $\bar\sigma$ is a differential,
  in both the domain and codomain of $\Phi$, to compute that
  $\Phi(\bar\sigma(1\otimes \sigma m_\ell)) = 0 =
  \bar\sigma(\Phi(1\otimes \sigma m_\ell))$.
\end{proof}

A consequence of \cite{BBLNR14}*{Prop.~3.8} or
\cite{NS18}*{Lem.~II.4.2} is that if $X$ is a bounded below spectrum
with an action of the circle group $\bT$, then the canonical map
$$
G\: X^{t\bT}\cong (X^{tC_p})^{\bar\bT} \longto (X^{tC_p})^{h\bar\bT}
$$
can be identified with $p$-adic completion.

Let $F \: X^{t\bT} \to X^{tC_p}$
denote the restriction map of Tate constructions associated to $C_p
\subset \bT$.  It agrees with the composite of~$G$ with the forgetful
map $(X^{tC_p})^{h\bar\bT} \to X^{tC_p}$, i.e., the restriction map
associated to $\{e\} \subset \bT$.

These considerations also apply when
replacing $X$ by $H\wedge X$, and thus we can compute $H_*^c(X^{t\bT})$
by the $\bar\bT$-homotopy fixed point spectral sequence
$$
E^2 = H^{*}(\bar\bT; H_*^c(X^{tC_p})) \implies
H_*^c(X^{t\bT})\,.
$$
This is a left half-plane spectral sequence with entering
differentials $d^2(t^r\otimes x) = t^{r+1}\otimes \bar\sigma(x)$,
where $\bar\sigma$ is induced by the residual $\bar\bT$-action on
$H_*^c(X^{tC_p})$.  Its edge homomorphism $F_* \: H^c_*(X^{t\bT}) \to
H^c_*(X^{tC_p})$ agrees with the homomorphism induced
by~$F$ for $H \wedge X$.

In the case of $X=\THH(MU)$, Lemma~\ref{lemma:sigmaone-exact} and
Proposition~\ref{prop:theta-is-dga-morphism} imply that the
$\bar\bT$-homotopy fixed point spectral sequence collapses at the
$E^3$-term, with
$$
E^\infty_{s,*} = E^3_{s,*} =
\begin{cases}
  \ker(\bar\sigma) & \text{for $s=0$}\\
  0 & \text{for $s<0$}\,,
\end{cases}
$$
and that the edge homomorphism
$$
F_*\: H_*^c(\THH(MU)^{t\bT}) \longto H_*^c(\THH(MU)^{tC_p})
$$
is an injective morphism of rbb filtered right
$\sA$-module algebras, with image equal to $\ker(\bar\sigma)$.

Recall the injective
homomorphism~$F_+ \: C_+(M_*; \sigma) \to R_+(M_*)$ from
Definition~\ref{dfn:homological-C-singer}.

\begin{theorem} \label{thm:finale-MU}
  There is a commutative square of rbb complete filtered right
  $\sA$-module algebras (or rbb complete filtered left
  $\sA_*$-comodule algebras)
  \begin{equation*}
    \begin{aligned}
      \xymatrix@C16mm{
      H_*^c(\THH(MU)^{t\bT})
      \ar[r]^-{\Theta_{MU}^\bT}
      \ar@{ >->}[d]^{F_*}
      &
        C_+(H_*(\THH(MU)); \sigma)
        \ar@{ >->}[d]^{F_+}\\
      H_*^c(\THH(MU)^{tC_p})
      \ar[r]^-{\Theta_{MU}}
      &R_+(H_*(\THH(MU)))\,.
        }
    \end{aligned}
  \end{equation*}
  Both horizontal morphisms are isomorphisms of unfiltered
  graded $\bF_p$-vector spaces and induce $\Ext$-isomorphisms.
\end{theorem}
\begin{proof}
  We proved in Proposition~\ref{prop:theta} that $\Theta_{MU}$ is a
  morphism in the category $\alg(\sC^\wedge, \ctensor)$ of rbb
  complete right $\sA$-module algebras, and an isomorphism of
  unfiltered graded $\bF_p$-vector spaces.  By
  Proposition~\ref{prop:theta-is-dga-morphism} we can restrict
  $\Theta_{MU}$ to the kernel of the differential in both its domain
  and codomain.  The theorem then follows by the identification of
  $$
  \ker \bar\sigma \subset H_*^c (\THH(MU)^{tC_p})
  $$
  with the image of $F_*$ and the identification of
  $$
  \ker \bar\sigma \subset R_+(H_*(\THH(MU)))
  $$
  with $C_+(H_*(\THH(MU)); \sigma)$.
\end{proof}

Theorem~\ref{thm:finale-MU} also holds after replacing~$MU$ by the
$p$-local Brown--Peterson spectrum~$BP$, which is known to admit an
$E_4$ ring spectrum structure by Basterra--Mandell~\cite{BM13}.

\begin{theorem} \label{thm:finale-BP}
  There is a commutative square of rbb complete filtered right
  $\sA$-module algebras (or rbb complete filtered left
  $\sA_*$-comodule algebras)
  \begin{equation*}
    \begin{aligned}
      \xymatrix@C16mm{
      H_*^c(\THH(BP)^{t\bT})
      \ar[r]^-{\Theta_{BP}^\bT}
      \ar@{ >->}[d]^{F_*}
      &
        C_+(H_*(\THH(BP)); \sigma)
        \ar@{ >->}[d]^{F_+}\\
      H_*^c(\THH(BP)^{tC_p})
      \ar[r]^-{\Theta_{BP}}
      &R_+(H_*(\THH(BP)))\,.
        }
    \end{aligned}
  \end{equation*}
  Both horizontal morphisms are isomorphisms of unfiltered graded
  $\bF_p$-vector spaces and induce $\Ext$-isomorphisms.
\end{theorem}
\begin{proof}
  Recall the isomorphisms
  \[
    H_*(BP)\cong
    \begin{cases}
      P(\xibar_k\mid k\geq 1) \text{\ \ for $p>2$, and} \\
      P(\bar\zeta^2_k\mid k\geq 1) \text{\ \ for $p=2$}
    \end{cases}
  \]
  of right $\sA$-module algebras, identifying $H_*(BP)$ as a sub left
  $\sA_*$-comodule algebra of the dual Steenrod algebra.

  The proof is similar to the~$MU$ case.  Essentially, we must repeat
  the arguments made for $MU$ after replacing every class
  $m_{p^k-1}\in H_*(MU)$ with $\xibar_k\in H_*(BP)$ for $p>2$ and by
  $\bar\zeta_k^2$ for $p=2$, and ignoring classes $m_\ell$ where
  $\ell\neq p^k-1$.  Cf.~\cite{LNR11}*{Sec.~6.2}.
\end{proof}

\begin{corollary}
  For each of $B = S$, $MU$ and~$BP$ there is a multiplicative and
  strongly convergent limit Adams spectral sequence
  \begin{align*}
    E_2^{s,t}
	&= \Ext_{\sA}^{s,t}(\bF_p, C_+(H_*(\THH(B)); \sigma)) \\
	&\cong \cExt_{\sA_*}^{s,t}(\bF_p, C_+(H_*(\THH(B)); \sigma)) \\
	&\Longrightarrow \pi_{t-s}(TP(B)^\wedge_p) \,.
  \end{align*}
\end{corollary}

\begin{proof}
  This follows by combining
  Propositions~\ref{prop:inverse-limit-of-adams-ss},
  \ref{prop:left-right-ext} and~\ref{prop:ext-cup} with
  Theorems~\ref{thm:finale-MU} and~\ref{thm:finale-BP}, together with
  the standard calculation
  $$
  H^c_*(THH(S)^{t\bT}) \cong P(t^{\pm1}) \cong C_+(H_*(THH(S)); 0) \,.
  $$
\end{proof}

\subsection{A Ravenel type short exact sequence}

In the special case $M_* = \bF_p$, when
$C_+(\bF_p; 0) \cong H^c_*(S^{t\bT}) \cong H^c_*(\Sigma^2 \bC
P^\infty_{-\infty})$, the cohomological version of the following short
exact sequence appeared in~\cite{Rav84}*{2.3}.  This relation to
$\bC P^\infty$, together with the role of the circle group, motivated
our choice of the letter~$C$ in the notation $C_+(M_*; \sigma)$.
\begin{lemma}
  Let $(M_*, \sigma)$ be a differential graded right $\sA$-module, and assume
  that $\beta_*$ acts trivially.  There is a natural short exact sequence of
  filtered right $\sA$-modules
  \begin{equation}
    \label{eq:f-plus-ses}
    0\longto C_+(M_*;\sigma) \overset{F_+}\longto
    R_+(M_*)
    \overset{T_+}\longto S^{-1}\sh_1 C_+(M_*;\sigma)\longto 0\,. 
  \end{equation}
\end{lemma}
\begin{proof}
  Consider the sequence
  \begin{equation} \label{eq:f+t+ses}
    0 \longto c_+(M_*; \sigma) \overset{f_+}\longto
	r_+(M_*) \overset{t_+}\longto S^{-1} \sh_{1} c_+(M_*; \sigma)
	\longto 0
  \end{equation}
  with $f_+$ the strictly filtration-preserving right $\sA$-linear
  injection defined in~\eqref{eq:f+odd}--\eqref{eq:f+two}, and $t_+$
  the surjection given by $t_+(ut^r \otimes x) = S^{-1} t^r \otimes x$
  and $t_+(t^r \otimes x) = S^{-1} t^r \otimes \sigma(x)$.  A
  comparison of~\eqref{eq:steenrod-singer-tr}--\eqref{eq:steenrod-singer-squares}
  with~\eqref{eq:Ps-C-plus}--\eqref{eq:beta-C-plus} shows that $t_+$ is
  right $\sA$-linear, using in particular that
  $P_* \sigma = \sigma P_*$ and $\beta_* S^{-1} = - S^{-1} \beta_*$.
  When $M_* = \bF_p\{x\}$ or $M_* = \bF_p\{x, \sigma(x)\}$ it is
  elementary to check that the sequence is short exact, and the
  general case follows from this.  Moreover, the downward filtration shift
  by~$1$ ensures that~\eqref{eq:f+t+ses} is in fact a short exact
  sequence of filtered right $\sA$-modules.  By passing to completions
  we obtain the asserted sequence~\eqref{eq:f-plus-ses}, with $T_+ := t_+^\wedge$.

  Since completion with respect to a filtration is given by the limit
  of a tower of surjections, the induced sequence of completions is
  also short exact.
\end{proof}

We write $H_*(M_*, \sigma)$ for the homology of $M_*$ with respect to
the differential $\sigma\: S M_* \to M_*$.  Furthermore, let
$H_*(M_*, \beta_*) = \ker \beta_* / \im \beta_*$ be the Margolis
homology of $M_*$ with respect to the Bockstein.
\begin{lemma}
  The short exact sequence \eqref{eq:f-plus-ses} splits as
  right $\sA$-modules only if\\ $H_*(M_*,\sigma)=0$.
\end{lemma}
\begin{proof}
  By an application of Lemma~9.4, the Margolis homology of
  $C_+(M_*;\sigma)$ with respect to $\beta_*$ is isomorphic to
  $C_+(H_*(M_*, \sigma); 0)$, which is trivial if and only if
  $H_*(M_*, \sigma)$ is trivial.

  A splitting of \eqref{eq:f-plus-ses} as right $\sA$-modules would
  imply that $H_*(C_+(M_*;\sigma), \beta_*)$ splits off
  $H_*(R_+(M_*), \beta_*)$.  However, the latter is trivial since the
  kernel and image of $\beta_*\: R_+(M_*)\to R_+(M_*)$ are both equal
  to the completion of the vector subspace spanned by the elements of
  the form $t^r \otimes x$, again by Lemma~9.4.
\end{proof}

\end{document}